\documentclass[11pt,leqno]{amsart}
\usepackage{hyperref}
\usepackage{epsfig}
\usepackage{enumitem}
\usepackage{amssymb}
\usepackage{amscd}
\usepackage{tikz-cd}
\usepackage[all,cmtip,matrix,arrow]{xy}
\usepackage{mathrsfs}
\usepackage{graphicx}
\usepackage{mathabx}
\usepackage{array}
\usepackage{longtable}
\xyoption{rotate}
\xyoption{pdf}
\usepackage{xparse}

\newcommand{\Implies}[2]{$\text{\ref{#1}}\implies\text{\ref{#2}}$}
\newcommand{\Iff}[2]{$\text{\ref{#1}}\iff\text{\ref{#2}}$}

\makeatletter
\DeclareRobustCommand*{\bfseries}{%
  \not@math@alphabet\bfseries\mathbf
  \fontseries\bfdefault\selectfont
  \boldmath
}
\makeatother

\AtBeginEnvironment{tikzcd}{\setlength{\mathsurround}{0pt}}

\newtheorem{theo}{Theorem}[section]
\newtheorem{lemma}[theo]{Lemma}
\newtheorem{defi}[theo]{Definition}
\newtheorem{prop}[theo]{Proposition}

\newtheorem{cor}[theo]{Corollary}
\newtheorem{remark}[theo]{Remark}

\newtheorem{example}[theo]{Example}
\newtheorem{ques}[theo]{Question}
\numberwithin{equation}{section}

\newcommand{\laxtimes}[1]{\mathop{\times\mkern-13mu\raise1.3ex\hbox{$\scriptscriptstyle\to$}_{#1}}}

\mathchardef\mhyphen="2D

\def\A{{\mathbb A}}

\def\N{\mathbb{N}}
\def\bL{\mathbb{L}}

\def\R{\mathbb{R}}
\def\bS{\mathbb{S}}

\def\Z{\mathbb{Z}}
\def\Q{\mathbb{Q}}
\def\bG{\mathbb{G}}

\def\wt{\widetilde}

\def\D{{\mathcal{D}}}

\def\bR{{\mathbf R}}
\def\bL{{\mathbf L}}

\def\FF{{\mathbb F}}

\def\pre-tr{\operatorname{pre-tr}}
\def\h{\operatorname{h}}

\def\Hom{\operatorname{Hom}}
\def\Map{\operatorname{Map}}
\def\End{\operatorname{End}}

\DeclareMathOperator*{\colim}{colim}

\DeclareMathOperator*{\hocolim}{hocolim}

\newcommand{\tens}[1]{%
  \mathbin{\mathop{\otimes}\displaylimits_{#1}}%
}

\newcommand{\ms}[1]{\mathscr{#1}}
\newcommand{\msK}{\ms{K}}

\newcommand{\hy}{\mhyphen}

\newcommand{\indlim}[1][]{\mathop{\varinjlim}\limits_{#1}}
\newcommand{\inddlim}[1][]{{``{\indlim[#1]}"}}
\newcommand{\prolim}[1][]{\mathop{\varprojlim}\limits_{#1}}
\newcommand{\proolim}[1][]{{``{\prolim[#1]}"}}
\newcommand{\biggplus}[1][]{\mathop{\bigoplus}\limits_{#1}}

\newcommand{\prodd}[1][]{\mathop{\prod}\limits_{#1}}

\newcommand{\limproj}{\varprojlim}
\newcommand{\liminj}{\varinjlim}

\newcommand{\bbar}{\overline}

\newcommand{\toto}{\rightrightarrows}
\newcommand{\xto}{\xrightarrow}
\newcommand{\xlto}{\xleftarrow}
\newcommand{\hto}{\hookrightarrow}
\newcommand{\leftto}{\leftarrow}

\newcommand{\fin}{\operatorname{fin}}
\newcommand{\Fin}{\operatorname{Fin}}
\newcommand{\profin}{\operatorname{ProFin}}
\newcommand{\comphaus}{\operatorname{CompHaus}}
\newcommand{\Top}{\operatorname{Top}}
\newcommand{\Tate}{\operatorname{Tate}}
\newcommand{\Mot}{\operatorname{Mot}}
\newcommand{\Corr}{\operatorname{Corr}}

\newcommand{\Nuc}{\operatorname{Nuc}}

\newcommand{\Flat}{\operatorname{Flat}}
\newcommand{\fla}{\operatorname{flat}}
\newcommand{\Cat}{\operatorname{Cat}}
\newcommand{\Ring}{\operatorname{Ring}}
\newcommand{\Cone}{\operatorname{Cone}}
\newcommand{\Fiber}{\operatorname{Fiber}}

\newcommand{\el}{\operatorname{el}}

\newcommand{\loc}{\operatorname{loc}}
\newcommand{\cont}{\operatorname{cont}}
\newcommand{\lex}{\operatorname{lex}}

\newcommand{\red}{\operatorname{red}}
\newcommand{\acc}{\operatorname{acc}}
\newcommand{\bnd}{\operatorname{bnd}}
\newcommand{\ex}{\operatorname{ex}}
\newcommand{\strcont}{\operatorname{strcont}}

\newcommand{\st}{\operatorname{st}}
\newcommand{\cg}{\operatorname{cg}}

\newcommand{\dual}{\operatorname{dual}}

\newcommand{\bE}{{\mathbb E}}
\newcommand{\bbF}{{\mathbb F}}

\newcommand{\mk}{\mathrm k}
\newcommand{\ev}{\mathrm ev}

\newcommand{\cJ}{{\mathcal J}}

\newcommand{\cF}{{\mathcal F}}
\newcommand{\cG}{{\mathcal G}}
\newcommand{\cO}{{\mathcal O}}

\newcommand{\cM}{{\mathcal M}}

\newcommand{\cD}{{\mathcal D}}

\newcommand{\cA}{{\mathcal A}}
\newcommand{\cB}{{\mathcal B}}
\newcommand{\cI}{{\mathcal I}}
\newcommand{\cC}{{\mathcal C}}
\newcommand{\cE}{{\mathcal E}}

\newcommand{\cY}{{\mathcal Y}}
\newcommand{\cU}{{\mathcal U}}

\newcommand{\cS}{{\mathcal S}}
\newcommand{\cT}{{\mathcal T}}

\newcommand{\cH}{{\mathcal H}}

\newcommand{\veps}{\varepsilon}

\newcommand{\un}{\underline}
\newcommand{\la}{\langle}
\newcommand{\ra}{\rangle}

\newcommand{\Compass}{\operatorname{CompAss}}

\newcommand{\seminorm}{\operatorname{SemiNorm}}
\newcommand{\Ban}{\operatorname{Ban}}
\newcommand{\Coeq}{\operatorname{Coeq}}
\newcommand{\Eq}{\operatorname{Eq}}

\newcommand{\THH}{\operatorname{THH}}

\newcommand{\Fun}{\operatorname{Fun}}
\newcommand{\Tot}{\operatorname{Tot}}

\newcommand{\Ab}{\operatorname{Ab}}

\newcommand{\Perf}{\operatorname{Perf}}

\newcommand{\perf}{\operatorname{perf}}

\newcommand{\Kar}{\operatorname{Kar}}
\newcommand{\Kos}{\operatorname{Kos}}

\newcommand{\PSh}{\operatorname{PSh}}
\newcommand{\copsh}{\operatorname{coPSh}}
\newcommand{\Cosh}{\operatorname{Cosh}}
\newcommand{\Sh}{\operatorname{Shv}}
\newcommand{\Shv}{\operatorname{Shv}}
\newcommand{\supp}{\operatorname{Supp}}

\newcommand{\coker}{\operatorname{coker}}

\newcommand{\im}{\operatorname{Im}}

\newcommand{\Ext}{\operatorname{Ext}}
\newcommand{\Tor}{\operatorname{Tor}}

\newcommand{\Id}{\operatorname{Id}}

\newcommand{\Sp}{\operatorname{Sp}}

\newcommand{\add}{\operatorname{add}}

\newcommand{\res}{\operatorname{res}}
\newcommand{\cores}{\operatorname{cores}}

\newcommand{\Sym}{\operatorname{Sym}}

\newcommand{\Ind}{\operatorname{Ind}}
\newcommand{\Pro}{\operatorname{Pro}}
\newcommand{\Calk}{\operatorname{Calk}}

\newcommand{\Spec}{\operatorname{Spec}}

\newcommand{\Stab}{\rm Stab}

\newcommand{\Alg}{\operatorname{Alg}}

\newcommand{\id}{\operatorname{id}}

\newcommand{\coev}{\operatorname{coev}}

\newcommand{\tors}{\operatorname{tors}}

\newcommand{\acycl}{\operatorname{acycl}}

\newcommand{\Set}{\operatorname{Set}}

\newcommand{\pt}{\operatorname{pt}}

\newcommand{\Mod}{\operatorname{Mod}}
\newcommand{\clopen}{\operatorname{ClOpen}}
\newcommand{\Open}{\operatorname{Open}}
\newcommand{\mult}{\operatorname{mult}}

\usepackage{amscd}

\newcommand{\mfp}{\mathfrak{p}}
\newcommand{\mPr}{\mathfrak{Pr}}
\newcommand{\mfq}{\mathfrak{q}}

\newcommand{\m}{\mathfrak{m}}

\title[K-theory and localizing invariants of large categories]
{K-theory and localizing invariants of large categories}

\author{Alexander I. Efimov}
\address{Steklov Mathematical Institute of RAS}
\email{efimov@mccme.ru}

\begin{document}

\begin{abstract} In this paper we introduce and study the so-called continuous $K$-theory for a certain class of ``large'' stable $\infty$-categories, more precisely, for dualizable presentable categories. For compactly generated categories, the continuous $K$-theory is simply the usual (non-connective) $K$-theory of the full subcategory of compact objects. More generally, we show that any localizing invariant of small stable $\infty$-categories can be uniquely extended to a localizing invariant of dualizable categories.

We compute the continuous $K$-theory for categories of sheaves on locally compact Hausdorff spaces. Using the special case for sheaves on the real line, we give an alternative proof of the theorem of Kasprowski and Winges \cite{KW19} on the commutation of $K$-theory with infinite products for small stable $\infty$-categories.

We also study the general theory of dualizable categories. In particular, we give an ``explicit'' proof of Ramzi's theorem \cite{Ram24a} on the $\omega_1$-presentability of the category of dualizable categories. Among other things, we prove that dualizability is equivalent to ``flatness'' in the category of presentable stable categories.
\end{abstract}


\maketitle

\tableofcontents

\section{Introduction}

It is well known that the Grothendieck group $K_0(\cA)$ of an additive category vanishes whenever $\cA$ has countable direct sums. More generally, for any stable $\infty$-category $\cC$ with countable coproducts its $K$-theory spectrum is contractible. This vanishing also holds for any additive invariant, like (topological) Hochschild homology. Therefore, naturally, additive invariants are normally studied only for small stable infinity categories, and not for ``large'' categories. 

In this paper we introduce and study the notion of {\it continuous} non-connective $K$-theory (and more general localizing invariants) for a certain class of ``nice'' large categories. These are the so-called dualizable (also known as compactly assembled) presentable stable categories. Moreover, this continuous $K$-theory in fact extends the usual non-connective $K$-theory for small stable $\infty$-categories.

We now explain more precisely the basic idea. Let us denote by $\Cat^{\perf}$ the $\infty$-category of small idempotent-complete stable $\infty$-categories and exact functors between them. We denote by $\Sp$ the $\infty$-category of spectra. We consider the non-connective $K$-theory as a functor $K:\Cat^{\perf}\to\Sp.$ We refer to \cite{BGT, Bar} for the definition of (non-connective) algebraic $K$-theory of a small stable $\infty$-category, based on the Waldhausen's $S_{\bullet}$-construction \cite{Wal}. If $R$ is a usual (discrete) associative unital ring, then the $K$-theory space $K_{\geq 0}(\Perf(R))$ of the stable $\infty$-category of perfect complexes $\Perf(R)$ is equivalent to the Quillen's $K$-theory space of $R$ as defined in \cite{Qui}.        

One of the main properties of $K$-theory is the localization. We recall that a short exact sequence in $\Cat^{\perf}$ is a fiber-cofiber sequence $\cA\xto{F_1}\cB\xto{F_2}\cC,$ that is, a bicartesian square of the form
\begin{equation}\label{eq:intro_square}
\begin{tikzcd}
\cA \ar{r}{F_1}\ar{d} & \cB\ar{d}{F_2}\\
0\ar{r} & \cC.
\end{tikzcd}
\end{equation} 
In terms of the (triangulated) homotopy categories, the square \eqref{eq:intro_square} is cartesian  if and only if the functor $\h F_1$ induces an equivalence $\h\cA\xto{\sim}\ker(\h F_2).$ Assuming this, the square \eqref{eq:intro_square} is cocartesian if and only if the induced functor $\h\cB/\h\cA\to \h\cC$ is fully faithful and the Karoubi completion of its image is equivalent to $\h\cC.$ In other words, a fiber-cofiber sequence \eqref{eq:intro_square} induces a short exact sequence of triangulated categories, up to direct summands: $h\cA\to \h\cB\to\h\cC.$

Localization property for non-connective $K$-theory means that the $K$-theory functor sends each fiber-cofiber sequence \eqref{eq:intro_square} to a fiber (=cofiber) sequence in $\Sp.$

Now, consider the $\infty$-category $\Pr_{\st}^L$  of presentable stable $\infty$-categories and colimit-preserving functors. Recall that a presentable stable category $\cC$ is called dualizable if it is a dualizable object in the symmetric monoidal category $\Pr_{\st}^L,$ with the Lurie tensor product. This is equivalent to $\cC$ being a retract in $\Pr_{\st}^L$ of a compactly generated category, by \cite[Proposition D.7.3.1]{Lur18}.

We call an exact functor $F:\cC\to\cD$ of presentable stable $\infty$-categories {\it strongly continuous} if $F$ has a right adjoint which is colimit-preserving. We denote by $\Cat_{\st}^{\dual}$ the (non-full) subcategory of dualizable categories and strongly continuous functors. Further, we denote by $\Cat_{\st}^{\cg}\subset \Cat_{\st}^{\dual}$ the full subcategory of compactly generated categories.

It is well-known that the functor $\Ind:\Cat^{\perf}\to\Cat_{\st}^{\cg}$ is an equivalence, and an inverse equivalence is given by $\cC\mapsto\cC^{\omega}$ (the latter is the full subcategory of compact objects). Therefore, we have a fully faithful functor $\Ind:\Cat^{\perf}\to \Cat_{\st}^{\dual}.$ We claim that the non-connective $K$-theory can be naturally extended to dualizable presentable categories in the following sense: there is a functor $K^{\cont}:\Cat_{\st}^{\dual}\to \Sp,$ and a natural isomorphism $K^{\cont}\circ\Ind\cong K$ of functors $\Cat^{\perf}\to\Sp.$ Moreover, $K^{\cont}$ is uniquely determined (up to a contractible space of choices) by the localization property: it sends short exact sequences in $\Cat_{\st}^{\dual}$ to fiber sequences in $\Sp.$

Uniqueness essentially follows from another characterization of dualizable categories: a category $\cC\in\Pr_{\st}^L$ is dualizable if and only if there exists a short exact sequence $\cC\to\cD\to\cD'$ in $\Pr_{\st}^L$ with strongly continuous functors, such that $\cD$ and $\cD'$ compactly generated, see Proposition \ref{prop:dualizable_via_ses}. If we require $K^{\cont}$ to be a localizing invariant, then we are forced to have 
\begin{equation}\label{eq:naive_definition}K^{\cont}(\cC)\cong\Fiber(K(\cD^{\omega})\to K(\cD^{'\omega})).\end{equation} 

However, it is not very reasonable to take \eqref{eq:naive_definition} as a definition of $K^{\cont}$ for the following reasons:

\begin{itemize}
\item It is not immediately obvious that $K^{\cont}(\cC)$ does not depend on the choice of a short exact sequence; this can be deduced from \cite[Theorem 18]{Tam}. 

\item It is not clear that a strongly continuous functor $\cC\to\cC'$ of dualizable categories induces a map $K^{\cont}(\cC)\to K^{\cont}(\cC').$
\end{itemize}

Instead, we define the continuous $K$-theory using Calkin categories. Recall the notion of Calkin algebra from functional analysis: for a Hilbert space $\cH,$ the Calkin algebra is defined as the quotient $\Calk(\cH)=B(\cH)/C(\cH)$ of the algebra of bounded operators $B(\cH)$ by the ideal of compact operators $C(\cH).$ Since $C(\cH)$ is the closure of the ideal of bounded operators of finite rank, we have a natural discrete analogue. Namely, for a vector space $V$ over a field $\mk$ the Calkin algebra is defined as the quotient $\Calk(V)=\End(V)/(V^*\otimes V).$ This naturally generalizes to a Calkin category associated to a small stable $\infty$-category, which we now recall.

For $\cA\in\Cat^{\perf},$ the Calkin category is defined as the (Karoubi completion of) the quotient:
$$\Calk(\cA)=(\Ind(\cA)/\cA)^{\Kar}.$$ 
If we ignore set-theoretic issues for now (they are easily resolved), then the short exact sequence $\cA\xto{\cY}\Ind(\cA)\to\Calk(\cA)$ implies an isomorphism $K(\cA)\cong\Omega K(\Calk(\cA)),$ since $K(\Ind(\cA))\cong 0.$ 

We claim that Calkin categories naturally generalize to dualizable categories. Namely, one can define a functor $\Calk^{\cont}:\Cat_{\st}^{\dual}\to\Cat^{\perf},$ so that we have a natural equivalence $\Calk^{\cont}\circ\Ind\simeq\Calk$ of functors $\Cat^{\perf}\to \Cat^{\perf}.$  
For this we use one more characterization of dualizable categories: a stable presentable $\infty$-category $\cC$ is dualizable if and only if the colimit functor $\Ind(\cC)\to \cC$ has a left adjoint. Denoting it by $\hat{\cY}_{\cC},$ we see that the quotient $\Ind(\cC)/\hat{\cY}_{\cC}(\cC)$ is compactly generated. We define
$$\Calk^{\cont}(\cC):=(\Ind(\cC)/\hat{\cY}_{\cC}(\cC))^{\omega}.$$

It is easy to see that for a small stable $\infty$-category $\cA$ we have
$$\Calk^{\cont}(\Ind(\cA))\simeq\Calk(\cA).$$
Ignoring set-theoretic issues again, we define the continuous $K$-theory by the formula $K^{\cont}(\cC):=\Omega K(\Calk^{\cont}(\cC)).$ We refer to Subsection \ref{ssec:loc_invar_dualizable} for a more precise definition.

It is not difficult to check that the functor $\Calk^{\cont}$ preserves short exact sequences, hence $K^{\cont}$ is a localizing invariant. The above discussion also implies an isomorphism $K^{\cont}(\Ind(\cA))\cong K(\cA).$ Similarly, we can apply the same construction to any localizing invariant. We have the following result.

\begin{theo} Let $\cE$ be a stable $\infty$-category. The precomposition functor
$$\Fun(\Cat_{\st}^{\dual},\cE)\to \Fun(\Cat^{\perf},\cE),\quad F\mapsto F\circ\Ind,$$ induces an equivalence between the full subcategories of localizing invariants. The inverse equivalence is given by $F\mapsto F^{\cont}.$\end{theo}

The detailed proof is given in Subsection \ref{ssec:loc_invar_dualizable} below. Moreover, in Subsection \ref{ssec:F_cont_via_right_Kan_extension} we show that $F^{\cont}$ is simply the right Kan extension of $F$ via the functor $\Ind:\Cat^{\perf}\to\Cat_{\st}^{\dual}.$ 

The assignment $F\mapsto F^{\cont}$ preserves the important properties of localizing invariants. In particular, $F$ preserves filtered colimits if and only if so does $F^{\cont}.$

Examples of dualizable (but not necessarily compactly generated) categories naturally arise in various contexts, see Subsection \ref{ssec:examples}. One class of examples is provided by almost mathematics (Faltings \cite{Fa02}, Gabber-Ramero \cite{GR03}). Namely, let $R$ be an associative unital ring, and $\m\subset R$ a two-sided ideal such that $\m^2=\m$ and $\Tor_{n}^R(\m,\m)=0$ for $n>0$ (for example, $\m$ is flat as a right or left $R$-module). Then the extension of scalars functor $D(R)\to D(R/\m)$ is a localization, and its kernel $\cC_{R,\m}$ is a dualizable category. We have an equivalence $\cC_{R,\m}\simeq D(\Mod_a\mhyphen R),$ where $\Mod_a\mhyphen R$ is the abelian category of almost modules: $\Mod_a\mhyphen R=\Mod\mhyphen R/\Mod\hy (R/\m).$ Moreover, if $\m$ is contained in the Jacobson radical of $R,$ then the category $D(\Mod_a\mhyphen R)$ has no non-zero compact objects \cite{Kel94}.

Another class of examples which we are interested in this paper are categories of sheaves on a locally compact Hausdorff space $X.$ Let $\un{\cC}$ be a presheaf on $X$ with values in $\Cat_{\st}^{\dual}.$ We denote by $\Sh(X;\un{\cC})$ the category of sheaves with values in $\un{\cC}.$ By \cite{Lur18}, the category $\Shv(X;\un{\cC})$ is dualizable. We have the following general result, see Theorem \ref{th:U_loc_locally_compact}.

\begin{theo}\label{th:main_intro} With the above notation, let $\cE$ be a presentable stable $\infty$-category, and $F:\Cat^{\perf}\to \cE$ a localizing invariant which commutes with filtered colimits. Then we have a natural isomorphism
$$F^{\cont}(\Shv(X;\un{\cC}))\cong \Gamma_c(X,F^{\cont}(\un{\cC})).$$\end{theo}

We note that even when the presheaf $\un{\cC}$ takes values in compactly generated categories (and even when it is constant), the category $\Shv(X;\un{\cC})$ is usually {\it not} compactly generated, see Subsection \ref{ssec:compact_sheaves}. For example, when $X$ is non-compact and connected, and the presheaf $\un{\cC}$ is constant and non-zero, then the category $\Shv(X;\un{\cC})$ has no non-zero compact objects. This is a straightforward generalization of Neeman's theorem \cite{Nee01b}.

The following is a special case of Theorem \ref{th:main_intro}.

\begin{cor}Let $\cC$ be a dualizable category. Then for any $n\geq 0$ we have a natural isomorphism
$$K^{\cont}(\Shv(\R^n,\cC))\cong\Omega^n K^{\cont}(\cC).$$ In particular, we have $K_0^{\cont}(\Shv(\R^n,\cC))\cong K_n^{\cont}(\cC).$\end{cor}

Restricting to finite CW complexes, we see that the categories of $\cC$-valued sheaves in fact ``categorify'' the maps to $K^{\cont}(\cC)$ in the following sense.

\begin{cor}\label{cor:sheaves_on_finite_CW_complexes_intro} Let $X$ be a finite CW complex (hence a compact Hausdorff space). Let $\cC$ be a dualizable category. Then we have a natural isomorphism $$K^{\cont}(\Shv(X;\cC))\cong K^{\cont}(\cC)^X.$$ In particular, if $\cA$ is a small stable idempotent-complete $\infty$-category, then we have $K_0^{\cont}(\Shv(X;\Ind(\cA)))\cong [X,K_{\geq 0}(\cA)],$ where the RHS is the abelian group of homotopy classes of maps from $X$ to the $K$-theory space of $\cA.$\end{cor}

It turns out that Corollary \ref{cor:sheaves_on_finite_CW_complexes_intro} holds for arbitrary localizing invariants, not necessarily commuting with filtered colimits, see Theorem \ref{th:sheaves_on_finite_CW_complexes}.

Another interesting class of examples of dualizable categories comes from condensed mathematics: these are categories of (strongly) nuclear modules over analytic rings. We consider the following special case (see Example \ref{ex:nuc}).

Let $R$ be a noetherian commutative ring, and let $I\subset R$ be an ideal. Clausen and Scholze \cite{CS} defined the category $\Nuc(R_{\hat{I}})$ of nuclear solid modules over $R_{\hat{I}}.$ This is a dualizable category, and the full subcategory of compact objects in $\Nuc(R_{\hat{I}})$ is equivalent to the usual category of perfect complexes $\Perf(R_{\hat{I}}).$ We study the category $\Nuc(R_{\hat{I}})$ in \cite{E25a}, where we will in particular prove that
$$K^{\cont}(\Nuc(R_{\hat{I}}))\cong\prolim[n]K(R/I^n).$$

The paper is organized as follows.

In Part \ref{part:dualizable_cats} we study the theory of dualizable categories. In Section \ref{sec:presentable_stable} we mostly recall some definitions and results on presentable categories, focusing on the stable case. We also fix the terminology and notation, including the formulations of strong and weak versions of Grothendieck's AB axioms for abstract $\infty$-categories in Subsection \ref{ssec:AB_axioms}. In Subsection \ref{ssec:well_generated} we discuss the precise relation between presentable  stable $\infty$-categories and well generated triangulated categories in the terminology of \cite{Nee01}. Some results here might be new. In particular we prove that Neeman's $\kappa$-compact objects in $\h \cC$ are $\kappa$-compact in $\cC$ in the $\infty$-categorical sense, but not vice versa in general. 

In Section \ref{sec:general_theory_dualizable} we study the general theory of (not necessarily cocomplete) compactly assembled $\infty$-categories, mostly focusing on the dualizable presentable stable $\infty$-categories. In Subsection \ref{ssec:dualizable_cats} we start with some basic equivalent characterizations of dualizability from \cite{Lur18} and study the basic properties, including the $\omega_1$-presentability of dualizable categories (and more generally $\omega_1$-accessibility of compactly assembled categories). Before continuing the study we explain in Subsection \ref{ssec:analogy_dualizable_comphaus} a remarkable analogy between the category $\Cat_{\st}^{\dual}$ of dualizable categories and the category $\comphaus^{op},$ the opposite of the category of compact Hausdorff spaces. After giving a long table of analogies we explain the plan for the rest of Section \ref{sec:general_theory_dualizable}, illustrating the results with the corresponding familiar facts about compact Hausdorff spaces. Notably, the colimits in $\Cat_{\st}^{\dual}$ are nicely behaved, which corresponds to the nice behavior of limits in $\comphaus.$ On the other hand, the limits in $\Cat_{\st}^{\dual}$ are much more complicated, which corresponds to the complexity of colimits in $\comphaus$ (computed as the Stone-\v{C}ech compactification of a colimit in $\Top$).

Section \ref{sec:presentability_Cat_dual} contains an explicit proof of the $\omega_1$-presentability of $\Cat_{\st}^{\dual},$ including explicit characterizations of $\kappa$-compact objects for an uncountable regular $\kappa.$ In Section \ref{sec:Urysohn} we give an explicit $\omega_1$-compact generator (via colimits) of $\Cat_{\st}^{\dual},$ which can be thought of as an analogue of Urysohn's lemma for compact Hausdorff spaces. The latter can be understood as a statement about the generation of $\comphaus$ by the unit interval $[0,1]$ via limits.

The remaining three sections of Part \ref{part:dualizable_cats} contain certain subtle results on dualizable and presentable stable categories. 
In Section \ref{sec:dualizability_via_flatness} we prove a surprising statement: a presentable stable category is dualizable if and only if it is flat in $\Pr_{\st}^L.$ More precisely, $\cC$ is flat if the functor $\cC\otimes-$ preserves fully faithful functors. We also give a more general criterion for flatness in $\Pr^L_{\st,\kappa}$ for an uncountable regular cardinal $\kappa.$ In Section \ref{sec:extensions_of_comp_gen} we give a criterion of dualizability for categories which are extensions of compactly generated categories. If such an extension is dualizable, then it is in fact compactly generated. As an application, we recover a result from \cite[Appendix A]{CDH+20} stating that in $\Cat^{\perf}$ the extensions of $\cA$ by $\cB$ are classified by functors $\cA\to\Tate(\cB),$ where $\Tate(\cB)$ is the category of Tate objects, see \cite{Hen}. As another application, for a commutative noetherian ring $R$ we classify the localizing subcategories of $D(R)$ which are dualizable. They correspond bijectively to convex subsets of $\Spec R,$ see Theorem \ref{th:dualizability_equiv_convexity}. In Section \ref{sec:image_of_hom_epi} we show that for a dualizable category $\cC$ the essential image of the functor $\cC\to\Calk^{\cont}(\cC)$ is not a stable subcategory in general, although this is the case for compactly generated categories.

In Part \ref{part:loc_invar_dualizable} we develop the theory of localizing invariants of dualizable categories, in particular proving non-trivial results on the computation of these invariants. In Section \ref{sec:localizing_invariants} we study the general theory, in particular proving that
a localizing invariant of small categories extends uniquely to dualizable categories, using the Calkin construction. 
In Subsection \ref{ssec:sheaves_on_R} we give an example: we compute finitary localizing invariants of the category of sheaves on $\R$ and on $\R\cup\{-\infty\}.$ In Subsection \ref{ssec:K_theory_of_products} we do a similar computation for not necessarily finitary localizing invariants, deducing the commutation of K-theory with infinite products (Theorem \ref{th:map_of_localizing_invariants}).

In Section \ref{sec:sh_cosh_continuous_posets} we consider the sheaves and cosheaves on continuous posets (these posets are considered as categories, not as sets), and compute their (finitary) localizing invariants. We use this computation later to compute the localizing invariants for sheaves on locally compact Hausdorff spaces.  

In Section \ref{sec:sh_loc_comp} we study the sheaves on locally compact Hausdorff space with values in a presheaf of dualizable categories. First, in Subsection \ref{ssec:sheaves_finite_CW} we give a computation of arbitrary localizing invariants for sheaves on finite CW complexes (in the case of constant coefficients). In Subsection \ref{ssec:sheaves_K_sheaves_loc_comp} we explain how to approximate the category of sheaves on an arbitrary compact Hausdorff space by ``simpler'' categories (Proposition \ref{prop:approximating_sheaves}). Here it is crucial that we use the notion of a $\msK$-(pre)sheaf from \cite{Lur09}. In Subsection \ref{ssec:loc_invar_cats_of_sheaves} we apply the previous results to compute the finitary localizing invariants of categories of sheaves (Theorem \ref{th:U_loc_locally_compact}). In Subsection \ref{ssec:compact_sheaves} we give a simple description of compact objects in the categories of sheaves.

In Appendix \ref{app:Adams_representability} we prove two closely related versions of Adams representability theorem for $\omega_1$-compact dualizable categories (Theorems \ref{th:Adams_rep_covar} and \ref{th:Adams_rep_contravar}). 
In Appendix \ref{app:K_theory_products_exact} we show that $K$-theory of exact $\infty$-categories commutes with infinite products, using the results of C\'ordova Fedeli \cite{Cor} and Klemenc \cite{Kle}.

{\noindent {\bf Acknowledgements.}} I am grateful to Ko Aoki, Alexander Beilinson, Dustin Clausen, Adriano C\'ordova Fedeli, Vladimir Drinfeld, Boris Feigin, Dennis Gaitsgory, Marc Hoyois, Dmitry Kaledin, David Kazhdan, Maxim Kontsevich, Akhil Mathew, Thomas Nikolaus, Dmitri Orlov, Maxime Ramzi, Victor Saunier, Tomer Schlank, Peter Scholze, Vivek Shende, Vladimir Sosnilo, German Stefanich, Georg Tamme, Bertrand To\"en, Yakov Varshavsky and Christoph Winges for useful discussions. I am also grateful to an anonymous referee for many valuable suggestions and comments. Part of this work was done while I was a visitor in the Max Planck Institute for Mathematics in Bonn from December 2022 till February 2024, and I am grateful to the institute for their hospitality and support.  

\part{Theory of dualizable categories}
\label{part:dualizable_cats}

\section{Presentable stable categories}
\label{sec:presentable_stable}

\subsection{Infinity-categories, limits and colimits}
\label{ssec:infinity_cats_preliminaries}

We will freely use the theory of $\infty$-categories, functors, limits and colimits as developed in \cite{Lur09, Lur17}. Except for Section \ref{sec:dualizability_via_flatness} (where we use the formalism of ordinary $2$-categories) we will only deal with $(\infty,1)$-categories. Given an ordinary category $\cC,$ we identify it with its nerve $N(\cC),$ which is an $\infty$-category. In particular, we will consider a poset (partially ordered set) as an $\infty$-category. Sometimes we will say ``category'' instead of ``$\infty$-category'', when the meaning is clear from the context.

We denote by $\cS$ the $\infty$-category of spaces (or equivalently $\infty$-groupoids). It is freely generated by one object via colimits.

We will use the following convention: a functor $p:I\to J$ between $\infty$-categories is cofinal if for any $j\in J$ the $\infty$-category $I\times_J J_{j/}$ is weakly contractible. Equivalently, for any $\infty$-category $\cC$ and for any functor $F:J\to\cC,$
we have
$$\indlim[](J\xto{F}\cC)\cong\indlim[](I\xto{p}J\xto{F}\cC),$$
assuming that one of the colimits exists. Dually, a functor $p:I\to J$ is final if the functor $p^{op}:I^{op}\to J^{op}$ is cofinal.    

Recall that an infinite cardinal $\kappa$ is called regular if for any collection of cardinals $\{\kappa_i\}_{i\in I}$ such that $\kappa_i<\kappa$ and $|I|<\kappa$ we have $\sum\limits_{i\in I}\kappa_i<\kappa.$ For example, any (infinite) successor cardinal is regular.
If $\kappa$ is regular, then a set $A$ is called $\kappa$-small if $|A|<\kappa.$ 

We refer to \cite[Definition 5.3.1.7]{Lur09} for the notion of a $\kappa$-filtered $\infty$-category. By \cite[Proposition 5.3.1.18]{Lur09} for any $\kappa$-filtered $\infty$-category $I$ there exists a $\kappa$-directed poset $J$ and a cofinal functor $J\to I.$ Recall that a poset $J$ is called $\kappa$-directed if $J$ is nonempty and any $\kappa$-small collection of elements of $J$ has an upper bound. Recall that $\omega$-filtered $\infty$-categories resp. $\omega$-directed posets are simply called filtered $\infty$-categories resp. directed posets.

Given a small $\infty$-category $\cA$ and a regular cardinal $\kappa$ we denote by $\Ind_{\kappa}(\cA)$ the $\infty$-category which is freely generated by $\cA$ via $\kappa$-filtered colimits. The category $\Ind_{\kappa}(\cA)$ can be described as a full subcategory of the category of presheaves $\Fun(\cA^{op},\cS)$ which is formed by $\kappa$-filtered colimits of representable presheaves. We will usually denote the $\kappa$-ind objects of $\cA$ by $\inddlim[i\in I]x_i,$ where $I$ is a $\kappa$-filtered $\infty$-category, and we have a functor $I\to\cA,$ $i\mapsto x_i.$ For $\kappa=\omega,$ we simply write $\Ind(\cA)$ instead of $\Ind_{\omega}(\cA).$ 

If $\cA$ is not necessarily small (but still locally small), we will sometimes still consider the (locally small) category $\Ind(\cA),$ which is the directed union of $\Ind(\cB),$ where $\cB$ runs through small full subcategories of $\cA.$

If $\cC$ is an $\infty$-category with $\kappa$-filtered colimits, then an object $x\in\cC$ is called $\kappa$-compact if the functor $\Map_{\cC}(x,-)$ commutes with $\kappa$-filtered colimits. We denote by $\cC^{\kappa}\subset\cC$ the full subcategory of $\kappa$-compact objects. The $\omega$-compact objects are called compact objects. 

Recall that an $\infty$-category $\cC$ is called stable if it is pointed, has finite limits and colimits, and a square of the form
\begin{equation}\label{eq:fiber-cofiber}
\begin{tikzcd}
x \ar{r}{f} \ar{d} & y \ar{d}{g}\\
\ast \ar{r} & z
\end{tikzcd}
\end{equation}
is cartesian if and only if it is cocartesian. 

In a general pointed $\infty$-category, if a square of the form \eqref{eq:fiber-cofiber} is cartesian then it is also called a fiber sequence, and $x$ is called a fiber of $g.$ Dually, if a square of the form \eqref{eq:fiber-cofiber} is cocartesian then it is also called a cofiber sequence, and $z$ is called a cofiber of $f.$ 

If an $\infty$-category $\cC$ is stable, then the category $\h\cC$ has a natural triangulated structure. A functor $F:\cC\to\cD$ between stable $\infty$-categories is exact if it is pointed, and takes fiber sequences in $\cC$ to fiber sequences in $\cD.$ If $F$ is exact, then the induced functor $\h F:\h\cC\to \h\cD$ is an exact functor between triangulated categories. By default, the functors between stable categories will be assumed to be exact.

For an $\bE_1$-ring $A,$ we will denote by $\Mod\hy A$ the (cocomplete, stable) category of right $A$-modules. We denote by $\Perf(A)\subset\Mod\hy A$ the full subcategory of perfect $A$-modules; it is generated by $A$ as a stable idempotent-complete subcategory.

We denote by $\Cat_{\infty}$ the category of small $\infty$-categories. Further, we denote by $\Cat^{\ex}\subset\Cat_{\infty}$ the (non-full) subcategory of small stable categories and exact functors. We denote by $\Cat^{\perf}\subset \Cat^{\ex}$ the full subcategory of Karoubi complete (idempotent-complete) stable categories.

Given an $\infty$-category $\cA,$ we denote by $\cA^{\Kar}$ its Karoubi completion. In particular, the functor $(-)^{\Kar}:\Cat^{\ex}\to\Cat^{\perf}$ is the left adjoint to the inclusion.

By definition, a short exact sequence in $\Cat^{\perf}$ is a fiber-cofiber sequence
$$
\begin{tikzcd}
\cA \ar{r} \ar{d} & \cB \ar{d}\\
0 \ar{r} & \cC.
\end{tikzcd}
$$
This means that the functor $\cA\to\cB$ is fully faithful, the composition $\cA\to\cB\to\cC$ is zero, and the functor $(\cB/\cA)^{\Kar}\to\cC$ is an equivalence. We will simply write
$$0\to\cA\to\cB\to\cC\to 0$$
for such short exact sequences. We will also use the same terminology and notation for short exact sequences of ``large'' stable categories, namely presentable and in particular dualizable.

By \cite[Corollary 4.25]{BGT}, the categories $\Cat^{\ex}$ and $\Cat^{\perf}$ are compactly generated. See also Proposition \ref{prop:presentability_of_Pr_kappa} and Remark \ref{rem:Cat_ex_via_Thomason} for a different argument.

\subsection{Grothendieck's AB axioms}
\label{ssec:AB_axioms}

We recall some of the Grothendieck's axioms for abelian categories introduced in \cite{Gro}. In fact these axioms make sense for a much more general class of $\infty$-categories. Let $\cA$ be an abelian (ordinary) category, and let $\cC$ be an abstract $\infty$-category.

Recall that the (AB5) axiom for $\cA$ states that $\cA$ has filtered colimits which are exact (equivalently, $\cA$ has infinite coproducts and the filtered colimits are exact). This has two natural generalizations. Suppose that $\cC$ has finite limits and filtered colimits. We say that $\cC$ satisfies the {\it strong} (AB5) axiom if filtered colimits in $\cC$ commute with finite limits. We say that $\cC$ satisfies the {\it weak} (AB5) axiom if the class of monomorphisms in $\cC$ is closed under filtered colimits. Clearly, the strong (AB5) implies the weak (AB5).

Next, recall that the (AB4*) axiom for $\cA$ states that $\cA$ has infinite products which are exact. The generalizations for $\cC$ are similar to the above. Namely, suppose that $\cC$ has infinite products and finite colimits. We say that $\cC$ satisfies the {\it strong} (AB4*) axiom if the infinite products commute with finite colimits. We say that $\cC$ satisfies the {\it weak} (AB4*) axiom if the class of epimorphisms in $\cC$ is closed under infinite products.

As in the case of abelian categories, we say that $\cC$ satisfies the (strong or weak) (AB5*) resp. (AB4) axiom if the category $\cC^{op}$ satisfies the (strong or weak) (AB5) resp. (AB4*) axiom.

Finally, recall the (AB6) axiom for $\cA.$ It states that $\cA$ has filtered colimits, infinite products, and for any collection of filtered categories $J_i,$ $i\in I,$ and for any functors $J_i\to \cA,$ $j_i\mapsto x_{j_i},$ the map $$\indlim[(j_i)_i\in\prod\limits_{i} J_i]\prod_{i\in I} x_{j_i}\to \prod_{i\in I}\indlim[j_i\in J_i]x_{j_i}$$ is an isomorphism. This axiom makes sense for an abstract $\infty$-category $\cC$ with exactly the same formulation. We refer to Subsection \ref{ssec:AB6} for a more detailed discussion.

Below we will frequently mention the above axioms: they turn out to be very relevant for the study of dualizable categories. 

\subsection{Presentable and accessible stable $\infty$-categories}
\label{ssec:pres_acc_cats}

An $\infty$-category $\cC$ is called $\kappa$-accessible for some regular cardinal $\kappa$ if $\cC\simeq\Ind_{\kappa}(\cC_0),$ where $\cC_0$ is a small $\infty$-category. Note that in this case $\cC$ is stable if and only if the category $\cC_0^{\Kar}$ is stable. 

Further, an $\infty$-category is accessible if it is $\kappa$-accessible for some $\kappa.$ If $\cC$ and $\cD$ are accessible $\infty$-categories with $\kappa$-filtered colimits (but $\cC$ and $\cD$ are not necessarily $\kappa$-accessible), we say that a functor $F:\cC\to\cD$ is $\kappa$-continuous if $F$ commutes with $\kappa$-filtered colimits. Further, $F$ is {\it continuous} if it is $\omega$-continuous, i.e. commutes with filtered colimits, assuming that $\cC$ and $\cD$ have filtered colimits.

A $\kappa$-continuous functor $F:\cC\to\cD$ is called $\kappa$-accessible if moreover $\cC$ and $\cD$ are $\kappa$-accessible. A functor is called accessible if it is $\kappa$-accessible for some $\kappa.$

An $\infty$-category $\cC$ is $\kappa$-presentable if $\cC$ is $\kappa$-accessible and cocomplete. Equivalently, this means that $\cC\simeq\Ind_{\kappa}(\cC_0),$ where the category $\cC_0$ has $\kappa$-small colimits. Further, $\cC$ is presentable if it is $\kappa$-presentable for some $\kappa.$ Following \cite{Lur09}, we denote by $\Pr^L$ the $(\infty,1)$-category of all presentable categories and colimit-preserving functors. Note that the category $\Pr^L$ is not locally small. 

A functor $F:\cC\to\cD$ between presentable categories is colimit-preserving if and only if it has a right adjoint. On the other hand, $F$ has a left adjoint if and only if it is accessible and commutes with limits. This gives an equivalence $(\Pr^L)^{op}\simeq\Pr^R,$ where the $1$-morphisms in $\Pr^R$ are accessible functors which commute with limits.

If $\cC$ is presentable and $\cD$ is cocomplete, we will denote by $\Fun^L(\cC,\cD)$ the category of colimit-preserving functors. The category $\Pr^L$ has a natural symmetric monoidal structure, for which the internal $\Hom$ from $\cC$ to $\cD$ is given by $\Fun^L(\cC,\cD).$ The tensor product of presentable categories is usually called the Lurie tensor product. It is easy to show that $\cC\otimes\cD\simeq\Fun^L(\cC,\cD^{op})^{op}$ for $\cC,\cD\in\Pr^L.$ The unit object is given by $\cS$ -- the $\infty$-category of spaces.

For a regular cardinal $\kappa,$ we denote by $\Pr^L_{\kappa}\subset \Pr^L$ the non-full subcategory of $\kappa$-presentable categories, where $1$-morphisms in $\Pr^L_{\kappa}$ are colimit-preserving functors which preserve $\kappa$-compact objects. The category $\Pr^L_{\kappa}$ is locally small. Moreover, the assignment $\cC\mapsto\cC^{\kappa}$ defines an equivalence between $\Pr^L_{\kappa}$ and the category of idempotent-complete small $\infty$-categories with $\kappa$-small colimits, where the $1$-morphisms are functors which commute with $\kappa$-small colimits. If $\kappa>\omega,$ then the existence of $\kappa$-small colimits automatically implies idempotent-completeness. The Lurie tensor product induces a well-defined symmetric monoidal structure on $\Pr^L_{\kappa}.$ 

For any regular cardinal $\kappa,$ the inclusion functor $\Pr^L_{\kappa}\to \Pr^L$ commutes with colimits. For $\kappa<\lambda,$ the right adjoint to the functor $\Pr^L_{\kappa}\to \Pr^L_{\lambda}$ is given by $\cC\mapsto \Ind_{\kappa}(\cC^{\lambda}).$

We denote by $\Pr^L_{\st}\subset \Pr^L$ the full subcategory of presentable stable categories, and similarly for $\Pr^L_{\st,\kappa}\subset\Pr^L_{\kappa}.$ 
The inclusions $\Pr^L_{\st}\to\Pr^L$ and $\Pr^L_{\st,\kappa}\to\Pr^L_{\kappa}$ commute with limits and colimits.  

Below we will mostly deal with presentable stable categories. By default, the functors between stable categories will be assumed to be exact. Note that an exact functor $F:\cC\to\cD$ between presentable stable $\infty$-categories is continuous if and only if it is colimit-preserving or equivalently commutes with coproducts. Equivalently, this means that the functor $\h F:\h\cC\to \h\cD$ between the homotopy categories commutes with coproducts.

A cocomplete stable category $\cC$ is $\omega$-presentable if and only if it is compactly generated by a set of objects.
The assignment $\cC\mapsto \cC^{\omega}$ defines an equivalence $\Pr^L_{\st,\omega}\xto{\sim}\Cat^{\perf}.$

\begin{defi}\label{defi:gluing_general} 1) Given an exact functor $F:\cD\to\cC$ between small stable categories, we denote by $\cC\oright_F \cD$ the semi-orthogonal gluing of $\cC$ and $\cD$ via $F.$ More precisely, $\cC\oright_F \cD$ is the category of triples $(x,y,\varphi),$ where $x\in\cC,$ $y\in\cD$ and $\varphi:x\to F(y).$ Equivalently, $\cC\oright_F\cD$ is the oplax limit of the diagram $F:\cD\to\cC$ in the $(\infty,2)$-category of stable categories. As explained in \cite[Theorem 2.4.2]{CDW23} (and references therein), the four categories obtained by taking an (op)lax (co)limit of the diagram $F:\cD\to\cC$ are naturally equivalent.
	
2) In the situation of 1) we denote by $i_1:\cC\to\cC\oright_F \cD$ and $i_2:\cD\to\cC\oright_F \cD$ the inclusion functors given by $i_1(x)=(x[-1],0,0),$ $i_2(x)=(0,x,0).$ 

3) If $\cC$ and $\cD$ are cocomplete stable categories and $F:\cD\to\cC$ is an arbitrary exact functor, we denote by $\cC\oright_F\cD$ the same semi-orthogonal gluing as above, and similarly for the functors $i_1:\cC\to\cC\oright_F\cD$ and $i_2:\cC\to\cC\oright_F\cD.$

4) We will also use a closely related construction. Given exact functors between (small or large) stable categories $F:\cA\to\cC,$ $G:\cB\to\cC,$ we denote by $\cA\laxtimes{\cC}\cB$ or by $\cA\laxtimes{F,\cC,G}\cB$ the lax pullback, i.e. the category of triples $(x,y,f),$ where $x\in\cA,$ $y\in \cB$ and $f:F(x)\to G(y).$ If $F$ has a right adjoint $F^R,$ then we clearly have $\cA\laxtimes{\cC}\cB\simeq \cA\oright_{F^R G}\cB.$\end{defi} 

We make the following observation.

\begin{prop}\label{prop:gluing_presentable} Let $F:\cD\to\cC$ be an exact functor between presentable stable categories. The following are equivalent:

\begin{enumerate}[label=(\roman*),ref=(\roman*)]

\item  the category $\cE=\cC\oright_F \cD$ is presentable; \label{acc1}

\item the functor $F$ is accessible. \label{acc2}
\end{enumerate}

Moreover, if $F$ is $\kappa$-accessible, then $\cE$ is $\kappa$-presentable.
\end{prop}

\begin{proof}We consider the functors $i_1:\cC\to\cE$ and $i_2:\cD\to\cE$ as in Definition \ref{defi:gluing_general}. Clearly, both $i_1$ and $i_2$ are continuous.

\Implies{acc1}{acc2}. Suppose that $\cE$ is presentable. Then the functor $i_1^R:\cE\to\cC$ (the right adjoint to $i_1$) is accessible. Hence, so is the composition $i_1^R\circ i_2:\cD\to\cC,$ which is isomorphic to $F.$ 

\Implies{acc2}{acc1}. It is sufficient to prove the ``moreover'' statement. Suppose that $\cC$ and $\cD$ are $\kappa$-presentable and $F$ commutes with $\kappa$-filtered colimuts. Then an object $(x,y,\varphi)$ is $\kappa$-compact in $\cE$ whenever $x\in\cC^{\kappa},$ $y\in\cD^{\kappa}.$ We conclude that $\cE$ is $\kappa$-presentable, and the full subcategory $\cE^{\kappa}\subset\cE$ is generated by $i_1(\cC^{\kappa})$ and $i_2(\cD^{\kappa})$ (as a stable subcategory).
\end{proof}

\begin{remark}If we consider the $(\infty,2)$-category of presentable stable categories and accessible exact functors, then the category $\cC\oright_F\cD$ has again four interpretations as an (op)lax (co)limit of the arrow $\cD\xto{F}\cC.$\end{remark}

The statements of the following proposition are well known. We include the proof for completeness.

\begin{prop}\label{prop:presentability_of_Pr_kappa} Let $\kappa$ be a regular cardinal.

1) The category $\Cat_{\infty}$ is generated by colimits by a single compact object $[1]$ (the linearly ordered set with $2$ elements). In particular, $\Cat_{\infty}$ is compactly generated.

2) The category $\Pr^L_{\kappa}$ is generated by colimits by a single $\kappa$-compact object $\Fun([1],\cS)$ -- the arrow category of spaces. In particular, $\Pr^L_{\kappa}$ is $\kappa$-presentable. 

3) The category $\Pr^L_{\st,\kappa}$ is generated by colimits by a single $\kappa$-compact object $\Sp.$ In particular, $\Pr^L_{\st,\kappa}$ is $\kappa$-presentable and $\Cat^{\perf}\simeq \Pr^L_{\st,\omega}$ is compactly generated.
\end{prop}

\begin{proof}For 1) it suffices to show that if $F:\cC\to\cD$ is a functor between small $\infty$-categories such that the map $\Fun([1],\cC)^{\simeq}\to\Fun([1],\cD)^{\simeq}$ is an equivalence of spaces, then $F$ is an equivalence.

It is clear that $F$ is essentially surjective and the map $F^{\simeq}:\cC^{\simeq}\to\cD^{\simeq}$ is an equivalence of spaces. For $x,y\in\cC,$ the space $\Map_{\cC}(x,y)$ is the fiber of the map $\Fun([1],\cC)^{\simeq}\to\cC^{\simeq}\times\cC^{\simeq}$ over $(x,y),$ and similarly for $\cD.$ This implies that $F$ is fully faithful. 

To deduce 2) from 1) it suffices to recall that the functor $\Pr^L_{\kappa}\to \Cat_{\infty},$ $\cC\mapsto\cC^{\kappa},$ commutes with $\kappa$-filtered colimits.

3) The $\kappa$-compactness of $\Sp$ in $\Pr^L_{\st,\kappa}$ follows from the above commutation with $\kappa$-filtered colimits. It remains to show that the functor $\Pr^L_{\st,\kappa}\to\cS,$ $\cC\mapsto (\cC^{\kappa})^{\simeq},$ is conservative.

Let $F:\cA\to\cB$ be an exact functor between small stable categories such that the induced functor $F^{\simeq}:\cA^{\simeq}\to\cB^{\simeq}$ is an equivalence of spaces. We need to show that $F$ is an equivalence. It is clear that $F$ is essentially surjective. It remains to observe that for $x,y\in\cA$ and for each $n\in\Z$ the abelian group $\pi_n\cA(x,y)$ is naturally a retract of $\pi_2(\cA^{\simeq},x\oplus y[1-n]),$ and similarly for $\pi_n\cB(F(x),F(y)).$
\end{proof}

\begin{remark}\label{rem:Cat_ex_via_Thomason} The assignment $\cA\mapsto (\cA^{\Kar},K_0(\cA))$ defines an equivalence between $\Cat^{\ex}$ and the category of pairs $(\cB,\Gamma),$ where $\cB\in\Cat^{\perf}$ and $\Gamma\subset K_0(\cB)$ is a subgroup. This follows from Thomason's theorem \cite[Theorem 2.1]{Th}. 

In particular, the category $\Cat^{\ex}$ is compactly generated and $\cA\in (\Cat^{\ex})^{\omega}$ iff $\cA^{\Kar}\in (\Cat^{\perf})^{\omega}$ and the group $K_0(\cA)$ is finitely generated.
If $\kappa$ is an uncountable regular cardinal, then $\cA\in\Cat^{\ex}$ is a $\kappa$-compact object if and only if $\cA^{\Kar}\in (\Cat^{\perf})^{\kappa}$ and $K_0(\cA)$ is $\kappa$-small.\end{remark}

We will need an explicit characterization of $\kappa$-compact objects in $\Cat^{\perf}$ for uncountable $\kappa.$ 

\begin{prop}\label{prop:kappa_compact_small_cats} Let $\kappa$ be an uncountable regular cardinal and let $\cA$ be a small idempotent-complete stable category. The following are equivalent.

\begin{enumerate}[label=(\roman*),ref=(\roman*)]
\item The object $\cA$ is $\kappa$-compact in $\Cat^{\perf}.$ \label{kappa_comp_in_Cat_perf1}


\item The triangulated category $\h\cA$ is generated (as an idempotent-complete triangulated subcategory) by a $\kappa$-small set of objects, and moreover for any $x,y\in\cA$ the spectrum $\cA(x,y)$ is $\kappa$-compact in $\Sp.$ \label{kappa_comp_in_Cat_perf2}
\end{enumerate}
\end{prop}

\begin{proof} \Implies{kappa_comp_in_Cat_perf1}{kappa_comp_in_Cat_perf2} Let $\cA$ be in $(\Cat^{\perf})^{\kappa},$ and choose a $\kappa$-small ind-system $(\cA_i)_i$ of compact objects of $\Cat^{\perf}$ such that $\cA=\indlim[i]\cA_i.$ Then for each $i$ the spectra $\cA_i(x,y)$ are $\omega_1$-compact, $x,y\in\cA_i$. Hence, the spectra $\cA(x,y)$ are $\kappa$-compact, $x,y\in\cA.$

Since each $\h\cA_i$ is generated by a single object, we conclude that $\h\cA$ is generated by a $\kappa$-small set of objects.

\Implies{kappa_comp_in_Cat_perf2}{kappa_comp_in_Cat_perf1} First suppose that $\h\cA$ is generated by a single object $x,$ and consider the $\bE_1$-algebra $A=\End_{\cA}(x).$ Since the underlying spectrum of $A$ is $\kappa$-compact, it follows from Proposition \ref{prop:E_1-algebras_in_monoidal_cats} that $A$ is $\kappa$-compact in the category of $\bE_1$-rings. Now the category of $\bE_1$-rings is equivalent to the full subcategory of the category $(\Cat^{\perf})_{\Sp^{\omega}/},$ formed by pairs $(\cB,F)$ such that $F(\bS)$ generates $\cB.$ It follows that $\cA$ is $\kappa$-compact.

In general, if $\h\cA$ is generated by a $\kappa$-small collection of objects $S,$ then $\cA$ is a directed union of idempotent-complete stable subcategories $\cA_T,$ generated by a finite subset of objects $T\subset S.$ By the above, each category $\cA_T$ is $\kappa$-compact in $\Cat^{\perf},$ hence so is $\cA.$ 
\end{proof}

We recall that the the colimit of a functor $I\to \Pr^L,$ $i\mapsto \cC_i,$ is equivalent to the limit of the corresponding functor $I^{op}\to\Pr^R,$ where the transition functors are given by the right adjoints. We denote this colimit by $\indlim[i]^{\cont}\cC_i$ (to avoid confusion with the colimits of small stable categories). 

We call an exact functor $F:\cC\to\cD$ between presentable stable $\infty$-categories {\it strongly continuous} if it has a right adjoint which is continuous (in other words, the right adjoint of $F$ has a further right adjoint). It is convenient to introduce the non-full subcategory $\Pr^{LL}_{\st}\subset\Pr^L_{\st}$ with the same objects and with $1$-morphisms being strongly continuous functors. Note that $\Pr^{LL}_{\st}$ is locally small. Given presentable stable categories $\cC$ and $\cD,$ we denote by $\Fun^{LL}(\cC,\cD)$ the (small) category of strongly continuous functors. 

Similarly, for a regular cardinal $\kappa$ we denote by $\Pr^{LL}_{\st,\kappa}\subset \Pr^{LL}_{\st}$ the full subcategory of $\kappa$-presentable categories. In particular, $\Pr^{LL}_{\st,\kappa}$ is a subcategory of $\Pr^L_{\st,\kappa}$ (for $\kappa=\omega$ we have an equality $\Pr^{LL}_{\st,\omega}=\Pr^{L}_{\st,\omega}$).

Note that if $\cC$ and $\cD$ are $\kappa$-presentable stable categories, then a colimit-preserving functor $F:\cC\to\cD$ preserves $\kappa$-compact objects (i.e. $F$ is a $1$-morphism in $\Pr^L_{\st,\kappa}$) if and only if its right adjoint commutes with $\kappa$-filtered colimits. For this reason, we will sometimes say that such a functor $F$ is $\kappa$-strongly continuous. 

\begin{prop}\label{prop:colimits_in_Pr^LL} The category $\Pr^{LL}_{\st}$ is cocomplete and the functor $\Pr^{LL}_{\st}\to \Pr^L_{\st}$ commutes with colimits. In particular, the categories $\Pr^{LL}_{\st,\kappa}$ are also cocomplete and all the inclusion functors $\Pr^{LL}_{\st,\kappa}\to\Pr^{LL}_{\st,\lambda}$ and $\Pr^{LL}_{\st,\kappa}\to\Pr^L_{\st,\kappa}$ commute with colimits.\end{prop}

\begin{proof} This follows directly from the description of a colimit in $\Pr^L_{\st}$ as a limit of the opposite diagram. Namely, consider a small $\infty$-category $I$ and a functor $I\to\Pr^{LL}_{\st},$ $i\mapsto \cC_i,$ and put $\cC:=\indlim[i]^{\cont}\cC_i$ (the colimit is taken in $\Pr^L_{\st}$).  Identify $\cC$ with the limit $\prolim[i\in I^{op}]\cC_i.$ Since the transition functors in this inverse system are continuous, the coproducts in the limit category are computed ``pointwise''. 

Let $\cD\in\Pr^L_{\st}.$ We need to show that a continuous functor $F:\cC\to\cD$ is strongly continuous if and only if the compositions $F_i:\cC_i\to\cC\to\cD$ are strongly continuous. But the functor $F^R:\cD\to\cC$ corresponds to the compatible system of functors $(F_i^R:\cD\to\cC_i)$ under the identification of $\cC$ with the inverse limit. Hence, $F^R$ is continuous if and only if all $F_i^R$ are continuous.\end{proof}

It follows from Proposition \ref{prop:compact_objects_in_colimits_of_presentable} that the category $\Pr^{LL}_{\st,\omega_1}$ is not presentable, more precisely, the object $\Sp\in\Pr^{LL}_{\st,\omega_1}$ is not $\kappa$-compact for any regular cardinal $\kappa.$ A different example showing this was explained to me by German Stefanich.

We make an observation on the automatic strong continuity.

\begin{prop}\label{prop:automatic_strcont} Let $G:\cD\to\cE$ be a continuous functor between presentable stable categories. Let $F_i:\cC_i\to\cD,$ $i\in I,$ be a family of strongly continuous functors, where $\cC_i$ are presentable stable categories. Suppose that the images of $F_i$ generate $\cD,$ and the compositions $G\circ F_i$ are strongly continuous. Then $G$ is also strongly continuous.\end{prop}

\begin{proof}
Indeed, our assumptions imply that the right adjoints $F_i^R,$ $i\in I,$ are continuous and they form a conservative family. Further, the compositions $F_i^R\circ G^R$ are continuous. Hence, the functor $G^R$ is also continuous, as required.\end{proof}

Given a presentable stable category $\cC,$ we will call a full subcategory $\cD\subset\cC$ localizing if $\cD$ is presentable stable and the inclusion functor commutes with colimits. Equivalently, this means that $\cD$ is generated as a cocomplete stable subcategory by a set of objects (or equivalently by a single object).

\subsection{(Co)localizations and homological epimorphisms}

We will frequently use the notion of a Verdier quotient for (small or large) stable categories, and a closely related notion of a homological epimorphism. To avoid confusion we give precise definitions, specifying the assumptions on the categories and the functors.

\begin{defi}\label{defi:localizations_and_hom_epi}
Let $F:\cA\to\cB$ be an exact functor between stable categories. We denote by $\ker(F)\subset \cA$ the full subcategory of objects $x\in\cA$ such that $F(x) = 0.$

1) Suppose that $\cA$ and $\cB$ are presentable categories. We say that $F$ is a localization (also known as a left Bousfield localization) if it has a fully faithful right adjoint $F^R:\cB\to\cA.$ Similarly, we say that $F$ is a colocalization (also known as a right Bousfield localization) if it has a fully faithful left adjoint $F^L:\cB\to\cA.$

2) Suppose that $\cA$ and $\cB$ are idempotent-complete small categories. We say that $F$ is a Verdier localization (up to direct summands) if the induced functor $(\cA/\ker(F))^{\Kar}\to\cB$ is an equivalence. More generally, we say that $F$ is a homological epimorphism if the functor $\Ind(F):\Ind(\cA)\to\Ind(\cB)$ is a localization.
\end{defi}

\begin{remark}\label{rmk:remarks_on_quotients}
1) Suppose that $\cA$ is a large locally small stable category, and $\cC\subset\cA$ is a full stable subcategory. Then the quotient $\cA/\cC$ might be not locally small (for us this would mean that the quotient does not exist). However, if the inclusion has a right adjoint (i.e. $\cC$ is coreflective in $\cA$), then we have $\cA/\cC\simeq \cC^{\perp},$ where $\cC^{\perp}\subset\cA$ is the full subcategory of objects $x$ such that $\Hom_{\cA}(y,x) = 0$ for $y\in \cC.$ Dually, if the inclusion has a left adjoint (i.e. $\cC$ is reflective in $\cA$), then we have $\cA/\cC\simeq ^{\perp}\cC\subset\cA.$

2) In the situation of 1), if $\cC$ is small (but not reflective or coreflective), then the quotient $\cA/\cC$ is locally small. More precisely, we consider $\cA/\cC$ as a directed union of $\cA'/\cC,$ where $\cA'$ runs through small full stable subcategories of $\cA$ containing $\cC.$

3) An accessible exact functor $F:\cA\to\cB$ between presentable stable categories is a localization resp. a colocalization if and only if the full subcategory $\ker(F)\subset\cA$ is coreflective resp. reflective and we have ab equivalence $\bar{F}:\cA/\ker(F)\xto{\sim}\cB.$

4) One can consider more general accessible exact functors $F:\cA\to\cB$ between presentable stable categories such that the category $\ker(F)$ is presentable (a priori it is only accessible) and we have an equivalence $\bar{F}:\cA/\ker(F)\to\cB$ (in particular the quotient is locally small). Such a situation deserves a separate study, but in the present paper we will only deal with the case when $F$ either commutes with colimits (a localization) or with limits (a colocalization).  

5) In $\Cat^{\perf},$ the fiber of $F:\cA\to\cB$ is given by $\ker(F),$ and the cofiber of $F$ is given by $(\cB/\la F(\cA)\ra)^{\Kar},$ where $\la F(\cA)\ra\subset\cB$ is the stable subcategory generated by $F(\cA)$ \cite[Section 5]{BGT}. In particular, $F$ is a Verdier localization if and only if $\cB$ is identified with the cofiber of the fiber of $F.$ On the other hand, $F$ is a homological epimorphism if and only if it is an epimorphism in $\Cat^{\perf}:$ this follows from Proposition \ref{prop:mono_epi_Pr^L} 2).
\end{remark}

We recall that a Verdier localization between small stable categories is a homological epimorphism: this follows from \cite[Theorem 2.1]{Nee92b} and also from \cite[Proposition 5.13]{BGT}. However, the class of homological epimorphisms is much larger. Historically, the first example of a homological epimorphism which is not a Verdier localization is due to Keller \cite[Section2]{Kel94}, it comes from almost mathematics (see also Example \ref{ex:almost_mathematics} below). 

The following stability properties of (co)localizations and Verdier localizations are quite standard, but we give some of the proofs for completeness. For convenience we denote by $\Pr^{\acc}_{\st}$ the category of presentable stable categories and accessible exact functors

\begin{prop}
1) In $\Pr^{\acc}_{\st},$ both localizations and colocalizations are closed under retracts.

2) In $\Cat^{\perf},$ both Verdier localizations and homological epimorphisms are closed under retracts.
\end{prop}

\begin{proof}
1) We give the proof for localizations, and the argument for colocalizations is dual. Suppose that $F:\cA\to\cB$ is a retract of $F':\cA'\to\cB'$ in $\Pr^{\acc}_{\st},$ so we have a diagram of the following form:
\begin{equation*}
\begin{tikzcd}
\cA \ar{r}{i}\ar{d}{F} & \cA' \ar{r}{p}\ar{d}{F'} & \cA \ar{d}{F}\\
\cB \ar{r}{j} & \cB' \ar{r}{\pi} & \cB.
\end{tikzcd}
\end{equation*}
Suppose that $F'$ is a localization, and let $F^{'R}$ be its (fully faithful) right adjoint. Then the right adjoint to $F$ exists, and it is given by $F^R = p\circ F^{'R}\circ j.$ The adjunction counit for $F$ is given by the composition of isomorphisms $F\circ F^R\cong F\circ p\circ F^{'R}\circ j\cong \pi\circ F'\circ F^{'R}\circ j\xto{\sim} \pi\circ j\cong \Id_{\cB}.$ Hence, $F$ is fully faithful.

2) The assertion for Verdier localizations is clear: the formation of fibers and cofibers in $\Cat^{\perf}$ is compatible with retracts. The assertion about homological epimorphisms follows directly from 1). 
\end{proof}

\begin{prop}\label{prop:localizations_stable_under_pullbacks} \cite{RSW25}
The class of localizations is closed under pullbacks in $\Pr^L_{\st},$ and the class of colocalizations is closed under pullbacks in $\Pr^R.$
\end{prop}

\begin{proof}
This is a special case of \cite[Lemma 2.7]{RSW25}.
\end{proof}



\subsection{Monomorphisms and epimorphisms of presentable stable categories}
\label{ssec:mono_epi_pres}

Here we show that the monomorphisms and epimorphisms in $\Pr^L_{\st}$ and $\Pr^{LL}_{\st}$ are as expected: these are respectively the fully faithful functors and the localizations.

\begin{prop}\label{prop:mono_epi_Pr^L} Let $F:\cC\to\cD$ be a continuous functor between presentable stable categories. 
	
	1) $F$ is a monomorphism in $\Pr^L_{\st}$ if and only if $F$ is fully faithful. 
	
	2) $F$ is an epimorphism in $\Pr^L_{\st}$ if and only if $F$ is a localization.\end{prop}

\begin{proof} 1) The ``if'' part is clear. For the ``only if'' part, we see that the map of groupoids
	$$\cC^{\simeq}\simeq \Fun(\Sp,\cC)^{\simeq}\to \Fun(\Sp,\cD)^{\simeq}\simeq\cD^{\simeq}$$ is a monomorphism, hence $F$ is fully faithful (since $\cC$ and $\cD$ are stable).
	
	2) Again, the ``if'' part is clear. For the ``only if'' part, note that $\cD\sqcup_{\cC}^{\cont}\cD\simeq\cD\times_{\cC}\cD,$ where both functors $\cD\to\cC$ are $F^R.$ Hence, the functor $\cD\to\cD\times_{\cC}\cD$ is an equivalence, in particular, it is fully faithful. Hence, for any $x,y\in\cD$ the square
	$$
	\begin{tikzcd}
		\cD(x,y) \ar{r}\ar{d} & \cD(x,y)\ar{d}\\
		\cD(x,y) \ar{r} & \cC(F^R(x),F^R(y))
	\end{tikzcd}
	$$
	is bicartesian, thus $F^R$ is fully faithful and $F$ is a localization.\end{proof}

\begin{prop}\label{prop:mono_epi_Pr^LL}Let $F:\cC\to\cD$ be a strongly continuous functor between presentable stable categories.
	
	1) $F$ is a monomorphism in $\Pr^{LL}_{\st}$ if and only if $F$ is fully faithful. 
	
	2) $F$ is an epimorphism in $\Pr^{LL}_{\st}$ if and only if $F$ is a localization.
\end{prop}

\begin{proof} 1) The ``if'' part is clear. We prove the ``only if'' part.
	
	Suppose that $F$ is a monomorphism. Then for a presentable stable category $\cE$ and strongly continuous functors $G_1,G_2:\cE\to\cC,$ we have
	$\Hom(G_1,G_2)\xto{\sim}\Hom(F\circ G_1,F\circ G_2).$ By adjunction, we have
	\begin{equation}\label{eq:vanishing_Hom_comp} \Hom(G_1,\Cone(\Id_{\cC}\to F^RF)\circ G_2)=0.\end{equation}
	Put $\Phi:=\Cone(\Id_{\cC}\to F^RF),$ and consider the semi-orthogonal gluing $\cT:=\cC\oright_{\Phi}\cC$ (Definition \ref{defi:gluing_general}). Since $\Phi$ is continuous, the inclusions of the two copies of $\cC$ are strongly continuous. Denote by $\pi_1,\pi_2:\cT\to\cC$ the semi-orthogonal projections (respectively, the left and the right adjoints to the inclusions). Then $\pi_1$ and $\pi_2$ are strongly continuous. Applying \eqref{eq:vanishing_Hom_comp} to $\cE=\cT,$ $G_1=\pi_1[-1],$ $G_2=\pi_2,$ we obtain
	$$0=\Hom(\pi_1[-1],\Phi\circ\pi_2)\cong\Hom(\pi_1[-1]\circ\pi_2^R,\Phi)\cong\Hom(\Phi,\Phi),$$
	hence $\Phi=0,$ which exactly means that $F$ is fully faithful.  `
	
	2) follows from Proposition \ref{prop:mono_epi_Pr^L} since the functor $\Pr_{\st}^{LL}\to\Pr_{\st}^L$ is conservative and commutes with colimits.
\end{proof}

\subsection{Well generated triangulated categories}
\label{ssec:well_generated}

The notion of presentability for a stable $\infty$-category $\cC$ can be formulated in terms of its triangulated homotopy category $\h\cC.$ Namely, $\cC$ is presentable if and only if the triangulated category $\h\cC$ is well generated in the terminology of Neeman \cite{Nee01}. This is however not immediately straightforward because of the difference of the notions of $\kappa$-compact objects in the two contexts. We will recall the key definitions and then formulate the precise comparison statements. We note that a simpler characterization of well generated categories was given by H. Krause in \cite{Kr01}. However, it will be convenient for us to use the original definition from \cite{Nee01}.

Let $T$ be a triangulated category. Recall that $T$ is said to satisfy (TR5) if it has all small direct sums (coproducts). We assume from now on that this holds, and all direct sums are assumed to be small. Note that $T$ is automatically idempotent complete. A full subcategory $S\subset T$ is called {\it localizing} if $S$ is a triangulated subcategory closed under all direct sums.

Given a regular cardinal $\kappa,$ an object $x\in T$ is called {\it $\kappa$-small} if for any family $\{y_j\}_{j\in J}$ of objects of $T$ and for any morphism $f:x\to\biggplus[j\in J]y_j$ there exists a subset $J'\subseteq J$ with $|J'|<\kappa,$ such that $f$ factors through $\bigoplus\limits_{j\in J'}y_j$ \cite[Definition 4.1.1]{Nee01}. We denote by $T^{(\kappa)}\subset T$ the full subcategory of $\kappa$-small objects (note that $T^{(\kappa)}$ is not necessarily essentially small). Note that $T^{(\kappa)}$ is an idempotent-complete triangulated subcategory, and it is closed under $\kappa$-small direct sums in $T$ \cite[Lemmas 4.1.5 and 4.1.6]{Nee01}. 

By definition, $x\in T$ is $\omega$-small if and only if $x$ is compact, i.e. the functor $\Hom_T(x,-)$ commutes with small direct sums. The category of compact objects in $T$ is also denoted by $T^c,$ so we have $T^c = T^{(\omega)}.$

Now the notion of a $\kappa$-compact object in $T$ is more subtle, and it is based on the following definition of a $\kappa$-perfect class.

\begin{defi}\cite[Definition 3.3.1]{Nee01} \label{defi:kappa_perfect_class}
Let $T$ be a triangulated category satisfying (TR5): all small direct sums exist. Let $\kappa$ be a regular cardinal.

A class of objects $S\subset Ob(T)$ is said to be $\kappa$-perfect if the following conditions hold.
\begin{enumerate}[label=(\roman*),ref=(\roman*)]
	\item We have $0\in S.$ \label{zero_in_kappa_perfect}
	\item Let $x\in S,$ and let $\{y_j\}_{j\in J}$ be a collection of objects of $T,$ where $|J|<\kappa.$ Let $f:x\to\biggplus[j\in J]y_j$ be a map. Then there exist objects $z_j\in S,$ $j\in J,$ together with maps $f_j:z_j\to y_j,$ and a map $x\to\biggplus[j\in J]z_j$ such that the composition
	\begin{equation*}
	x\to \biggplus[j\in J]z_j\xto{(g_j)}\biggplus[j\in J] y_j
	\end{equation*}
	equals $f.$ \label{kappa_perfect_surjectivity}
	\item Let $x\in S,$ and let $\{f_j:z_j\to y_j\}_{j\in J}$ be a collection of morphisms in $T,$ where $|J|<\kappa$ and $z_j\in S$ for all $j\in J.$ Let $x\to \biggplus[j\in J]z_j$ be a map such that the composition
	\begin{equation*}
	x\to \biggplus[j\in J]z_j\xto{(f_j)}\biggplus[j\in J] y_j
	\end{equation*} 
	vanishes. Then for each $j\in J$ it is possible to factor $f_j$ as
	\begin{equation*}
	z_j\xto{g_j} w_j\xto{h_j} y_j 
	\end{equation*}
	such that $w_j\in S$ and the composition
	\begin{equation*}
		x\to \biggplus[j\in J]z_j\xto{(g_j)}\biggplus[j\in J] w_j
	\end{equation*} 
	already vanishes. \label{kappa_perfect_injectivity}
\end{enumerate}
\end{defi}

Note that any class $S\subset Ob(T)$ containing $0$ is automatically $\omega$-perfect by the same argument as in \cite[Example 3.3.16]{Nee01}. In general if $S\subset T$ is a $\kappa$-perfect class, then for any regular cardinal $\lambda$ the class of objects of $\la S\ra^{\lambda}$ is also $\kappa$-perfecct \cite[Lemma 3.3.7]{Nee01}. Here $\la S\ra^{\lambda}$ is the smallest triangulated subcategory of $T$ containing $S$ and closed under $\lambda$-small direct sums.

For a triangulated category $T$ satisfying (TR5) and for a regular cardinal $\kappa$ Neeman defines the class of {\it $\kappa$-compact} objects to be the largest $\kappa$-perfect class $S\subset Ob(T)$ which is contained in $Ob(T^{(\kappa)})$ \cite[Definition 4.2.2]{Nee01}. By \cite[Corollary 3.3.10]{Nee01} this largest class exists, and by \cite[Corollaries 3.3.13 and 3.3.14]{Nee01} the objects of $S$ form an idempotent-complete triangulated subcategory of $T,$ closed under $\kappa$-small direct sums. It is denoted by $T^{\kappa}.$ Note that $T^{\omega}=T^{(\omega)}=T^c.$

For a regular cardinal $\kappa,$ a triangulated category $T$ is called {\it $\kappa$-compactly generated} if $T$ satisfies (TR5) and there exists a small set of objects $S\subset T$ such that $S\subset T^{\kappa}$ and for any non-zero $x\in T$ there exists $y\in S,$ $n\in \Z$ and a non-zero map $y[n]\to x.$ This definition is equivalent to \cite[Definition 1.15]{Nee01} in view of \cite[Proposition 8.4.1]{Nee01}: $T$ is automatically generated by $S$ as a localizing subcategory. Moreover, for any regular cardinal $\lambda\geq\kappa$ we have $T^{\lambda}=\la S\ra^{\lambda},$ and this category is essentially small \cite[Lemma 4.4.5, Proposition 3.2.5]{Nee01}.

Finally, a triangulated category is {\it well generated} if it is $\kappa$-compactly generated for some regular cardinal $\kappa.$

Now if $\cC$ is a cocomplete stable category, then $\h\cC$ is a triangulated category satisfying (TR5). It is clear that for any regular cardinal $\kappa$ we have an inclusion $\h(\cC^{\kappa})\subset (\h \cC)^{(\kappa)},$ i.e. $\kappa$-compact objects of $\cC$ are $\kappa$-small in $\h\cC.$ For $\kappa=\omega$ we have $\h (\cC^{\omega})=(\h \cC)^{\omega}.$ However, for general regular $\kappa$ it is not immediately clear what is the relation between the usual notion of $\kappa$-compactness in $\cC$ and Neeman's notion of $\kappa$-compactness in $\h\cC.$ We explain this relation in the following proposition, which is conceptually similar to \cite[Lemma 6.2.5]{Nee01} and \cite[Theorem C]{Kr01}. We formulate the result for a general cocomplete stable category, but we deal with small $\kappa$-perfect classes in $\h \cC$ to avoid set-theoretic issues.

\begin{prop}\label{prop:two_notions_of_kappa_compactness}
Let $\cC$ be a cocomplete stable category and let $\kappa$ be a regular cardinal. Let $\cD\subset\cC$ be an essentially small full stable subcategory closed under $\kappa$-small coproducts. Consider the (continuous) functor $F:\Ind_{\kappa}(\cD)\to\cC,$ induced by the inclusion $\iota:\cD\hto\cC.$ Then the following are equivalent.
\begin{enumerate}[label=(\roman*),ref=(\roman*)]
	\item We have $\h \cD\subset (\h \cC)^{(\kappa)}$ and the objects of $\h\cD$ form a $\kappa$-perfect class in $\h\cC$ (hence $\h\cD\subset (\h\cC)^{\kappa}$). \label{kappa_perfect_in_kappa_small}
	\item The functor $F:\Ind_{\kappa}(\cD)\to\cC$ is strongly continuous (hence $\cD\subset \cC^{\kappa}$). \label{functor_from_Ind_kappa_is_strcont}
\end{enumerate}
In particular, if the category $(\h \cC)^{\kappa}$ is essentially small, then we have an inclusion $(\h\cC)^{\kappa}\subset \h(\cC^{\kappa}).$
\end{prop} 

\begin{proof}
Before proving the equivalence we make some observations. First, we may and will assume that $\cD$ is small. Since $\cD$ has $\kappa$-small colimits, the category $\Ind_{\kappa}(\cD)$ is $\kappa$-presentable. Since the inclusion $\iota:\cD\to\cC$ commutes with $\kappa$-small colimits, the functor $F:\Ind_{\kappa}(\cD)\to\cC$ is indeed continuous. The presentability of $\Ind_{\kappa}(\cD)$ implies that $F$ has a right adjoint $F^R:\cC\to\Ind_{\kappa}(\cD).$ We will identify $\Ind_{\kappa}(\cD)$ with the category $\Fun^{\kappa\hy\lex}(\cD^{op},\Sp)$ of $\kappa$-left exact functors. Under this identification the functor $F^R$ is given by $F^R(x)=\Hom_{\cC}(\iota(-),x).$ Note that the inclusion $\Fun^{\kappa\hy\lex}(\cD^{op},\Sp)\to\Fun^{\ex}(\cD^{op},\Sp)$ commutes with $\kappa$-filtered colimits, i.e. these colimits are computed objectwise. We also note that the conclusion in brackets in \ref{kappa_perfect_in_kappa_small} follows from the definition of $(\h \cC)^{\kappa}.$ 
If $(\h \cC)^{\kappa}$ is essentially small, then denoting by $\cD$ the full subcategory of $\cC$ with $\h \cD = (\h \cC)^{\kappa},$ we see that $\cD$ satisfies \ref{kappa_perfect_in_kappa_small}. Assuming the equivalence, we conclude that $(\h\cC)^{\kappa}\subset \h(\cC^{\kappa}).$ Now we prove the implications.

\Implies{kappa_perfect_in_kappa_small}{functor_from_Ind_kappa_is_strcont}. It suffices to prove that $F^R:\cC\to\Ind_{\kappa}(\cD)$ commutes with all direct sums. Let $\{x_j\in\cC\}_{j\in J}$ be a family of objects of $\cC.$ By assumption we have $\h \cD\subset (\h \cC)^{(\kappa)},$ hence by the above observations we have
\begin{equation*}
F^R(\biggplus[j\in J] x_j)\cong\indlim[\substack{J'\subset J\\ |J'|<\kappa}] F^R(\biggplus[j\in J'] x_j).
\end{equation*}
Thus, we may and will assume that $J$ is $\kappa$-small. For each $j\in J$ the category $\cD_{/x_j}$ is $\kappa$-filtered. We choose a $\kappa$-directed poset $I_j$ with a cofinal functor $I_j\to \cD_{/x_j},$ $i_j\mapsto (y_{i_j}\to x_j).$ Then the product $\prodd[j\in J]I_j$ is also a $\kappa$-directed poset, and we have
\begin{equation*}
\biggplus[j\in J]F^R(x_j)\cong \indlim[(i_j)_j\in\prod\limits_{j\in J}I_j] \Hom_{\cD}(-,\biggplus[j\in J] y_{i_j})
\end{equation*}
(recall that $\kappa$-directed colimits in $\Fun^{\kappa\hy\lex}(\cD^{op},\Sp)$ are computed objectwise). Therefore, we need to prove that for $z\in\cD$ the map of spectra
\begin{equation}\label{eq:map_of_spectra_comparison}
\indlim[(i_j)_j\in\prod\limits_{j\in J}I_j] \Hom_{\cD}(z,\biggplus[j\in J] y_{i_j})\to \Hom_{\cC}(\iota(z),\biggplus[j\in J]x_j)
\end{equation}
is an equivalence. Now recall that by assumption the objects of $\h\cD$ form a $\kappa$-perfect class in $\h \cC.$ The condition \ref{kappa_perfect_surjectivity} of Definition \ref{defi:kappa_perfect_class} (together with the cofinality of $I_j\to\cD_{/x_j}$) implies that the map \eqref{eq:map_of_spectra_comparison} induces epimorphisms on homotopy groups. Similarly, the condition \ref{kappa_perfect_injectivity}  of Definition \ref{defi:kappa_perfect_class} implies that \eqref{eq:map_of_spectra_comparison} induces monomorphisms on homotopy groups. We conclude that \eqref{eq:map_of_spectra_comparison} is indeed an equivalence, which proves that $F^R:\cC\to\Ind_{\kappa}(\cD)$ commutes with all direct sums.

To be more precise, we used not only the cofinality of the functors $I_j\to \cD_{/x_j},$ but also the fullness of the functors $\h (\cD_{/x_j})\to (\h \cD)_{/x_j}.$

\Implies{functor_from_Ind_kappa_is_strcont}{kappa_perfect_in_kappa_small}. Since $F$ is strongly continuous, it preserves $\kappa$-compact objects, so we have $\cD\subset\cC^{\kappa},$ in particular $\h\cD\subset (\h \cC)^{(\kappa)}.$ To see that the objects of $\h \cD$ form a $\kappa$-perfect class in $\h \cC$ we simply reverse the above argument.
\end{proof}

The following is now immediate.

\begin{cor}\label{cor:presentable_and_well_generated}
Let $\cC$ be a cocomplete stable category. Then $\cC$ is presentable if and only if the triangulated category $\h \cC$ is well generated.

More precisely, if $\kappa$ is a regular cardinal, then $\cC$ is $\kappa$-presentable if and only if $\h \cC$ is $\kappa$-compactly generated. In this case for any regular cardinal $\lambda\geq \kappa$ we have equality $\h (\cC^{\lambda}) = (\h\cC)^{\lambda}$ of full subcategories of $\h \cC.$
\end{cor}

\begin{proof}
Suppose that $\cC$ is $\kappa$-presentable. Then the stable subcategory $\cC^{\kappa}\subset\cC$ is essentially small and is closed under $\kappa$-small colimits. The functor $\Ind_{\kappa}(\cC^{\kappa})\to\cC$ is an equivalence, hence by Proposition \ref{prop:two_notions_of_kappa_compactness} we have $\h(\cC^{\kappa})\subset (\h\cC)^{\kappa}.$ Since $\h\cC$ is generated by $\h(\cC^{\kappa})$ as a localizing subcategory, we conclude that $\h\cC$ is $\kappa$-compactly generated and 
$\h(\cC^{\kappa}) = (\h\cC)^{\kappa}.$ 

Conversely, suppose that $\h \cC$ is $\kappa$-compactly generated. Denote by $\cD\subset \cC$ the full subcategory with the same objects as $(\h \cC)^{\kappa}\subset \h \cC.$ Since  $\cD$ is essentially small, by Proposition \ref{prop:two_notions_of_kappa_compactness} we have $\cD\subset\cC^{\kappa}.$ Also, $\cD$ generates $\cC$ via colimits (since $\h \cD$ generates $\h \cC$ as a localizing subcategory). It follows that $\cC$ is $\kappa$-presentable.

Finally, if $\cC$ is $\kappa$-presentable and $\lambda\geq \kappa$ is a regular cardinal, then $\cC$ is also $\lambda$-presentable, hence by the above we have $\h (\cC^{\lambda}) = (\h\cC)^{\lambda}.$
\end{proof}

For completeness we give an example showing that in general the inclusion $(\h \cC)^{\kappa}\subset \h (\cC^{\kappa})$ is strict (if $\kappa>\omega$). In the following proposition for a presentable stable category $\cD$ we use the identification $\Pro(\cD)\simeq \Fun^{\acc}(\cD,\Sp)^{op}.$ We also consider complexes of abelian groups as (Eilenberg-Mac Lane) spectra. 

\begin{prop}
Let $\mk$ be a field and let $\kappa$ be an uncountable regular cardinal. Denote by $j:D(\mk)\hto \Pro(D(\mk))$ the inclusion. Choose a $\mk$-vector space $U$ of dimension $\geq\kappa,$ i.e. $U\cong\biggplus[I]\mk,$ $|I|\geq \kappa.$ Consider the functor $\Phi_1:D(\mk)\to \Pro(D(\mk)),$ such that its opposite functor $\Phi_1^{op}:D(\mk)^{op}\to\Fun^{\acc}(D(\mk),\Sp)$ is given by
\begin{equation*}
\Phi_1^{op}(V)(W) = \Hom_{\mk}(V, \biggplus[\N]\Hom_{\mk}(U,W)).
\end{equation*}
Consider another functor
\begin{equation*}
\Phi_2:D(\mk)\to \Pro(D(\mk)),\quad \Phi_2(V) = j(\prodd[\N] V).
\end{equation*}
Define $\Phi:D(\mk)\to\Pro(D(\mk))$ to be the direct sum $\Phi = \Phi_1\oplus \Phi_2.$ Consider the lax pullback
\begin{equation*}
\cC = D(\mk)\laxtimes{\Phi,\Pro(D(\mk)),j} D(\mk),
\end{equation*}
i.e. the category of triples $(V,W,f),$ where $V,W\in D(\mk)$ and $f:\Phi(V)\to j(W)$ (Definition \ref{defi:gluing_general}). Then $\cC$ is presentable and we have $(\h \cC)^{\kappa} = 0$ but $\cC^{\kappa}\ne 0.$ More precisely, $\cC^{\kappa}$ consists of objects of the form $(0,W,0)$ with $W\in D(\mk)^{\kappa}.$
\end{prop}

\begin{proof}
For convenience we identify $D(\mk)$ with the full subcategory of pro-constant objects of $\Pro(D(\mk)),$ i.e. with the essential image of $j.$ Note that the inclusion commutes with colimits.

 We first show that $\cC$ is cocomplete, i.e. it has infinite coproducts. By construction, $\Phi_1$ commutes with all colimits and $\Phi_2$ takes values in pro-constant objects. Hence, for a family $\{(V_s,W_s,f_s)\}_{s\in S}$ of objects of $\cC$ the pushout 
\begin{equation}\label{eq:pushout}
X = \Phi(\biggplus[s]V_s)\sqcup_{\biggplus[s]\Phi(V_s)} \biggplus[s] W_s\in \Pro(D(\mk))
\end{equation}
is pro-constant, and we have
\begin{equation}\label{eq:formula_for_coproduct}
\biggplus[s] (V_s,W_s,f_s) = (\biggplus[s]V_s,X,g),
\end{equation}
where $g:\Phi(\biggplus[s]V_s)\to X$ is the structural map to the pushout.

Next, we show that $\cC$ is presentable. By construction, $\Phi$ commutes with $\omega_1$-filtered colimits. 
In particular, the $\omega_1$-filtered colimits in $\cC$ are computed objectwise. This shows that the object $(0,\mk,0)$ is $\omega_1$-compact in $\cC.$ Next, let $\lambda$ be a regular cardinal such that $\lambda$ is strictly greater than the dimensions of $U$ and $\prodd[\N]\mk$ (hence $\lambda>\kappa$ by assumption). Then the object $(\mk,0,0)$ is $\lambda$-compact. We see that $\cC$ is generated by colimits by $\lambda$-compact objects, hence $\cC$ is $\lambda$-presentable.

Next, we describe the $\kappa$-compact objects of $\cC$ (in the $\infty$-categorical sense). We already observed that $(0,\mk,0)\in \cC^{\omega_1}\subset \cC^{\kappa},$ hence also $(0,W,0)\in\cC^{\kappa}$ for $W\in D(\mk)^{\kappa}.$ On the other hand, if $W\in D(\mk)$ is not $\kappa$-compact, then $(0,W,0)\not\in\cC^{\kappa}.$ It remains to prove that if $(V,W,f)\in \cC^{\kappa}$ then $V = 0.$ Suppose that $V\ne 0.$ Testing the $\kappa$-compactness for $\kappa$-filtered systems of objects of the form $(0,X,0),$ we see that the pro-object $\Cone(\Phi(V)\to W)$ corresponds to a functor $D(\mk)\to\Sp$ which commutes with $\kappa$-filtered colimits. Equivalently, the latter functor is left Kan-extended from $D(\mk)^{\kappa}.$ This means that $\Cone(\Phi(V)\to W)\in\Pro(D(\mk)^{\kappa})\subset \Pro(D(\mk)).$ In particular, the image of $\Phi(V)$ in $\Pro(D(\mk)/D(\mk)^{\kappa})$ is pro-constant. Since for some $n\in\Z$ the object $\mk[n]$ is a direct summand of $V,$ we see that the image of $\Phi(\mk)$ in $\Pro(D(\mk)/D(\mk)^{\kappa})$ is pro-constant. But modulo pro-constant objects $\Phi(\mk)$ is identified with $\Phi_1(\mk)\cong \proolim[n]\bigoplus\limits_{k=0}^n U.$ By assumption on $U$ the image of the latter pro-object in $\Pro(D(\mk)/D(\mk)^{\kappa})$ is not pro-constant, a contradiction.

Finally, we show that $(\h \cC)^{\kappa} = 0.$ Assume the contrary. By Proposition \ref{prop:two_notions_of_kappa_compactness} we have $(\h \cC)^{\kappa} \subset \h (\cC^{\kappa}).$ By the above description of $\cC^{\kappa},$ any nonzero $\kappa$-compact object generates $\h(\cC^{\kappa})$ via $\kappa$-small coproducts, cones and shifts. Hence, we must have $(\h \cC)^{\kappa} = \h (\cC^{\kappa}).$ The localizing subcategory $\la \cC^{\kappa}\ra\subset \cC$ is identified with $D(\mk)$ via the fully faithful inclusion $\iota:D(\mk)\to \cC,$ $\iota(W) = (0,W,0).$ Applying Proposition \ref{prop:two_notions_of_kappa_compactness} again, we conclude that $\iota$ must be strongly continuous. This is however not the case: we have $\iota^R(V,W,f) = W,$ hence by \eqref{eq:pushout} and \eqref{eq:formula_for_coproduct} the functor $\iota^R$ does not commute with countable coproducts (since $\Phi$ does not by construction). This gives a contradiction, therefore $(\h \cC)^{\kappa} = 0,$ as stated.
\end{proof}

Let now $\cC$ and $\cD$ be cocomplete stable categories, and let $\kappa$ be a regular cardinal. We summarize the above discussion.

\begin{itemize}
\item $\cC$ is $\kappa$-presentable if and only if $\h\cC$ is $\kappa$-compactly generated;

\item An exact functor $F:\cC\to\cD$ is continuous if and only if the functor $\h F:\h\cC\to\h\cD$ commutes with coproducts;

\item If $\cC$ is presentable and $x\in \h\cC$ is $\kappa$-compact in the above sense, then $x$ is $\kappa$-compact in $\cC$ in the $\infty$-categorical sense. The converse holds if $\cC$ is $\kappa$-presentable.

\item If $\cC$ is presentable, then a full subcategory $\cE\subset\cC$ is localizing if and only if $\h\cE\subset\h\cC$ is generated as a localizing subcategory by a set of objects. 
\end{itemize}

Recall that for a triangulated category $T$ an additive functor $F:T\to\Ab$ is called homological if for any exact triangle $x\to y\to z$ the sequence $F(x)\to F(y)\to F(z)$ is exact. It is convenient to call a functor $F:T^{op}\to\Ab$ with the same property cohomological. By \cite[Theorem 1.17]{Nee01}, a well generated triangulated category $T$ satisfies Brown representability for contravariant functors: an additive functor $F:T^{op}\to\Ab$ is representable if and only if $F$ is cohomological and commutes with products. In particular, an exact functor between well generated triangulated categories $\Phi:T\to S$ has a right adjoint iff $\Phi$ commutes with coproducts. 


\section{General theory of dualizable categories}
\label{sec:general_theory_dualizable}

\subsection{Dualizable and compactly assembled categories}
\label{ssec:dualizable_cats}


We first recall the notion of a compactly assembled (not necessarily stable and not necessarily cocomplete) category from \cite{Lur18}.

\begin{defi} Let $\cC$ be an accessible $\infty$-category with filtered colimits. Then $\cC$ is called compactly assembled if the colimit functor $\colim:\Ind(\cC)\to\cC$ has a left adjoint.\end{defi}

If the $\infty$-category $\cC$ is compactly assembled, then we denote by $\hat{\cY}_{\cC}:\cC\to \Ind(\cC)$ (or simply $\hat{\cY}$) the (fully faithful) functor which is the left adjoint to $\colim.$ Note that the right adjoint to $\colim$ is always given by the Yoneda embedding $\cY_{\cC}:\cC\to\Ind(\cC).$ Any compactly generated (i.e. $\omega$-accessible) category is compactly assembled. Indeed, if $\cC\simeq\Ind(\cA),$ then the functor $\hat{\cY}$ is obtained by applying $\Ind$ to the Yoneda embedding $\cA\to\Ind(\cA).$

Compactly assembled ordinary categories were studied in great detail by Johnstone and Joyal \cite{JJ} under the name ``continuous categories''. This is a generalization of the notion of a continuous poset, see Subsection \ref{ssec:cont_posets} for a more precise definition and  \cite{GHKLMS, BH81} for a detailed account.

We will need the following notion of strong continuity for functors between compactly assembled categories.

\begin{defi}\label{defi:strcont_general}1) Let $\cC$ and $\cD$ be compactly assembled accessible categories. A functor $F:\cC\to\cD$ is called strongly continuous if $F$ is continuous and the natural transformation $\hat{\cY}_{\cD}\circ F\to \Ind(F)\circ\hat{\cY}_{\cC}$ is an isomorphism.

2) We denote by $\Compass$ the $(\infty,1)$-category of compactly assembled accessible categories, where the $1$-morphisms are given by strongly continuous functors. 
\end{defi}

For stable compactly assembled categories, the two notions of strong continuity of exact functors are equivalent by Proposition \ref{prop:strong_continuity_via_hat_Y}.


A presentable stable $\infty$-category $\cC$ is called {\it dualizable} if it is a dualizable object of the symmetric monoidal category $(\Pr^L_{\st},\otimes)$ (or equivalently of the symmetric monoidal homotopy category $\h \Pr^L_{\st}$). That is, there exists an $\infty$-category $C^{\vee}\in \Pr_{\st}^L,$ and continuous functors $\coev:\Sp\to \cC^{\vee}\otimes\cC$ and $\ev:\cC\otimes\cC^{\vee}\to\Sp$ such that the compositions
$$\cC\xto{\Id\boxtimes\coev}\cC\otimes\cC^{\vee}\otimes\cC\xto{\ev\boxtimes\Id}\cC,\quad \cC^{\vee}\xto{\coev\boxtimes\Id}\cC^{\vee}\otimes\cC\otimes\cC^{\vee}\xto{\Id\boxtimes\ev}\cC^{\vee}$$
are isomorphic to the identity.

It turns out that a presentable stable category $\cC$ is compactly assembled if and only if $\cC$ is dualizable. This is part of the following result due to Lurie.

\begin{prop}\label{prop:Lurie_criterions_dual}\cite[Proposition D.7.3.1]{Lur18} Let $\cC$ be a presentable stable $\infty$-category. The following are equivalent.

\begin{enumerate}[label=(\roman*),ref=(\roman*)]
\item The $\infty$-category $\cC$ is dualizable as an object of $\Pr^L_{\st}.$ \label{Lur1}

\item $\cC$ is compactly assembled. \label{Lur2}

\item $\cC$ is a retract in $\Pr^L_{\st}$ of a compactly generated category. \label{Lur3}

\item Let $\cD$ be another presentable stable category and let $F:\cD\to\cC$ be a localization, i.e. the right adjoint $F^R:\cC\to\cD$ is fully faithful. Then $F$ has a continuous section.\label{Lur4}
\end{enumerate}
\end{prop}

\begin{remark}\label{rem:functor_-^vee} Given a dualizable category $\cC,$ we have a natural functor $(-)^{\vee}:\cC^{op}\to\cC^{\vee},$ which can be described (at least) in two ways. The first one is the following: given $x\in\cC,$ the object $x^{\vee}\in\cC^{\vee}$ corresponds to the continuous functor $$\Hom_{\Ind(\cC)}(\cY(x),\hat{\cY}(-)):\cC\to\Sp.$$

Another description is via the universal property: we have 
$$\Hom_{\cC^{\vee}}(y,x^{\vee})\cong \Hom_{\cC^{\vee}\otimes\cC}(y\boxtimes x,\coev(\bS)).$$

In the special case $\cC=\Mod\hy A,$ where $A$ is an $\bE_1$-ring, we have $\cC^{\vee}\simeq A\hy\Mod$ -- the category of left $A$-modules. For a right $A$-module $M$ we have $M^{\vee}=\Hom_A(M,A)$ -- the usual dual left module of a right module. 
\end{remark}

\begin{remark}Essentially rephrasing the previous remark: given a dualizable category $\cC$ and an object $x\in\cC,$ the ind-object $\hat{\cY}(x)$ corresponds to the functor $$\ev(x\boxtimes (-)^{\vee}):\cC^{op}\to\Sp.$$\end{remark}

\begin{prop}\label{prop:strong_continuity_via_hat_Y} Let $F:\cC\to\cD$ be a continuous functor between dualizable categories. Then the right adjoint $F^R:\cD\to\cC$ is continuous if and only if the natural transformation $\hat{\cY}_{\cD}\circ F\to \Ind(F)\circ\hat{\cY}_{\cC}$ is an isomorphism.\end{prop}

\begin{proof}The corresponding natural transformation between the right adjoints is given by $\varphi:\colim\circ\Ind(F^R)\to F^R\circ\colim$ of functors $\Ind(\cD)\to\cC.$ Tautologically, $\varphi$ is an isomorphism if and only if $F^R$ commutes with filtered colimits, or equivalently with all colimits (since $F^R$ is exact).\end{proof}

We now define the category of dualizable categories.

\begin{defi}The category $\Cat_{\st}^{\dual}$ is the (non-full) subcategory of $\Pr^L_{\st}$ formed by dualizable categories and strongly continuous functors. Equivalently, $\Cat_{\st}^{\dual}$ is the full subcategory of $\Pr^{LL}_{\st}$ formed by dualizable categories.\end{defi}

In particular, the category $\Cat_{\st}^{\dual}$ is locally small (since so is $\Pr^{LL}_{\st}$). Note that $\Cat_{\st}^{\dual}\subset \Compass$ is a (non-full) subcategory. 

We make an immediate observation about short exact sequences.

\begin{prop}\label{prop:ses_Pr^LL} Consider a short exact sequence $0\to\cC\xto{F}\cD\xto{G}\cE\to 0$ in $\Pr^{LL}_{\st}.$ If $\cD$ is dualizable, then so are $\cC$ and $\cE.$\end{prop}

\begin{proof}Indeed, both $\cC$ and $\cE$ are retracts of $\cD$ in $\Pr^L_{\st}:$ the right adjoint functors $F^R$ and $G^R$ are continuous, and we have $F^R F\cong\Id_{\cC},$ $G G^R\cong\Id_{\cE}.$\end{proof}

We need one more characterization of dualizable categories, which is easy to deduce from Proposition \ref{prop:Lurie_criterions_dual} and which is implicit in the proof of \cite[Proposition D.7.3.1]{Lur18}.

\begin{prop}\label{prop:dualizable_via_ses} Let $\cC$ be a presentable stable $\infty$-category. Then $\cC$ is dualizable if and only if there exists a small stable $\infty$-category $\cA$ and a strongly continuous fully faithful functor $\cC\to\Ind(\cA).$ Equivalently, there is a short exact sequence in $\Cat_{\st}^{\dual}$ of the form \begin{equation}\label{eq:presentation_as_kernel}0\to \cC\to\Ind(\cA)\to\Ind(\cA')\to 0.\end{equation}\end{prop}

\begin{proof}If we have a strongly continuous fully faithful functor $\Phi:\cC\to\Ind(\cA),$ then $\cC$ is dualizable by Proposition \ref{prop:ses_Pr^LL}.

Conversely, suppose that $\cC$ is dualizable. By the criterion \ref{Lur2} of Proposition \ref{prop:Lurie_criterions_dual}, we have a fully faithful strongly continuous functor $\hat{\cY}:\cC\to \Ind(\cC).$ As explained in \cite[Proof of Proposition D.7.3.1]{Lur18} the essential image of $\hat{\cY}$ is contained in $\Ind(\cC_0)$ for some small stable subcategory $\cC_0\subset\cC.$ We obtain a fully faithful strongly continuous functor $\cC\to\Ind(\cC_0),$ as required.\end{proof}

We make the following observation about compact objects, which follows easily from theorems of Neeman and Thomason.

\begin{prop}\label{prop:compact_objects_in_quotients} Let $\cC$ be a dualizable category, and let $\cA\subset \cC^{\omega}$ be a full stable subcategory. Then we have an equivalence $(\cC^{\omega}/\cA)^{\Kar}\xto{\sim}(\cC/\Ind(\cA))^{\omega}.$ For any compact object $x\in (\cC/\Ind(\cA))^{\omega},$ the object $x\oplus x[1]$ can be lifted to a compact object of $\cC.$\end{prop}

\begin{proof}If $\cC$ is compactly generated, then the result follows from \cite[Theorem 2.1]{Nee96}. 

In general, we choose a fully faithful strongly continuous functor $\cC\to\cD,$ where $\cD$ is compactly generated. Then the equivalence $(\cD^{\omega}/\cA)^{\Kar}\xto{\sim}(\cD/\Ind(\cA))^{\omega}$ induces an equivalence $(\cC^{\omega}/\cA)^{\Kar}\xto{\sim}(\cC/\Ind(\cA))^{\omega}.$ 

The second assertion of the proposition follows from Thomason's theorem \cite[Theorem 2.1]{Th}.\end{proof}

Next, we show that any dualizable category is $\omega_1$-compactly generated. This follows from the following result which is (to the best of our knowledge) due to H.~Krause.

\begin{prop}\label{prop:generating_the_kernel}Let $F:\cA\to \cB$ be a functor between small stable $\infty$-categories. Then the category $\ker(\Ind(F):\Ind(\cA)\to \Ind(\cB))$ is $\omega_1$-compactly generated, and its $\omega_1$-compact objects are exactly of the form
\begin{equation}\label{eq:objects_of_ker}\inddlim(x_1\xto{f_1} x_2\xto{f_2}\dots),\quad \text{ where }F(f_n)=0\text{ in }\cB\text{ for }n\geq 1.\end{equation}
\end{prop}

\begin{proof}This is a special case of \cite[Theorem 7.4.1]{Kr}.\end{proof}

\begin{remark}Note that the statement of \cite[Theorem 7.4.1]{Kr} implies that for any uncountable regular cardinals $\kappa<\lambda$ the (non-full) inclusion $\Pr^L_{\st,\kappa}\to\Pr^L_{\st,\lambda}$ commutes with fibers (i.e. with kernels). It is easy to deduce that this inclusion commutes with all $\kappa$-small limits. 

Indeed, it is sufficient to prove this for fiber products and for $\kappa$-small products. Given $\kappa$-strongly continuous functors $F:\cA\to\cC$ and $G:\cB\to\cC$ between $\kappa$-presentable categories, we can first form the lax fiber product $\cA\laxtimes{\cC}\cB\simeq \cA\oright_{F^R G}\cB$ (Definition \ref{defi:gluing_general}). The functor $F^R G$ is $\kappa$-accessible, hence by Proposition \ref{prop:gluing_presentable} the category $\cA\laxtimes{\cC}\cB$ is $\kappa$-presentable. Consider the functor $H:\cA\laxtimes{\cC}\cB\to\cC,$ given by $H(x,y,\varphi)=\Cone(F(x)\xto{\varphi} G(y)).$ Then $H$ preserves $\kappa$-compact objects. Hence, by \cite[Theorem 7.4.1]{Kr} the kernel $\ker(H) \simeq \cA\times_{\cC}\cB$ is $\kappa$-compactly generated by those $\kappa$-compact objects of $\cA\laxtimes{\cC}\cB$ which belong to $\ker(H).$ Finally, if $\{\cC_i\}_{i\in I}$ is a $\kappa$-small collection of $\kappa$-presentable categories, then the category $\prodd[i]\cC_i$ is also $\kappa$-presentable, and $(\prodd[i]\cC_i)^{\kappa}=\prodd[i]\cC_i^{\kappa}.$\end{remark}



\begin{cor}\label{cor:dualizable_is_aleph_1_generated} Any dualizable category $\cC$ is $\omega_1$-compactly generated. In particular, $\Cat_{\st}^{\dual}$ is naturally a full subcategory of $\Pr^{LL}_{\st,\omega_1}$ (and a non-full subcategory of $\Pr^L_{\st,\omega_1}$).\end{cor}

\begin{proof}This is a direct consequence of Propositions \ref{prop:dualizable_via_ses} and \ref{prop:generating_the_kernel}.
\end{proof}

\begin{cor}\label{cor:hat_Y_into_Ind_C_omega_1} For a dualizable category $\cC,$ the essential image of the functor $\hat{\cY}:\cC\to\Ind(\cC)$ is contained in the full subcategory $\Ind(\cC^{\omega_1}).$\end{cor}

\begin{proof}Since $\cC$ is $\omega_1$-compactly generated, the colimit functor factors through $\Ind(\cC^{\omega_1}):$
$$\colim:\Ind(\cC)\to\Ind(\cC^{\omega_1})\to\cC.$$
Here the first functor is the right adjoint to the inclusion. Hence, the functor $\hat{\cY}$ factors through $\Ind(\cC^{\omega_1}).$
\end{proof}

\begin{remark}In general, let $\cC$ be a dualizable category and let $\cA\subset\cC$ be a full small stable subcategory. 
Then the image of $\hat{\cY}:\cC\to\Ind(\cC)$ is contained in (the essential image of) $\Ind(\cA)$ if and only if the functor $\Ind(\cA)\to\cC$ is a localization. This condition is much stronger than the condition that $\cA$ generates $\cC$ by colimits.

For example, let $\mk$ be a field and consider the $3$-dimensional $\mk$-algebra $A=\mk[x,y]/(x^2,xy,y^2),$ and let $\cB\subset D(\cA)$ be the stable subcategory generated by $A^*$ -- $\mk$-linear dual of $A.$ Then $\cB$ generates $D(A)$ by colimits, but the functor $\Ind(\cB)\to D(A)$ is not a localization.\end{remark}

We give a proof that any compactly assembled category (not necessarily cocomplete) is $\omega_1$-compactly generated.

\begin{prop}\label{prop:comp_ass_omega_1_gen} Let $\cC$ be an accessible compactly assembled category. Then $\cC$ is $\omega_1$-accessible. Moreover, the $\omega_1$-compact objects of $\cC$ are exactly of the form $x=\indlim[](x_1\to x_2\to\dots),$ where for each $n$ the morphism $\cY(x_n)\to\cY(x_{n+1})$ factors through $\hat{\cY}(x_{n+1}).$\end{prop}

\begin{proof}Suppose that $x=\indlim[x]x_n$ is a sequential colimit as above. Then $$\hat{\cY}(x)\cong\indlim[n]\hat{\cY}(x_n)\cong\indlim[n]\cY(x_n).$$ Hence, the object $\hat{\cY}(x)$ is $\omega_1$-compact in $\Ind(\cC).$ Since $\hat{\cY}$ is fully faithful and commutes with filtered colimits, we see that $x$ is $\omega_1$-compact in $\cC.$ Conversely, suppose that $x$ is $\omega_1$-compact. Since the right adjoint to $\hat{\cY}$ commutes with filtered colimits, we see that $\hat{\cY}(x)$ is $\omega_1$-compact in $\Ind(\cC).$ Hence, $\hat{\cY}(x)=\inddlim(x_1\to x_2\to\dots)$ for some sequence $(x_n)_{n\geq 1}.$ Then $\hat{\cY}(x)\cong\indlim[n]\hat{\cY}(x_n).$ Now for $n\geq 1$ the object $\cY(x_n)$ is compact in $\Ind(\cC),$ hence the map $\cY(x_n)\to \hat{\cY}(x)$ factors through some $\hat{\cY}(x_m),$ where $m\geq n.$ But the mapping space $\Map(\cY(x_n),\hat{\cY}(x))$ is identified with the colimit of $\Map(\cY(x_n),\cY(x_k)).$ Hence, for some $k>m$ the composition $\cY(x_n)\to\hat{\cY}(x_m)\to\cY(x_m)\to\cY(x_k)$ is homotopic to the transition map $\cY(x_n)\to \cY(x_k).$ The composition $\hat{\cY}(x_m)\to\cY(x_m)\to\cY(x_k)$ is homotopic to the composition $\hat{\cY}(x_m)\to \hat{\cY}(x_k)\to \cY(x_k).$ We conclude that for each $n\geq 1$ there is some $k>n$ such that the transition map $\cY(x_n)\to\cY(x_k)$ factors through $\hat{\cY}(x_k).$ Passing to a subsequence, we may assume that $k=n+1,$ as required.

It remains to show that any object of $\cC$ is an $\omega_1$-directed colimit of $\omega_1$-compact objects. Let $x\in\cC,$ and $\hat{\cY}(x)=\inddlim[i\in I]x_i,$ where $I$ is directed. Let $J''$ be the poset of order-preserving maps $f:\N\to I.$ Then $J''$ is directed and
$$x=\indlim[f\in J'']\indlim[n]x_{f(n)}.$$
Let $J'\subset J''$ be the subset of maps $f:\N\to I,$ such that for each $n\geq 0$ the map $\cY(x_{f(n)})\to\cY(x_{f(n+1)})$ factors through $\hat{\cY}(x_{f(n+1)}).$ Arguing as above, we see that the inclusion $J'\to J''$ is cofinal, in particular $J'$ is directed. 

Denote by $J$ the poset of equivalence classes $J'/\sim,$ where $f\sim g$ if $f(n)=g(n)$ for $n\gg 0.$ The partial order on $J$ is the following: given elements $\bar{f},\bar{g}\in J$ with representatives $f,g\in J',$ we have $\bar{f}\leq\bar{g}$ if $f(n)\leq g(n)$ for $n\gg 0.$ Then the poset $J$ is $\omega_1$-directed, and the map $J'\to J$ is cofinal. Moreover, the functor $J'\to \cC,$ $f\mapsto \indlim[n]x_{f(n)},$ factors uniquely through $J.$ Therefore, we have
$$x=\indlim[\bar{f}\in J]\indlim[n]x_{f(n)}.$$ This shows that $x$ is an $\omega_1$-directed colimit of $\omega_1$-compact objects of $\cC.$\end{proof}


\begin{remark} One can show that the category $\Compass$ is $\omega_1$-presentable. More precisely, it is generated by colimits by a single $\omega_1$-compact object, given by the poset $\R\cup\{+\infty\}$ with the usual order, see Section \ref{sec:Urysohn}, more precisely, Remark \ref{rem:Urysohn_for_Compass}. In fact, the category $\Compass$ shares a lot of good properties with the category $\Cat_{\st}^{\dual}.$ However, in this paper we will mostly restrict our attention to dualizable categories, i.e. compactly assembled stable categories.\end{remark}

\subsection{Analogy between dualizable categories and compact Hausdorff spaces}
\label{ssec:analogy_dualizable_comphaus}

Before continuing the study of dualizable categories we explain an interesting analogy between the category $\Cat_{\st}^{\dual}$ and the category $\comphaus^{op}$ -- the opposite of the category of compact Hausdorff spaces. This analogy is explained by Table \ref{tab:analogy} below, and it is helpful for an intuitive understanding of the theory of dualizable categories. After the table we summarize the plan for the rest of this section and for Sections \ref{sec:presentability_Cat_dual} and \ref{sec:Urysohn} below, illustrating some ideas using familiar results on compact Hausdorff spaces.

Denote by $\profin\subset\comphaus$ the full subcategory of profinite sets. We have a functor $\comphaus^{op}\to\Cat_{\st}^{\dual},$ $X\mapsto \Sh(X;\Sp)$ (Example \ref{ex:sheaves}, Proposition \ref{prop:Shv_is_dualizable}). By Corollary \ref{cor:when_sheaves_comp_gen}, the category $\Sh(X;\Sp)$ is compactly generated if and only if $X$ is profinite. 

In the table below for a topological space $X$ we denote by $Q(X)$ its set of quasi-components, with the quotient topology. Recall that quasi-components of $X$ are equivalence classes of points, where $x\sim y$ if any open-closed subset of $X$ containing $x$ also contains $y.$ The Grothendieck's AB axioms (and their strong and weak versions) mentioned in the table are explained in Subsection \ref{ssec:AB_axioms}. The table also mentions the category $\Shv_{\geq 0}(\R;\Sp),$ which is first introduced in Section \ref{sec:presentability_Cat_dual} with a direct description, and it is also studied in Section \ref{sec:Urysohn} and Subsections \ref{ssec:sheaves_on_R}, \ref{ssec:K_theory_of_products}. One of its equivalent definitions is the following: it is the full subcategory of $\Shv(\R;\Sp)$ which consists of sheaves with singular support in $\R\times\R_{\geq 0}\subset T^*\R$ \cite{KS, KS18, RS}.
 
\begin{longtable}{ | m{20em} | m{20em} | } 
	\caption{Dualizable categories and compact Hausdorff spaces}
	\label{tab:analogy}\\
	\hline
	The category $\Cat_{\st}^{\dual}$ of dualizable categories & The opposite category $\comphaus^{op}$ of compact Hausdorff spaces\\
	\hline
	The full subcategory $\Cat_{\st}^{\cg}\subset\Cat_{\st}^{\dual}$ of compactly generated categories & The full subcategory $\profin^{op}\subset\comphaus^{op}$ of profinite sets\\
	\hline
	The functor $\Cat_{\st}^{\dual}\to\Cat_{\st}^{\cg},$ $\cC\mapsto \Ind(\cC^{\omega})$ & The functor $\comphaus^{op}\to\profin^{op},$ $X\mapsto Q(X)$ (the space of quasi-components)\\
	\hline
	For $\cC\in\Cat_{\st}^{\dual},$ we have $(\cC/\Ind(\cC^{\omega}))^{\omega}=0$ (by Proposition \ref{prop:compact_objects_in_quotients}) & For $X\in\comphaus^{op},$ we have $Q(X)\cong \pi_0(X),$ i.e. quasi-components are connected\\
	\hline
	The category $\Cat_{\st}^{\cg}\simeq\Cat^{\perf}$ is generated by colimits by the single compact object $\Sp^{\omega}$ (by Proposition \ref{prop:presentability_of_Pr_kappa}) & The category $\profin^{op}\simeq\Ind(\Fin^{op})$ is generated by colimits by the single compact object $\{0,1\}$ \\
	\hline
	The category $\Cat_{\st}^{\dual}$ is generated by colimits by the single $\omega_1$-compact object $\Sh_{\geq 0}(\R;\Sp)$ (by Theorem \ref{th:Urysohn_dualizable}) & The category $\comphaus^{op}$ is generated by colimits by the single $\omega_1$-compact object $[0,1]$\\
	\hline
	The category $\Cat_{\st}^{\dual}$ satisfies weak (AB4*) (Proposition \ref{prop:AB4*_for_Cat^dual}) and (AB6) (Proposition \ref{prop:AB6_for_Cat^dual}) & The category $\comphaus^{op}$ satisfies weak (AB4*) and (AB6)\\
	\hline
	Weak (AB5) holds in $\Cat_{\st}^{\dual}:$ the class of fully faithful functors is closed under filtered colimits (by Proposition \ref{prop:weak_AB5_Pr^LL}) & Weak (AB5*) holds in $\comphaus:$ the class of surjective maps is closed under cofiltered limits\\
	\hline
	The functor $(-)^{\omega}:\Cat_{\st}^{\dual}\to\Cat^{\perf}$ commutes with filtered colimits (Proposition \ref{prop:compact_objects_in_filtered_colimits}) & The functor $Q(-):\comphaus\to\profin$ commutes with cofiltered limits\\
	\hline
	The inclusion $\Cat_{\st}^{\cg}\to\Cat_{\st}^{\dual}$ commutes with infinite products (by Proposition \ref{prop:products_in_Cat^dual}) & The inclusion $\profin\to\comphaus$ commutes with infinite coproducts\\
	\hline
	In $\Cat_{\st}^{\dual}$ sequential colimits do not commute with finite limits (Corollary \ref{cor:seq_colimits_finite_limits}) & In $\comphaus$ sequential limits do not commute with finite colimits\\
	\hline
	Localizations (=epimorphisms by Corollary \ref{cor:epi_in_Cat^dual}) in $\Cat_{\st}^{\dual}$ & Injective maps (=monomorphisms) in $\comphaus$\\
	\hline
	Fully faithful functors (=monomorphisms by Corollary \ref{cor:mono_in_Cat^dual}) in $\Cat_{\st}^{\dual}$ & Surjective maps (=epimorphisms) in $\comphaus$\\
	\hline
	Localizations are stable under pullbacks in $\Cat_{\st}^{\dual}$ (by Proposition \ref{prop:nice_pullbacks}) & Injective maps are stable under pushouts in $\comphaus$\\
	\hline
	Fully faithful functors are stable under pushouts in $\Cat_{\st}^{\dual}$ (by Proposition \ref{prop:pushouts_of_mono}) & Surjective maps are stable under pullbacks in $\comphaus$\\
	\hline
	The non-full inclusion $\Cat_{\st}^{\dual}\to\Pr^L_{\st}$ is conservative and commutes with colimits (by Proposition \ref{prop:colimits_of_dualizable}) & The (faithful) forgetful functor $\comphaus\to\Set$ is conservative and commutes with limits\\
	\hline
	Pullbacks of localizations are preserved by the functor $\Cat_{\st}^{\dual}\to\Pr^L_{\st}$ (by Proposition \ref{prop:nice_pullbacks}) & Pushouts of injective maps are preserved by the forgetful functor $\comphaus\to\Set$\\
	\hline
\end{longtable}

First we explain the general plan for the subsequent subsections of this section, and how they are related to the statements in the left column. We start with some examples of dualizable (and more generally compactly assembled) categories in Subsection \ref{ssec:examples}. In Subsection \ref{ssec:cont_stab} we explain the so-called continuous stabilization: a universal way to make a compactly assembled accessible category into a dualizable category. This construction will be used repeatedly below, in particular in the study of localizing invariants. Then in Subsection \ref{ssec:duality_covariant} we explain that the operation of taking the dual category is a covariant involution of $\Cat_{\st}^{\dual}.$ This has no analogue for compact Hausdorff spaces, but there is a similar operation for a more general class of the so-called stably compact spaces, see for example \cite{Leh26}. In Subsections \ref{ssec:dual_via_compact} and \ref{ssec:AB6} we explain two more equivalent characterizations of dualizability. The latter one is quite abstract via the axiom (AB6), and the former one is more ``explicit'' via the so-called compact morphisms. Intuitively, these are the morphisms which behave like they factor through a compact object, and for the dualizability we need to have sufficiently many such morphisms in a certain sense. In Subsection \ref{ssec:autom_dualizable} we explain a condition for ``automatic dualizability'', which in particular implies that the images in $\Cat_{\st}^{\dual}$ can be computed naively (namely, as a localizing subcategory generated by the usual essential image). This could be included into the above table, but this might lead to a confusion: in $\comphaus$ the images coincide with coimages, which is not the case for $\Cat_{\st}^{\dual}.$ In Subsection \ref{ssec:Calkin} we introduce the Calkin construction for dualizable categories, which in the compactly generated case is simply the quotient by the compact objects (up to an idempotent completion). This is crucial for the definition and the study of localizing invariants in Part \ref{part:loc_invar_dualizable}. In Subsection \ref{ssec:colimits_of_dualizable} we study the colimits in $\Cat_{\st}^{\dual}.$ These are nicely behaved: they are the same as in $\Pr^L_{\st},$ which by the above table corresponds to the nice behavior of limits in $\comphaus$ (the forgetful functor to $\Set$ commutes with limits). The subtle point is the weak (AB5) axiom: the class of fully faithful functors is closed under filtered colimits. This allows to prove that the functor $(-)^{\omega}:\Cat_{\st}^{\dual}\to\Cat^{\perf}$ commutes with filtered colimits, which by the above table corresponds to the commutation of $Q:\comphaus\to\profin$ with cofiltered limits. We also prove a simple but useful statement: the class of fully faithful functors is closed under pushouts. In Subsection \ref{ssec:SOD} we study the semi-orthogonal decompositions in $\Cat_{\st}^{\dual}:$ both finite and infinite. This has no interesting analogue for compact Hausdorff spaces, but again there is an analogue for stably compact spaces. In Subsections \ref{ssec:finite_limits_of_dualizable} and \ref{ssec:general_limits_Cat^dual} we study the limits in $\Cat_{\st}^{\dual},$ first the finite limits and then the general limits. These are much more complicated than the colimits, just like in $\comphaus$ the colimits are much more complicated than limits. We obtain a general description of limits, and also prove a useful fact: the class of localizations is closed under pullbacks, and these pullbacks can be computed naively (i.e. in $\Pr^L_{\st}$). In Subsection \ref{ssec:fpqc_descent} we give a non-trivial example of a computation of a limit in $\Cat_{\st}^{\dual},$ in particular proving the fpqc descent for the assignment $R\mapsto D(R),$ where $R$ is an ordinary commutative ring. Finally, in Subsection \ref{ssec:products_of_dualizable} we study the products in $\Cat_{\st}^{\dual},$ and prove certain results showing that they are nicely behaved. Note that in $\comphaus$ an infinite coproduct is given by $\beta(\sqcup_i X_i)$ -- the Stone-\v{C}ech compactification of the disjoint union. Although for a general topological space this compactification is extremely complicated, in our situation it behaves in a controllable way. In particular, if all $X_i$ are profinite, so is $\beta(\sqcup_i X_i),$ and in terms of Boolean algebras (of open-closed subsets) this is simply the operation of taking the infinite product. This corresponds to the fact that the inclusion $\Cat_{\st}^{\cg}\to\Cat_{\st}^{\dual}$ commutes with infinite products. The proofs of the weak (AB4*) and the (AB6) axioms for $\Cat_{\st}^{\dual}$ are subtle, but the analogous statements for $\comphaus^{op}$ are relatively elementary (Corollary \ref{cor:AB4*_AB6_for_comphaus}).

In Section \ref{sec:presentability_Cat_dual} we prove that the category $\Cat_{\st}^{\dual}$ is $\omega_1$-presentable, and also describe the $\kappa$-compact objects for an uncountable regular cardinal $\kappa.$ The $\omega_1$-presentability of $\comphaus^{op}$ is explained below. It can also be deduced from the Gelfand duality: it is easy to show that the category of commutative unital $C^*$-algebras is $\omega_1$-presentable. In Section \ref{sec:Urysohn} we prove the analogue of Urysohn's lemma for dualizable categories. More precisely, in $\comphaus^{op}$ the closed unit interval $[0,1]$ is an $\omega_1$-compact generator (via colimits), and for $\Cat_{\st}^{\dual}$ the same holds for the object $\Sh_{\geq 0}(\R;\Sp).$ 
Of course, there are many other $\omega_1$-compact generators, both in $\comphaus^{op}$ and in $\Cat_{\st}^{\dual}.$ 

Now we prove the non-standard statements from the right column of Table \ref{tab:analogy}. Consider the category $\Ban_1$ of Banach vector spaces over $\R,$ where the morphisms are contractive linear maps (i.e. the bounded maps of norm $\leq 1$). We have a functor $\comphaus^{op}\to\Ban_1,$ $X\mapsto C(X,\R).$

\begin{lemma}\label{lem:axioms_for_Ban_1} The category $\Ban_1$ satisfies weak (AB4*) and (AB6).\end{lemma}

\begin{proof}The product in $\Ban_1$ of a family $\{(V_i,||\cdot ||_i)\}_{i\in I}$ is given by the subspace of the naive product, formed by elements $(v_i)_{i\in I}$ such that $\sup_{i\in I}||v_i||_i<\infty,$ with the norm given by $$||(v_i)_{i\in I}||=\sup_i ||v_i||_i.$$ Epimorphisms in $\Ban_1$ are maps with a dense image. It follows that $\Ban_1$ satisfies weak (AB4*), i.e. the class of epimorphisms is closed under products. 
	
	Filtered colimits in $\Ban_1$ are obtained by taking the completion (of the separation of) of the naive filtered colimit. Here the norm on the naive (filtered) colimit 
	$\indlim[j\in J]^{naive}V_j$ is defined by $$||v||=\inf_{(j\in J,w\in V_j,w\mapsto v)}||w||_j,\quad v\in \indlim[j\in J]^{naive}V_j.$$
	To show that (AB6) holds in $\Ban_1,$ it is convenient to use the presentable compactly assembled category $\seminorm_1$ (of seminormed vector spaces and contractive maps) from Example \ref{ex:seminorm} below. Since $\seminorm_1$ is compactly assembled, it satisfies (AB6) by Remark \ref{rem:compass_via_AB5_and_AB6}. It remains to observe that the completion functor $\seminorm_1\to\Ban_1$ commutes with all colimits and with infinite products.\end{proof}

\begin{lemma}\label{lem:cont_functions_properties} The functor $C(-,\R):\comphaus^{op}\to\Ban_1$ is conservative, detects epimorphisms and commutes with products and with filtered colimits.\end{lemma}

\begin{proof} The functor $C(-,\R)$ is conservative by Urysohn's lemma. By the Tietze extension theorem, the functor $C(-,\R)$ detects epimorphisms (and also detects regular epimorphisms).
	
	Now, the product of $\{X_i\}_{i\in I}$ in $\comphaus^{op}$ is given by $\beta(\coprod_i X_i)$ -- the Stone-\v{C}ech compactification of the disjoint union. For any topological space $Y,$ continuous functions on $\beta Y$ are identified with bounded continuous functions on $Y.$ It follows that the functor $C(-,\R)$ commutes with products.
	
	Finally, consider a codirected system $J^{op}\to \comphaus,$ $j\mapsto X_j,$ where $J$ is directed, and put $X=\prolim[j]X_j.$ By Stone-Weierstrass theorem, the union of the images $\im(C(X_j,\R)\to C(X,\R))$ is a dense subalgebra. It remains to observe that the map from the naive colimit $\indlim[j] C(X_j,\R)$ to $C(X,\R)$ is isometric. Indeed, this follows from the fact that for any $j\in J$ and for any neighborhood of $U\supset \im(X\to X_j)$ there is some $k\geq j$ such that $\im(X_k\to X_j)\subset U.$ It follows that the functor $C(-,\R)$ commutes with filtered colimits.
\end{proof}

\begin{cor}\label{cor:AB4*_AB6_for_comphaus} The category $\comphaus^{op}$ satisfies weak (AB4*) and (AB6).\end{cor}

\begin{proof} This follows directly from Lemmas \ref{lem:axioms_for_Ban_1} and \ref{lem:cont_functions_properties}. 
\end{proof} 

\begin{prop}The functor $\profin\to\comphaus$ commutes with coproducts.\end{prop}

\begin{proof}Consider a family of profinite sets $\{X_i\}_{i\in I}.$ We need to show that the space $\beta(\coprod_i X_i)$ is totally disconnected. It suffices to show that the locally constant $\R$-valued functions on $\beta(\coprod_i X_i)$ are dense in the space of all continuous functions.
	
	Consider a function $f:\beta(\coprod_i X_i)\to\R$ corresponding to a family of continuous functions $f_i:X_i\to\R,$ such that for some $C>0$ we have $||f_i||\leq C,$ $i\in I.$ Take some $\veps>0,$ and choose an integer $N>\frac{C}{\veps}.$ Since $X_i$ is profinite, we can find a locally constant function $g_i:X_i\to\R$ such that $||g_i-f_i||\leq \veps$ and $g_i$ takes values in the finite set $\{\frac{kC}{N}: -N\leq k\leq N\}.$ Then the corresponding function $g:\beta(\coprod_i X_i)\to\R$ is also locally constant and $||g-f||\leq\veps.$\end{proof}

\begin{prop}\label{prop:comphaus_op_presentability} The category $\comphaus^{op}$ is generated by colimits by the closed unit interval $[0,1],$ which is an $\omega_1$-compact object. In particular, the category $\comphaus^{op}$ is $\omega_1$-presentable, and the $\omega_1$-compact objects are precisely the metrizable spaces.\end{prop}

\begin{proof} To see that the object $[0,1]$ is $\omega_1$-compact in $\comphaus^{op},$ it suffices to observe that $\omega_1$-filtered colimits in $\Ban_1$ can be computed naively, i.e. the naive colimit is already complete (and separated). By Urysohn's lemma, the functor $C(-,[0,1]):\comphaus^{op}\to\Set$ is conservative. Hence, the object $[0,1]$ generates the category $\comphaus^{op}$ by colimits.
	
	Now, the full subcategory of metrizable spaces in $\comphaus$ contains the final object and is closed under fiber products and countable products. Hence, metrizable spaces in $\comphaus^{op}$ are closed under $\omega_1$-small colimits. It remains to show that they are generated by $[0,1]$ via $\omega_1$-small colimits.
	
	
	Let $X$ be a metrizable compact Hausdorff space. Choose an injective map $X\to[0,1]^{\N}.$ Denote by $Y$ the pushout $[0,1]^{\N}\sqcup_X [0,1]^{\N}.$ Then $Y$ is also metrizable,  
	and we choose an injective map $Y\to [0,1]^{\N}.$ Composing it with the two maps $[0,1]^{\N}\to Y,$ we obtain two self-maps $f,g:[0,1]^{\N}\to [0,1]^{\N}.$ The equalizer $\Eq(f,g)$ is isomorphic to $X.$\end{proof}

In the proof of the following proposition for a topological space $X$ we denote by $\clopen(X)$ the Boolean algebra of open-closed subsets of $X.$

\begin{prop}\label{prop:Q_commutes_with_cofiltered_limits} The functor $Q(-)\cong\pi_0(-):\comphaus\to\profin$ commutes with cofiltered limits.\end{prop}

\begin{proof}Consider a codirected system $J^{op}\to\comphaus,$ $j\mapsto X_j,$ and put $X=\prolim[j]X_j,$ $\pi_j:X\to X_j,$ $\pi_{kj}:X_k\to X_j$ for $k\geq j.$ Passing to Boolean algebras, we need to show that the map $\indlim[j]\clopen(X_j)\to\clopen(X)$ is an isomorphism. Given an open-closed $U\subset X,$ consider its characteristic function $\chi_U.$ Then for some $j$ we have a continuous function $f:X_j\to \R$ such that $||\pi_j^*f-\chi_U\leq \frac{1}{3}\||.$ Then $\{x: f(x)\geq \frac{2}{3}\}$ is an open-closed subset of $X_j,$ and its preimage in $X$ equals $U.$
	
	Suppose that $V\subset X_j$ is an open-closed subset such that $\pi_j^{-1}(V)=\emptyset.$ Then $\pi_j(X)\subset X_j\setminus V,$ hence for some $k\geq j$ we have $\pi_{kj}(X_k)\subset X_j\setminus V,$ i.e. $\pi_{kj}^{-1}(V)=\emptyset.$\end{proof}

\begin{prop}Sequential limits in $\comphaus$ do not commute with finite colimits.\end{prop}

\begin{proof}Consider the surjective map $\{0,1\}^{\N}\to [0,1],$ $(a_n)_{n\geq 0}\mapsto\sum\limits_{n\geq 0}a_n 2^{-n-1}.$ Put $X=\{0,1\}^{\N}$ and $Y=X\times_{[0,1]} X.$ Then $[0,1]\cong \Coeq(Y\toto X).$
	
	Put $X_n=\{0,1\}^n,$ 
	Then the diagram $Y\toto X$ is the limit of $(Y\toto X_n)_n,$ but $\Coeq(Y\toto X_n)$ is a one-element set for $n\geq 0.$ Hence, the sequential limits do not commute with coequalizers in $\comphaus.$\end{proof}



\subsection{Examples of compactly assembled categories, stable and non-stable} 
\label{ssec:examples}

We first give a few examples of dualizable (=compactly assembled stable) categories. 

\begin{example}Let $X$ be a quasi-compact quasi-separated scheme. By \cite[Theorem 3.1.1]{BVdB03} the derived category $D_{qc}(X)$ of complexes of $\cO_X$-modules with quasi-coherent cohomology (considered as a stable $\infty$-category) is compactly generated, hence dualizable.\end{example}

\begin{example}\label{ex:almost_mathematics} (Faltings \cite{Fa02}, Gabber-Ramero \cite{GR03}) Let $A$ be an (associative unital) ring, and let $J\subset A$ be a two-sided ideal such that $J^2=J$ (we do not assume that $J$ or $J\tens{A}J$ is flat as a left or right module). Then $\Mod\hy A/J\subset\Mod\hy A$ is a Serre subcategory. Denote by $\Mod_a\hy A$ the Serre quotient $\Mod\hy A/\Mod\hy (A/J).$ Then the derived category $D(\Mod_a\hy A)$ is dualizable. 

Indeed, the subcategory $D_{\Mod\hy A/J}(A)\subset D(A)$ is (presentable and) closed under products and coproducts, hence the inclusion functor has both left and right adjoints. Since $D(\Mod_a\hy A)\simeq D(A)/D_{\Mod\hy A/J}(A),$ we have a short exact sequence in $\Cat_{\st}^{\dual}:$
\begin{equation}\label{eq:ses_for_Mod_a_general} 0\to D(\Mod_a\hy A)\to D(A)\xto{\Phi} D_{\Mod\hy A/J}(A)\to 0,\end{equation}
where the functor $\Phi$ is the left adjoint to the inclusion.

It is formal that we have a DG ring $B$ such that $D_{\Mod\hy A/J}(A)\simeq D(B),$ and $\Phi$ is identified with $-\stackrel{\bL}{\tens{A}}B.$ We refer to \cite[Theorem B]{HS24} for the description of $B$ without flatness assumptions on $J$ or $J\tens{A}J$ (in a more general setting of connective ring spectra). We only mention that $B$ is connective (i.e. $H_{<0}(B)=0$), $H_0(B)\cong A/J,$ and $H_0(\Fiber(A\to B))\cong J\tens{A}J$ (the non-derived tensor product).  If $J$ is flat as a right (or left) $A$-module, then we have $B=A/J.$ More generally, if $J\tens{A}J$ is flat as a right (or left) $A$-module, then $B$ is the DG ring $\Cone(J\tens{A}J\to A).$ Here $J\tens{A}J$ is considered as an (idempotent) quasi-ideal in $A.$ More precisely, by Lemma \ref{lem:ideals_and_quasi-ideals} any idempotent quasi-ideal arises in this way assuming its flatness as a right $A$-module.\end{example}

\begin{example}\label{ex:sheaves} Let $X$ be a locally compact Hausdorff space. Then the category $\Sh(X;\Sp)$ of sheaves of spectra is dualizable (note that we do not require hyperdescent for sheaves). More generally, for any dualizable category $\cC$ the category $\Sh(X;\cC)\simeq\Sh(X;\Sp)\otimes\cC$ is dualizable. Even more generally, if $\un{\cC}$ is a presheaf on $X$ with values in $\Cat_{\st}^{\dual},$ then the category $\Sh(X;\un{\cC})$ is dualizable by Proposition \ref{prop:Shv_is_dualizable}. We will prove in Section \ref{sec:sh_loc_comp} that $$K^{\cont}(\Sh(X;\un{\cC}))\cong \Gamma_c(X,(K^{\cont}(\un{\cC}))^{\sharp}),$$ where $(-)^{\sharp}$ denotes the sheafification.\end{example}

\begin{example}\label{ex:nuc} Let $R$ be a (for simplicity) noetherian commutative ring, and let $I\subset R$ be an ideal. Clausen and Scholze \cite{CS} defined the category of nuclear solid modules, which we denote by $\Nuc^{CS}(R_{\hat{I}}),$ where the completion $R_{\hat{I}}$ is considered as an analytic ring in the terminology of loc. cit. This category is dualizable but not compactly generated (unless $I$ is nilpotent). The compact objects are given by the usual category of perfect complexes $\Perf(R_{\hat{I}}).$ Moreover, the category $\Nuc^{CS}(R_{\hat{I}})$ is closely related with another category $\Nuc(R_{\hat{I}})$ which can be defined as an inverse limit in $\Cat_{\st}^{\dual}:$
$$\Nuc(R_{\hat{I}})=\prolim[n]^{\dual}D(R/I^n).$$
We refer to Subsection \ref{ssec:general_limits_Cat^dual} for a discussion of general inverse limits in $\Cat_{\st}^{\dual}.$ We study both categories $\Nuc^{CS}(R_{\hat{I}})$ and $\Nuc(R_{\hat{I}})$ in great detail in \cite{E25a}. In particular, in loc. cit. prove that there is a natural fully faithful strongly continuous functor $\Nuc^{CS}(R_{\hat{I}})\to \Nuc(R_{\hat{I}})$ which induces an equivalence on compact objects, and we have
$$K^{\cont}(\Nuc^{CS}(R_{\hat{I}}))\cong K^{\cont}(\Nuc(R_{\hat{I}}))\cong\prolim[n]K(R/I^n).$$\end{example}

We now give some examples of compactly assembled non-stable categories.

\begin{example}Given an associative unital ring $A,$ the category $\Flat\hy A$ of flat right $A$-modules is compactly generated by Lazard's theorem, hence it is compactly assembled.\end{example}

\begin{example}\label{ex:cont_posets} Let $P$ be a poset considered as a category. Then $P$ is compactly assembled if and only if $P$ is a continuous poset, see Subsection \ref{ssec:cont_posets} for the characterization of such posets. For example, the poset $\R\cup\{+\infty\}$ with the usual order is continuous, and we have $\hat{\cY}(x)=\inddlim[y<x]y.$ Another example is the poset $\Open(X)$ of open subsets of a locally compact Hausdorff space. In this case we have $\hat{\cY}(U)=\inddlim[V\Subset U]V.$\end{example}

\begin{example}Let $X$ be a locally compact Hausdorff space. Then the ordinary category $\Sh(X;\Set)$ is compactly assembled, see \cite{JJ}. By \cite{Lur18}, the $\infty$-category $\Sh(X;\cS)$ of sheaves of spaces is compactly assembled.\end{example}

The following example might be new (to the best of our knowledge).

\begin{example}\label{ex:seminorm} Denote by $\seminorm_1$ the category of seminormed $\R$-vector spaces, where the morphisms are given by contractive linear maps (i.e. the maps $F:X\to Y$ such that $||F(x)||\leq ||x||$). Then $\seminorm_1$ is presentable and compactly assembled. Indeed, the category $\seminorm_1$ has colimits, and it is generated by colimits  by finite-dimensional {\it normed} vector spaces and by the object $(\R;0)$ ($1$-dimensional vector space with zero seminorm). We have $$\hat{\cY}((\R;0))=\inddlim[\veps>0](\R;\veps |\cdot|),$$
and for a finite-dimensional normed space $V$ we have 
$$\hat{\cY}((V;||\cdot||))=\inddlim[c>1](V;c||\cdot||).$$\end{example}

\subsection{Continuous stabilization}
\label{ssec:cont_stab}

We describe the left adjoint to the (non-full) inclusion $\Cat_{\st}^{\dual}\to\Compass.$ If $\cC$ is an accessible $\infty$-category with filtered colimits, we denote by $\Stab^{\cont}(\cC)$ the (presentable stable) category of cofiltered limit-preserving functors $\cC^{op}\to\Sp.$ If $\cC$ is $\kappa$-accessible, we can identify the category $\Stab^{\cont}(\cC)$ with the category of functors $(\cC^{\kappa})^{op}\to\Sp$ which commute with $\kappa$-small cofiltered limits. The inclusion $\Stab^{\cont}(\cC)\to\Fun((\cC^{\kappa})^{op},\Sp)$ commutes with limits and $\kappa$-filtered colimits, hence it has a left adjoint $F\mapsto F^{\sharp}.$

For $x\in\cC,$ we denote by $h_x:(\cC^{\kappa})^{op}\to\Sp$ the ``representable'' presheaf, given by $h_x(y)=\bS[\cC(y,x)],$ $y\in\cC^{\kappa}.$ For any presentable stable category $\cD$ we have an equivalence
\begin{equation}\label{eq:cont_stab_general} \Fun^{L}(\Stab^{\cont}(\cC),\cD)\xto{\sim} \Fun^{\cont}(\cC,\cD),\quad F\mapsto(x\mapsto F(h_x^{\sharp})).\end{equation}

\begin{prop}\label{prop:continuous_stab}Let $\cC$ be a compactly assembled accessible category. Then the category $\Stab^{\cont}(\cC)$ is dualizable. The left adjoint to the inclusion $\Cat_{\st}^{\dual}\to \Compass$ is given by $\cC\mapsto\Stab^{\cont}(\cC).$\end{prop}

\begin{proof} By Proposition \ref{prop:comp_ass_omega_1_gen} the category $\cC$ is $\omega_1$-compactly generated. If $\kappa\geq \omega_1$ is a regular cardinal, then the left adjoint to the inclusion $\Stab^{\cont}(\cC)\to\Fun((\cC^{\kappa})^{op},\Sp)$ has an explicit description:
\begin{equation}\label{eq:sheafification_on_compass_cats}F\mapsto F^{\sharp},\quad F^{\sharp}(x)=\prolim[j]F(x_j),\quad\text{where }x\in\cC^{\kappa},\,\hat{\cY}(x)=\inddlim[j]x_j\in\Ind(\cC^{\omega_1}).\end{equation}

We observe that the (a priori partially defined) functor $\hat{\cY}:\Stab^{\cont}(\cC)\to\Ind(\Stab^{\cont}(\cC))$ is defined on objects of the form $h_x^{\sharp}$ and is given by $$\hat{\cY}(h_x^{\sharp})=\inddlim[j]h_{x_j}^{\sharp},\quad x\in\cC,\,\hat{\cY}(x)=\inddlim[j]x_j.$$  Indeed, let $(F_i)_i$ is an ind-system in $\Stab^{\cont}(\cC),$ and choose a regular cardinal $\kappa\geq \omega_1$ such that $x\in\cC^{\kappa}.$ Then applying \eqref{eq:sheafification_on_compass_cats} we get
$$\Hom(h_x^{\sharp},\indlim[i]F_i)\cong\prolim[j]\indlim[i]F_i(x_j)\cong\Hom(\inddlim[j]h_{x_j}^{\sharp},\inddlim[i]F_i).$$
Since the category $\Stab^{\cont}(\cC)$ is generated by the objects $h_x^{\sharp},$ we conclude that $\Stab^{\cont}(\cC)$ is dualizable and the functor $x\mapsto h_x^{\sharp}$ is strongly continuous.


By the above, if $\cD$ is dualizable, then a continuous functor $F:\Stab^{\cont}(\cC)\to\cD$ is strongly continuous iff we have $\hat{\cY}(F(h_x^{\sharp}))\cong \Ind(F)(\hat{\cY}(h_x^{\sharp})).$ This exactly means that the functor $x\mapsto F(h_x^{\sharp})$ is strongly continuous. This proves the proposition.\end{proof}

Below we will be especially interested in the special case of Proposition \ref{prop:continuous_stab} for the poset $\cC=(\R\cup\{+\infty\},\leq),$ considered as a compactly assembled category.

\begin{remark}\label{rem:Stab_cont_otimes_something} Note that if $\cC$ is an accessible $\infty$-category with filtered colimits, and $\cD$ is a presentable stable category, then the tensor product $\Stab^{\cont}(\cC)\otimes\cD$ is identified with the category of cofiltered limit-preserving functors $\cC^{op}\to\cD.$ This follows from \eqref{eq:cont_stab_general} and from the equivalence
$$\Stab^{\cont}(\cC)\otimes\cD\simeq\Fun^L(\Stab^{\cont}(\cC),\cD^{op})^{op}.$$\end{remark}

\subsection{Duality as a covariant involution on dualizable categories}
\label{ssec:duality_covariant}

Note that the category $\Cat^{\perf}$ has a covariant involution $\cA\mapsto\cA^{op}.$ We observe that $\Cat_{\st}^{\dual}$ also has a covariant involution $(-)^{\vee},$ so that $\Ind(\cA)^{\vee}=\Ind(\cA^{op})$ for $\cA\in\Cat^{\perf}.$

Namely, given a strongly continuous functor $F:\cC\to\cD$ between dualizable categories, we see that the dual functor $F^{\vee}:\cD^{\vee}\to\cC^{\vee}$ has a left adjoint $F^{\vee,L}:\cC^{\vee}\to\cD^{\vee}.$ More precisely, we have $F^{\vee,L}=(F^R)^{\vee},$ where $F^R:\cD\to\cC$ is the right adjoint to $F.$

This gives a well-defined functor $(-)^{\vee}:\Cat_{\st}^{\dual}\to \Cat_{\st}^{\dual},$ which is naturally an involution. In particular, the functor $(-)^{\vee}$ commutes with all limits and colimits in $\Cat_{\st}^{\dual}.$

\begin{remark}In fact, the involution $(-)^{\vee}:\Cat_{\st}^{\dual}\to \Cat_{\st}^{\dual}$ can be naturally extended to an involution $(-)^{\cont\hy op}:\Compass\to\Compass.$ This follows from the following statement, which we will address elsewhere.

Denote by $\Cat^{\dual}$ the non-full subcategory of $\Pr^L$ formed by dualizable objects; the $1$-morphisms in $\Cat^{\dual}$ are functors which have a colimit-preserving right adjoint. We claim that there is a natural equivalence $\Compass\xto{\sim}\Cat^{\dual}.$ This equivalence sends a compactly assembled accessible $\infty$-category $\cC$ to the category $\hat{\cC}$ of cofiltered limit-preserving functors $\cC^{op}\to\cS.$ The inverse functor sends a dualizable (non-stable) category $\cD$ to the category of left exact colimit-preserving functors $\cD^{\vee}\to\cS.$\end{remark}

\begin{remark} Explicitly, the involution $(-)^{\cont\hy op}:\Compass\to\Compass$ can be described as follows. Given $\cC\in\Compass,$ the category $\cC^{\cont\hy op}\subset\Ind(\cC^{op})\simeq\Pro(\cC)^{op}$ is the full subcategory formed by pro-objects $\proolim[i\in I]x_i,$ where $I$ is codirected, such that for any $i\in I$ there exists $j\leq i$ such that the map $\cY(x_j)\to\cY(x_i)$ in $\Ind(\cC)$ factors through $\hat{\cY}(x_i).$

For example, if $X$ is a locally compact Hausdorff space, then $\Open(X)^{\cont\hy op}\cong \msK(X)^{op}$ -- the poset of compact subsets of $X$ with the reverse inclusion order.\end{remark} 

\subsection{Dualizability via compact morphisms}
\label{ssec:dual_via_compact}

We give a criterion for dualizability of a presentable stable category $\cC$ in terms of the homotopy category $\h\cC.$ At least conceptually, most of this subsection is covered by the results of \cite{Kr00, Kr05}, applied to the smashing subcategory $\h\hat{\cY}(\cC)\subset\h\Ind(\cC^{\omega_1})$ (assuming that $\cC$ is dualizable). In particular, the ideal of compact maps corresponds to this smashing subcategory as in \cite[Theorem A]{Kr00} and \cite[Theorem 1]{Kr05}. However, the criterion for dualizability does not follow formally from the results in loc. cit. Definition \ref{defi:compact_morphism} and a version of Theorem \ref{th:dualizability_via_compact_maps} below were explained to me by Dustin Clausen. Proposition \ref{prop:compact_maps_and_wavy_arrows} seems to be new.

Let $\cC$ be a presentable stable category. 

\begin{defi}\label{defi:compact_morphism} A morphism $f:x\to y$ in $\cC$ is called compact if for any ind-object $\inddlim[i]z_i\in\Ind(\cC)$ and for any morphism $g:y\to \indlim[i]z_i,$ the composition $g\circ f$ factors through some $z_i.$\end{defi}

\begin{remark}1) Definition \ref{defi:compact_morphism} also makes sense for non-stable presentable categories. However, in this subsection we consider only stable categories.

2) For an object $x\in\cC,$ compactness of the identity map $\id_x$ is equivalent to the compactness of $x.$
\end{remark}

The following observation is essentially the motivation for Definition \ref{defi:compact_morphism}.

\begin{prop}\label{prop:compact_maps_in_dualizable_cats} Suppose that $\cC$ is dualizable. The following are equivalent for a morphism $f:x\to y$ in $\cC.$

\begin{enumerate}[label=(\roman*),ref=(\roman*)]
\item $f$ is compact. \label{comp1}

\item The map $\cY(x)\to\cY(y)$ in $\Ind(\cC)$ factors through $\hat{\cY}(y).$ \label{comp2}
\end{enumerate}
\end{prop}

\begin{proof}\Implies{comp1}{comp2}. This follows from the definition of compactness, applied to the identity map $y\to\colim(\hat{\cY}(y))=y.$

\Implies{comp2}{comp1}. For any ind-object $\inddlim[i]z_i$ and a map $g:y\to\indlim[i]z_i,$ we can identify $g$ with the map $h:\hat{\cY}(y)\to\inddlim[i]z_i.$ Composing $h$ with the map $\cY(x)\to\hat{\cY}(y),$ we get an element of $\indlim[i]\pi_0(\Hom(x,z_i)).$ It is represented by some map $x\to z_i.$ The composition $x\to z_i\to\indlim[i]z_i$ is homotopic to $g\circ f.$ Hence, $f$ is compact.
\end{proof}

\begin{cor}\label{cor:decomposing_compact_morphisms} In a dualizable category $\cC,$ any compact morphism is (homotopic to) a composition of two compact morphisms.\end{cor}

\begin{proof} Let $f:x\to y$ be a compact morphism, and choose its lift $\tilde{f}:\cY(x)\to\hat{\cY}(y).$ Let $\hat{\cY}(y)=\inddlim[i\in I]y_i,$ where $I$ is filtered. Since $\hat{\cY}$ commutes with colimits, we have $\hat{\cY}(y)\simeq\indlim[i]\hat{\cY}(y_i).$ Hence, the morphism $\tilde{f}$ factors through some $\hat{\cY}(y_{i_0}).$ We obtain compact morphisms $x\to y_{i_0},$ $y_{i_0}\to y,$ and their composition is $f.$\end{proof}

\begin{prop}\label{prop:compactness_of_morphisms_definitions} Let $f:x\to y$ be a morphism in a presentable stable category $\cC.$ The following are equivalent.

\begin{enumerate}[label=(\roman*),ref=(\roman*)]
\item $f$ is compact. \label{compdef1}

\item for any ind-object $\inddlim[i]z_i\in\Ind(\cC)$ such that $\indlim[i]z_i=0,$ the map of abelian groups
$$\indlim[i]\pi_0\Hom(y,z_i)\to \indlim[i]\pi_0\Hom(x,z_i)$$
is zero. \label{compdef2}
\end{enumerate}
\end{prop}

\begin{proof}\Implies{compdef1}{compdef2}. Suppose that $\indlim[i\in I]z_i=0,$ where $I$ is directed, and take some map $h:y\to z_{i_0},$ $i_0\in I.$ Denote by $g$ the identity map
$$g:y\to\indlim[i\geq i_0]\Fiber(y\to z_i)\cong y.$$
Since $f$ is compact, the composition $g\circ f$ factors through some object $\Fiber(y\to z_i),$ $i\geq i_0.$ It follows that the composition $x\xto{f}y\to z_i$ is zero, as required.

\Implies{compdef2}{compdef1}. Consider an arbitrary ind-object $\inddlim[i]z_i\in\Ind(\cC),$ where $I$ is directed. Take a map $g:y\to\indlim[i]z_i.$ Put $w_i:=\Cone(z_i\to\indlim[j]z_j).$ Then $\indlim[i]w_i=0.$ Pick some $i_0\in I,$ and consider the composition $$h:x\xto{f} y\xto{g}\indlim[i]z_i\to w_{i_0}.$$ By \ref{compdef2} for some $i\geq i_0$ the composition $x\xto{h}w_{i_0}\to w_i$ is zero. It follows that the composition $g\circ f$ factors through $z_i.$ Hence, $f$ is compact. \end{proof}

\begin{prop}\label{prop:compact_maps_via_direct_sums} Let $f:x\to y$ be a morphism in a presentable stable category $\cC.$ Suppose that $y$ is $\omega_1$-compact. The following are equivalent.

\begin{enumerate}[label=(\roman*),ref=(\roman*)]
\item $f$ is compact. \label{comp_countable1}

\item for any sequence of objects $z_1,z_2,\dots$ in $\cC,$ the image of the map
$$\pi_0\Hom(y,\biggplus[n]z_n)\to \pi_0\Hom(x,\biggplus[n]z_n)$$
is contained in $\biggplus[n]\pi_0\Hom(x,z_n).$ \label{comp_countable2}
\end{enumerate}
\end{prop}

\begin{proof}\Implies{comp_countable1}{comp_countable2}. This is clear: apply the definition of compactness to the sequence $(\biggplus[1\leq k\leq n]z_k)_{n\geq 1}.$

\Implies{comp_countable2}{comp_countable1}. Consider a directed system $(z_i)_{i\in I},$ and a map $g:y\to\indlim[i]z_i.$ Since $y$ is $\omega_1$-compact, we may assume that $I=\N.$ Indeed, $\omega_1$-compactness of $y$ implies that for some sequence $i_0\leq i_1\leq\dots$ of elements of $I$ the map $g$ factors through $\indlim[n]z_{i_n}.$

Assume $I=\N.$ Recall that we have a cofiber sequence
$$\biggplus[n]z_n\to\biggplus[n]z_n\to\indlim[n]z_n.$$ 
By Corollary \ref{cor:decomposing_compact_morphisms} the morphism $f$ can be factored as $x\xto{f'}w\xto{f''}y,$ where both $f'$ and $f''$ are compact.

Consider the composition 
$$h:y\xto{g}\indlim[n]z_n\to\biggplus[n]z_n[1].$$
By \ref{comp_countable2} the map $h\circ f''$ factors through a finite direct sum $\biggplus[0\leq n\leq k]z_n[1].$ But this implies that $h\circ f''$ is a zero map.

Hence, $g\circ f'':w\to \indlim[n] z_n$ can be lifted to a map $r:w\to\biggplus[n]z_n.$ By \ref{comp_countable2} the composition $r\circ f':x\to\biggplus[n] z_n$ factors through a finite direct sum $\biggplus[0\leq n\leq k]z_n.$ It follows that $g\circ f$ factors through $z_k,$ hence $f$ is compact.   
\end{proof}

We can now give a criterion of dualizability on the level of homotopy categories. Proposition \ref{prop:compact_maps_via_direct_sums} justifies the following definition.

\begin{defi}Let $T$ be a well generated triangulated category, and let $x\in T,$ $y\in T^{\omega_1}.$ A morphism $f:x\to y$ is called compact if for any sequence of objects $(z_1,z_2,\dots)$ in $T$ and for any morphism $g:y\to \biggplus[n] z_n,$ the composition $g\circ f$ factors through a finite direct sum $\biggplus[1\leq n\leq k]z_n$ for some $k.$\end{defi}

 Recall that in a triangulated category $T$ with countable coproducts we have sequential homotopy colimits, defined up to a non-canonical isomorphism:
$$\hocolim(x_1\xto{f_1} x_2\xto{f_2}\dots)=\Cone(\biggplus[n\geq 1] x_n\xto{\id-(f_n)} \biggplus[n\geq 1] x_n).$$

\begin{theo}\label{th:dualizability_via_compact_maps} Let $\cC$ be a presentable stable category. Then $\cC$ is dualizable if and only if the triangulated homotopy category $\h\cC$ satisfies the following condition:

$(*)$ the category $\h\cC$ is generated (as a localizing subcategory) by a set of objects $x_j,$ $j\in I,$ such that each $x_j$ is isomorphic to a homotopy colimit of a sequence $y_1\to y_2\to\dots,$ where each $y_n$ is in $(\h \cC)^{\omega_1}$ and each map $y_n\to y_{n+1}$ is compact. 
\end{theo}

\begin{proof}Suppose that $\cC$ is dualizable. By Corollary \ref{cor:dualizable_is_aleph_1_generated} $\cC$ is $\omega_1$-presentable, hence by Corollary \ref{cor:presentable_and_well_generated} we have $(\h \cC)^{\omega_1} = \h(\cC^{\omega_1}).$ Take some $x\in\cC^{\omega_1}.$ Then $\hat{\cY}(x)$ is contained in $\Ind(\cC^{\omega_1})^{\omega_1},$ so $\hat{\cY}(x)\cong\indlim[n]y_n$ for some sequence $y_1\to y_2\to\dots$ of objects of $\cC^{\omega_1}.$ Since $\hat{\cY}$ commutes with colimits, we have $$\indlim[n]\hat{\cY}(y_n)\xto{\sim}\indlim[n]\cY(y_n).$$ It follows that for each $n$ there is some $m>n$ such that the map $\cY(y_n)\to\cY(y_m)$ factors through $\hat{\cY}(y_m).$ By Proposition \ref{prop:compact_maps_in_dualizable_cats} this means that the map $y_n\to y_m$ is compact. Passing to a subsequence, we may assume that all maps $y_n\to y_{n+1}$ are compact. Hence, the category $\h\cC$ satisfies $(*):$ we can take the collection of all $\omega_1$-compact objects.

Now suppose that $\h\cC$ satisfies $(*).$ By Proposition \ref{prop:two_notions_of_kappa_compactness} we have $(\h \cC)^{\omega_1}\subset \h(\cC^{\omega_1}).$ To show that $\cC$ is dualizable, it suffices to observe the following.

{\noindent {\bf Claim.} For any sequence $y_1\xto{f_1} y_2\xto{f_2} \dots$ in $\cC^{\omega_1}$ such that each map $f_n$ is compact, the ind-object $\inddlim[n]y_n$ is in the left orthogonal to $\ker(\colim:\Ind(\cC)\to\cC).$}

\begin{proof}[Proof of Claim.] Suppose that an ind-system $(z_i)_{i\in I}$ satisfies $\indlim[i]z_i=0.$ Then each map of spectra
$$\indlim[i]\Hom(y_{n+1},z_i)\to \indlim[i]\Hom(y_n,z_i)$$
induces zero maps on the homotopy groups. Hence,
$$\Hom(\inddlim[n]y_n,\inddlim[i]z_i)\cong\prolim[n]\indlim[i]\Hom(y_n,z_i)=0.$$
This proves the claim.
\end{proof} 

It follows that the (a priori partially defined) functor $\hat{\cY}$ is defined on each $x_j.$ Since the objects $x_j$ generate $\cC,$ it follows that $\hat{\cY}$ is defined on the whole category $\cC.$ 
\end{proof}

The strong continuity of functors between dualizable stable categories can also be interpreted in terms of compact morphisms in the homotopy categories.

\begin{prop}\label{prop:strong_continuity_of_functor_via_compact_maps} A continuous functor $F:\cC\to\cD$ between dualizable categories is strongly continuous if and only if the functor $\h F$ preserves $\omega_1$-compact objects and compact maps between them.\end{prop}

\begin{proof}Suppose that $F$ is strongly continuous. Then $F$ preserves $\omega_1$-compact objects and $\Ind(F)\circ\hat{\cY}_{\cC}\cong \hat{\cY}_{\cD}\circ F.$ It follows from Proposition \ref{prop:compact_maps_in_dualizable_cats} that $F$ preserves all compact maps, in particular, $\h F$ preserves compact maps between $\omega_1$-compact objects.

Suppose that $\h F$ preserves $\omega_1$-compact objects and compact maps between them. By the proof of Theorem \ref{th:dualizability_via_compact_maps} we see that for each $\omega_1$-compact object $x$ the map $\hat{\cY}_{\cD}(F(x))\to\Ind(F)(\hat{\cY}_{\cC}(x))$ is an isomorphism. Since $\cC$ is $\omega_1$-compactly generated, we conclude that $F$ is strongly continuous.
\end{proof}

Finally, given a dualizable category $\cC$ and two objects $x,y\in\cC,$ we discuss the relation between the following two abelian groups:

\begin{itemize}
\item the subgroup of $\pi_0\Hom(x,y),$ which consists of compact maps;

\item the abelian group $\pi_0\Hom_{\Ind(\cC)}(\cY(x),\hat{\cY}(y)).$ 
\end{itemize}

We recall the general notions of an ideal and a quasi-ideal in a small additive category. Given small additive categories $\cA$ and $\cB,$ a left resp. right $\cA$-module is an additive functor $\cA\to\Ab$ resp. $\cA^{op}\to\Ab,$ and an $\cA\hy\cB$-bimodule is a biadditive bifunctor $\cA\times\cB^{op}\to\Ab.$ Recall that an ideal $J$ of a small additive category $\cA$ is a subbimodule of the diagonal bimodule $\cA(-,-):\cA^{op}\times\cA\to\Ab.$ More generally, a quasi-ideal in $\cA$ is a bimodule $I:\cA^{op}\times\cA\to \Ab$ together with a map $\alpha:I\to\cA(-,-)$ (not necessarily injective) such that for any $f\in I(x,y),$ $g\in I(y,z)$ we have $\alpha(g)f=g\alpha(f).$

We have the usual notion of a tensor product of left and right $\cA$-modules. In particular, given bimodules $F,G:\cA^{op}\times\cA\to\Ab,$ we have a bimodule $F\tens{\cA} G,$ given by $$(F\tens{\cA} G)(x,y)=F(-,y)\tens{\cA}G(x,-).$$ Note that the tensor product of quasi-ideals is naturally a quasi-ideal. We say that an $\cA\hy\cA$ bimodule $F$ is flat on the right if each right $\cA$-module $F(-,x)$ is flat, $x\in\cA.$

Given a quasi-ideal $(I,\alpha),$ we have a natural map $\wt{\mult}:I\tens{\cA} I\to I,$ $\wt{\mult}(g\otimes f)=\alpha(g)f=g\alpha(f).$ We say that a quasi-ideal $(I,\alpha)$ is idempotent if $\wt{\mult}$ is an isomorphism.

\begin{lemma}\label{lem:ideals_and_quasi-ideals} Let $\cA$ be a small additive category. The assignment $(I,\alpha)\mapsto \alpha(I)$ gives a bijection 
\[
\xymatrix@1{
\left\{
\text{
\begin{minipage}[c]{2.5in}
isomorphism classes of idempotent quasi-ideals $(I,\alpha)$ in $\cA$  such that $I$ is flat on the right
\end{minipage}
}
\right\}
\ar[r]^-{\sim}
&
\left\{
\text{
\begin{minipage}[c]{2.2in}
ideals $J$ in $\cA$ such that $J^2=J$ and $J\tens{\cA} J$ is flat on the right.
\end{minipage}
}
\right\}
}
\]
The inverse map sends an ideal $J$ to the quasi-ideal $J\tens{\cA}J,$ where the map $J\tens{\cA}J\to\cA(-,-)$ is given by the multiplication (i.e. composition).\end{lemma}

\begin{proof} Let $(I,\alpha)$ be an idempotent quasi-ideal in $\cA,$ and let $J=\alpha(I).$ It is clear that $J^2=J.$ We need to show that $J\tens{\cA}J\cong I$ as a quasi-ideal. Note that if $f\in I(x,y),$ $g\in I(y,z)$ are elements such that $\alpha(f)=0$ or $\alpha(g)=0,$ then $\wt{\mult}(g\otimes f)=0.$ It follows that the isomorphism $\wt{\mult}$ factors through $J\tens{\cA}J,$ hence $J\tens{\cA}J\cong I.$

Conversely, let $J$ be an ideal in $\cA$ such that $J^2=J$ and $I=J\tens{\cA}J$ is flat on the right. We
 need to show that $I$ is idempotent. It suffices to show that the map
\begin{equation}
\label{eq:quasi-ideal_idemp}
(J\tens{\cA}J)\tens{\cA} J\to (J\tens{\cA}J)\tens{\cA}\cA(-,-)\cong J\tens{\cA}J
\end{equation}
is an isomorphism. Since $J^2=J,$ the map \eqref{eq:quasi-ideal_idemp} is surjective. On the other hand, since $J\tens{\cA}J$ is flat on the right, the map \eqref{eq:quasi-ideal_idemp} is injective. Hence, this map is an isomorphism.
\end{proof}

\begin{prop}\label{prop:compact_maps_and_wavy_arrows} Let $\cC$ be a dualizable category, and let $T=\h\cC$ be its homotopy category. Denote by $J$ the ideal of compact maps in $T^{\omega_1}.$ Consider the bifunctor $$I:T^{\omega_1,op}\times T^{\omega_1}\to\Ab,\quad I(x,y)=\pi_0\Hom_{\Ind(\cC)}(\cY_{\cC}(x),\hat{\cY}_{\cC}(y)).$$ Then $I$ is an idempotent quasi-ideal in $T^{\omega_1},$ which is flat on the right. Moreover, $I$ corresponds to $J$ via the bijection from Lemma \ref{lem:ideals_and_quasi-ideals}. In particular, we have an isomorphism of quasi-ideals $I\cong J\tens{T^{\omega_1}}J.$\end{prop}

\begin{proof}We first check that $I$ is a quasi-ideal. First, for $x,y\in T^{\omega_1}\simeq \h \cC^{\omega_1}$ the structural map $\alpha = \alpha_{(x,y)}:I(x,y)\to T^{\omega_1}(x,y)$ is induced by the functor $\colim:\Ind(\cC^{\omega_1})\to\cC.$ Equivalently, $\alpha$ sends a map $f:\cY(x)\to \hat{\cY}(y)$ to the composition $\cY(x)\xto{f}\hat{\cY}(y)\to\cY(y),$ considered as an element of $\pi_0\Hom(\cY(x),\cY(y))\cong T^{\omega_1}(x,y).$ If on the other hand we identify $T^{\omega_1}(x,y)$ with $\pi_0 \Hom(\hat{\cY}(x),\hat{\cY}(y)),$ then $\alpha(f)$ corresponds to the composition $\hat{\cY}(x)\to\cY(x)\xto{f}\hat{\cY}(g).$
	
We need to check that for any objects $x,y,z\in T^{\omega_1}$ and for $f\in I(x,y),$ $g\in I(y,z)$ we have $\alpha(g)f = g \alpha(f).$ This follows from the commutative diagram 
\begin{equation*}
\begin{tikzcd}
\cY(x)\ar{r}{\cY(\alpha(f))}\ar{d}{f} & \cY(y)\ar{d}{g}\\
\hat{\cY}(y) \ar{r}{\hat{\cY}(\alpha(g))}\ar{ru}{} & \hat{\cY}(z).
\end{tikzcd}
\end{equation*}
This proves that the pair $(I,\alpha)$ is a quasi-ideal.

If $x\in T^{\omega_1}$ and $\hat{\cY}_{\cC}(x)=\inddlim[](x_1\to x_2\dots),$ then $I(-,x)=\indlim[n] T^{\omega_1}(-,x_n),$ hence $I$ is flat on the right. Since $\hat{\cY}(x)\cong\indlim[n]\hat{\cY}(x_n),$ we see that $I$ is idempotent. By Proposition \ref{prop:compact_maps_in_dualizable_cats} we see that $J$ is the image of $I.$ The rest follows from Lemma \ref{lem:ideals_and_quasi-ideals}.\end{proof}

\begin{remark}We remark that both Lemma \ref{lem:ideals_and_quasi-ideals} and Proposition \ref{prop:compact_maps_and_wavy_arrows} hold in the non-additive resp. non-stable context. Namely, if $\cA$ is any (discrete) small category, then we can consider the (bi)functors with values in the category of sets instead of abelian groups. The tensor product and flatness are defined in the usual way, and the notion of a (quasi-)ideal is the same. The statement of Lemma \ref{lem:ideals_and_quasi-ideals} still holds, and the proof is the same.

If $\cC$ is a compactly assembled accessible $\infty$-category (not necessarily cocomplete), then the analogue of Proposition \ref{prop:compact_maps_and_wavy_arrows} still holds, and the proof is the same. If $\cC$ is an ordinary category, then $I$ is the quasi-ideal of wavy arrows in the terminology of \cite[Section 2]{JJ}.\end{remark}

\begin{remark}Another generalization of Proposition \ref{prop:compact_maps_and_wavy_arrows} is the following. Let $T$ be a compactly generated triangulated category, and consider a smashing subcategory $S\subset T.$ This means that $S$ is a full triangulated subcategory closed under coproducts, and the inclusion functor has a right adjoint $\Phi:T\to S$ which commutes with coproducts. Then we have a quasi-ideal $I$ in $T^{\omega},$ given by $I(x,y)=\Hom_T(x,\Phi(y)).$ Moreover, $I$ is idempotent and flat on the right. The image of $I$ in $T^{\omega}$ is the ideal $J$ of maps which factor through an object of $S.$ We have $J^2=J,$ and we obtain an isomorphism of quasi-ideals $J\tens{T^{\omega}}J\cong I.$ Proposition \ref{prop:compact_maps_and_wavy_arrows} is a special case of this statement with $T=\h\Ind(\cC^{\omega_1}),$ and $S=\h\hat{\cY}(\cC).$\end{remark}

\subsection{The axiom (AB6) and dualizability}
\label{ssec:AB6}

For the formulation of the axiom (AB6) we refer to Subsection \ref{ssec:AB_axioms}. Recall that it makes sense in any $\infty$-category with infinite products and filtered colimits. The most important observation is that (AB6) holds in the category of sets, hence also in the $\infty$-category of spaces. We deduce the following statement.




\begin{prop}\label{prop:products_in_Ind_C} Let $\cC$ be an $\infty$-category with infinite products. Then the category $\Ind(\cC)$ also has infinite products, and the following isomorphism holds:
$$\prod_{i\in I}\inddlim[j_i\in J_i]x_{j_i}\cong \inddlim[(j_i)_i\in\prod_{i\in I}J_i]\prod_{i\in I}x_{j_i},$$ where $J_i$ are filtered categories, 
$i\in I,$ and we have functors $J_i\to \cC,$ $j_i\mapsto x_{j_i}.$\end{prop}

\begin{proof}Since (AB6) holds in the $\infty$-category of spaces, for any object $y\in\cC$ we have a chain of isomorphisms:
\begin{multline*}\Map(y,\inddlim[(j_i)_i\in\prod_{i\in I}J_i]\prod_{i\in I}x_{j_i})\cong \indlim[(j_i)_i\in\prod_{i\in I}J_i]\Map(y,\prod_{i\in I}x_{j_i})
\cong \indlim[(j_i)_i\in\prod_{i\in I}J_i]\prod_{i\in I}\Map(y,x_{j_i})\\
\cong \prod_{i\in I}\indlim[j_i\in J_i]\Map(y,x_{j_i})\cong \prod_{i\in I}\Map(y,\inddlim[j_i\in J_i]x_{j_i})\cong \Map(y,\prod_{i\in I}\inddlim[j_i\in J_i]x_{j_i}).\end{multline*}

This proves the proposition.\end{proof}

We obtain another criterion of dualizability, which is due to Clausen and Scholze.

\begin{prop}\label{prop:AB6_criterion} (Clausen, Scholze) Let $\cC$ be a presentable stable category. Then $\cC$ is dualizable if and only if (AB6) holds in $\cC.$\end{prop}

\begin{proof}Suppose that $\cC$ is dualizable. Then the colimit functor $\colim:\Ind(\cC)\to\cC$ has a left adjoint $\hat{\cY}:\cC\to \Ind(\cC).$ Hence, the functor $\colim$ commutes with infinite products. By Proposition \ref{prop:products_in_Ind_C}, this exactly means that (AB6) holds in $\cC.$ 

Conversely, suppose that $(AB6)$ holds in $\cC.$ Arguing as above, we see that the functor $\colim:\Ind(\cC)\to\cC$ commutes with infinite products. Now, the category $\Ind(\cC)$ is not presentable, so we need the following additional observation to deduce the existence of the left adjoint to $\colim.$

Let $\kappa$ be a regular cardinal such that the category $\cC$ is $\kappa$-presentable. Denote by $(\colim)_{\kappa}:\Ind(\cC^{\kappa})\to\cC$ the colimit functor. Denote by $G:\Ind(\cC)\to \Ind(\cC^{\kappa})$ the right adjoint to the inclusion. Note that $G$ commutes with colimits, and its restriction to $\cC\subset\Ind(\cC)$ is given by the inclusion $\cC\simeq\Ind_{\kappa}(\cC^{\kappa})\hto \Ind(\cC^{\kappa}).$ We see that the natural transformation $(\colim)_{\kappa}\circ G\to \colim$ is an equivalence of functors $\Ind(\cC)\to \cC.$ Since both functors $\colim:\Ind(\cC)\to\cC$ and $G:\Ind(\cC)\to \Ind(\cC^{\kappa})$ commute with infinite products, we deduce that the functor $(\colim)_{\kappa}:\Ind(\cC^{\kappa})\to\cC$ also commutes with infinite products. Hence, it has a left adjoint $\cC\to \Ind(\cC^{\kappa}).$ Then the composition $\cC\to \Ind(\cC^{\kappa})\to \Ind(\cC)$ is left adjoint to $\colim,$ hence $\cC$ is dualizable.\end{proof}

\begin{remark}\label{rem:compass_via_AB5_and_AB6} The same argument shows that a (not necessarily stable) presentable category $\cC$ is compactly assembled if and only if it satisfies (AB6) and strong (AB5): filtered colimits commute with finite limits. If $\cC$ is a complete lattice, then this statement is well known, see \cite[Theorem I.2.3]{GHKLMS}. 

Below we will see that the category $\Cat_{\st}^{\dual}$ satisfies (AB6) (Proposition \ref{prop:AB6_for_Cat^dual}) but the strong (AB5) does not hold in $\Cat_{\st}^{\dual}$ (Corollary \ref{cor:seq_colimits_finite_limits}). \end{remark}


\subsection{Automatic dualizability and images of functors}
\label{ssec:autom_dualizable}

We have the following sufficient condition for dualizability.

\begin{prop}\label{prop:condition_for_dualizability} Let $\cD$ be a presentable stable category, and consider a family of strongly continuous functors $F_i:\cC_i\to\cD,$ where each $\cC_i$ is dualizable. Suppose that $\cD$ is generated by the images of $F_i$ (by colimits). Then $\cD$ is dualizable.\end{prop}

\begin{proof} We first reduce to the case when $I$ consists of a single element. Consider the category $\cC=\prodd[i]\cC_i.$ Then $\cC$ is dualizable: it is sufficient to prove this for compactly generated categories, and in this case we have $\cC\simeq\Ind(\biggplus[i]\cC_i^{\omega}).$ By Proposition \ref{prop:colimits_in_Pr^LL}, the induced functor $F:\cC\to\cD$ is strongly continuous and its image generates $\cD.$

So assume that $I=\{i\}$ and put $\cC=\cC_i,$ $F=F_i.$ We give two proofs.

{\noindent {\bf First proof.}} Note that the right adjoint functor $F^R:\cD\to\cC$ is conservative (since the image of $F$ generates), and it commutes with products and filtered colimits. By Proposition \ref{prop:AB6_criterion} the axiom (AB6) (from Subsection \ref{ssec:AB_axioms}) holds in $\cC.$  We deduce that (AB6) also holds in $\cD.$ Applying the same proposition again, we see that $\cD$ is dualizable.

{\noindent {\bf Second proof.}} Let $x\in \cC$ be a generating object, and $\hat{\cY}(x)=\inddlim[i]x_i.$ Then for an ind-object $\inddlim[j]y_j\in \Ind(\cD)$ we have an equivalence \begin{multline*}\Hom(\inddlim[i]F(x_i),\inddlim[j]y_j)\cong \Hom(\inddlim[i]x_i,\inddlim[j]F^R(y_j))\cong \Hom(x,\indlim[j] F^R(y_j))\\ \cong
\Hom(x,F^R(\indlim[j] y_j))\cong
 \Hom(F(x),\indlim[j] y_j).\end{multline*} Hence, the (a priori partially defined) functor $\hat{\cY}_{\cD}$ is defined on $F(x).$ Since the objects $F(x)$ generates $\cD$ by assumption, we conclude that the functor $\hat{\cY}_{\cD}:\cD\to \Ind(\cD)$ exists, hence $\cD$ is dualizable.\end{proof}

An important corollary is that we have well-behaved images of $1$-morphisms in $\Cat_{\st}^{\dual}.$

\begin{defi}\label{defi:image} Let $F:\cC\to\cD$ be a continuous functor between presentable stable categories. We denote by $\im(F)\subset\cD$ the localizing subcategory generated by the image of $F.$ Equivalently, $\im(F)$ is the fiber of the cofiber of $F$ in $\Pr^L_{\st}.$\end{defi}

\begin{cor}\label{cor:images_in_Cat^dual} Let $\cC$ and $\cD$ be presentable stable categories, and suppose that $\cC$ is dualizable. Let $F:\cC\to\cD$ be a strongly continuous functor. Then the category $\im(F)$ is dualizable, and the inclusion $\im(F)\to\cD$ is strongly continuous.\end{cor}

\begin{proof}The functor $F:\cC\to \im(F)$ is strongly continuous, hence Proposition \ref{prop:condition_for_dualizability} implies that $\im(F)$ is dualizable. Applying Proposition \ref{prop:automatic_strcont} to the functors $\cC\to\im(F)\to\cD,$ we deduce that the inclusion $\im(F)\to\cD$ is strongly continuous.\end{proof}

\begin{remark}Under the assumptions of Corollary \ref{cor:images_in_Cat^dual}, if $\cD$ is dualizable, then $\im(F)$ is also the fiber of the cofiber of $F$ in $\Cat_{\st}^{\dual}.$\end{remark}

\begin{remark}Note that we have a (not very explicit) description of the functor $\Pr^{LL}_{\st}\to\Cat_{\st}^{\dual},$ right adjoint to the inclusion. Let $\cC$ be a presentable stable category, and denote by $P$ the poset of strictly full dualizable subcategories $\cD\subset\cC,$ such that the inclusion is strongly continuous. Here ``strictly full'' means that if some $x\in\cC$ is isomorphic to an object of $\cD,$ then $x\in\cD.$ Note that $P$ is indeed a small set: any such $\cD$ is generated by $\omega_1$-compact objects of $\cC.$ Moreover, $P$ is a complete lattice: the supremum (join) of a collection $\{\cD_i\}_{i\in I}$ is given by the image of the functor $\prodd[i]\cD_i\to\cC$ in the sense of Definition \ref{defi:image}. In particular, $P$ has the largest element $\cD_{\max}.$ Then $\cD_{\max}$ is exactly the image of $\cC$ under the functor $\Pr^{LL}_{\st}\to\Cat_{\st}^{\dual}.$\end{remark}

\subsection{Calkin categories}
\label{ssec:Calkin}

Fix an uncountable regular cardinal $\kappa.$ We will define the Calkin and $\kappa$-Calkin categories as follows.

Given a small stable Karoubi complete category $\cA,$ we put $$\Calk(\cA)=(\Ind(\cA)/\cA)^{\Kar},\quad\Calk_{\kappa}(\cA)=(\Ind(\cA)^{\kappa}/\cA)^{\Kar}.$$
It turns out that this construction has a direct analogue for dualizable categories.

Let $\cC$ be a dualizable category. By Corollary \ref{cor:dualizable_is_aleph_1_generated}, $\cC$ is $\omega_1$-presentable. By Corollary \ref{cor:hat_Y_into_Ind_C_omega_1} the image of the functor $\hat{\cY}:\cC\to \Ind(\cC)$ is contained in the full subcategory $\Ind(\cC^{\omega_1}).$ Hence, we have a fully faithful strongly continuous functor $\hat{\cY}:\cC\to \Ind(\cC^{\kappa}).$

\begin{defi} For a dualizable presentable stable category $\cC$ we define its continuous $\kappa$-Calkin category $\Calk_{\kappa}^{\cont}(\cC)$ as the category of compact objects in the quotient $\Ind(\cC^{\kappa})/\hat{\cY}(\cC)\simeq\ker(\colim:\Ind(\cC^{\kappa})\to\cC).$ We have a short exact sequence
$$0\to\cC\xto{\hat{\cY}} \Ind(\cC^{\kappa})\to \Ind(\Calk_{\kappa}^{\cont}(\cC))\to 0$$
in $\Cat_{\st}^{\dual}.$ 

We have fully faithful functors $\Calk_{\kappa}^{\cont}(\cC)\to\Calk_{\lambda}^{\cont}(\cC)$ for $\kappa<\lambda.$ We denote by $\Calk^{\cont}(\cC)$ the colimit of $\Calk_{\lambda}^{\cont}(\cC).$\end{defi}

We observe that this Calkin construction is functorial: we have a well-defined functor $\Calk_{\kappa}^{\cont}:\Cat_{\st}^{\dual}\to \Cat^{\perf}.$ Indeed, if $F:\cC\to\cD$ is a strongly continuous functor between dualizable categories, then we have a commutative square
$$
\begin{tikzcd}
\cC \ar{r}{F} \ar{d}{\hat{\cY}_{\cC}} &[1em] \cD \ar{d}{\hat{\cY}_{\cD}}\\
\Ind(\cC^{\kappa}) \ar{r}{\Ind(F^{\kappa})} & \Ind(\cC^{\kappa}).
\end{tikzcd}
$$

Hence, we have a well-defined functor $\Calk_{\kappa}^{\cont}(F):\Calk_{\kappa}^{\cont}(\cC)\to \Calk_{\kappa}^{\cont}(\cD).$

\begin{prop}\label{prop:Calk_of_cg} If $\cC$ is compactly generated, we have a natural equivalence $\Calk_{\kappa}(\cC^{\omega})\simeq \Calk_{\kappa}^{\cont}(\cC).$\end{prop}

\begin{proof}Indeed, in this case the functor $\hat{\cY}:\cC\to \Ind(\cC^{\kappa})$ is identified with $\Ind(\cY:\cC^{\omega}\to \cC^{\kappa}).$ Applying $\Ind$ to the short exact sequence $$0\to \cC^{\omega}\to\cC^{\kappa}\to \Calk_{\kappa}(\cC^{\omega})\to 0,$$
we obtain an equivalence $\Calk_{\kappa}(\cC^{\omega})\simeq \Calk_{\kappa}^{\cont}(\cC).$\end{proof}

\begin{remark}Note that if $\cC$ is compactly generated, then the essential image of the functor $\cC^{\kappa}\to\Calk_{\kappa}^{\cont}(\cC)$ is a stable subcategory: it is equivalent to the Verdier quotient $\cC^{\kappa}/\cC^{\omega}.$ However, in general this essential image does not have to be a stable subcategory even for $\kappa=\omega_1,$ see Proposition \ref{prop:image_not_a_stable_subcat} below.\end{remark}

\begin{prop}\label{prop:exactness_of_Calk} The functor $\Calk_{\kappa}^{\cont}:\Cat_{\st}^{\dual}\to \Cat^{\perf}$ takes short exact sequences to short exact sequences.\end{prop}

\begin{proof}The only assertion which requires a proof is the following: if $F:\cC\to\cD$ is a strongly continuous fully faithful functor between dualizable categories, then the functor $\Calk_{\kappa}^{\cont}(F):\Calk_{\kappa}^{\cont}(\cC)\to \Calk_{\kappa}^{\cont}(\cD)$ is fully faithful. Passing to ind-completions, it suffices to check that the functor $\Ind(\cC^{\kappa})/\hat{\cY}(\cC)\to \Ind(\cD^{\kappa})/\hat{\cY}(\cD)$ is fully faithful. But this is the same as the functor $$\ker(\colim:\Ind(\cC^{\kappa})\to\cC)\to \ker(\colim:\Ind(\cD^{\kappa})\to\cD).$$ The latter functor is fully faithful, because the functor $\Ind(\cC^{\kappa})\to \Ind(\cD^{\kappa})$ is fully faithful.\end{proof}

We conclude this subsection with the description of morphisms in the Calkin category.

\begin{prop}\label{prop:Homs_in_Calk} Let $\cC$ be a dualizable category. The category $\Calk^{\cont}(\cC)$ is identified with the idempotent completion of the essential image of the functor $\Cone(\hat{\cY}_{\cC}\to\cY_{\cC}):\cC\to\Ind(\cC).$ In particular, for $x,y\in\cC,$ we have
\begin{multline*}\Hom_{\Calk^{\cont}(\cC)}(x,y)\cong\Cone(\Hom_{\Ind(\cC)}(\cY(x),\hat{\cY}(y))\to\Hom_{\cC}(x,y))\\ \cong
\Cone(\ev_{\cC}(y\boxtimes x^{\vee})\to \Hom_{\cC}(x,y)).
\end{multline*}
\end{prop}

\begin{proof}This follows directly from the definitions. Namely, the left adjoint to the inclusion $\ker(\colim)\to\Ind(\cC)$ is given by the left Kan extension of the functor \begin{equation*}\Cone(\hat{\cY}_{\cC}\to\cY_{\cC}):\cC\to \ker(\colim).\qedhere\end{equation*}\end{proof}

Note that if $\cC=\Mod\hy A$ for some $\bE_1$-ring $A,$ then  for any right $A$-modules $M$ and $N$ Proposition \ref{prop:Homs_in_Calk} gives an isomorphism
$$\Hom_{\Calk(A)}(M,N)\cong\Cone(N\tens{A}\Hom_A(M,A)\to\Hom_A(M,N)).$$

\subsection{Colimits of dualizable categories}
\label{ssec:colimits_of_dualizable}

We first observe the following.

\begin{prop}\label{prop:colimits_of_dualizable} The full subcategory $\Cat_{\st}^{\dual}\subset \Pr_{\st}^{LL}$ is closed under colimits. In particular, the category $\Cat_{\st}^{\dual}$ is cocomplete and the functor $\Cat_{\st}^{\dual}\to \Pr_{\st}^{L}$ commutes with colimits.\end{prop}

\begin{proof}The second statement follows from the first one by Proposition \ref{prop:colimits_in_Pr^LL}.

To prove the first statement, consider a functor $I\to\Cat_{\st}^{\dual},$ $i\mapsto\cC_i.$ Applying Proposition \ref{prop:colimits_in_Pr^LL} again, we see that the functors $\cC_i\to \indlim[i]^{\cont}\cC_i$ are strongly continuous, and their images generate the target. Hence by Proposition \ref{prop:condition_for_dualizability} the category $\indlim[i]^{\cont}\cC_i$ is dualizable. 
\end{proof}

We deduce the description of epimorphisms in $\Cat_{\st}^{\dual}.$

\begin{cor}\label{cor:epi_in_Cat^dual}
Let $F:\cC\to\cD$ be a strongly continuous between dualizable categories. Then $F$ is an epimorphism in $\Cat_{\st}^{\dual}$ if and only if $F$ is a localization.
\end{cor}

\begin{proof}
This follows from Proposition \ref{prop:mono_epi_Pr^L} since the functor $\Cat_{\st}^{\dual}\to\Pr^L_{\st}$ is conservative and commutes with colimits by Proposition \ref{prop:colimits_of_dualizable}.
\end{proof}




We need the following general facts about the directed colimits in $\Pr^{LL}_{\st}$ and $\Pr^L_{\st}.$

\begin{prop}\label{prop:colimit_of_presentable_compositions_of_adjoints} Let $(\cC_i)_{i\in I}$ be a directed system of presentable stable categories, and $\cC=\indlim[i]^{\cont}\cC_i.$ We denote by $F_{ij}:\cC_i\to\cC_j$ the transition functors, and by $F_i:\cC_i\to \cC$ the structural functors to the colimit. Also, denote by $F_{ij}^R$ and $F_i^R$ the corresponding right adjoint functors.

1) We have $$\indlim[i] F_i F_i^R\cong \Id_{\cC}.$$

2) Suppose that the functors $F_{ij}$ are strongly continuous, i.e. the functors $F_{ij}^R$ are continuous. Then for any $i,j\in I$ we have an isomorphism
\begin{equation}\label{eq:colimit_adjoints_compositions} \indlim[k\geq i,j] F_{jk}^R F_{ik}\cong F_j^R F_i.\end{equation}\end{prop}

\begin{proof}We use once again the natural equivalence
$\cC\simeq \prolim[j]\cC_j,$
where the transition functors in the inverse system are $F_{ij}^R.$ Moreover, the composition $\cC\simeq \prolim[j]\cC_j\to\cC_i$ is identified with the functor $F_i^R.$

To prove 1), we simply observe that for $x,y\in\cC$ we have isomorphisms
$$\Hom_{\cC}(x,y)\cong \prolim[i]\Hom(F_i^R(x),F_i^R(y))\cong\prolim[i] \Hom_{\cC}(F_i F_i^R(x),y)\cong\Hom_{\cC}(\indlim[i] F_i F_i^R(x),y).$$

We prove 2). By Proposition \ref{prop:colimits_in_Pr^LL} the functors $F_i^R$ are continuous. Using again the continuity of $F_{ij}^R,$ we obtain a well-defined functor $F_i':\cC_i\to\prolim[j]\cC_j\simeq \cC,$ given by
$$F_i'(x)=(\indlim[k\geq i,j]F_{jk}^R F_{ik}(x))_j.$$ Then \eqref{eq:colimit_adjoints_compositions} is equivalent to the statement that $F_i'\cong F_i.$ To check this, take some objects $x\in\cC_i,$ $y=(y_j)_j\in\prolim[j]\cC_i\simeq \cC.$ We obtain the following isomorphisms:
\begin{multline*}\Hom_{\cC}(F_i'(x),y)\cong\prolim[j]\Hom_{\cC_j}(\indlim[k\geq i,j]F_{jk}^R F_{ik}(x),y_j)\cong \prolim[k\geq j\geq i]\Hom_{\cC_j}(F_{jk}^R F_{ik}(x),y_j)\\ \cong
\prolim[j=k\geq i]\Hom_{\cC_j}(F_{ij}(x),y_j)\cong\prolim[j\geq i]\Hom_{\cC_i}(x,F_{ij}^R(y_j))=\Hom_{\cC_i}(x,y_i)=\Hom_{\cC_i}(x,F_i^R(y)).\end{multline*}
Hence, $F_i'\cong F_i,$ as required.\end{proof}

In Section \ref{sec:sh_loc_comp} we will need the following closely related statement.

\begin{lemma}\label{lem:auxiliary_on_colimits} Let $I$ be a directed poset, and consider a functor $I\to\Pr^{LL}_{\st},$ $i\mapsto\cC_i.$  Consider a section of the associated cocartesian fibration over $I,$ $i\mapsto x_i.$ Denote by $F_{ij}:\cC_i\to\cC_j$ the transition functors, and denote by $F_i:\cC_i\to\cC=\indlim[i]^{\cont}\cC_i$ the functors to the colimit. Then for any $i\in I,$ we have
	\begin{equation}\label{eq:auxiliary_on_colimits}
		\indlim[j\geq i]F_{ij}^R (x_j)\cong F_i^R(\indlim[j]F_j(x_j)).
	\end{equation}
\end{lemma}

\begin{proof}Fix $i\in I.$ Consider the full subposet $I_{\geq i}\subset I$ consisting of $j\geq i.$ Denote by $J\subset I\times I$ the full subposet of pairs $(j,j')$ such that $j'\geq j\geq i.$ Then the inclusion $I_{\geq i}\to J,$ $j\mapsto (j,j),$ is cofinal. Hence, we have an isomorphism
	\begin{equation}\label{eq:isomorphism_by_cofinality_abstract}
		\indlim[j\geq i]F_{ij}^R (x_j)\xto{\sim} \indlim[j'\geq j\geq i]F_{ij'}^R (F_{jj'}(x_j)).
	\end{equation}
	Applying Proposition \ref{prop:colimit_of_presentable_compositions_of_adjoints}, for a fixed $j\geq i$ we obtain an isomorphism
	\begin{equation*}
		\indlim[j'\geq j] F_{ij'}^R(F_{jj'}(x_j))\xto{\sim} F_i^R(F_j(x_j)).
	\end{equation*}
	This allows to simplifty the target of \eqref{eq:isomorphism_by_cofinality_abstract}:
	\begin{equation}\label{eq:isomorphism_by_playing_with_adjoints_abstract}
		\indlim[j'\geq j\geq i]F_{ij'}^R (F_{jj'}(x_j))\xto{\sim} \indlim[j\geq i] F_i^R(F_j(x_j)).
	\end{equation}
	Finally, by the strong continuity of $F_i$ we have an isomorphism
	\begin{equation}\label{eq:isom_by_strcont_of_F_i}
		\indlim[j\geq i] F_i^R(F_j(x_j))\xto{\sim} F_i^R(\indlim[j\geq i]F_j(x_j)).
	\end{equation}
	The composition of the isomorphisms \eqref{eq:isomorphism_by_cofinality_abstract}, \eqref{eq:isomorphism_by_playing_with_adjoints_abstract} and \eqref{eq:isom_by_strcont_of_F_i} gives the stated isomorphism \eqref{eq:auxiliary_on_colimits}.
\end{proof}

Recall that by Propositions \ref{prop:mono_epi_Pr^L} and \ref{prop:mono_epi_Pr^LL} the monomorphisms in $\Pr^L_{\st}$ and $\Pr^{LL}_{\st}$ are exactly the fully faithful functors. By Corollary \ref{cor:epi_in_Cat^dual} below the same holds for $\Cat_{\st}^{\dual}.$

\begin{prop}\label{prop:weak_AB5_Pr^LL} The categories $\Pr^{LL}_{\st}$ and $\Cat_{\st}^{\dual}$ satisfy the weak (AB5) axiom: the class of fully faithful functors is closed under filtered colimits. In particular, a filtered colimit of short exact sequences in $\Cat_{\st}^{\dual}$ is again a short exact sequence.\end{prop}


\begin{proof} By Proposition \ref{prop:colimits_of_dualizable}, it is sufficient to prove the statement for $\Pr^{LL}_{\st}.$

Consider a directed system of strongly continuous fully faithful functors $(\Phi_i:\cC_i\to\cD_i)_{i\in I},$ where the transition functors $F_{ij}:\cC_i\to\cC_j$ and $G_{ij}:\cD_i\to\cD_j$ are strongly continuous. Put $\cC=\indlim[i]^{\cont}\cC_i,$ $\cD=\indlim[i]^{\cont}\cD_i,$ and denote by $\Phi:\cC\to\cD$ the colimit of $\Phi_i.$ Denote by $F_i:\cC_i\to\cC,$ $G_i:\cC_i\to\cC$ the structural functors. For $i,j\in I$ we can directly compute the composition $F_j^R\Phi^R\Phi F_i:\cC_i\to\cC_j.$ Namely, applying Proposition \ref{prop:colimit_of_presentable_compositions_of_adjoints} together with the strong continuity and fully faithfulness of $\Phi_i$ we obtain
\begin{multline*}F_j^R\Phi^R\Phi F_i\cong \Phi_j^R G_j^RG_i\Phi_i\cong\Phi_j^R(\indlim[k\geq i,j]G_{jk}^R G_{ik})\Phi_i\\ \cong
 \indlim[k\geq i,j]\Phi_j^R G_{jk}^R G_{ik} \Phi_i\cong \indlim[k\geq i,j]F_{jk}^R\Phi_k^R\Phi_k F_{ik}\cong \indlim[k\geq i,j]F_{jk}^R F_{ik}\cong F_j^R F_i.\end{multline*}
Hence, $\Phi^R\Phi\cong\Id,$ i.e. $\Phi$ is fully faithful.
\end{proof}

\begin{cor}\label{cor:directed_unions} Let $\cC$ be a presentable stable category, and suppose that we have a directed family of presentable stable full subcategories $\cC_i\subset\cC,$ $i\in I,$ where the inclusion functors $\cC_i\to\cC$ are strongly continuous. If $\cC$ is generated by $\cC_i$ as a localizing subcategory, then the natural functor $\indlim[i]^{\cont}\cC_i\to \cC$ is an equivalence.\end{cor}

\begin{proof}By Proposition \ref{prop:weak_AB5_Pr^LL}, the functor $\indlim[i]^{\cont}\cC_i\to\cC$ is fully faithful. By assumption, $\cC$ is generated by $\cC_i,$ so we actually have an equivalence.\end{proof}

\begin{remark} The following example shows that the analogues of Proposition \ref{prop:weak_AB5_Pr^LL} and \ref{cor:directed_unions} fail in $\Pr^L_{\st}.$ Denote by $p_n$ the $n$-th prime number, $n\geq 1,$ i.e. $p_1=2,$ $p_2=3,\dots.$ Put $\cC=D(\Z)$ and $\cC_n=D(A_n),$ where $A_n=\Z[p_k^{-1},\, k\geq n],$ $n\geq 1.$ We consider $\cC_n$ as a full subcategory of $\cC$ via the restriction of scalars, and denote by $F_n:\cC_n\to\cC$ the inclusion functor. Hence, we have
$$\cC_1\subset\cC_2\subset\dots\subset\cC.$$
Note that the subcategories $\cC_n$ generate $\cC$ as a localizing subcategory. Indeed, we have $\Q=A_1\in\cC_1,$ and $\bbF_{p_n}=A_{n+1}/p_n A_{n+1}\in\cC_{n+1}.$ However, we claim that the functor $\Phi:\indlim[n]^{\cont}\cC_n\to\cC$ is not an equivalence.

Indeed, otherwise by Proposition \ref{prop:colimit_of_presentable_compositions_of_adjoints} 1) we would have $\indlim[n] F_n F_n^R\cong\Id_{\cC}.$ But this clearly fails:
$$\Cone(\indlim[n] F_n F_n^R(\Z)\to\Z)=\indlim[n]\bR\Hom(A_n/\Z[-1],\Z)\cong\Cone(\bigoplus\limits_p\Z_p\to\prodd[p]\Z_p)\ne 0.$$ Hence, $\Phi$ is not an equivalence.\end{remark}

We have the following important observation about compact objects.

\begin{prop}\label{prop:compact_objects_in_filtered_colimits} The functor $(-)^{\omega}:\Cat_{\st}^{\dual}\to\Cat^{\perf}$ commutes with filtered colimits. Equivalently, the object $\Sp\in\Cat_{\st}^{\dual}$ is compact.\end{prop}

\begin{proof}Consider an ind-system $(\cC_i)_i$ of dualizable categories. By Proposition \ref{prop:weak_AB5_Pr^LL}, we have a short exact sequence
$$0\to\indlim[i]^{\cont}\cC_i\to\Ind(\indlim[i]\cC_i^{\omega_1})\to\Ind(\indlim[i]\Calk_{\omega_1}^{\cont}(\cC_i))\to 0.$$
It follows that 
$$(\indlim[i]^{\cont}\cC_i)^{\omega}\simeq\ker(\indlim[i]\cC_i^{\omega_1}\to\indlim[i]\Calk_{\omega_1}^{\cont}(\cC_i))\simeq \indlim[i]\ker(\cC_i^{\omega_1}\to \Calk_{\omega_1}^{\cont}(\cC_i))=\indlim[i]\cC_i^{\omega}.$$
Here we used the fact that in $\Cat^{\perf}$ filtered colimits commute with finite limits.\end{proof}

We now give examples showing that the object $\Sp\in\Pr^{LL}_{\st,\omega_1}$ is not small (i.e. not $\kappa$-compact for any regular cardinal $\kappa$). A different argument for this was given previously by German Stefanich, using the generalizations of the derived category of quasi-coherent sheaves on $B\bG_{a,\FF_p}.$  

We use the following notation. For a pair of ordinals $\lambda\leq \mu$ we denote by $[\lambda,\mu]$ the version of a long closed interval: we put
\begin{equation*}
[\lambda,\mu]=(\{\nu\mid \lambda\leq \nu<\mu\}\times [0,1)) \cup\{\mu\},
\end{equation*}
and consider this set as a topological space with the interval topology with respect to the following total order. The order on $\{\nu\mid \lambda\leq \nu<\mu\}\times [0,1)\subset [\lambda,\mu]$ is lexicographic, and the element $\mu$ is the largest one. Then $[\lambda,\mu]$ is a connected compact Hausdorff space.

\begin{prop}\label{prop:compact_objects_in_colimits_of_presentable} Let $\mk$ be a field. Take some regular cardinal $\kappa,$ which we also consider as an ordinal. For an ordinal $\lambda<\kappa,$ denote by $X_{\lambda}$ the rectangle $[\lambda,\kappa]\times [0,1]$ (considered as a topological space). Denote by $\cC_{\lambda}\subset\Shv(X_{\lambda};D(\mk))$ the full subcategory of sheaves whose pullback to $\{\kappa\}\times (0,1]$ vanishes. Then the pullback functors $\cC_{\lambda}\to\cC_{\mu}$ are strongly continuous for $\lambda\leq \mu<\kappa,$ the colimit $\indlim[\lambda<\kappa]^{\cont}\cC_{\lambda}\simeq D(\mk)$ is non-zero and compactly generated, but we have $\cC_{\lambda}^{\omega}=0$ for $\lambda<\kappa.$

Therefore, the functor $(-)^{\omega}:\Pr^{LL}_{\st,\omega_1}\to\Cat^{\perf}$ does not commute with $\kappa$-filtered colimits, i.e. the object $\Sp\in\Pr^{LL}_{\st,\omega_1}$ is not $\kappa$-compact. In particular, the category $\Pr^{LL}_{\st,\omega_1}$ is not presentable.\end{prop}

\begin{proof} We first consider the special case $\kappa=\omega,$ for which the argument is easier to understand. Identifying $[0,\omega]$ with $\bR_{\geq 0}\cup\{+\infty\},$ we have a homeomorphism $[0,\omega]\xto{\sim}[0,1],$ $x\mapsto \frac{x}{x+1}.$ Then for $n\in\N$ the category $\cC_n$ is identified with the full subcategory of $\Shv([1-\frac{1}{n+1},1]\times [0,1]; D(\mk))$ whose pullback to $\{1\}\times (0,1]$ vanishes. It is easy to see that each $\cC_n$ is $\omega_1$-presentable: it is a kernel of a continuous functor between dualizable categories with an $\omega_1$-continuous right adjoint. Clearly, the functors $\cC_n\to\cC_{n+1}$ are strongly continuous. Identifying the colimit with the limit of the inverse system, we see that $\indlim[n\in\N]^{\cont}\cC_n$ is equivalent to the full subcategory of $\cC_0$ consisting of sheaves whose support is contained in $[1-\frac{1}{n+1},1]\times [0,1]$ for all $n\in\N$ (and whose pullback to $\{1\}\times (0,1]$ vanishes). This is simply the category of sheaves supported at the point $\{(1,0)\},$ which is equivalent to $D(\mk).$
	
It remains to show that $\cC_n^{\omega}=0$ for $n\in\N.$ Since all the categories $\cC_n$ are equivalent, we may assume $n = 0.$ Suppose that $\cF\in\cC_0^{\omega}$ is a non-zero compact object. For $n\in\N,$ put 
$Y_n=([1-\frac{1}{n+1},1]\times\{0\})\cup ([0,1-\frac{1}{n+1}]\times [0,1]).$ 
The pullback functor $\cC_0\to\Shv(Y_n;D(\mk))$ is strongly continuous, hence $\cF_{\mid Y_n}$ is a compact sheaf on $Y_n.$ Since $Y_n$ is connected and compact, it follows from Proposition \ref{prop:compact_sheaves} that $\cF_{\mid Y_n}$ is a nowhere zero locally constant sheaf. Since this holds for all $n\in \N,$ we deduce that $\cG = \cF_{\mid [0,1)\times\{1\}}$ is a nowhere zero locally constant sheaf on $Z=[0,1)\times\{1\}$. The pullback functor $\cC_0\to\Shv(Z; D(\mk))$ is strongly continuous, hence $\cG$ is also a compact sheaf. But $\supp(\cG)=Z$ is not compact, a contradiction with Proposition \ref{prop:compact_sheaves}.
	
For a general regular cardinal $\kappa$ the proof is almost identical. We indicate the modifications. As above we see that the categories $\cC_{\lambda}$ are $\omega_1$-presentable and each transition functor $\cC_{\lambda}\to\cC_{\mu}$ is strongly continuous. The same argument as above shows that the colimit $\indlim[\lambda<\kappa]^{\cont}\cC_{\lambda}$ is identified with the full subcategory of $\cC_0$ consisting of sheaves supported at the point $\{(\kappa,0)\}.$ This subcategory is equivalent to $D(\mk).$

The proof of the vanishing of $\cC_{\lambda}^{\omega}\simeq\cC_0^{\omega}$ is completely analogous. Instead of the above subsets $Y_n$ we consider for $\lambda<\kappa$ the subset $Y_{\lambda}'=([\lambda,\kappa]\times\{0\})\cup ([0,\lambda]\times [0,1]),$ which is again connected and compact. Repeating the above argument, we see that for a non-zero compact object $\cF\in\cC_0^{\omega}$ its restriction to $Z'=[0,\kappa)\times\{1\}$ would be a nowhere zero locally constant sheaf $\cG,$ which has to also be a compact object of $\Shv(Z'; D(\mk)).$ This again gives a contradiction with Proposition \ref{prop:compact_sheaves} since $\supp(\cG) = Z'$ is not compact.
\end{proof}

We observe that the Calkin construction commutes with sufficiently filtered colimits.

\begin{prop}\label{prop:Calkin_construction_accessible} Let $\kappa$ be an uncountable regular cardinal. The functor $\Calk_{\kappa}^{\cont}:\Cat_{\st}^{\dual}\to\Cat^{\perf}$ commutes with $\kappa$-filtered colimits. In particular, the functor $\Calk_{\kappa}:\Cat^{\perf}\to \Cat^{\perf}$ commutes with $\kappa$-filtered colimits.\end{prop}

\begin{proof}Recall that the functor $\Pr^L_{\st,\kappa}\to\Cat^{\perf},$ $\cC\mapsto\cC^{\kappa},$ commutes with $\kappa$-filtered colimits. Hence, so does the composition $(-)^{\kappa}:\Cat_{\st}^{\dual}\to \Pr^L_{\st,\kappa}\to\Cat^{\perf}.$ It follows that the functor $\Cat_{\st}^{\dual}\to\Cat_{\st}^{\cg},$ $\cC\mapsto\Ind(\cC^{\kappa})/\hat{\cY}(\cC),$ commutes with $\kappa$-filtered colimits. Passing to compact objects, we deduce that the functor $\cC\mapsto (\Ind(\cC^{\kappa})/\hat{\cY}(\cC))^{\omega}=\Cat_{\kappa}^{\cont}(\cC)$ commutes with $\kappa$-filtered colimits.\end{proof}

We conclude this subsection with the following statement about pushouts.

\begin{prop}\label{prop:pushouts_of_mono} In each of the categories $\Cat_{\st}^{\dual},$ $\Pr_{\st}^{LL}$ and $\Pr_{\st}^L,$ the class of fully faithful functors is closed under pushouts.\end{prop}

\begin{proof}It suffices to prove this for $\Pr^L_{\st}.$ Passing to the right adjoints, we that the statement is exactly the assertion of Proposition \ref{prop:localizations_stable_under_pullbacks} for colocalizations.\end{proof}

\subsection{Semi-orthogonal decompositions}
\label{ssec:SOD}

We have the natural notion of a semi-orthogonal decomposition in $\Pr_{\st}^{LL}$ and in $\Cat_{\st}^{\dual}.$

\begin{defi}\label{defi:finite_SOD} A (finite) semi-orthogonal decomposition in $\Pr^{LL}_{\st}$ (resp. $\Cat_{\st}^{\dual}$) of a presentable stable (resp. dualizable) category $\cC$ is a collection of presentable stable (resp. dualizable) subcategories $\cC_1,\dots,\cC_n\subset\cC$ such that 

i) the inclusion functors $\cC_i\to\cC$ are strongly continuous;

ii) For $i>j,$ $x\in \cC_i,$ $y\in\cC_j,$ we have $\Hom(x,y)=0;$

iii) The categories $\cC_i$ generate $\cC$ as a stable subcategory.

In this case we write $\cC=\langle\cC_1,\dots,\cC_n\rangle.$\end{defi}

\begin{remark}\label{rem:SOD_PR^LL} Let $i_1:\cC\to\cE$ and $i_2:\cD\to\cE$ be strongly continuous functors between presentable stable categories. Then $i_1$ and $i_2$ define a semi-orthogonal decomposition $\cE=\la \cC,\cD\ra$ in $\Pr_{\st}^{LL}$ if and only if the following conditions hold:

\begin{itemize}
\item $i_1$ (resp. $i_2$) has a strongly continuous left (resp. right) adjoint $i_1^L$ (resp. $i_2^R$);

\item we have $i_1^L i_1\xto{\sim}\Id_{\cC}$ and $\Id_{\cD}\xto{\sim}i_2^R i_2;$

\item we have $i_1^L i_2=0$ (equivalently, $i_2^R i_1=0$);

\item we have a cofiber sequence $i_2 i_2^R\to\Id_{\cE}\to i_1 i_1^L$ in $\Fun^{LL}(\cE,\cE).$
\end{itemize}

This is a straightforward generalization of a similar statement for small stable categories. Note that the strong continuity of $i_2^R$ follows from the cofiber sequence in the last condition.
\end{remark}

The following observations are immediate.

\begin{prop}\label{prop:finite_SOD_dualizable} Let $\cC=\la \cC_1,\dots,\cC_n\ra$ be a semi-orthogonal decomposition in $\Pr_{\st}^{LL}.$ Then $\cC$ is dualizable if and only if each $\cC_i$ is dualizable. In this case we also have a semi-orthogonal decomposition on the level of compact objects: $\cC^{\omega}=\la \cC_1^{\omega},\dots,\cC_n^{\omega}\ra.$\end{prop}

\begin{proof}If $\cC$ is dualizable, then each $\cC_i$ is a retract of $\cC$ in $\Pr^L_{\st},$ hence $\cC_i$ is dualizable. Conversely, if each $\cC_i$ is dualizable, then $\cC$ is dualizable by Proposition \ref{prop:condition_for_dualizability}.

For the second statement we may assume that $n=2.$ Denote by $i_1:\cC_1\to\cC,$ $i_2:\cC_2\to\cC$ the inclusion functors. Then the functors $i_1^{\omega}:\cC_1^{\omega}\to\cC^{\omega}$ and $i_2^{\omega}:\cC_2^{\omega}\to\cC^{\omega}$ satisfy the version of conditions of Remark \ref{rem:SOD_PR^LL} for small stable categories. Hence, they induce a semi-orthogonal decomposition $\cC^{\omega}=\la \cC_1^{\omega},\cC_2^{\omega}\ra.$\end{proof}

The usual construction of a semi-orthogonal gluing via a bimodule works in the context of dualizable categories.

\begin{prop}\label{prop:gluing_of_dualizable} Let $F:\cD\to\cC$ be an accessible functor between dualizable categories, and consider the semi-orthogonal gluing $\cE=\cC\oright_F \cD$ (Definition \ref{defi:gluing_general}), with the inclusion functors $i_1:\cC\to\cE,$ $i_2:\cD\to\cE,$ given by $i_1(x) = (x[-1],0,0),$ $i_2(x)=(0,x,0).$ The following are equivalent.
\begin{enumerate}[label=(\roman*),ref=(\roman*)]
\item the category $\cE$ is dualizable and the functors $i_1$ and $i_2$ are strongly continuous. \label{cont1}

\item the functor $F$ is continuous. \label{cont2}
\end{enumerate}
\end{prop}

\begin{proof} Note that the functor $i_2$ is automatically strongly continuous. 

\Implies{cont1}{cont2}. It suffices to recall that $F\cong i_1^R\circ i_2,$ where $i_1^R$ is the right adjoint to $i_1.$ Note that we did not use the dualizability of $\cE.$

\Implies{cont2}{cont1}. Suppose that $F$ is continuous. Note that the functor $i_1^R$ is given by
$$i_1^R(x,y,\varphi)=\Cone(x\xto{\varphi}F(y)).$$ Hence, $i_1$ is strongly continuous. The dualizability of $\cE$ follows from Proposition \ref{prop:finite_SOD_dualizable}.\end{proof}

We have the following expected statement about the functors to and from a dualizable category with a semi-orthogonal decomposition.

\begin{prop}\label{prop:functors_in_and_out_SOD} Let $\cE$ be a presentable stable category with a semi-orthogonal decomposition $\cE=\la i_1(\cC), i_2(\cD)\ra$ in $\Pr_{\st}^{LL}.$ Denote by $\pi_1:\cE\to\cC$ the left adjoint to $i_1,$ and $\pi_2:\cE\to\cD$ the right adjoint to $i_2.$ Let $\cT$ be another presentable stable category.

1) A continuous functor $F:\cT\to\cE$ is strongly continuous if and only if the compositions $\pi_1\circ F:\cT\to\cC$ and $\pi_2\circ F:\cT\to\cD$ are strongly continuous. We have a semi-orthogonal decomposition
$$\Fun^{LL}(\cT,\cE)=\la \Fun^{LL}(\cT,\cC), \Fun^{LL}(\cT,\cD)\ra.$$

2) A continuous functor $G:\cE\to\cT$ is strongly continuous if and only if the compositions $G\circ i_1:\cC\to\cT$ and $G\circ i_2:\cD\to\cT$ are strongly continuous. We have a semi-orthogonal decomposition
$$\Fun^{LL}(\cE,\cT)=\la \Fun^{LL}(\cC,\cT), \Fun^{LL}(\cD,\cT)\ra.$$ \end{prop}

\begin{proof} We first prove part 2). The criterion of strong continuity follows from Proposition \ref{prop:automatic_strcont} applied to the functors $i_1:\cC\to\cE,$ $i_2:\cD\to\cE.$ 
The functors $i_1':\Fun^{LL}(\cC,\cT)\to \Fun^{LL}(\cE,\cT)$ and $i_2':\Fun^{LL}(\cD,\cT)\to \Fun^{LL}(\cE,\cT)$ are given respectively by precomposition with $\pi_1$ and $\pi_2.$ The left adjoint to $i_1'$ is given by precomposition with $i_1,$ and the right adjoint to $i_2'$ is given by precomposition with $i_2.$ It follows formally that $i_1'$ and $i_2'$ are fully faithful and the composition $i_1'^L i_2'$ is zero. It remains to observe that the cofiber sequence $i_2\pi_2\to\Id\to i_1\pi_1$ in $\Fun^{LL}(\cE,\cE)$ implies that we have a cofiber sequence $i_2' i_2'^R\to\Id\to i_1'i_1'^L$ in $\Fun(\Fun^{LL}(\cE,\cT),\Fun^{LL}(\cE,\cT)).$

We now prove part 1). The proof of a semi-orthogonal decomposition is similar; here the inclusion functors are given by the compositions with $i_1$ and $i_2.$ The ``only if'' direction of the criterion is clear. To prove the ``if'' direction we note that the functors $i_1\pi_1 F$ and $i_2\pi_2 F$ are strongly continuous, and the cofiber sequence $i_2\pi_2 F\to F\to i_1\pi_1 F$ implies that $F$ is also strongly continuous.
\end{proof}

We also define the notion of a poset-indexed (possibly infinite) semi-orthogonal decomposition.

\begin{defi}\label{defi:infinite_SOD}Let $I$ be a partially ordered set, and let $\cC$ be a presentable stable (resp. dualizable) category. An $I$-indexed semi-orthogonal decomposition of $\cC$ in $\Pr^{LL}_{\st}$ (resp. in $\Cat_{\st}^{\dual}$) is a collection of presentable stable (resp. dualizable) subcategories $\cC_i\subset\cC,$ $i\in I,$ such that

i) the inclusion functors $\cC_i\to\cC,$ $i\in I,$ are strongly continuous;

ii) For $i\not\leq j,$ $x\in \cC_i,$ $y\in\cC_j,$ we have $\Hom(x,y)=0;$

iii) The categories $\cC_i$ generate $\cC$ as a localizing subcategory.

In this case we write $\cC=\la \cC_i; i\in I\ra.$\end{defi}

\begin{prop}\label{prop:infinite_SOD_of_dualizable} Let $I$ be a poset and let $\cC=\la \cC_i; i\in I\ra$ be a semi-orthogonal decomposition in $\Pr_{\st}^{LL}.$ Then $\cC$ is dualizable iff each $\cC_i$ is dualizable. In this case we have a semi-orthogonal decomposition $\cC^{\omega}=\la \cC_i^{\omega}; i\in I\ra.$\end{prop}

\begin{proof}The first statement is proved in the same way as Proposition \ref{prop:finite_SOD_dualizable}.

To prove the second statement, we introduce for each finite subset $J\subset I$ the subcategory $\cC_J\subset\cC$ generated by $\cC_j,$ $j\in J.$ By Propositions \ref{prop:condition_for_dualizability} and \ref{prop:automatic_strcont} the category $\cC_J$ is dualizable and the inclusion $\cC_J\to\cC$ is strongly continuous. Moreover, we have a semi-orthogonal decomposition $\cC_J=\la \cC_j, j\in J\ra.$ Note that we can choose a refinement of the partial order on $J$ to a total order, so that we also have a semi-orthogonal decomposition of $\cC_J$ in the sense of Definition \ref{defi:finite_SOD}. By Proposition \ref{prop:finite_SOD_dualizable}, the category $\cC_J^{\omega}$ is generated as a stable category by $\cC_j^{\omega},$ $j\in J.$ Combining Corollary \ref{cor:directed_unions} and Proposition \ref{prop:compact_objects_in_filtered_colimits} we deduce that the category $\cC^{\omega}\simeq\indlim[J\subset I]\cC_J^{\omega}$ is generated as a stable subcategory by $\cC_i^{\omega},$ $i\in I.$ Hence, we have a semi-orthogonal decomposition $\cC^{\omega}=\la \cC_i^{\omega}; i\in I\ra,$ as required.\end{proof}

\subsection{Finite limits of dualizable categories}
\label{ssec:finite_limits_of_dualizable}

The limits in the category $\Cat_{\st}^{\dual}$ of dualizable categories are much more complicated than colimits. We start their study with the case of kernels (i.e. fibers). We denote by $\Q_{\leq}$ the poset of rational numbers with the usual order.

\begin{prop}\label{prop:functors_from_Q} Let $\cC$ be a dualizable stable category. For any $x\in\cC^{\omega_1}$ there exists a functor $\Phi:\Q_{\leq}\to\cC$ such that

\begin{itemize}
\item $\indlim[a\in\Q]\Phi(a)\cong x$

\item for any $a\in\Q,$ we have $\indlim[b<a]\Phi(b)\xto{\sim}\Phi(a);$

\item for any $a<b,$ $a,b\in\Q,$ the morphism $\Phi(a)\to\Phi(b)$ is compact.
\end{itemize}
\end{prop}

\begin{proof}It is convenient to replace the poset $\Q_{\leq}$ with the equivalent poset $\Z[1/2]_{\leq}.$

Choose a sequence $x_1\to x_2\to\dots$ in $\cC$ such that $\indlim[n] x_n\cong x$ and each map $x_n\to x_{n+1}$ is compact. Define the functor $\Phi'_0:\Z_{\leq}\to \cC$ by the formula $\Phi'_0(a)=\begin{cases}x_a & \text{ for }a\geq 1;\\
0 & \text{ for }a\leq 0.\end{cases}$ Since any compact morphism in $\cC$ is a composition of two compact morphisms, we can define inductively the functors $\Phi'_n:(\frac{1}{2^n}\Z)_{\leq}\to \cC$ such that $(\Phi'_{n+1})_{\mid \frac{1}{2^n}\Z}\cong \Phi'_n,$ and for any $a,b\in\frac{1}{2^n}\Z$ such that $a<b,$ the map $\Phi'_n(a)\to\Phi'_n(b)$ is compact. The sequence $(\Phi'_n)_{n\geq 1}$ defines a functor $\Phi':\Z[\frac{1}{2}]_{\leq}\to\cC.$

Finally, we define a functor $\Phi:\Z[\frac{1}{2}]_{\leq}\to\cC$ by the formula $\Phi(a)=\indlim[b<a]\Phi'(a).$ Then $\Phi$ satisfies the required properties (since $\Z[\frac12]_{\leq}$ is a dense linearly ordered set).\end{proof}

The following lemma will be useful.

\begin{lemma}\label{lem:largest_dualizable} Let $\cC$ be a dualizable category, and let $\cD\subset\cC$ be a localizing subcategory. Then there exists the largest dualizable subcategory $\cE\subset\cD$ such that the inclusion functor $\cE\to\cC$ is strongly continuous.

Moreover, $\cE$ is generated by the objects of the form $\indlim[](F:\Q_{\leq}\to\cD),$ where $F$ is a functor such that for any rational numbers $a<b$ the map $F(a)\to F(b)$ is compact in $\cC.$\end{lemma}

\begin{proof}Denote by $\cE'\subset\cD$ the full subcategory generated (as a localizing subcategory) by the objects described in the second assertion. We first prove that $\cE'$ is dualizable and the inclusion $\cE'\to\cC$ is strongly continuous. 

Consider a functor $F:\Q_{\leq}\to\cD$ as above. We may and will assume that for each rational $a$ the map $\indlim[b<a]F(b)\to F(a)$ is an equivalence. Indeed, as in the proof of Proposition \ref{prop:functors_from_Q} we may replace $F$ by the functor $F':\Q_{\leq}\to\cD,$ given by $F'(a)=\indlim[b<a]F(b).$ 

Now, we observe that our assumptions on $F$ imply that each object $F(a)$ is contained in $\cE'$ (because the linearly ordered set $\Q_{<a}$ is equivalent to $\Q_{\leq}$). We also have $\hat{\cY}_{\cC}(\indlim[](F))=\inddlim(F(0)\to F(1)\to\dots),$ since by assumption each map $F(n)\to F(n+1)$ is compact in $\cC$ for $n\in\N.$ We conclude that the functor $\hat{\cY}_{\cC}$ takes $\cE'$ to the essential image of $\Ind(\cE')$ in $\Ind(\cC).$ In other words, $\cE'$ is dualizable and the inclusion functor $\cE'\to \cC$ commutes with $\hat{\cY},$ hence it is strongly continuous. 

Now, let $\cE''\subset \cD$ be some full dualizable subcategory, such that the inclusion $\cE''\to\cC$ is strongly continuous. We claim that $\cE''$ is contained in $\cE'.$ To see this, take any countably presented object $x\in(\cE'')^{\omega_1}.$ Choose a functor $F:\Q_{\leq}\to \cE''$ such that each map $F(a)\to F(b)$ is compact in $\cE''$ (hence in $\cC$) and $\indlim[](F)\cong x.$ We see that $x\in\cE'$ as required.\end{proof}

We can now give a description of kernels in $\Cat_{\st}^{\dual}.$

\begin{prop}\label{prop:ker^dual} Let $F:\cC\to\cD$ be a strongly continuous functor between dualizable categories. Denote by $\ker(F)\subset\cC$ the usual kernel of $F$ (the fiber in $\Pr^L_{\st}$). Then the dualizable kernel of $F$ (i.e. the fiber in $\Cat_{\st}^{\dual}$) is given by the largest dualizable full subcategory $\cE\subset\ker(F)$ such that the inclusion $\cE\to\cC$ is strongly continuous. More precisely, $\cE$ can be described as in Lemma \ref{lem:largest_dualizable}.\end{prop}

\begin{proof}Let $\cT$ be a dualizable category, and $G:\cT\to\cC$ a strongly continuous functor such that $F\circ G=0.$ We want to show that the image of $G$ is contained in $\cE.$ But Corollary \ref{cor:images_in_Cat^dual} implies that the category $\im(G)$ (as in Definition \ref{defi:image}) is dualizable and the inclusion $\im(G)\to\cC$ is strongly continuous. Since $\im(G)\subset\ker(F),$ we conclude that $\im(G)\subset\cE$ as required.\end{proof}

We denote by $\ker^{\dual}(F)$ the category described in Proposition \ref{prop:ker^dual}.

\begin{example}Let $R$ be a discrete valuation ring with the maximal ideal $\m\subset R,$ and with the field of fractions $K.$ Then $\ker^{\dual}(D(R)\to D(R/\m))=0.$ Indeed, the usual kernel is given by $\ker(D(R)\to D(R/\m))=D(K),$ but the inclusion functor $D(K)\to D(R)$ is not strongly continuous. The only localizing subcategory of $D(K)$ which is not equal to $D(K)$ is zero.\end{example}

\begin{cor}\label{cor:seq_colimits_finite_limits} Sequential colimits in $\Cat_{\st}^{\dual}$ do not commute with finite limits. In particular, the category $\Cat_{\st}^{\dual}$ is not compactly assembled.\end{cor}

\begin{proof}Take some prime number $p$ and denote by $R_n=\Z_p(\zeta_{p^n})$ the ring of integers of the cyclotomic extension $\Q_p(\zeta_{p^n})\supset \Q_p.$ Then $\ker^{\dual}(D(R_n)\to D(\FF_p))=0.$ On the other hand, the colimit $R=\indlim[n]R_n$ is a non-discrete valuation ring, hence $$\ker^{\dual}(\indlim[n]^{\cont}D(R_n)\to D(\FF_p))\simeq\ker^{\dual}(D(R)\to D(\FF_p))\simeq D(\Mod_a\hy R)\ne 0.$$
Hence, sequential colimits in $\Cat_{\st}^{\dual}$ do not commute with kernels. This implies that $\Cat_{\st}^{\dual}$ is not compactly assembled.\end{proof}

We have the following general statement about fiber products in $\Cat_{\st}^{\dual}.$

\begin{prop}\label{prop:nice_pullbacks} Consider a pair of strongly continuous functors between dualizable categories $F:\cA\to\cC,$ $G:\cB\to\cC.$ 

1) The fiber product $\cA\times^{\dual}_{\cC}\cB$ exists in $\Cat_{\st}^{\dual},$ and the natural functor $\cA\times^{\dual}_{\cC}\cB\to \cA\times_{\cC}\cB$ is fully faithful and strongly continuous. More precisely, the category $\cA\times^{\dual}_{\cC}\cB$ is the largest dualizable subcategory $\cE\subset \cA\times_{\cC}\cB,$ such that the functors $\cE\to\cA$ and $\cE\to\cB$ are strongly continuous.

2) If moreover $G$ is a localization, then we have $\cA\times^{\dual}_{\cC}\cB\simeq \cA\times_{\cC}\cB,$ and the functor $\cA\times^{\dual}_{\cC}\cB\to\cA$ is a localization.\end{prop}

\begin{proof}1) Consider the lax fiber product $\cA\laxtimes{\cC}\cB\simeq \cA\oright_{F^R G}\cB$ (Definition \ref{defi:gluing_general}). By Proposition \ref{prop:gluing_of_dualizable} it is dualizable. It follows from Proposition \ref{prop:functors_in_and_out_SOD} that $\cA\times^{\dual}_{\cC}\cB\simeq\ker^{\dual}(\Phi:\cA\laxtimes{\cC}\cB\to \cC),$ where $\Phi(x,y,\varphi)=\Cone(F(x)\to G(y)).$ The usual fiber product (taken in $\Pr^L_{\st}$) is simply $\ker(\Phi).$ The rest follows from Proposition \ref{prop:ker^dual}.

2) By Proposition \ref{prop:localizations_stable_under_pullbacks}, in $\Pr^L_{\st}$ a pullback of a localization is a localization. To show that $\cA\times_{\cC}\cB$ is also the fiber product in $\Cat_{\st}^{\dual},$ it suffices to show that the above functor $\Phi:\cA\laxtimes{\cC}\cB\to \cC$ is a localization (assuming that $G$ is a localization). This is standard. Explicitly, the right adjoint $\Phi^R$ is given by $\Phi^R(x)=(0,G^R(x),0),$ and $\Phi\circ\Phi^R\cong \Id_{\cC}.$\end{proof}

We deduce the description of monomorphisms in $\Cat_{\st}^{\dual}.$

\begin{cor}\label{cor:mono_in_Cat^dual}
Let $F:\cC\to\cD$ be a strongly continuous between dualizable categories. Then $F$ is a monomorphism in $\Cat_{\st}^{\dual}$ if and only if $F$ is fully faithful.
\end{cor}

\begin{proof}
The ``if'' direction is clear. We prove the ``only if'' direction. 

Suppose that $F$ is a monomorphism in $\Cat_{\st}^{\dual}.$ By Proposition \ref{prop:nice_pullbacks} the functor
$\cC\times_{\cD}^{\dual}\cC\to \cC\times_{\cD}\cC$ is fully faithful. Hence, the functor $\cC\to \cC\times_{\cD}\cC$ is fully faithful. Arguing as in the proof of part 2) of Proposition \ref{prop:mono_epi_Pr^L} (namely, the proof of the fully faithfulness of $F^R$ in the notation of loc. cit.) we conclude that $F$ is fully faithful.
\end{proof}

We will need the following statement about the class of situations when the fiber product of compactly generated categories is automatically compactly generated. We recall that the notion of a homological epimorphism is given in Definition \ref{defi:localizations_and_hom_epi}.

\begin{prop}\label{prop:fiber_product_hom_epi} Let $F:\cA\to\cB$ be a homological epimorphism of small stable idempotent-complete categories. Then we have equivalences
$$\Ind(\cA\times_{\cB}\cA)\simeq \Ind(\cA)\times_{\Ind(\cB)}^{\dual}\Ind(\cA)\simeq \Ind(\cA)\times_{\Ind(\cB)}\Ind(\cA).$$\end{prop}

\begin{proof}By Proposition \ref{prop:nice_pullbacks} we only need to show that the category $\Ind(\cA)\times_{\Ind(\cB)}\Ind(\cA)$ is compactly generated. It suffices to show that the objects of the form $(x,0),$ $x\in\ker(\Ind(F)),$ are generated by colimits by the essential image of the diagonal functor $\cA\to\cA\times_{\cB}\cA.$

By Proposition \ref{prop:generating_the_kernel}, we may assume that $x$ is of the form $\inddlim[](x_1\xto{f_1} x_2\xto{f_2}\dots),$ where $x_n\in\cA$ and $F(f_n):F(x_n)\to F(x_{n+1})$ is a zero map in $\cB.$ Then we have
$$(x,0)\cong \indlim((x_1,x_1)\xto{g_1} (x_2,x_2)\xto{g_2}\dots)\quad\text{in }\Ind(\cA)\times_{\Ind(\cB)}\Ind(\cA),$$
where the maps $g_n:(x_n,x_n)\to (x_{n+1},x_{n+1})$ are chosen (non-uniquely) in such a way that they induce the pairs of maps $(f_n:x_n\to x_{n+1},\,0:x_n\to x_{n+1}).$ This proves the proposition.\end{proof}

\subsection{General limits of dualizable categories}
\label{ssec:general_limits_Cat^dual}

To describe the general limits in $\Cat_{\st}^{\dual},$ we will use the following non-standard adjunction statement.

\begin{prop}\label{prop:nonstandard_adjunction} Let $\kappa$ be an uncountable regular cardinal. The (non-full) subcategory inclusion functor $\Cat_{\st}^{\dual}\to \Pr^L_{\st,\kappa}$ has a right adjoint, given by $\cC\mapsto \Ind(\cC^{\kappa}).$\end{prop}

\begin{proof}We construct the adjunction unit and counit as follows. 

Let $\cC$ be a dualizable stable category. Then the adjunction unit on $\cC$ is given by the functor $\hat{\cY}:\cC\to \Ind(\cC^{\kappa})$ (we know that this functor is strongly continuous).

Let $\cD$ be a $\kappa$-presentable category. Then the adjunction counit is given by the colimit functor $\colim:\Ind(\cD^{\kappa})\to \cD.$ Note that this functor is $\kappa$-strongly continuous.

If $\cC$ is dualizable, then the composition $\cC\xto{\hat{\cY}} \Ind(\cC^{\kappa})\xto{\colim} \cC$ is isomorphic to the identity functor. If $\cD$ is $\kappa$-presentable, then the composition $$\Ind(\cD^{\kappa})\xto{\Ind(\cY)} \Ind(\Ind(\cD^{\kappa})^{\kappa})\xto{\colim} \Ind(\cD^{\kappa})$$ is also isomorphic to the identity. This proves the adjunction.
\end{proof}

\begin{remark}Given a dualizable category $\cC$ and a $\kappa$-presentable category $\cD$ (where $\kappa$ is uncountable), a $\kappa$-strongly continuous functor $F:\cC\to\cD$ corresponds via the above adjunction to the composition $\cC\xto{\hat{\cY}}\Ind(\cC^{\kappa})\xto{\Ind(F^{\kappa})} \Ind(\cD^{\kappa}).$\end{remark}

\begin{theo}\label{th:limit_of_dualizable} Let $I$ be an $\infty$-category, and $\Phi:I\to\Cat_{\st}^{\dual}$ be a functor, $i\mapsto \cC_i.$ Then the limit of $\Phi$ exists, and for any uncountable regular cardinal $\kappa$ we have an equivalence
$$\prolim[i]^{\dual}\cC_i=\ker^{\dual}(\Ind(\prolim[i]\cC_i^{\kappa})\to \Ind(\prolim[i] \Calk_{\kappa}^{\cont}(\cC_i))).$$\end{theo}

\begin{proof}By Proposition \ref{prop:nonstandard_adjunction} we see that the limit $\prolim[i]^{\dual}\Ind(\cC_i^{\kappa})$ exists and is given by $\Ind(\prolim[i]\cC_i^{\kappa}).$ 
Let $\cD$ be a dualizable category. Note that a strongly continuous functor $\cD\to \Ind(\cC_i^{\kappa})$ corresponds to a strongly continuous functor $\cD\to\cC_i$ if and only if the composition $\cD\to \Ind(\cC_i^{\kappa})\to \Ind(\Calk_{\kappa}^{\cont}(\cC_i))$ is zero. It follows formally that $\prolim[i]^{\dual}\cC_i$ exists and it is the largest dualizable subcategory $$\cE\subset\ker(\Ind(\prolim[i]\cC_i^{\kappa})\to\prodd[i] \Ind(\Calk_{\kappa}^{\cont}(\cC_i)))$$ such that the inclusion functor $\cE\to \Ind(\prolim[i]\cC_i^{\kappa})$ is strongly continuous.

Note that for any functor $I\to\Cat_{\st}^{\dual},$ $i\mapsto\cD_i,$ we have a strongly continuous fully faithful embedding $\Ind(\prolim[i] \cD_i^{\omega})\to \prolim[i]^{\dual}\cD_i.$ To conclude, we simply observe that $\prolim[i]^{\dual}\cC_i$ is equivalent to $$\ker^{\dual}(\prolim[i]^{\dual}\Ind(\cC_i^{\kappa})\to \prolim[i]^{\dual}\Ind(\Calk_{\kappa}^{\cont}(\cC_i))),$$ and we have a factorization
$$\prolim[i]^{\dual}\Ind(\cC_i^{\kappa})\simeq \Ind(\prolim[i]\cC_i^{\kappa})\to \Ind(\prolim[i] \Calk_{\kappa}^{\cont}(\cC_i))\hto \prolim[i]^{\dual}\Ind(\Calk_{\kappa}^{\cont}(\cC_i)),$$
where the latter functor is fully faithful.

This proves the theorem.
\end{proof}

\subsection{Example: fpqc descent}
\label{ssec:fpqc_descent}

To give a feeling of dualizable inverse limits, we prove two statements about descent for functors with values in $\Cat_{\st}^{\dual}.$ We denote by $\Ring_{\bE_{\infty}}$ the $\infty$-category of $\bE_{\infty}$-ring spectra (without connectivity assumptions). Recall that the descendable topology on $\Ring_{\bE_{\infty}}^{op}$ is given by the pretopology in which the covers are the maps $R\to S$ such that $S$ is a descendable $R$-algebra, see \cite[Definition 3.18]{Mat}. Namely, this means that the unit object $R$ is contained in the stable idempotent-complete ideal of $\Mod\hy R$ generated by $S.$ 

\begin{prop}\label{prop:descendable_descent} The functor $\Ring_{\bE_{\infty}}\to \Cat_{\st}^{\dual},$ $R\mapsto \Mod\hy R,$ is a sheaf in the descendable topology.
\end{prop}

This can be deduced from \cite[Corollary 4.85]{Ram24b}, but we would like to give an explicit proof using the description of the limit from Theorem \ref{th:limit_of_dualizable}. We need the following lemma.

\begin{lemma}\label{lem:detects_kappa_compactness} Let $F:\cC\to\cD$ be a strongly continuous functor between presentable stable categories, and denote by $F^R$ its right adjoint. Suppose that the identity functor $\Id_{\cC}$ considered as an object of $\Fun^L(\cC,\cC)$ is contained in the idempotent-complete stable subcategory generated by functors of the form $F^R\circ G.$ where $G\in\Fun^L(\cC,\cD).$ Then for any regular cardinal $\kappa$ the functor $F$ preserves and reflects $\kappa$-compact objects.\end{lemma}

\begin{proof}Since $F^R$ is continuous, $F$ takes $\cC^{\kappa}$ to $\cD^{\kappa}.$ 

Let $x\in\cC$ be an object such that $F(x)\in\cD$ is $\kappa$-compact. Take a $\kappa$-directed poset $I$ and a functor $\Phi:I\to \cC,$ $\Phi(i) = y_i.$ Let $G:\cC\to\cD$ be a continuous functor. By the $\kappa$-compactness of $F(x)$ and by the continuity of $F^R$ we have
\begin{multline*}
\indlim[i\in I]\Hom_{\cC}(x,F^R(G(y_i)))\cong \indlim[i\in I]\Hom_{\D}(F(x),G(y_i))\xto{\sim} \Hom_{\cD}(F(x),\indlim[i\in I] G(y_i))\\
\cong \Hom_{\cC}(x,F^R(\indlim[i\in I] G(y_i)))\cong \Hom_{\cC}(x,\indlim[i\in I]F^R(G(y_i))).
\end{multline*}
By assumption $\Phi$ considered as an object of $\Fun(I,\cC)$ is contained in the idempotent-complete stable subcategory generated by $F^R\circ G\circ \Phi,$ where $G\in\Fun^L(\cC,\cD).$ Therefore, we have an equivalence 
\begin{equation*}
\indlim[i\in I]\Hom_{\cC}(x,y_i)\xto{\sim} \Hom_{\cC}(x,\indlim[i\in I]y_i).
\end{equation*} 
This proves that $x$ is $\kappa$-compact in $\cC,$ so $F$ reflect $\kappa$-compactness.
\end{proof}



\begin{proof}[Proof of Proposition \ref{prop:descendable_descent}] Let $R\to S$ be a map such that $S$ is descendable over $R.$ By \cite[Proposition 3.22]{Mat} we have an equivalence $\Mod\hy R \simeq\Tot(\Mod\hy S^{\otimes\bullet+1}).$ By Lemma \ref{lem:detects_kappa_compactness}, it induces an equivalence $(\Mod\hy R)^{\omega_1}\simeq\Tot((\Mod\hy S^{\otimes\bullet+1})^{\omega_1}).$ By Theorem \ref{th:limit_of_dualizable}, we obtain
$$\Tot^{\dual}(\Mod\hy S^{\otimes\bullet+1})\simeq\ker^{\dual}(\Ind((\Mod\hy R)^{\omega_1})\to \Ind(\Tot(\Calk_{\omega_1}(S^{\otimes\bullet+1})))).$$
It suffices to prove the following statement:

$(*)$ the functor $\Calk_{\omega_1}(R)\to \Tot(\Calk_{\omega_1}(S^{\otimes\bullet +1}))$ is fully faithful.

To prove $(*),$ recall that by Proposition \ref{prop:Homs_in_Calk} for $M,N\in (\Mod\hy R)^{\omega_1}$ the spectrum of morphisms between $M$ and $N$ in the Calkin category is given by
$$\Hom_{\Calk_{\omega_1}(R)}(M,N)=\Cone(\Hom_R(M,R)\tens{R} N\to \Hom_R(M,N)).$$

Since we know the descent for $(\Mod\hy R)^{\omega_1},$ it suffices to show that the map
$$\Hom_R(M,R)\tens{R} N\to\Tot\left(\Hom_{S^{\otimes \bullet+1}}(S^{\otimes\bullet+1}\tens{R} M,S^{\otimes\bullet+1})\tens{S^{\otimes\bullet+1}}(S^{\otimes\bullet+1}\tens{R} N)\right)$$ is an isomorphism. This is equivalent to proving the following:

$(**)$ the map $\Hom_R(M,R)\tens{R} N\to \Tot(\Hom_R(M,S^{\otimes\bullet+1})\tens{R} N)$ is an isomorphism.

This follows easily from the descendability of $S$ over $R.$ Indeed, by \cite[Proposition 3.20]{Mat} the inverse sequence $(\Tot_n S^{\otimes \bullet+1})_n$ is pro-constant, hence we have equivalences
\begin{multline*}\Tot(\Hom_R(M,S^{\otimes\bullet+1})\tens{R}N)\cong \prolim[n] \Tot_n(\Hom_R(M,S^{\otimes\bullet+1})\tens{R}N)\\ \cong
 \prolim[n] \Hom_R(M,\Tot_n(S^{\otimes\bullet+1}))\tens{R}N\cong \Hom_R(M,\prolim[n] \Tot_n(S^{\otimes\bullet+1}))\tens{R}N\\ \cong
  \Hom_R(M,R)\tens{R}N. 
\end{multline*}
This proves the proposition.
\end{proof}

\begin{remark}It seems plausible that for any descendable map $R\to S$ and for any uncountable regular cardinal $\kappa$ we have an equivalence
$$\Calk_{\kappa}(R)\simeq \Tot(\Calk_{\kappa}(S^{\otimes \bullet +1})).$$
However, we only checked the fully faithfulness statement (the proof for $\kappa>\omega_1$ is the same). It is not clear how to prove the essential surjectivity even for $\kappa=\omega_1.$\end{remark}

Denote by $\Ring$ the category of (discrete) commutatie rings.

\begin{prop}\label{prop:fpqc_descent} The functor $\Ring\to \Cat_{\st}^{\dual},$ $R\mapsto D(R),$ is a sheaf in the fpqc topology.\end{prop}

\begin{proof} Descent for finite disjoint unions is obvious. We need to prove that for a faithfully flat map $R\to S$ we have $D(R)\simeq\Tot^{\dual}(D(S^{\otimes \bullet + 1})).$ As in the proof of Proposition \ref{prop:descendable_descent}, we only need to show that $D(R)^{\omega_1}\simeq\Tot(D(S^{\otimes\bullet+1})^{\omega_1})$ and to prove the fully faithfulness statement $(*)$ above.

By Proposition \ref{prop:Faithfully_flat_aleph_1_accessible} below, $S$ is isomorphic to an $\omega_1$-directed colimit $\indlim[i\in I]S_i,$ where $S_i$ are descendable faithfully flat $R$-algebras (and $I$ is an $\omega_1$-directed poset). We know that the functors $D(-)^{\omega_1}:\Ring\to\Cat^{\perf}$ and $\Calk_{\omega_1}(-):\Ring\to\Cat^{\perf}$ commute with $\omega_1$-filtered colimits (by Proposition \ref{prop:Calkin_construction_accessible}). Also, the functor $$\Tot(-):(\Cat^{\perf})^{\Delta}\to\Cat^{\perf}$$ commutes with $\omega_1$-filtered colimits, since the category $\Delta$ is countable and the category $\Cat^{\perf}$ is compactly generated. We deduce from the proof of Proposition \ref{prop:descendable_descent} that $D(R)^{\omega_1}\simeq\Tot(D(S^{\otimes\bullet+1})^{\omega_1})$ and the functor $\Calk_{\omega_1}(R)\to \Tot(\Calk_{\omega_1}(S^{\otimes\bullet +1}))$ is fully faithful. This implies that $D(R)\simeq\Tot^{\dual}(D(S^{\otimes \bullet + 1})),$ as required.
\end{proof}

The following result is due to Peter Scholze (private communication). We give a short proof based on Positselski's results on accessibility \cite{Pos}.

\begin{prop}[Scholze]\label{prop:Faithfully_flat_aleph_1_accessible} Let $R$ be a commutative ring. The category of faithfully flat $R$-algebras is $\omega_1$-accessible. A faithfully flat $R$-algebra $S$ is $\omega_1$-compact if and only if $S$ is a countably presented $R$-module (equivalently, a countably presented $R$-algebra).

In particular, any faithfully flat $R$-algebra is an $\omega_1$-directed colimit of descendable faithfully flat $R$-algebras.\end{prop}

\begin{proof} The last assertion follows from the fact that a faithfully flat countably presented $R$-algebra $S$ is descendable by \cite[Corollary 3.33]{Mat}.

By Lazard's theorem, the category $\Flat_R$ of flat $R$-modules is equivalent to the ind-completion $\Ind(P(R)),$ where $P(R)$ is the category of finite projective $R$-modules. The category of flat commutative $R$-algebras $\Alg^{\fla}_{R}$ is equivalent to the category of modules over the monad $\Sym_R^*(-):\Flat_R\to\Flat_R.$ The functor $\Sym_R^*(-)$ commutes with filtered colimits and preserves $\omega_1$-compact objects. Applying Proposition \ref{prop:modules_over_accessible_monads} below, we see that the category $\Alg^{\fla}_{R}$ is $\omega_1$-accessible, and a flat $R$-algebra $S$ is $\omega_1$-compact if and only if $S$ is a countably presented $R$-module.

It remains to observe that if $S'\to S$ is a map of flat $R$-algebras such that $S$ is faithfully flat over $R,$ then $S'$ is also faithfully flat over $R.$
\end{proof}

Recall the following $2$-categorical notion.

\begin{defi}\cite{AR} Consider a pair of functors $F,G:\cC\to\cD$ between ordinary categories, and let $\varphi,\psi:F\to G$ be a pair of natural transformations. The equifier of $(\varphi,\psi)$ is the full subcategory $\Eq(\varphi,\psi)\subset \cC$ formed by objects $x\in\cC$ such that $\varphi_x=\psi_x.$\end{defi}

\begin{prop}\label{prop:modules_over_accessible_monads} Let $\kappa$ be an uncountable regular cardinal, and let $\cC$ be a $\kappa$-accessible (ordinary) category which has filtered colimits. Let $(T:\cC\to\cC,\eta:\Id_{\cC}\to T,\mu:T^2\to T)$ be a monad such that $T$ commutes with filtered colimits and preserves $\kappa$-compact objects. Then the category of $T$-modules $\Mod_T(\cC)$ is $\kappa$-accessible, and a $T$-module $M\in\Mod_T(\cC)$ is $\kappa$-compact in $\Mod_T(\cC)$ if and only if its underlying object is $\kappa$-compact in $\cC.$\end{prop}

\begin{proof}


Let $\cD$ be the category of pairs $(x,\alpha),$ where $x\in\cC,$ $\alpha:T(x)\to x.$ Clearly, $\cD$ has filtered colimits and the forgetful functor $\pi:\cD\to\cC$ commutes with filtered colimits. By \cite[Theorem 4.1]{Pos} the category $\cD$ is $\kappa$-accessible, and $\pi:\cD\to\cC$ preserves and reflects $\kappa$-compactness. 

Consider $\alpha$ as a natural transformation $T\circ\pi\to\pi.$ Then the full subcategory $\Mod_T(\cC)\subset\cD$ is exactly the joint equifier of the pairs $(\id_{\pi},\alpha\circ\eta_{\pi}):\pi\to\pi$ and $(T(\alpha),\mu_{\pi}):T^2\circ\pi\to T\circ\pi.$ Applying \cite[Theorem 3.1]{Pos}, we conclude that the category $\Mod_T(\cC)$ is $\kappa$-accessible, and the forgetful functor to $\cC$ preserves and reflects $\kappa$-compactness.
\end{proof}

\subsection{Products of dualizable categories}
\label{ssec:products_of_dualizable}

In this subsection we discuss the products in $\Cat_{\st}^{\dual}.$ It turns out that they are very well-behaved: the functor $\Ind(-):\Cat^{\perf}\to\Cat_{\st}^{\dual}$ commutes with products, and the category $\Cat_{\st}^{\dual}$ satisfies the weak (AB4*) axiom and the (AB6) axiom. We refer to Subsection \ref{ssec:AB_axioms} for the formulations of the AB axioms. 

First, we make an observation about infinite products of ind-objects. Let $\cA_i,$ $i\in I,$ be a family of small stable categories. Given ind-objects $M_i\in\Ind(\cA_i),$ $i\in I,$ we can consider each $M_i$ as an object of $\Ind(\prodd[i]\cA_i).$ We denote by $\prodd[i] M_i$ the product of these objects in $\Ind(\prodd[i]\cA_i).$ We denote by $\prodd[i]:\prodd[i]\Ind(\cA_i)\to\Ind(\prodd[i]\cA_i)$ the resulting functor. This is simply the right adjoint to the natural functor $\Ind(\prodd[i]\cA_i)\to\prodd[i]\Ind(\cA_i).$ We give another two equivalent descriptions of the functor $\prodd[i].$

\begin{prop}\label{prop:product_of_ind_of_a_i}1) For a small stable category $\cA,$ let us identify the category $\Ind(\cB)$ with the category $\Fun(\cB^{op},\Sp)$ of exact functors to $\Sp.$ Then for a collection of ind-objects $M_i\in \Ind(\cA_i),$ $i\in I,$ the ind-object $\prodd[i] M_i\in \Ind(\prodd[i]\cA_i)$ is identified with the functor $$\prodd[i] \cA_i^{op}\to\Sp,\quad (\prodd[i] M_i)((x_i)_i) = \prodd[i] M_i(x_i).$$

2) If $J_i,$ $i\in I,$ are filtered categories such that $M_i=\inddlim[j_i\in J_i]x_{j_i},$ then the ind-object $\prodd[i] M_i\in \Ind(\prodd[i]\cA_i)$ is given by the formula $$\prodd[i] \inddlim[j_i\in Ji]x_{j_i}=\inddlim[(j_i)_i\in\prod_{i\in I}J_i](x_{j_i})_i.$$
\end{prop}

\begin{proof}Part 1) follows directly from the definition, and part 2) follows from Proposition \ref{prop:products_in_Ind_C}.\end{proof}

\begin{cor}\label{cor:comm_square} Let $\{F:\cA_i\to\cB_i\}_{i\in I}$ be a family of functors between small stable categories. Then we have a commutative square of categories
$$
\begin{tikzcd}
\prodd[i] \Ind(\cA_i) \ar{r}{\prodd[i]\Ind(F_i)} \ar{d}{\prodd[i]} &[1em] \prodd[i] \Ind(\cB_i) \ar{d}{\prodd[i]}\\
\Ind(\prodd[i]\cA_i) \ar{r}{\Ind(\prodd[i] F_i)} & \Ind(\prodd[i]\cB_i).
\end{tikzcd}
$$
\end{cor}

\begin{proof}
This follows immediately from Proposition \ref{prop:product_of_ind_of_a_i}, part 2).
\end{proof}

Given a small stable category $\cB,$ we denote by $-\tens{\cB}-:\Ind(\cB)\otimes \Ind(\cB^{op})\to\Sp$ the evaluation functor. It is identified with the functor $\THH(\cB,-):\Ind(\cB\otimes\cB^{op})\to \Sp.$

\begin{prop}\label{prop:tensor_product_of_products} Let $\{\cA_i\}_{i\in I}$ be a collection of small stable categories, and take some ind-objects $M_i\in \Ind(\cA_i),$ $N_i\in\Ind(\cA_i^{op}),$ $i\in I.$ Then the natural map $$(\prodd[i] M_i)\tens{\prodd[i]\cA_i}(\prodd[i] N_i)\to\prodd[i] (M_i\tens{\cA_i} N_i)$$ is an isomorphism.\end{prop}

\begin{proof}Let us write $M_i=\inddlim[j_i'\in J_i']x_{j_i'}',$ $N_i=\inddlim[j_i''\in J_i'']x_{j_i''}'',$ $i\in I,$ where $J_i'$ and $J_i''$ are filtered categories and $x_{j_i'}'\in\cA_i,$ $x_{j_i''}''\in\cA_i^{op}$. Then applying Proposition \ref{prop:product_of_ind_of_a_i} and (AB6) for $\Sp$ we get a chain of equivalences 
\begin{multline*}(\prodd[i] M_i)\tens{\prodd[i]\cA_i}(\prodd[i] N_i)\cong (\inddlim[(j_i')_i\in\prod_{i\in I}J_i'](x_{j_i'}')_i)\tens{\prodd[i]\cA_i}(\inddlim[(j_i'')_i\in\prod_{i\in I}J_i''](x_{j_i''}'')_i)\\ \cong
\indlim[((j_i')_i,(j_i'')_i)\in\prod\limits_{i} J_i'\times\prod\limits_{i} J_i'']\prodd[i]\cA_i(x_{j_i''}'',x_{j_i'}')\cong\prodd[i]\indlim[(j_i',j_i'')\in J_i'\times J_i'']\cA_i(x_{j_i''}'',x_{j_i'}')\\ \cong
\prodd[i] \left((\inddlim[j_i'\in J_i']x_{j_i'}')\tens{\cA_i}(\inddlim[j_i''\in J_i'']x_{j_i''}'')\right)\cong \prodd[i] (M_i\tens{\cA_i}N_i).
 \end{multline*}
This proves the proposition.
\end{proof}


We can now deduce the following statement about infinite products of homological epimorphisms.

\begin{prop}\label{prop:products_of_hom_epi}Let $\{F_i:\cA_i\to\cB_i\}_{i\in I}$ be a family of homological epimorphisms between small stable categories. 

1) The functor $\prodd[i] F_i:\prodd[i]\cA_i\to\prodd[i]\cB_i$ is a homological epimorphism.

2) For a family of objects $\{M_i\in \Ind(\cA_i)\}_{i\in I},$ the object $\prodd[i] M_i$ is contained in $\ker(\Ind(\prodd[i] F_i))\subset\Ind(\prodd[i]\cA_i)$ if and only if each $M_i$ is contained in $\ker(\Ind(F_i))\subset\Ind(\cA_i).$ Moreover, the category $\ker(\Ind(\prodd[i] F_i))$ is generated by the objects of this form.\end{prop}

\begin{proof} Put $\cA:=\prodd[i]\cA_i,$ $\cB:=\prodd[i]\cB_i,$ $F:=\prodd[i] F_i.$ We denote by $\Ind(F_i)^R:\Ind(\cB_i)\to\Ind(\cA_i)$ resp. $\Ind(F)^R:\Ind(\cB)\to\Ind(\cA)$ the right adjoint to $\Ind(F_i)$ resp. $\Ind(F).$

1) We observe that we have a commutative diagram
\begin{equation*}
\begin{tikzcd}
\prodd[i]\Ind(\cB_i) \ar{r}{\prodd[i]\Ind(F_i)^R}\ar{d}{\prodd[i]} &[2em] \prodd[i]\Ind(\cA_i) \ar{r}{\prodd[i]\Ind(F_i)}\ar{d}{\prodd[i]} &[1em] \prodd[i]\Ind(\cB_i)\ar{d}{\prodd[i]}\\
\Ind(\prodd[i]\cB_i)\ar{r}{\Ind(F)^R} & \Ind(\prodd[i]\cA_i)\ar{r}{\Ind(F)} & \Ind(\prodd[i]\cB_i).
\end{tikzcd}
\end{equation*}
Namely, the right square commutes by Corollary \ref{cor:comm_square}, and the commutativity of the left square is easily seen by taking the left adjoints. By assumption each adjunction counit $\Ind(F_i)\circ\Ind(F_i)^R\to \Id$ is an isomorphism for $i\in I.$ From the above diagram we conclude that the morphism $\Ind(F)\circ\Ind(F)^R\circ\prodd[i]\to\prodd[i]$ in $\Fun^{\acc}(\prodd[i]\Ind(\cB_i),\Ind(\prodd[i]\cB_i))$ is an isomorphism. But the essential image of the functor $\prodd[i]:\prodd[i]\Ind(\cB_i)\to\Ind(\prodd[i]\cB_i)$ contains the compact objects, therefore the adjunction counit $\Ind(F)\circ\Ind(F)^R\to\Id$ is an isomorphism. This means that $F$ is a homological epimorphism, as stated.


2) The first statement holds even without the assumption that $F_i$ are homological epimorphisms, it follows from Corollary \ref{cor:comm_square}. For the ``moreover'' statement, note that the category $\ker(\Ind(F))$ is generated by the objects of the form $$\Fiber(x\to\Ind(F)^R\Ind(F)(x))\in\Ind(\cA),\quad x=(x_i)_i\in\cA=\prodd[i]\cA_i.$$ But we have
$$\Fiber(x\to\Ind(F)^R\Ind(F)(x))\cong\prodd[i] \Fiber(x_i\to\Ind(F_i)^R\Ind(F_i)(x_i)),$$ and $$\Fiber(x_i\to\Ind(F_i)^R\Ind(F_i)(x_i))\in\ker(\Ind(F_i)).$$ This proves part 2).
\end{proof}

Now we can describe the products in the category $\Cat_{\st}^{\dual}.$

\begin{prop}\label{prop:products_in_Cat^dual} Let $\{\cC_i\}_{i\in I}$ be a collection of dualizable categories, and let $\kappa$ be an uncountable regular cardinal.

1) We have a short exact sequence \begin{equation}\label{eq:exact_seq_for_product}0\to\prodd[i]^{\dual}\cC_i\to\Ind(\prodd[i]\cC_i^{\kappa})\to\Ind(\prodd[i]\Calk_{\kappa}^{\cont}(\cC_i))\to 0.\end{equation}

2) If the categories $\cC_i$ are compactly generated, then we have an equivalence $$\prodd[i]^{\dual}\cC_i\simeq\Ind(\prodd[i]\cC_i^{\omega}).$$
In other words, the functor $\Ind:\Cat^{\perf}\to\Cat_{\st}^{\dual}$ commutes with products.
\end{prop}

\begin{proof}By Theorem \ref{th:limit_of_dualizable}, we have an equivalence $$\prodd[i]^{\dual}\cC_i\simeq\ker^{\dual}(\Ind(\prodd[i]\cC_i^{\kappa})\to\Ind(\prodd[i]\Calk_{\kappa}^{\cont}(\cC_i))).$$
Now part 1) follows from Proposition \ref{prop:products_of_hom_epi}, applied to the homological epimorphisms $\cC_i^{\kappa}\to\Calk_{\kappa}^{\cont}(\cC_i).$

Part 2) follows from the fact that an infinite product of short exact sequences in $\Cat^{\perf}$ is a short exact sequence \cite[Lemma 5.2]{KW19}, applied to the sequences
\begin{equation*}0\to\cC_i^{\omega}\to\cC_i^{\omega_1}\to\Calk_{\omega_1}^{\cont}(\cC_i)\to 0.\end{equation*}
For completeness we explain this by a direct computation. Recall that $\Calk_{\omega_1}^{\cont}(\cC_i)\simeq \Calk_{\omega_1}(\cC_i^{\omega}).$ It suffices to check that the functor $\prodd[i]\cC_i^{\omega_1}/\prodd[i]\cC_i^{\omega}\to \prodd[i]\Calk_{\omega_1}(\cC_i^{\omega})$ is fully faithful (the essential surjectivity up to direct summands is clear). Take some objects $x=(x_i)_i, y=(y_i)_i\in\prodd[i]\cC_i^{\omega_1},$ and denote by $\bar{x}, \bar{y}$ their images in the quotient. For each $i\in I$ choose a direct sequence $(z_{i,n})_{n\geq 0}$ in $\cC_i^{\omega}$ with an isomorphism $y_i\cong \indlim[n]z_{i,n}.$ Then we have a cofinal functor $\N^I\to (\prodd[i]\cC_i^{\omega})_{/y},$ $f\mapsto (z_{i,f(i)})_i.$ Using (AB6) for $\Sp,$ we obtain
\begin{multline*}
\Hom_{\prodd[i]\cC_i^{\omega_1}/\prodd[i]\cC_i^{\omega}}(\bar{x},\bar{y})\cong \indlim[f:I\to\N] \Cone(\prodd[i]\Hom_{\cC_i^{\omega_1}}(x_i,z_{i,f(i)})\to \prodd[i]\Hom_{\cC_i^{\omega_1}}(x_i,y_i))\\
\cong\prodd[i] \Cone(\indlim[n]\Hom_{\cC_i^{\omega_1}}(x_i,z_{i,n})\to\Hom_{\cC_i^{\omega_1}}(x_i,y_i))\cong \prodd[i]\Hom_{\Calk_{\omega_1}(\cC_i^{\omega})}(x_i,y_i).  
\end{multline*}
This proves the fully faithfulness assertion.\end{proof}

Now we want to prove that an infinite product of short exact sequences in $\Cat_{\st}^{\dual}$ is again a short exact sequence. Here the non-trivial part is the weak (AB4*) for $\Cat_{\st}^{\dual}:$ an infinite product of epimorphisms (localizations) in $\Cat_{\st}^{\dual}$ is an epimorphism. We need the following observation.

\begin{lemma}\label{lem:dual_of_products_of_dualizable} Let $\kappa$ be an uncountable regular cardinal.
	
	1) Let $\{\cC_i\}_{i\in I}$ be a collection of dualizable categories. We have a fully faithful strongly continuous functor
$$\Phi:(\prodd[i]^{\dual}\cC_i)^{\vee}\to\Ind(\prodd[i](\cC_i^{\kappa})^{op}).$$ The image of $\Phi$ contains all objects of the form $\prodd[i]\ev_{\cC_i}(-,x_i),$ where $x_i\in\cC_i^{\vee}$ and $\ev_{\cC_i}:\cC_i\otimes\cC_i^{\vee}\to \Sp$ is the evaluation functor for $\cC_i.$ Moreover, the image of $\Phi$ is generated by the objects of this form.

2) Let $\{F_i:\cC_i\to\cD_i\}_{i\in I}$ be a family of strongly continuous functors between dualizable categories. Then the following square commutes:
\begin{equation}\label{eq:comm_square_for_inclusion_of_dual_of_dualizable_product}
\begin{tikzcd}
(\prodd[i]^{\dual}\cD_i)^{\vee} \ar{r}{G} \ar{d}{\Psi} & (\prodd[i]^{\dual}\cC_i)^{\vee} \ar{d}{\Phi}\\
\Ind(\prodd[i](\cD_i^{\kappa})^{op}) \ar{r}{G'} & \Ind(\prodd[i](\cC_i^{\kappa})^{op}).
\end{tikzcd}
\end{equation}
Here the functors $\Phi$ and $\Psi$ are as in 1). The functor $G$ (resp. $G'$) is $(\prodd[i]^{\dual} F_i)^{\vee}$ (resp. the right adjoint to $\Ind(\prodd[i] (F_i^{\kappa})^{op})$).\end{lemma}

\begin{proof}1) The functor $\Phi$ is obtained by applying $(-)^{\vee,L}$ to the inclusion in the short exact sequence \eqref{eq:exact_seq_for_product}. Note that the (fully faithful) functor $\cC_i^{\vee}\to \Ind((\cC_i^{\kappa})^{op})=\Fun(\cC_i^{\kappa},\Sp)$ takes $x$ to the functor $\ev_{\cC_i}(-,x).$ The assertion now follows from Proposition \ref{prop:products_of_hom_epi}, part 2), applied to the homological epimorphisms $(\cC_i^{\kappa})^{op}\to\Calk_{\kappa}^{\cont}(\cC_i)^{op}.$

2) We need to show that the image of $G'\circ\Psi$ is contained in the image of $\Phi.$ By part 1), it suffices to show that for any $x_i\in\cD_i^{\vee},$ the object $G'(\prodd[i]\ev_{\cD_i}(-,x_i))\in\Ind(\prodd[i](\cC_i^{\kappa})^{op})$ is contained in the image of $\Phi.$ We have  $$G'(\prodd[i]\ev_{\cD_i}(-,x_i))\cong \prodd[i]\ev_{\cD_i}(F_i(-),x_i)\cong \prodd[i]\ev_{\cC_i}(-,F_i^{\vee}(x_i)).$$ By part 1), the latter object is in the image of $\Phi.$ This proves the proposition.\end{proof}

We can now deduce that the weak (AB4*) holds in the category $\Cat_{\st}^{\dual}.$

\begin{prop}\label{prop:AB4*_for_Cat^dual} Consider a collection $\{0\to\cA_i\to\cB_i\xto{F_i}\cC_i\to 0\}$ of short exact sequences in $\Cat_{\st}^{\dual}.$ Then the sequence
$$0\to\prodd[i]^{\dual}\cA_i\to\prodd[i]^{\dual}\cB_i\to\prodd[i]^{\dual}\cC_i\to 0$$ is also short exact.\end{prop}

\begin{proof} The only statement that requires a proof is that the functor $\prodd[i]^{\dual}\cB_i\to\prodd[i]^{\dual}\cC_i$ is a localization.  Equivalently, we need to show that the dual functor $(\prodd[i]^{\dual}\cC_i)^{\vee}\to(\prodd[i]^{\dual}\cB_i)^{\vee}$ is fully faithful. By Lemma \ref{lem:dual_of_products_of_dualizable}, the latter functor is the upper horizontal arrow in a commutative square
\begin{equation}\label{eq:trick_with_duals}
\begin{tikzcd}
(\prodd[i]^{\dual}\cC_i)^{\vee} \ar{r}\ar{d} & (\prodd[i]^{\dual}\cB_i)^{\vee} \ar{d}\\
\Ind(\prodd[i](\cC_i^{\omega_1})^{op}) \ar{r} & \Ind(\prodd[i](\cB_i^{\omega_1})^{op}).
\end{tikzcd}
\end{equation}
Recall from loc. cit. that the lower horizontal arrow is right adjoint to the functor $\Ind(\prodd[i] (F_i^{\omega_1})^{op}).$ We have short exact sequences in $\Cat^{\perf}:$
\begin{equation*}
0\to \cA_i^{\omega_1}\to \cB_i^{\omega_1}\xto{F_i^{\omega_1}}\cC_i^{\omega_1}\to 0. 
\end{equation*}
Passing to the opposite categories and taking the product over $i,$ by \cite[Lemma 5.2]{KW19} we obtain the short exact sequence
\begin{equation*}
0\to \prodd[i] (\cA_i^{\omega_1})^{op}\to \prodd[i] (\cB_i^{\omega_1})^{op}\to \prodd[i] (\cC_i^{\omega_1})^{op}\to 0.
\end{equation*}
Therefore, in \eqref{eq:trick_with_duals} the lower horizontal arrow is fully faithful. Clearly, both vertical arrows are also fully faithful. It follows that the upper horizontal arrow is also fully faithful, as required.\end{proof} 

It turns out that the (AB6) axiom holds in the category $\Cat_{\st}^{\dual}$ itself.

\begin{prop}\label{prop:AB6_for_Cat^dual} The axiom (AB6) holds in $\Cat_{\st}^{\dual}.$ That is, for any filtered categories $J_i,$ $i\in I,$ for any functors $J_i\to\Cat_{\st}^{\dual},$ $j_i\mapsto \cC_{j_i},$ the natural functor
\begin{equation}\label{eq:comparison_for_AB6}\indlim[(j_i)_i\in\prod\limits_{i} J_i]^{\cont}\prodd[i]^{\dual}\cC_{j_i}\to \prodd[i]^{\dual}\indlim[j_i\in J_i]^{\cont}\cC_{j_i}\end{equation}
is an equivalence.
\end{prop}

\begin{proof}We first observe that (AB6) holds in the category $\Cat^{\perf}$ of small stable Karoubi-complete categories, since the category $\Cat^{\perf}$ is compactly generated (one can also check (AB6) for $\Cat^{\perf}$ directly, by showing that the comparison functor is fully faithful and essentially surjective). Since the functor $\Ind:\Cat^{\perf}\to \Cat_{\st}^{\dual}$ commutes both with products and with filtered colimits, we deduce that \eqref{eq:comparison_for_AB6} is an equivalence when all $\cC_{j_i}$ are compactly generated.

By Propositions \ref{prop:weak_AB5_Pr^LL} and \ref{prop:AB4*_for_Cat^dual}, filtered colimits and products of short exact sequences in $\Cat_{\st}^{\dual}$ are short exact. Consider the map of short exact sequences
$$
\begin{tikzcd}
\liminj^{\cont}\limits_{(j_i)_i\in\prod\limits_{i} J_i}\prodd[i]^{\dual}\cC_{j_i} \ar{r}\ar{d} & \Ind(\indlim[(j_i)_i\in\prod\limits_{i} J_i]\prodd[i]\cC_{j_i}^{\omega_1}) \ar{r}\ar{d} & \Ind(\indlim[(j_i)_i\in\prod\limits_{i} J_i]\prodd[i]\Calk_{\omega_1}^{\cont}(\cC_{j_i})) \ar{d}\\
\prodd[i]^{\dual}\liminj^{\cont}\limits_{j_i\in J_i}\cC_{j_i} \ar{r} & \Ind(\prodd[i]\indlim[j_i\in J_i]\cC_{j_i}^{\omega_1}) \ar{r} & \Ind(\prodd[i]\indlim[j_i\in J_i]\Calk_{\omega_1}^{\cont}(\cC_{j_i})).\\
\end{tikzcd}
$$
By the above discussion, the middle vertical and the right vertical arrows are equivalences. Hence, so is the left vertical arrow. This proves (AB6) for $\Cat_{\st}^{\dual}.$\end{proof}

For a future reference we mention the following basic statements about the functor $\prodd[i]:\prodd[i]\cC_i\to\prodd[i]^{\dual}\cC_i$ when $\cC_i$ are dualizable categories, $i\in I.$ Again, we define this functor to be the right adjoint to the natural functor. As in the compactly generated case, for $x=(x_i)_i\in\prodd[i]\cC_i$ we put $\prodd[i] x_i = \prodd[i](x)\in \prodd[i]^{\dual}\cC_i;$ this object is indeed a product of the images of $x_i$ under the functors $\cC_i\to\prodd[i]^{\dual}\cC_i.$ The following is a generalization of Corollary \ref{cor:comm_square}.

\begin{cor}\label{cor:comm_square_dual}
Let $\{F_i:\cC_i\to\cD_i\}_{i\in I}$ be a collection of strongly continuous functors between dualizable categories. Then the following square commutes:
\begin{equation*}
\begin{tikzcd}
\prodd[i]\cC_i \ar{r}{\prodd[i] F_i} \ar{d}{\prodd[i]} & \prodd[i]\cD_i \ar{d}{\prodd[i]}\\
\prodd[i]^{\dual}\cC_i \ar{r}{\prodd[i]^{\dual} F_i} & \prodd[i]^{\dual}\cD_i.
\end{tikzcd}
\end{equation*}
\end{cor}

\begin{proof}
It suffices to prove a version of this statement when $\prodd[i]\cC_i$ resp. $\prodd[i]\cD_i$ is replaced with the full subcategory $\prodd[i]\cC_i^{\kappa}$ resp. $\prodd[i]\cD_i^{\kappa},$ where $\kappa$ is an uncountable regular cardinal. This is equivalent to the commutativity of the following (lax commutative) square:
\begin{equation*}
	\begin{tikzcd}
		\Ind(\prodd[i]\cC_i^{\kappa}) \ar{r}{\Ind(\prodd[i] F_i^{\kappa})} \ar{d}{F} & \Ind(\prodd[i]\cD_i^{\kappa}) \ar{d}{G}\\
		\prodd[i]^{\dual}\cC_i \ar{r}{\prodd[i]^{\dual} F_i} & \prodd[i]^{\dual}\cD_i,
	\end{tikzcd}
\end{equation*}
where $F$ resp. $G$ is the left Kan extension of $\prodd[i]$ from $\prodd[i]\cC_i^{\kappa}$ resp. $\prodd[i]\cD_i^{\kappa}.$ But this is simply the dual of the commutative square \eqref{eq:comm_square_for_inclusion_of_dual_of_dualizable_product}.
\end{proof}

We also mention the following statement on the evaluation functor for the category $\prodd[i]^{\dual}\cC_i,$ which is a generalization of Proposition \ref{prop:tensor_product_of_products}. We identify $(\prodd[i]^{\dual}\cC_i)^{\vee}$ with $\prodd[i]^{\dual}\cC_i^{\vee}$ (recall from Subsection \ref{ssec:duality_covariant} that $(-)^{\vee}$ is an autoequivalence of $\Cat_{\st}^{\dual}$).

\begin{cor}
	Let $\{\cC_i\}_{i\in I}$ be a collection of dualizable categories, and put $\cC=\prodd[i]^{\dual}\cC_i.$ Then for $(x_i)_i\in\prodd[i]\cC_i$ and $(y_i)_i\in\prodd[i]\cC_i^{\vee}$ we have an isomorphism
	\begin{equation}\label{eq:product_pairing_formula}
		\ev_{\cC}(\prodd[i] x_i, \prodd[i] y_i)\xto{\sim}\prodd[i]\ev_{\cC_i}(x_i,y_i).
	\end{equation}
\end{cor}

\begin{proof}
	We first consider the case when all $\cC_i$ are compactly generated. Then we have $\prodd[i]^{\dual}\cC_i\simeq\Ind(\prodd[i]\cC_i^{\omega})$ by Proposition \ref{prop:products_in_Cat^dual}, and the isomorphism \eqref{eq:product_pairing_formula} holds by Proposition \ref{prop:tensor_product_of_products}. 
	
	Next, we reduce the general case to the above special case. For each $i\in I$ choose a fully faithful strongly continuous functor $F_i:\cC_i\to\cD_i,$ where $\cD_i$ is compactly generated. Put $\cD = \prodd[i]^{\dual}\cD_i$ and $F=\prodd[i]^{\dual} F_i:\cC\to\cD.$ Then $F^{\vee,L}$ is identified with $\prodd[i]^{\dual}F_i^{\vee,L}.$ By Corollary \ref{cor:comm_square_dual} we have $F(\prodd[i] x_i)\cong \prodd[i]F_i(x_i),$ $F^{\vee,L}(\prodd[i] y_i)\cong \prodd[i]F_i^{\vee,L}(y_i).$ Since $F$ and $F_i$ are fully faithful, we have the isomorphisms $\ev_{\cC}(x,y)\xto{\sim}\ev_{\cD}(F(x),F^{\vee,L}(y))$ and $\ev_{\cC_i}(x_i,y_i)\xto{\sim}\ev_{\cD_i}(F_i(x_i),F_i^{\vee,L}(y_i)).$ Hence, we can replace $\cC_i$ with $\cD_i,$ in which case we already know the isomorphism \eqref{eq:product_pairing_formula}.
\end{proof}

\section{Presentability of the category of dualizable categories}
\label{sec:presentability_Cat_dual}

In this section we give an alternative, ``explicit'' proof of the absolute version of Ramzi's theorem \cite{Ram24a} which states that the category $\Cat_{\st}^{\dual}$ is $\omega_1$-presentable. We also describe explicitly the $\kappa$-compact objects in $\Cat_{\st}^{\dual}$ for any uncountable regular cardinal $\kappa.$ We start with the following observation.

\begin{prop}\label{prop:categories_under_C_and_monads} Let $\cC$ be a dualizable stable category. Consider the full subcategory $\cE_{\cC}\subset (\Cat_{\st}^{\dual})_{\cC/}$ which consists of pairs $(\cD,F:\cC\to\cD)$ such that the image of $F$ generates $\cD$ by colimits. Then the functor $\cE_{\cC}\to \Alg_{\bE_1}(\Fun^{L}(\cC,\cC)),$ $(\cD,F)\mapsto F^R\circ F,$ is an equivalence of categories.
	
	The inverse functor is given by $T\mapsto \Mod_T(\cC).$\end{prop} 

\begin{proof}Indeed, for any $(\cD,F)\in \cE_{\cC}$ the functor $F^R:\cD\to\cC$ is conservative and commutes with all colimits, in particular, with geometric realizations. Hence, the functor $F^R$ is monadic by Lurie-Barr-Beck theorem, and we have an equivalence $\cD\xto{\sim}\Mod_{F^R\circ F}(\cC).$
	
	It remains to check that for any $T\in\Alg_{\bE_1}(\Fun^{L}(\cC,\cC))$ the category $\Mod_T(\cC)$ is dualizable and the functor $\cC\to\Mod_T(\cC)$ is strongly continuous. The latter assertion is obvious, and the former one follows from Proposition \ref{prop:condition_for_dualizability}.\end{proof}

We will repeatedly use the following.

\begin{prop}\label{prop:E_1-algebras_in_monoidal_cats} Let $\cC$ be a presentable stable $\bE_1$-monoidal category. Let $\kappa$ be a regular cardinal such that $\cC$ is $\kappa$-presentable, the tensor product of $\kappa$-compact objects is $\kappa$-compact, and the unit object is $\kappa$-compact (equivalently, $\cC$ is an $\bE_1$-algebra in $\Pr^L_{\st,\kappa}$). 
	
	1) The category $\Alg_{\bE_1}(\cC)$ of $\bE_1$-algebras in $\cC$ is $\kappa$-compactly generated.
	
	2) If $\kappa$ is uncountable, then an $\bE_1$-algebra $A$ is $\kappa$-compact in $\Alg_{\bE_1}(\cC)$ if and only if the underlying object of $A$ is $\kappa$-compact in $\cC.$ 
\end{prop}

\begin{proof}1) For an object $x\in\cC,$ denote by $T(x)$ the tensor algebra of $x,$ i.e. the free $\bE_1$-algebra generated by $x.$ Then $T(-)$ is the left adjoint to the forgetful functor. If $x\in\cC^{\kappa},$ then our assumptions imply that $T(x)\in\cC^{\kappa}.$ Since the forgetful functor commutes with $\kappa$-filtered colimits, the objects $T(x),$ $x\in\cC^{\kappa},$ are $\kappa$-compact in $\Alg_{\bE_1}(\cC).$ Since the forgetful functor is conservative and $\cC$ is $\kappa$-presentable, we deduce that the category $\Alg_{\bE_1}(\cC)$ is $\kappa$-presentable. More precisely, the category $\Alg_{\bE_1}(\cC)^{\kappa}$ is the smallest full subcategory of $\Alg_{\bE_1}(\cC)$ which contains the objects $T(x),$ $x\in\cC^{\kappa},$ and is closed under $\kappa$-small colimits.
	
	2) The forgetful functor $\Alg_{\bE_1}(\cC)\to\cC$ is monadic, and the corresponding monad on $\cC$ is given by the functor $T(-):\cC\to\cC.$ Since $T(-)$ preserves $\kappa$-compact objects, the assertion follows from \cite[Lemma 6.17]{BSY}.\end{proof} 



We will apply Proposition \ref{prop:E_1-algebras_in_monoidal_cats} to monoidal categories of the form $\Fun^{L}(\cD,\cD),$ where $\cD$ is a dualizable category. Since this category of functors is also dualizable, it is $\omega_1$-presentable.

\begin{prop}\label{prop:compactness_in_E-modules} Let $\kappa$ and $\cC$ be as in Proposition \ref{prop:E_1-algebras_in_monoidal_cats}, and suppose that $\kappa$ is uncountable. Let $\cM$ be a left $\cC$-module in $\Pr^L_{\st,\kappa},$ i.e. $\cM$ is $\kappa$-presentable and the action $\cC\otimes\cM\to\cM$ is $\kappa$-strongly continuous. Let $A$ be a $\kappa$-compact $\bE_1$-algebra in $\cC,$ and $N\in\cM$ an $A$-module. The following are equivalent.
	
	\begin{enumerate}[label=(\roman*),ref=(\roman*)]
		\item $N$ is $\kappa$-compact in $\Mod_A(\cM).$ \label{kappa_comp_E_1_alg1}
		
		\item the underlying object of $N$ is $\kappa$-compact in $\cM.$ \label{kappa_comp_E_1_alg2}
	\end{enumerate}
\end{prop}

\begin{proof} Indeed, the forgetful functor $\Mod_A(\cM)\to\cM$ is continuous and monadic, and the monad $A\otimes-:\cM\to\cM$ preserves $\kappa$-compact objects. Hence, the equivalence \Iff{kappa_comp_E_1_alg1}{kappa_comp_E_1_alg2} follows from \cite[Lemma 6.17]{BSY}.\end{proof}

In Subsection \ref{ssec:sheaves_on_R} below we introduce the notation $\Sh_{\geq 0}(\R;\Sp)$ for the category of sheaves of spectra on $\R$ with singular support in $\R\times\R_{\geq 0}.$ Here we will only need the description of this category as a special case of continuous stabilization from Subsection \ref{ssec:cont_stab}. Namely, recall from Example \ref{ex:cont_posets} that the poset $(\R\cup\{+\infty\})_{\leq}$ is a compactly assembled accessible category. This is a special case of a continuous poset, such posets are discussed in Subsection \ref{ssec:cont_posets} below. 
Then we have $\Sh_{\geq 0}(\R;\Sp)\simeq\Stab^{\cont}((\R\cup\{+\infty\})_{\leq}).$ By the general description of continuous stabilization, the latter category is identified with the full subcategory of functors $F:(\R\cup\{+\infty\})_{\leq}^{op}\to\Sp,$ such that $F(a)\cong \prolim[b<a]F(b)$ for $a\in \R\cup\{+\infty\}.$ For $a\in\R$ we denote by $\bS_{<a}\in \Sh_{\geq 0}(\R;\Sp)$ the functor given by $\bS_{<a}(b)=\begin{cases}\bS & \text{for }b\leq a\\ 0 & \text{for }b>a.\end{cases}$ (here the subscript is ``$<a$'' since this functor corresponds to the extension by zero of the constant sheaf on the open ray $(-\infty,a)$). We denote by $\bS_{\R}\in \Sh_{\geq 0}(\R;\Sp)$ the constant functor with value $\bS.$

\begin{prop}\label{prop:functors_from_Sh_geq_0}Let $\cC$ be a dualizable stable category. Given a strongly continuous functor $F:\Sh_{\geq 0}(\R;\Sp)\to\cC,$ we define a functor $\Phi_F:\Q_{\leq}\to \cC,$ $\Phi_F(a)=F(\bS_{<a}).$ The assignment $F\mapsto\Phi_F$ defines an equivalence between $\Fun^{LL}(\Sh_{\geq 0}(\R;\Sp),\cC)$ and the category of functors $\Phi:\Q_{\leq}\to\cC$ such that
	\begin{itemize}
		\item for any $a\in\Q,$ we have $\Phi(a)\cong\indlim[b<a]\Phi(b);$
		
		\item for any $a<b,$ $a,b\in\Q,$ the morphism $\Phi(a)\to\Phi(b)$ is compact.
	\end{itemize}
\end{prop}

\begin{proof}Since $\Q$ is dense in $\R,$ the category of functors $\Q_{\leq}\to\cC$ with required properties is equivalent to the category of functors $(\R\cup\{+\infty\})_{\leq}\to\cC$ with analogous properties. The assertion now follows as a special case of Proposition \ref{prop:continuous_stab} applied to the category $(\R\cup\{+\infty\})_{\leq}.$\end{proof}

\begin{cor}\label{cor:many_functors_from_Sh_on_R} For any dualizable category $\cC$ and for any countably presented object $x\in\cC^{\omega_1}$ there exists a strongly continuous functor $F:\Sh_{\geq 0}(\R;\Sp)\to\cC$ such that $F(\bS_{\R})\cong x.$\end{cor}

\begin{proof}This follows directly from Propositions \ref{prop:functors_from_Sh_geq_0} and \ref{prop:functors_from_Q}.\end{proof}

We prove the following result on the presentability of the category $\Cat_{\st}^{\dual}$ and on the $\kappa$-compact objects of this category.

\begin{theo}\label{th:presentability_of_Cat^dual}Put $\cC_0:=\Sh_{\geq 0}(\R;\Sp).$ 
	
	1) The category of dualizable categories $\Cat_{\st}^{\dual}$ is $\omega_1$-presentable. 
	
	2) Let $\kappa$ be an uncountable regular cardinal and $\cC\in\Cat_{\st}^{\dual}$ a dualizable category. The following are equivalent.
	
	\begin{enumerate}[label=(\roman*),ref=(\roman*)]
		\item $\cC$ is a $\kappa$-compact object of $\Cat_{\st}^{\dual}.$ \label{presentability1}
		
		\item $\cC$ is $\kappa$-compact in $\Pr^L_{\st,\kappa}.$ \label{presentability2}
		
		\item The functors $\ev:\cC\otimes\cC^{\vee}\to\Sp$ and $\coev:\Sp\to\cC^{\vee}\otimes\cC$ are $\kappa$-strongly continuous, i.e. their right adjoints commute with $\kappa$-filtered colimits. Equivalently, $\cC$ is dualizable in $\Pr^L_{\st,\kappa}.$ \label{presentability3}
		
		\item There exists a short exact sequence $0\to\cC\to\Ind(\cA)\to\Ind(\cB)\to 0$ in $\Cat_{\st}^{\dual},$ where $\cA$ and $\cB$ are $\kappa$-compact in $\Cat^{\perf}.$ \label{presentability4}
		
		\item There exists a $\kappa$-small set $I$ and a $\kappa$-compact monad $A\in\Alg_{\bE_1}(\Fun^{L}(\prod\limits_I\cC_0,\prod\limits_I\cC_0))^{\kappa},$ such that $$\cC\simeq\Mod_A(\prod\limits_I\cC_0).$$
		Moreover, if $\kappa=\omega_1,$ then $I$ can be assumed to be a one-element set. \label{presentability5}
	\end{enumerate}
\end{theo}

We need a few auxiliary statements. Let $\kappa$ be an uncountable regular cardinal.

\begin{lemma}\label{lem:kappa_compactness_and_pseudocompactness} Let $\cC$ and $\cD$ be dualizable categories, and suppose that the functors $\ev_{\cC}$ and $\coev_{\cC}$ are $\kappa$-strongly continuous. The following are equivalent for a continuous functor $F:\cC\to\cD.$
	
	\begin{enumerate}[label=(\roman*),ref=(\roman*)]
		\item the object $F$ is $\kappa$-compact in $\Fun^{L}(\cC,\cD)\simeq\cC^{\vee}\otimes\cD.$ \label{kappa_strcont1}
		
		\item the functor $F$ is $\kappa$-strongly continuous. \label{kappa_strcont2}
	\end{enumerate}
\end{lemma}

\begin{proof}This is essentially a standard statement, we include the proof for completeness.
	
	\Implies{kappa_strcont1}{kappa_strcont2}. Since the functors $\Sp\xto{-\otimes F}\cC^{\vee}\otimes\cD$ and $\ev_{\cC}$ are $\kappa$-strongly continuous, so is the composition
	$$\cC\xto{\Id\boxtimes F} \cC\otimes\cC^{\vee}\otimes\cD\xto{\ev_{\cC}\boxtimes\Id} \cD.$$
	This composition is isomorphic to $F.$
	
	\Implies{kappa_strcont2}{kappa_strcont1}. Since the functors $F$ and $\coev_{\cC}$ are $\kappa$-strongly continuous, so is the composition
	$$\Sp\xto{\coev}\cC^{\vee}\otimes\cC\xto{\Id\boxtimes F}\cC^{\vee}\otimes\cD.$$
	This composition sends $\bS$ to $F.$\end{proof}

\begin{lemma}\label{lem:kappa_compactness_for_ev_and_coev}  Consider a short exact sequence in $\Cat_{\st}^{\dual},$
	$$0\to\cC\xto{F} \cD\xto{G} \cE\to 0.$$
	
	1) If $\ev_{\cD}$ is $\kappa$-strongly continuous, then so is $\ev_{\cC}.$
	
	2) If $\coev_{\cD}$ is $\kappa$-strongly continuous, then so is $\coev_{\cE}.$
	
	3) If $\ev_{\cD},$ $\coev_{\cD}$ and $\ev_{\cE}$ are $\kappa$-strongly continuous, then so is $\coev_{\cC}.$
	
	4) If $\ev_{\cD},$ $\coev_{\cD}$ and $\coev_{\cC}$ are $\kappa$-strongly continuous, then so is $\ev_{\cE}.$
\end{lemma}

\begin{proof}1) The functor $\ev_{\cC}$ is the composition of $\kappa$-strongly continuous functors
	$$\cC\otimes\cC^{\vee}\xto{F\boxtimes F^{\vee,L}} \cD\otimes\cD^{\vee}\xto{\ev_{\cD}} \Sp,$$
	hence it is $\kappa$-strongly continuous.
	
	2) Similarly, the functor $\coev_{\cE}$ is the composition of $\kappa$-strongly continuous functors 
	$$\Sp\xto{\coev_{\cD}}\cD^{\vee}\otimes\cD\xto{G^{\vee,L}\boxtimes G}\cE^{\vee}\otimes\cE,$$
	hence it is $\kappa$-strongly continuous.
	
	3) It suffices to prove that $(F^{\vee,L}\boxtimes F)(\coev_{\cC}(\bS))$ is in $(\cD^{\vee}\otimes\cD)^{\kappa}.$ We have an exact triangle in the category $\Fun^L(\cD,\cD):$
	\begin{equation}
		\label{eq:exact_triangle_endofunctors} F\circ F^R\to \Id_{\cD}\to G^R\circ G.
	\end{equation}
	Identifying $\Fun^L(\cD,\cD)$ with $\cD^{\vee}\otimes\cD,$ we can rewrite \eqref{eq:exact_triangle_endofunctors} as follows: 
	$$(F^{\vee,L}\boxtimes F)(\coev_{\cC}(\bS))\to\coev_{\cD}(\bS)\to (G^{\vee,L}\boxtimes G)^R(\coev_{\cE}(\bS)).$$
	Since $\coev_{\cD}(\bS)$ is $\kappa$-compact, it suffices to show that $(G^{\vee,L}\boxtimes G)^R(\coev_{\cE}(\bS))$ is $\kappa$-compact.  The latter object corresponds under the self-duality of $\cD^{\vee}\otimes\cD$ to the composition $$\cD\otimes\cD^{\vee}\xto{G\boxtimes G^{\vee,L}}\cE\otimes\cE^{\vee}\xto{\ev_{\cE}}\Sp.$$ The latter functor is $\kappa$-strongly continuous by our assumptions. Applying Lemma \ref{lem:kappa_compactness_and_pseudocompactness} (to the categories $\cD\otimes\cD^{\vee}$ and $\Sp$) we conclude that $(G^{\vee,L}\boxtimes G)^R(\coev_{\cE}(\bS))$ is $\kappa$-compact.
	
	4) This is proved by reversing the proof of 3).\end{proof}

\begin{lemma}\label{lem:suff_cond_for_smallness_of_ev_coev} Let $F:\cC\to\cD$ be a strongly continuous functor between dualizable categories, and denote by $F^R$ its right adjoint. Suppose that the image of $F$ generates $\cD,$ i.e. $F^R$ is conservative. Suppose that the functors $\ev_{\cC},$ $\coev_{\cC}$ and $F^R$ are $\kappa$-strongly continuous. Then the functors $\ev_{\cD}$ and $\coev_{\cD}$ are also $\kappa$-strongly continuous.\end{lemma}

\begin{proof}We first show that $\ev_{\cD}$ is $\kappa$-strongly continuous. The assumption on $F$ implies that the image of the functor $F^{\vee,L}:\cC^{\vee}\to\cD^{\vee}$ generates $\cD^{\vee}$ (since $(-)^{\vee}$ is a covariant involution on $\Cat_{\st}^{\dual}$). Hence, it suffices to show that the composition
	$$\cD\otimes\cC^{\vee}\xto{\Id\boxtimes F^{\vee,L}}\cD\otimes\cD^{\vee}\xto{\ev_{\cD}}\Sp$$ is $\kappa$-strongly continuous. But it is identified with another composition
	$$\cD\otimes\cC^{\vee}\xto{F^R\otimes\Id}\cC\otimes\cC^{\vee}\xto{\ev_{\cC}}\Sp.$$ The latter composition is $\kappa$-strongly continuous since both functors are $\kappa$-strongly continuous.
	
	Next, we show that the functor $\coev_{\cD}$ is $\kappa$-strongly continuous, i.e. the object $\coev_{\cD}(\bS)\in\cD^{\vee}\otimes\cD$ is $\kappa$-compact. Denote by $A=F^R\circ F$ the continuous monad on $\cC$ associated with $F.$ Then $A^{\vee}\boxtimes A$ is a monad on $\cC^{\vee}\otimes\cC,$ and we have $\cD^{\vee}\otimes\cD\simeq\Mod_{A^{\vee}\boxtimes A}(\cC^{\vee}\otimes\cC).$ By Lemma \ref{lem:kappa_compactness_and_pseudocompactness} we have $A\in\Fun^L(\cC,\cC)^{\kappa},$ hence also  $A^{\vee}\boxtimes A\in\Fun^L(\cC^{\vee}\otimes\cC,\cC^{\vee}\otimes\cC)^{\kappa}.$ Applying Proposition \ref{prop:compactness_in_E-modules} to the $\Fun^L(\cC^{\vee}\otimes\cC,\cC^{\vee}\otimes\cC)$-module $\cC^{\vee}\otimes\cC$ and the $\bE_1$-algebra $A^{\vee}\otimes A,$ we see that the $\kappa$-compactness of $\coev_{\cD}(\bS)\in\cD^{\vee}\otimes\cD$ is equivalent to the $\kappa$-compactness of $(F^{\vee,L}\boxtimes F)^R(\coev_{\cD}(\bS))\in\cC^{\vee}\otimes\cC.$ The latter object is isomorphic to $A$ under the identification $\cC^{\vee}\otimes\cC\simeq\Fun^{L}(\cC,\cC).$ By the above, this object is $\kappa$-compact.
\end{proof}

\begin{proof}[Proof of Theorem \ref{th:presentability_of_Cat^dual}.] We first show the following.
	
	{\noindent{\bf Claim.}} {\it For any dualizable category $\cC$ there exists a set $I$ and a monad $A\in\Alg_{\bE_1}(\Fun^{L}(\prod\limits_I\cC_0,\prod\limits_I\cC_0)),$ such that $$\cC\simeq\Mod_A(\prod\limits_I\cC_0).$$}
	
	\begin{proof}[Proof of Claim] Choose some generating collection of $\omega_1$-compact objects $(x_i\in\cC^{\omega_1})_{i\in I}.$ By Corollary \ref{cor:many_functors_from_Sh_on_R} we find strongly continuous functors $F_i:\cC_0\to\cC,$ $i\in I,$ such that $F_i(\bS_{\R})\cong x_i.$ We obtain the strongly continuous functor
		$$F:\prod\limits_I\cC_0\simeq\coprod\limits_I^{\cont}\cC_0\xto{(F_i)_i}\cC.$$
		Since the image of $F$ generates $\cC$ by colimits, it follows from Proposition \ref{prop:categories_under_C_and_monads} that 
		\begin{equation*}
			\cC\simeq\Mod_{F^R\circ F}(\prod\limits_I\cC_0).\qedhere\end{equation*} 
	\end{proof}
	
	Next, we prove some implications from part 2). 
	
	\Implies{presentability2}{presentability1}. Let $(\cD_i)_i$ be a $\kappa$-directed system in $\Cat_{\st}^{\dual},$ and $\cD=\indlim[i]\cD_i.$ We only need to prove the following: if $\cC\to\cD_{i_0}$ is some $\kappa$-strongly continuous functor such that the composition $\cC\to\cD_{i_0}\to\cD$ is strongly continuous, then there exists some $i\geq i_0$ such that the composition $\cC\to\cD_{i_0}\to\cD_i$ is $\kappa$-strongly continuous. To see this, note that the composition $$\cC\to\Ind(\cD_{i_0}^{\kappa})\to \Ind(\Calk_{\kappa}^{\cont}(\cD_{i_0}))\to \Ind(\Calk_{\kappa}^{\cont}(\cD))$$
	is zero. By Proposition \ref{prop:Calkin_construction_accessible}, we have $\Calk_{\kappa}^{\cont}(\cD)\simeq\indlim[i]\Calk_{\kappa}^{\cont}(\cD_i).$ We deduce that for some $i\geq i_0$ the composition
	$$\cC\to\Ind(\cD_{i_0}^{\kappa})\to \Ind(\Calk_{\kappa}^{\cont}(\cD_{i_0}))\to \Ind(\Calk_{\kappa}^{\cont}(\cD_i))$$ is zero. This exactly means that the composition $\cC\to\cD_{i_0}\to\cD_i$ is strongly continuous.
	
	\Implies{presentability3}{presentability2}. Recall that $\Sp$ is a $\kappa$-compact object in $\Pr^L_{\st,\kappa}.$ Hence, any dualizable object in $\Pr^L_{\st,\kappa}$ is also $\kappa$-compact. This proves the implication.
	
	\Implies{presentability4}{presentability3}. Note that it is sufficient to prove that if $\cA$ is $\kappa$-compact in $\Cat^{\perf},$ then the functors $\ev$ and $\coev$ for $\Ind(\cA)$ are $\kappa$-strongly continuous. Indeed, assuming this, we can apply Lemma \ref{lem:kappa_compactness_for_ev_and_coev}.
	
	So let $\cA$ be in $(\Cat^{\perf})^{\kappa}.$ It follows directly from Proposition \ref{prop:kappa_compact_small_cats} that the functor $\ev_{\Ind(\cA)}$ is $\kappa$-strongly continuous: this exactly means that the spectra $\cA(x,y)$ are $\kappa$-compact. To show that $\Delta_{\cA}=\coev_{\Ind(\cA)}(\bS)$ (the diagonal $\cA\hy\cA$-bimodule) is $\kappa$-compact, choose a $\kappa$-small ind-system $(\cA_i)_i$ of compact (finitely presented) objects of $\Cat^{\perf}$ such that $\cA=\indlim[i]\cA_i,$ and denote by $F_i:\Ind(\cA_i)\to\Ind(\cA)$ the functors to the colimit. By a straightforward generalization of \cite[Proposition 2.14]{TV07}, for each $i$ the object $\Delta_{\cA_i}$ is compact in $\Ind(\cA_i\otimes\cA_i^{op}).$ Hence, the object $\Delta_{\cA}$ is $\kappa$-compact: it is a $\kappa$-small filtered colimit of compact objects, namely
	$$\Delta_{\cA}\cong\indlim[i]F_i^{\vee,L}\boxtimes F_i(\Delta_{\cA_i}).$$
	This proves the implication.
	
	
	\Implies{presentability5}{presentability3}. By Proposition \ref{prop:resolution_of_Sh_geq_0} below we have a short exact sequence
	$$0\to \cC_0\to\Fun(\Q_{\leq}^{op},\Sp)\to \prodd[\Q]\Sp\to 0.$$
	Here the inclusion functor is left adjoint to the functor $F\mapsto (a\mapsto \prolim[b<a] F(b)).$ Hence, $\cC_0$ satisfies \ref{presentability4} and by the above it also satisfies \ref{presentability3}. If $I$ is a $\kappa$-small set, then $\prod\limits_I \cC_0$ also satisfies \ref{presentability3}. If $A\in\Alg_{\bE_1}(\Fun^{L}(\prod\limits_I \cC_0,\prod\limits_I \cC_0))$ is $\kappa$-compact, then by Lemma \ref{lem:suff_cond_for_smallness_of_ev_coev} we see that the category $\Mod_A(\prod\limits_I \cC_0)$ satisfies \ref{presentability3}.
	
	We now prove 1). By the above, it suffices to prove that the category $\Cat_{\st}^{\dual}$ is generated by colimits by the categories of the form $\Mod_A(\cC_0),$ where $A\in\Fun^{L}(\cC_0,\cC_0)$ is an $\omega_1$-compact monad. Let $\cC$ be an arbitrary dualizable category. Choose a generating collection $S$ of $\omega_1$-compact objects of $\cC$ such that for each $s\in S$ the object $\hat{\cY}(s)$ is a formal colimit of a sequence of objects of $S.$ Then $S$ is a directed union of countable subsets $T\subset S$ with the same property. Denoting by $\cC_T\subset\cC$ the full subcategory generated by $T$ by colimits, we see that $\cC_T$ is dualizable, and the inclusion $\cC_T\to\cC$ is strongly continuous. By Corollary \ref{cor:directed_unions}, we have $\cC\simeq\indlim[T\subset S]^{\cont}\cC_T.$ Hence, we may assume that $S$ is countable. Then $\cC$ is generated by a single $\omega_1$-compact object. Arguing as in 1), we obtain an equivalence $\cC\simeq\Mod_A(\cC_0)$ for some monad $A\in\Fun^{L}(\cC_0,\cC_0).$ By Proposition \ref{prop:E_1-algebras_in_monoidal_cats}, we can find an $\omega_1$-directed system $(A_i)_i$ of $\omega_1$-compact monads such that $A\cong\indlim[i]A_i.$ Then $\cC\simeq\indlim[i]^{\cont}\Mod_{A_i}(\cC_0).$ This proves 1).
	
	\Implies{presentability1}{presentability5}. Suppose that $\cC$ is $\kappa$-compact in $\Cat_{\st}^{\dual}.$ We argue as in the proof of Claim. By the above, we see that $\cC$ is generated by a $\kappa$-small collection of $\omega_1$-compact objects $\{x_i\}_{i\in I}.$ We obtain a monad $A\in\Fun^{L}(\prod\limits_I \cC_0,\prod\limits_I \cC_0)$ such that $\cC\simeq\Mod_A(\prod\limits_I \cC_0).$ Since $\prod\limits_I \cC_0$ is $\kappa$-compact in $\Cat_{\st}^{\dual},$ we deduce that $\cC$ is $\kappa$-compact in the category of dualizable categories {\it under} $\prod\limits_I \cC_0.$ By Proposition \ref{prop:categories_under_C_and_monads}, this implies that $A$ is $\kappa$-compact. 
	
	We already proved that the conditions \ref{presentability1}, \ref{presentability2}, \ref{presentability3} and \ref{presentability5} are equivalent, and they are implied by \ref{presentability4}. It remains to prove the implication \Implies{presentability1}{presentability4}. So let $\cC$ be $\kappa$-compact in $\Cat_{\st}^{\dual}.$ By \ref{presentability5}, we have a strongly continuous functor $\prod\limits_I\cC_0\to\cC,$ $|I|<\kappa,$ such that the image generates. We will again use the strongly continuous fully faithful functor $\cC_0\to \Fun(\Q_{\leq}^{op},\Sp)$ from Proposition \ref{prop:resolution_of_Sh_geq_0} below. 
	Define the category $\cE$ by the pushout square in $\Cat_{\st}^{\dual}:$
	
	$$
	\begin{tikzcd}
		\prodd[I]\cC_0 \ar{r}\ar{d}{F} & \cC \ar{d}{G}\\
		\cD \ar{r} & \cE.
	\end{tikzcd}
	$$    
	
	Since $\cD$ is compactly generated, so is $\cE.$ Since the categories $\prod\limits_I\cC_0,$ $\cC$ and $\cD$ are $\kappa$-compact in $\Cat_{\st}^{\dual},$ so is $\cE.$ It follows that the category $\cE^{\omega}$ is $\kappa$-compact in $\Cat^{\perf}.$
	
	Since the functor $F$ is fully faithful by construction, so is $G$ by Proposition \ref{prop:pushouts_of_mono} (fully faithful functors are closed under pushouts in $\Cat_{\st}^{\dual}$). Consider the quotient $\cE'=\cE/G(\cC).$ Then $\cE'$ is compactly generated and $\cE'$ is $\kappa$-compact in $\Cat_{\st}^{\dual}.$ Therefore, $\cE'^{\omega}$ is $\kappa$-compact in $\Cat^{\perf}.$ Hence, the short exact sequence
	$$0\to\cC\to\cE\to\cE'\to 0$$ has the required properties.\end{proof}

\section{Urysohn's lemma for dualizable categories}
\label{sec:Urysohn}

In this section we prove the following result. As in the previous section we use the notation $\Shv_{\geq 0}(\R;\Sp)$ for the category $\Stab^{\cont}((\R\cup\{+\infty\})_{\leq}).$

\begin{theo}\label{th:Urysohn_dualizable} The category $\Cat_{\st}^{\dual}$ is generated by colimits by a single $\omega_1$-compact object, namely by the category $\Sh_{\geq 0}(\R;\Sp).$\end{theo}

Note that the classical Urysohn's lemma for compact Hausdorff spaces essentially says that the category
$\comphaus^{op}$ (the opposite category of compact Hausdorff spaces) is generated by colimits by the unit interval $[0,1].$ The latter object is $\omega_1$-compact in $\comphaus^{op}$ by Proposition \ref{prop:comphaus_op_presentability}.

\begin{proof}[Proof of Theorem \ref{th:Urysohn_dualizable}]
	We will consider the set $\R\cup\{+\infty\}$ as a linearly ordered set with the usual order. For $a\in\R$ we consider the subsets $\R_{\leq a},\,\R_{>a}\cup\{+\infty\}\subset \R\cup\{+\infty\}$ also as linearly ordered sets. We will again use the equivalence 
	$$\Fun^{LL}(\Sh_{\geq 0}(\R;\Sp),\cC)\simeq \Fun^{\strcont}(\R\cup\{+\infty\},\cC),$$
	given by Proposition \ref{prop:continuous_stab}, where $\cC$ is a dualizable category. 
	
	By Theorem \ref{th:presentability_of_Cat^dual} we already know that the category $\Cat_{\st}^{\dual}$ is presentable and the object $\Shv_{\geq 0}(\R;\Sp)$ is $\omega_1$-compact. Thus, we only need to prove the following: if a strongly continuous functor between dualizable categories $F:\cC\to\cD$ induces an equivalence
	\begin{equation}\label{eq:Fun_from_Sh_geq_0_to_F}
		\Fun^{\strcont}(\R\cup\{+\infty\},\cC)\xto{\sim} \Fun^{\strcont}(\R\cup\{+\infty\},\cD),
	\end{equation}
	then $F^{\omega_1}:\cC^{\omega_1}\to\cD^{\omega_1}$ is an equivalence. We first recall that the functor $$\Fun^{\strcont}(\R\cup\{+\infty\},\cD)\to\cD^{\omega_1},\quad \Phi\mapsto\Phi(+\infty),$$ is essentially surjective by Corollary \ref{cor:many_functors_from_Sh_on_R}, hence $F^{\omega_1}$ is also essentially surjective. To prove the fully faithfulness of $F^{\omega_1},$ we need the following observation: for any pair of strongly continuous functors $G:\R_{\leq 0}\to \cC,$ $H:\R_{>0}\cup\{+\infty\}\to\cC$ we have a pullback squares of spaces:
	\begin{equation*}
		\begin{tikzcd}
			\prolim[c>0]\Map_{\cC}(G(0),H(c))\ar[r]\ar[d] & \Fun^{\strcont}(\R\cup\{+\infty\},\cC)^{\simeq}\ar[d]\\
			\pt\ar{r}{(G,H)} & \Fun^{\strcont}(\R_{\leq 0},\cC)^{\simeq}\times\Fun^{\strcont}(\R_{>0}\cup\{+\infty\},\cC)^{\simeq}
		\end{tikzcd}
	\end{equation*}
	and
	\begin{equation*}
		\begin{tikzcd}
			\prolim[c>0]\Map_{\cD}(F(G(0)),F(H(c)))\ar[r]\ar[d] & \Fun^{\strcont}(\R\cup\{+\infty\},\cD)^{\simeq}\ar[d]\\
			\pt\ar{r}{(F\circ G,F\circ H)} & \Fun^{\strcont}(\R_{\leq 0},\cD)^{\simeq}\times\Fun^{\strcont}(\R_{>0}\cup\{+\infty\},\cD)^{\simeq}.
		\end{tikzcd}
	\end{equation*}
	Since the posets $\R_{\leq 0}$ and $\R_{>0}\cup\{+\infty\}$ are isomorphic to the poset $\R\cup\{+\infty\},$ we conclude that the map 
	\begin{equation}\label{eq:equivalence_of_limits_of_mapping_spaces}
		\prolim[c>0]\Map_{\cC}(G(0),H(c))\xto{\sim} \prolim[c>0]\Map_{\cD}(F(G(0)),F(H(c)))
	\end{equation} is an equivalence of spaces. Now take any objects $x,y\in\cC^{\omega_1},$ and choose strongly continuous functors $G,H:\R\cup\{+\infty\}\to\cC$ such that $x\cong G(+\infty),$ $y\cong H(+\infty).$ Then we obtain the following chain of equivalences of spaces, where all the limits and colimits are over real numbers:
	\begin{multline}\label{eq:chain_of_equivalences_mapping_spaces}
		\Map_{\cC}(x,y)\cong \Map_{\cC}(G(+\infty),\indlim[b]H(b))\cong \prolim[a]\indlim[b]\Map_{\cC}(G(a),H(b))\\
		\cong \prolim[a]\indlim[b]\prolim[c>b]\Map_{\cC}(G(a),H(c))\xto{\sim} \prolim[a]\indlim[b]\prolim[c>b]\Map_{\cD}(F(G(a)),F(H(c))) 
	\end{multline}
	Here the second equivalence follows from the isomorphism $\hat{\cY}(G(+\infty))\cong\inddlim[a]G(a).$ The third equivalence follows from the observation that for any $a\in\R$ we have an ind-equivalence $\inddlim[b]\Map_{\cC}(G(a),H(b))\cong \inddlim[b]\prolim[c>b]\Map_{\cC}(G(a),H(c)).$ The fourth map is an equivalence since \eqref{eq:equivalence_of_limits_of_mapping_spaces} is an equivalence and we have obvious equivalences of posets $\R_{\leq 0}\simeq \R_{\leq a},$ $\R_{>0}\cup\{+\infty\}\simeq \R_{>b}\cup\{+\infty\}.$
	
	Now, a similar chain of equivalences identifies the target of \eqref{eq:chain_of_equivalences_mapping_spaces} with $\Map_{\cD}(F(x),F(y)).$ This proves the fully faithfulness of $F^{\omega_1}$ and the theorem.
\end{proof}

\begin{remark}\label{rem:Urysohn_for_Compass} Suppose that we know the cocompleteness of the category $\Compass$ and the $\omega_1$-compactness of the poset $\R\cup\{+\infty\}$ in $\Compass$ (both are not difficult to check). Then the same argument as in the proof of Theorem \ref{th:Urysohn_dualizable} shows that the category $\Compass$ is generated via colimits by the poset $\R\cup\{+\infty\}.$
	
	A similar statement holds for the category of compactly assembled presentable categories and strongly continuous right exact functors. In this case an $\omega_1$-compact generating object is the category $\Shv_{\geq 0}(\R;\cS).$\end{remark}

For completeness we explain a direct way to recover the dualizable category $\cC$ from the category of strongly continuous functors $\R\cup\{+\infty\}\to\cC.$ Slightly abusing the terminology, we will call a functor $\R\to\cC$ strongly continuous if its left Kan extension to $\R\cup\{+\infty\}$ is strongly continuous. We denote the category of such functors by
$\Fun^{\strcont}(\R,\cC).$

\begin{prop}\label{prop:Urysohn_explicit} 1) Let $\cC$ be a dualizable category. The functor $$\Phi:\Fun^{\strcont}(\R,\cC)\to\cC^{\omega_1},\quad\Phi(F)=\indlim[a]F(a),$$ is a Verdier localization. Its kernel is the full subcategory of ind-zero functors, i.e. of functors $F$ such that for any $a\in\R$ there exists $b>a$ such that the map $F(a)\to F(b)$ is zero.
	
	2) Denote by $\Psi:\Ind(\cC^{\omega_1})\to \Ind(\Fun^{\strcont}(\R,\cC))$ the right adjoint to $\Ind(\Phi).$ Then for a strongly continuous functor $F:\R\to\cC,$ we have
	\begin{equation}\label{eq:right_adjoint_explicit} \Psi(\Phi(F))\cong \inddlim[f:\R\to\R]F\circ f.\end{equation}
	Here $f$ runs through functions satisfying the following property:
	
	$(\ast)$ $f$ is strictly monotone, $\lim\limits_{x\to +\infty}f(x)=+\infty$ and $\lim\limits_{x\to a-}f(x)=f(a)$ for $a\in\R.$\end{prop} 

The notation $\lim\limits_{x\to a-}$ simply means that we look only at $x<a.$ In other words, the last condition on $f$ means that the endofunctor $f:\R_{\leq}\to\R_{\leq}$ commutes with existing colimits.

We will need the following basic fact about sequential limits of directed colimits. It uses the axioms (AB5) and (AB6) from Subsection \ref{ssec:AB_axioms}.

\begin{lemma}\label{lem:seq_limits_of_filtered_colimits} Let $\cC$ be a presentable $\infty$-category such that filtered colimits in $\cC$ commute with finite limits (i.e. strong (AB5) holds in $\cC$) and (AB6) for countable products holds in $\cC$ (for example, $\cC$ can be any compactly assembled presentable category, such as $\cS$ or $\Sp$). Let $I$ be a directed poset, and consider $\N$ as a poset with the usual order. Let $F:\N^{op}\times I\to \cC$ be a functor. Then we have a natural isomorphism
	$$\prolim[n]\indlim[i]F(n,i)\cong\indlim[\varphi:\N\to I]\prolim[n\leq m]F(n,\varphi(m)),$$
	where $\varphi$ runs through order-preserving maps.
	Here the set $\{(n,m)\mid n\leq m\}$ is considered as a full subposet of $\N^{op}\times\N.$\end{lemma}

\begin{proof}Note that the inclusion
	$$\{(n,m)\mid n\leq m\leq n+1\}\to \{(n,m)\mid n\leq m\}$$
	is final. It follows that for any functor $G:\{n\leq m\}\to\cC$ we have
	$$\prolim[n\leq m]G(n,m)\simeq\Eq(\prodd[n]G(n,n)\toto\prodd[n]G(n,n+1)),$$
	where the two maps have components $G(n,n)\to G(n,n+1)$ resp. $G(n+1,n+1)\to G(n,n+1).$
	
	Note also that the poset of order-preserving maps $\N\to I$ is cofinal in the poset of all maps $\N\to I.$ Using our assumptions on $\cC,$ we obtain the isomorphisms
	\begin{multline*}\prolim[n]\indlim[i]F(n,i)\cong\Eq(\prodd[n]\indlim[i]F(n,i)\toto \prodd[n]\indlim[i]F(n,i))\\ \cong
		\Eq(\indlim[\varphi:\N\to I]\prodd[n]F(n,\varphi(n))\toto \indlim[\varphi:\N\to I]\prodd[n]F(n,\varphi(n+1)))\\ \cong \indlim[\varphi:\N\to I]\Eq(\prodd[n]F(n,\varphi(n))\toto \prodd[n]F(n,\varphi(n+1)))\cong \indlim[\varphi:\N\to I]\prolim[n\leq m]F(n,\varphi(m)).\qedhere\end{multline*}
\end{proof}

\begin{proof}[Proof of Proposition \ref{prop:Urysohn_explicit}] We already know that the functor $\Phi$ is essentially surjective by Corollary \ref{cor:many_functors_from_Sh_on_R}. Hence, it suffices to prove part 2). Indeed, \eqref{eq:right_adjoint_explicit} implies that $\Ind(\Phi)\circ\Psi\cong\Id,$ i.e. $\Phi$ is a homological epimorphism. Let $f:\R\to\R$ be a function satisfying $(\ast)$ such that $f(x)\geq x$ for all $x.$ Then for any strongly continuous functor $F:\R\to\cC,$ the functor $\Cone(F\to F\circ f)$ is an ind-zero functor $\R\to\cC$ (since $f$ is cofinal). Hence, the kernel of $\Ind(\Phi)$ is generated by colimits by ind-zero functors, as required.
	
	We now prove 2).  Denote by $J$ the poset of functions $f:\R\to\R$ satisfying $(\ast).$ Then $J$ is a cofinal subposet of the poset $P$ of all (not necessarily strictly) monotone functions $f:\R\to\R$ such that $\lim\limits_{x\to+\infty}f(x)=+\infty.$ Since $P$ is directed, so is $J.$ Consider another subposet $I\subset P$ which consists of monotone functions $f:\R\to\R$ such that $f$ takes values in $\N,$ $\lim\limits_{x\to+\infty} f(x)=+\infty,$ and $f$ is constant on $(-\infty,0]$ and on each semi-open interval $(n,n+1]$ for $n\in\N.$ Then $I$ is also cofinal in $P.$
	
	Consider two strongly continuous functors $F,G:\R\to\cC.$ Note that for $f\in I,$ the functor $G\circ f$ is right Kan-extended from $\N\subset\R.$ Applying Lemma \ref{lem:seq_limits_of_filtered_colimits} and using the isomorphism $\hat{\cY}(\indlim[a]F(a))\cong\inddlim[n\in\N]F(n),$ we obtain the isomorphisms
	\begin{multline*}\Hom(\indlim[a]F(a),\indlim[b]G(b))\cong\prolim[n\in\N]\indlim[k\in \N]\Hom_{\cC}(F(n),G(k))\\ \cong
		\indlim[\varphi:\N\to\N]\prolim[n\leq m]\Hom_{\cC}(F(n),G(\varphi(m)))\cong \indlim[\varphi:\N\to \N]\Hom_{\Fun(\N,\cC)}(F_{\mid \N},G\circ\varphi)\\ \cong
		\indlim[f\in I]\Hom_{\Fun(\R,\cC)}(F,G\circ f)\cong\indlim[f\in J]\Hom_{\Fun(\R,\cC)}(F,G\circ f).\end{multline*} This proves \eqref{eq:right_adjoint_explicit}.
\end{proof}



\section{Dualizability is equivalent to flatness}
\label{sec:dualizability_via_flatness}

In this section we give a non-trivial criterion of dualizability of a presentable stable category via a ``flatness'' condition. It can be skipped on a first reading.

Let $\cC$ be a presentable stable category. We can ask the following naive question: is it true that for any fully faithful colimit-preserving functor $\cD\to\cE$ between presentable stable categories, the functor $\cC\otimes\cD\to\cC\otimes\cE$ is also fully faithful?

\begin{defi}1) A category $\cC\in\Pr^L_{\st}$ is called flat in $\Pr^L_{\st}$ if the answer to the above question is affirmative, i.e. the functor $\cC\otimes-:\Pr^L_{\st}\to \Pr^L_{\st}$ preserves fully faithful functors (=monomorphisms).

2) Let $\kappa$ be a regular cardinal. A $\kappa$-presentable stable category $\cC$ is called flat in $\Pr^L_{\st,\kappa}$ if the functor $\cC\otimes-:\Pr^L_{\st,\kappa}\to \Pr^L_{\st,\kappa}$ preserves fully faithful functors.\end{defi}

Note that the flatness in $\Pr^L_{\st,\kappa}$ is a non-trivial condition only if $\kappa$ is uncountable. We obtain two equivalent characterizations of the class of flat categories, one of which is formulated in $2$-categorical terms.

To simplify matters we will only use ordinary $2$-categories. 
Let $\kappa$ be an uncountable regular cardinal.  Denote by $\mPr^{\acc}_{\st,\kappa}$ the $2$-category of $\kappa$-presentable stable categories and exact $\kappa$-accessible functors. More precisely, the objects of $\mPr^{\acc}_{\st,\kappa}$ are $\kappa$-presentable stable categories, and for $\cD,\cE\in Ob(\mPr^{\acc}_{\st,\kappa})$ the category of morphisms $\mPr^{\acc}_{\st,\kappa}(\cD,\cE)$ is given by $\h \Fun^{\kappa\hy\cont}(\cD,\cE)$ -- the homotopy category of the (presentable) stable category of $\kappa$-accessible exact functors from $\cD$ to $\cE.$ The composition in $\mPr^{\acc}_{\st,\kappa}$ is given by the composition of functors. We also denote by $\mPr^L_{\st,\kappa}\subset \mPr^{\acc}_{\st,\kappa}$ the locally full sub-$2$-category with the same objects, where for $\cD,\cE\in \mPr^L_{\st,\kappa}$ the category $\mPr^L_{\st,\kappa}(\cD,\cE)\subset \mPr^{\acc}_{\st,\kappa}(\cD,\cE)$ is the full subcategory of continuous functors. We have the following surprising result. Its formulation uses the axiom (AB6) from Subsection \ref{ssec:AB_axioms}, more precisely its special case for $\kappa$-small products.

\begin{theo}\label{th:flat=dualizable}
Let $\kappa$ be an uncountable regular cardinal. For a $\kappa$-presentable stable category $\cC$ the following conditions are equivalent.

\begin{enumerate}[label=(\roman*), ref=(\roman*)]
	\item $\cC$ is flat in $\Pr^L_{\st,\kappa}.$ \label{flatness}
	\item $\cC$ satisfies (AB6) for $\kappa$-small products. \label{AB6_kappa}
	\item There is a $2$-functor $\cC\otimes-:\mPr^{\acc}_{\st,\kappa}\to \mPr^{\acc}_{\st,\kappa}$ which induces the usual Lurie tensor product $2$-functor $\cC\otimes-:\mPr^L_{\st,\kappa}\to \mPr^L_{\st,\kappa}.$ \label{2_functoriality}
\end{enumerate}

In particular, a presentable stable category is flat in $\Pr^L_{\st}$ if and only if it is dualizable.
\end{theo}

Note that the final assertion follows from the equivalence \Iff{flatness}{AB6_kappa} by Proposition \ref{prop:AB6_criterion}. Moreover, the ``if'' direction is easy: dualizability of $\cC$ implies an equivalence $\cC\otimes -\simeq \Fun^{L}(\cC^{\vee},-).$ However, the implication \Implies{AB6_kappa}{flatness} is non-trivial, see below for the proof. We will also prove the implications \ref{flatness}$\implies$\ref{2_functoriality}$\implies$\ref{AB6_kappa}. The latter implication is relatively straightforward, but the former implication requires a non-standard equivalent description of $\mPr^{\acc}_{\st,\kappa}$ via semi-orthogonal gluings which we think of as correspondences. After the previous version of the paper was published, German Stefanich found a different direct proof of the implication \Implies{flatness}{AB6_kappa}. In fact his argument works in a much more general relative context (and for not necessarily stable categories).

Our principal observation is the following: we show that for a $\kappa$-presentable stable category $\cC$ the $2$-functor $\cC\otimes-: \mPr^L_{\st,\kappa}\to \mPr^L_{\st,\kappa}$ can be always extended to an oplax $2$-functor from $\mPr^{\acc}_{\st,\kappa}$ to itself. Then using a different interpretation of this oplax $2$-functor via semi-orthogonal gluings we deduce that flatness implies that this oplax $2$-functor is an actual $2$-functor.

For convenience we introduce the following notation: if $\cA$ is a small stable category with $\kappa$-small limits and $\cD$ is a presentable stable category, then we denote by
\begin{equation*}
(-)^{\kappa\hy\lex}:\Fun(\cA,\cD)\to \Fun^{\kappa\hy\lex}(\cA,\cD) 
\end{equation*}
the left adjoint to the inclusion. Here as usual $\Fun(\cA,\cD)$ denotes the category of exact functors. Note that the category $\Fun^{\kappa\hy\lex}(\cA,\cD)$ is naturally identified with $\Ind_{\kappa}(\cA^{op})\otimes\cD.$

\begin{prop}\label{prop:oplax_2_functor_direct_construction}
Let $\cC$ be a $\kappa$-presentable stable category. For $\cD,\cE\in\mPr^{\acc}_{\st,\kappa}$ and for a $\kappa$-accessible exact functor $F:\cD\to\cE$ we define $\cC\otimes F$ to be the following composition:
\begin{multline}
\label{eq:C_otimes_accessible_functor} \cC\otimes F:\cC\otimes\cD\simeq \Fun^{\kappa\hy\lex}((\cC^{\kappa})^{op},\cD)\xto{F\circ -} \Fun((\cC^{\kappa})^{op},\cE)\\
\xto{(-)^{\kappa\hy\lex}}\Fun^{\kappa\hy\lex}((\cC^{\kappa})^{op},\cE)\simeq \cC\otimes\cE.
\end{multline}
Then the assignments $\cD\mapsto \cC\otimes\cD$ and $F\mapsto\cC\otimes F$ naturally give rise to an oplax $2$-functor $\cC\otimes-:\mPr^{\acc}_{\st,\kappa}\to \mPr^{\acc}_{\st,\kappa}.$ It induces the usual Lurie tensor product $\cC\otimes-:\mPr^L_{\st,\kappa}\to \mPr^L_{\st,\kappa}.$
\end{prop}

\begin{proof}
By construction for $\cD,\cE\in \mPr^{\acc}_{\st,\kappa}$ we have a well-defined functor $\cC\otimes-:\mPr^{\acc}_{\st,\kappa}(\cD,\cE)\to \mPr^{\acc}_{\st,\kappa}(\cC\otimes\cD,\cC\otimes\cE).$ If $\cD=\cE,$ then the functor $\cC\otimes\Id_{\cD}:\cC\otimes\cD\to\cC\otimes\cD$ is naturally isomorphic to the identity functor.

Next, given $1$-morphisms $F:\cD_0\to\cD_1,$ $G:\cD_1\to\cD_2,$ we have a morphism $\alpha_{G,F}:\cC\otimes (G\circ F)\to (\cC\otimes G)\circ (\cC\otimes F)$ in $\h \Fun^{\kappa\hy\cont}(\cC\otimes\cD_0,\cC\otimes\cD_2),$ i.e. a $2$-morphism in $\mPr^{\acc}_{\st,\kappa}.$ It comes from the adjunction counit $\Id\to\iota\circ\iota^L,$ where $\iota:\Fun^{\kappa\hy\lex}((\cC^{\kappa})^{op},\cD_1)\to\Fun((\cC^{\kappa})^{op},\cD_1)$ is the inclusion. Explicitly, identifying $\cC\otimes\cD_i$ with $\Fun^{\kappa\hy\lex}((\cC^{\kappa})^{op},\cD_i),$ for a $\kappa$-left exact functor $\Phi:(\cC^{\kappa})^{op}\to\cD_0$ the map $(\alpha_{G,F})_{\Phi}$ is given by
\begin{multline*}
(\cC\otimes (G\circ F))(\Phi)\cong (G\circ F\circ\Phi)^{\kappa\hy\lex}\to (G\circ (F\circ \Phi)^{\kappa\hy\lex})^{\kappa\hy \lex}\\
\cong ((\cC\otimes G)\circ (\cC\otimes F))(\Phi).
\end{multline*}
Here the middle map is induced by the morphism $F\circ \Phi\to (F\circ \Phi)^{\kappa\hy\lex}.$ It is clear from the construction that the maps $\alpha_{G,F}$ are functorial in $F$ and $G.$

If $H:\cD_2\to\cD_3$ is another $1$-morphism, then the square
\begin{equation*}
\begin{tikzcd}
\cC\otimes(H\circ G\circ F) \ar{r}{\alpha_{H\circ G,F}}\ar{d}{\alpha_{H,G\circ F}} & (\cC\otimes(H\circ G))\circ (\cC\otimes F)\ar{d}{\alpha_{H,G}\circ\id}\\
(\cC\otimes H)\circ (\cC\otimes(G\circ F))\ar{r}{\id\circ\alpha_{G,F}} & (\cC\otimes H)\circ (\cC\otimes G)\circ (\cC\otimes F)
\end{tikzcd}
\end{equation*}
commutes in the (ordinary) category $\h \Fun^{\kappa\hy\cont}(\cC\otimes\cD_0,\cC\otimes\cD_3).$ Indeed, it is obtained by applying $(H\circ -)^{\kappa\hy\lex}$ to the commutative square in $\Fun(\cC\otimes\cD_0,\cC\otimes\cD_2):$ 
\begin{equation*}
\begin{tikzcd}
G\circ F\circ -\ar[r]\ar[d] & G\circ (F\circ -)^{\kappa\hy\lex}\ar[d]\\
(G\circ F\circ -)^{\kappa\hy\lex}\ar[r] & (G\circ (F\circ -)^{\kappa\hy\lex})^{\kappa\hy\lex}.
\end{tikzcd}
\end{equation*}
This shows that the data $(\cD\mapsto\cC\otimes\cD;F\mapsto\cC\otimes F;\alpha)$ together with the isomorphisms $\cC\otimes\Id_{\cD}\cong\Id_{\cC\otimes\cD}$ indeed defines an oplax $2$-functor $\cC\otimes-:\mPr^{\acc}_{\st,\kappa}\to \mPr^{\acc}_{\st,\kappa},$ as stated.

Now, suppose that $F:\cD\to\cE$ is a $1$-morphism in $\mPr^L_{\st,\kappa},$ i.e. $F$ is a continuous functor between $\kappa$-presentable stable categories. Then the standard functoriality of the tensor product with $\cC$ gives a functor which we temporarily denote by $G:\cC\otimes \cD\to \cC\otimes \cE.$ Its right adjoint $G^R$ is identified with the composition functor
\begin{equation*}
\cD\otimes\cE\simeq \Fun^{\kappa\hy\lex}((\cC^{\kappa})^{op},\cE)\xto{F^R\circ -} \Fun^{\kappa\hy\lex}((\cC^{\kappa})^{op},\cD)\simeq\cC\otimes\cD.
\end{equation*} 
It follows that $G$ is naturally isomorphic to the functor $\cC\otimes F$ defined by \eqref{eq:C_otimes_accessible_functor}. Similarly, passing to right adjoints we see that the above construction of the morphisms $\alpha_{G,F}$ agrees with the standard $2$-functoriality of the Lurie tensor product. This proves the proposition.
\end{proof}
 

Now we explain another equivalent model of the $2$-category $\mPr^{\acc}_{\st,\kappa}$ via semi-orthogonal gluings. We define the (ordinary) $2$-category $\Corr_{\kappa}$ as follows. Its objects are $\kappa$-presentable stable categories. For $\cD_0,\cD_1\in \Corr_{\kappa}$ the category of $1$-morphisms $\Corr_{\kappa}(\cD_0,\cD_1)$ is the homotopy category of the $\infty$-category of triples $(\cT,i_0,i_1),$ where $\cT$ is a $\kappa$-presentable stable category and $i_0:\cD_0\to\cT,$ $i_1:\cD_1\to\cT$ are $\kappa$-strongly continuous fully faithful functors such that we have a semi-orthogonal decomposition $\cT=\la i_1(\cD_1), i_0(\cD_0)\ra.$ For simplicity we write $\cT\in\Corr_{\kappa}(\cD_0,\cD_1)$ assuming the choice of $i_0$ and $i_1.$

The composition functor $\Corr_{\kappa}(\cD_1,\cD_2)\times \Corr_{\kappa}(\cD_0,\cD_1)\to\Corr_{\kappa}(\cD_0,\cD_2)$ is defined as follows.  For $\cT_{01}\in\Corr_{\kappa}(\cD_0,\cD_1),$ $\cT_{12}\in\Corr_{\kappa}(\cD_1,\cD_2)$ we form the pushout $\cT_{012}=\cT_{12}\sqcup_{\cD_1}\cT_{01},$ and define $\cT_{02}\subset\cT_{012}$ to be the full subcategory generated by the images of the (fully faithful) compositions $\cD_0\to\cT_{01}\to\cT_{012}$ and $\cD_2\to\cT_{12}\subset\cT_{012}.$ Then we put $\cT_{12}\circ\cT_{01}:=\cT_{02}$ with the above inclusions of $\cD_0$ and $\cD_2.$

The identity $1$-morphism $1_{\cD}\in\Corr_{\kappa}(\cD,\cD)$ is given by the category $\cT=\Fun([1],\cD)$ with inclusions $i_0(x)=(x\to x),$ $i_1(x)=(0\to x).$ Proposition \ref{prop:equiv_between_Corr_and_mPr} below shows that we obtain a well-defined $2$-category which is moreover equivalent to $\mPr^{\acc}_{\st,\kappa}.$ For the proof it is convenient to use the following notation. For a $\kappa$-accessible exact functor $F:\cD\to\cE$ between $\kappa$-presentable stable categories we put $\cD\laxtimes{F}\cE = \cD\laxtimes{F,\cE,\Id}\cE$ (the lax pullback from Definition \ref{defi:gluing_general}). This is the category of triples $(x,y,\psi),$ where $x\in\cD,$ $y\in \cE,$ $\psi:F(x)\to y.$ It is naturally equivalent to the category $\cE\oright_F\cD$ from loc. cit., namely we have
\begin{equation}\label{eq:equiv_between_gluings}
\cE\oright_F\cD\xto{\sim}\cD\laxtimes{F}\cE,\quad (y,x,\varphi:y\to F(x))\mapsto (x,\Cone(\varphi),\psi),
\end{equation}
where $\psi:F(x)\to\Cone(\varphi)$ is the natural map. The inverse equivalence of categories sends $(x,y,\psi)$ to $(\Fiber(\psi),x,\varphi),$ where $\varphi:\Fiber(\psi)\to F(x)$ is the natural map. The inclusion functors $\cD\to \cE\oright_F\cD$ and $\cE\to \cE\oright_F\cD$ (from Definition \ref{defi:gluing_general}), composed with the equivalence \eqref{eq:equiv_between_gluings}, are given by
\begin{equation*}
\cD\to\cD\laxtimes{F}\cE,\quad x\mapsto (x,F(x),\id),\quad \cE\to \cD\laxtimes{F}\cE,\quad y\mapsto (0,y,0).
\end{equation*}
Their right adjoints are given by $(x,y,\psi)\mapsto x$ resp. $(x,y,\psi)\mapsto y.$

\begin{prop}\label{prop:equiv_between_Corr_and_mPr} Let $\kappa$ be a regular cardinal.
	
1) The above constructions of the composition and of the identity $1$-morphisms in $\Corr_{\kappa}$ are well-defined and they naturally give rise to a $2$-category. 

2) We have an equivalence $\mPr^{\acc}_{\st,\kappa}\xto{\sim}\Corr_{\kappa},$ given on objects by $\cD\mapsto\cD,$ and on $1$-morphisms by $(F:\cD_0\to\cD_1)\mapsto \cD_1\oright_F\cD_0.$
\end{prop}

\begin{proof} We prove 1) and 2) simultaneously, essentially by unraveling the construction of the composition in $\Corr_{\kappa}.$
	
First, note that for $\kappa$-presentable stable categories $\cD,\cE$ we have an equivalence of ordinary categories 
\begin{equation*}
\h \Fun^{\kappa\hy\cont}(\cD_0,\cD_1)\xto{\sim}\Corr_{\kappa}(\cD_0,\cD_1),\quad F\mapsto (\cT=\cD_1\oright_F\cD_0,i_0,i_1),
\end{equation*} 
where $i_0:\cD_0\to\cT$ and $i_1:\cD_1\to\cT$ are as in Definition \ref{defi:gluing_general}. The inverse equivalence sends $(\cT,i_0,i_1)$ to $i_1^R\circ i_0,$ see the proof of Proposition \ref{prop:gluing_presentable}. 

Now let $F_{01}:\cD_0\to\cD_1$ and $F_{12}:\cD_1\to\cD_2$ be composable $1$-morphisms in $\mPr^{\acc}_{\st,\kappa}.$ We form the corresponding semi-orthogonal gluings $\cT_{01}=\cD_1\oright_{F_{01}}\cD_0$ and $\cT_{12}=\cD_2\oright_{F_{12}}\cD_1,$ considered as $1$-morphisms in $\Corr_{\kappa}.$ As above, we identify $\cT_{01}$ with $\cT_{01}'=\cD_0\laxtimes{F_{01}}\cD_1,$ and similarly $\cT_{12}$ with $\cT_{12}'=\cD_1\laxtimes{F_{12}}\cD_2.$ Then the pushout $\cT_{012}=\cT_{12}\sqcup_{\cD_1}\cT_{01}$ is identified with the pullback $\cT_{012}'=\cT_{12}'\times_{\cD_1}\cT_{01}'.$ The latter category is naturally identified with the lax limit of the diagram $\cD_0\xto{F_{01}}\cD_1\xto{F_{12}}\cD_2,$ i.e. with the category of quintuples $(x_0,x_1,x_2,\psi_{01},\psi_{12}),$ where $x_i\in\cD_i$ and $\psi_{01}:F_{01}(x_0)\to x_1,$ $\psi_{12}:F_{12}(x_1)\to x_2.$ It is also easy to describe the inclusions $i_0:\cD_0\to \cT_{012}\simeq\cT_{012}'$ and $i_2:\cD_2\to \cT_{012}\simeq \cT_{012}'.$ Namely, we know by the above discussion that their right adjoints are given by 
\begin{equation*}
i_0^R(x_0,x_1,x_2,\psi_{01},\psi_{12})=x_0,\quad i_2^R(x_0,x_1,x_2,\psi_{01},\psi_{12})=x_2.
\end{equation*}
Hence, putting $F_{02}:=F_{12}\circ F_{01},$ we obtain
\begin{equation*}
i_0(x_0) = (x_0,F_{01}(x_0),F_{02}(x_0),\id,\id),\quad i_2(x_2) = (0,0,x_2,0,0).
\end{equation*}
In particular, we see that $i_0$ and $i_2$ are fully faithful and we have the vanishing $i_0^R\circ i_2=0.$ 
This shows that the full subcategory $\cT_{02}\subset\cT_{012}'$ generated by $i_0(\cD_0)$ and $i_2(\cD_2)$ is indeed an object of $\Corr_{\kappa}(\cD_0,\cD_2),$ in particular we have semi-orthogonal decomposition $\cT_{02}=\la i_2(\cD_2),i_0(\cD_0)\ra.$ Moreover, we see that the composition $i_2^R\circ i_0:\cD_0\to\cD_2$ is identified with $F_{12}\circ F_{01},$ hence the composition of $1$-morphisms in $\Corr_{\kappa}$ defined above is identified with the composition in $\mPr^{\acc}_{\st,\kappa}.$

To finish the proof it remains to observe that the identity $1$-morphism in $\Corr_{\kappa}(\cD,\cD)$ is given by $\cD\oright_{\Id_{\cD}}\cD\simeq \cD\laxtimes{\Id_{\cD}}\cD\simeq \Fun([1],\cD),$ together with the inclusions of two copies of $\cD$ given by $x\mapsto (x\to x),$ $x\mapsto (0\to x).$ 
\end{proof}

To give a description of the oplax $2$-functor $\cC\otimes -$ in terms of the $2$-category $\Corr_{\kappa}$ we need the following simple observation about the behavior of the usual Lurie tensor product with respect to semi-orthogonal decompositions.

\begin{lemma}\label{lem:C_otimes_SOD} Let $\cC,\cT,\cD_1,\dots,\cD_n$ be presentable stable categories, and suppose that we have fully faithful continuous functors $i_k:\cD_k\to\cT,$ $1\leq k\leq n,$ which give a semi-orthogonal decomposition
	$$\cT=\la i_1(\cD_1),\dots,i_n(\cD_n)\ra.$$
	Then the functors $\cC\otimes i_k:\cC\otimes\cD_k\to\cC\otimes\cT$ are also fully faithful, and they give a semi-orthogonal decomposition
	$$\cC\otimes\cT=\la (\cC \otimes i_1)(\cC\otimes \cD_1),\dots,(\cC \otimes i_n)(\cC\otimes \cD_n)\ra.$$\end{lemma}

\begin{proof} It suffices to prove this for $n=2.$ Note that the functor $i_1$ has a left adjoint $i_1^L,$ and the functor $i_2$ has a continuous right adjoint $i_2^R.$ Hence, the functors $\cC\otimes i_1$ and $\cC\otimes i_2$ are fully faithful. Their images generate $\cC\otimes\cT$ by colimits. Finally, to see the the semi-orthogonality, note that $\cC\otimes i_1^L$ is the left adjoint to $\cC\otimes i_1,$ and we have
	$$(\cC\otimes i_1^L)\circ (\cC\otimes i_2)\cong \cC\otimes (i_1^L\circ i_2)=0.$$
	This is equivalent to the required semi-orthogonality.
\end{proof}

\begin{prop}\label{prop:tensoring_with_C_via_correspondences}
Let $\kappa$ be a regular cardinal and let $\cC$ be a $\kappa$-presentable stable category. Then we have an oplax $2$-functor $\cC\otimes-:\Corr_{\kappa}\to\Corr_{\kappa}$ which is given on objects by $\cD\mapsto \cC\otimes\cD,$ and on $1$-morphisms by $(\cT,i_0,i_1)\mapsto (\cC\otimes\cT,\cC\otimes i_0,\cC\otimes i_1).$ Within the above notation the structural $2$-morphism $\cC\otimes(\cT_{12}\circ \cT_{01})\to (\cC\otimes\cT_{12})\circ (\cC\otimes\cT_{01})$ is induced by the functor $\cC\otimes\cT_{02}\to \cC\otimes\cT_{012}.$ 

Moreover, this oplax $2$-functor is identified with the one from Proposition \ref{prop:oplax_2_functor_direct_construction} via the $2$-equivalence $\Corr_{\kappa}\simeq \mPr^{\acc}_{\st,\kappa}.$
\end{prop}

\begin{proof}
By Lemma \ref{lem:C_otimes_SOD} we have well-defined functors on categories of $1$-morphisms $\cC\otimes-:\Corr_{\kappa}(\cD_0,\cD_1)\to\Corr_{\kappa}(\cC\otimes\cD_0,\cC\otimes\cD_1).$ For $\cD_0=\cD_1$ this functor clearly sends the identity $1$-morphism of $\cD_0$ to the identity $1$-morphism of $\cC\otimes\cD_0.$ 
It can be checked directly that the $2$-morphisms $\cC\otimes(\cT_{12}\circ \cT_{01})\to (\cC\otimes\cT_{12})\circ (\cC\otimes\cT_{01})$ give a well-defined structure of an oplax $2$-functor on $\cC\otimes-:\Corr_{\kappa}\to\Corr_{\kappa},$ but this also follows from the second assertion of the proposition, which we now prove. More precisely, wee will see that these $2$-morphisms correspond to the $2$-morphisms defining the structure of an oplax $2$-functor on $\cC\otimes-:\mPr^{\acc}_{\st,\kappa}\to\mPr^{\acc}_{\st,\kappa}.$

Let $F:\cD_0\to\cD_1$ be a $1$-morphism in $\mPr^{\acc}_{\st,\kappa}.$ As above, we identify $\cD_1\oright_F\cD_0$ with $\cD_0\laxtimes{F}\cD_1.$ In the latter category the limits are computed componentwise, hence we have a tautological inclusion 
\begin{multline*}
\Fun^{\kappa\hy\lex}((\cC^{\kappa})^{op},\cD_0\laxtimes{F}\cD_1)\hto \Fun((\cC^{\kappa})^{op},\cD_0\laxtimes{F}\cD_1)\\
\simeq \Fun((\cC^{\kappa})^{op},\cD_0)\laxtimes{F\circ -}\Fun((\cC^{\kappa})^{op},\cD_1).
\end{multline*} 
Its essential image is equivalent to $\Fun^{\kappa\hy\lex}((\cC^{\kappa})^{op},\cD_0)\laxtimes{(F\circ -)^{\kappa\hy\lex}}\Fun^{\kappa\hy\lex}((\cC^{\kappa})^{op},\cD_1).$ Hence, the above functor $\cC\otimes-:\Corr_{\kappa}(\cD_0,\cD_1)\to \Corr_{\kappa}(\cC\otimes\cD_0,\cC\otimes\cD_1)$ is identified with the functor $\cC\otimes-:\mPr^{\acc}_{\st,\kappa}(\cD_0,\cD_1)\to\mPr^{\acc}_{\st,\kappa}(\cC\otimes\cD_0,\cC\otimes\cD_1)$ from Proposition \ref{prop:oplax_2_functor_direct_construction}.

Compatibility of the $2$-morphisms (defining the structures of the oplax $2$-functors) is essentially obtained by a direct computation. Namely, let $F_{01}:\cD_0\to\cD_1$ and $F_{12}:\cD_1\to\cD_2$ be $1$-morphisms in $\mPr^{\acc}_{\st,\kappa}.$ We put $F_{02}=F_{12}\circ F_{01},$ and $\cT_{ij}=\cD_i\laxtimes{F_{ij}}\cD_j$ for $0\leq i<j\leq 2.$ We also put $\cT_{012}=\cD_0\laxtimes{F_{01}}\cD_1\laxtimes{F_{12}}\cD_2.$ The inclusion $i_{02}:\cT_{02}\to\cT_{012}$ is given by $(x_0,x_2,\psi)\mapsto (x_0,F_{01}(x_0),x_2,\id,\psi).$ As above, we identify the category of (exact) functors $(\cC^{\kappa})^{op}\to\cT_{02}$ with the corresponding lax limit of categories of functors $(\cC^{\kappa})^{op}\to\cD_i,$ $i=0,2,$ and similarly for functors $(\cC^{\kappa})^{op}\to\cT_{012}.$ Consider the inclusion functor
\begin{equation*}
j:\Fun^{\kappa\hy\lex}((\cC^{\kappa})^{op},\cT_{012})\to \Fun((\cC^{\kappa})^{op},\cT_{012}).
\end{equation*}
Its left adjoint $j^L=(-)^{\kappa\hy\lex}$ can be described explicitly in terms of similar functors for $\cD_i$ instead of $\cT_{012}.$ The formula is slightly cumbersome, but we only need the values of $j^L$ on objects $(\Phi_0,\Phi_1,\Phi_2,\psi_{01},\psi_{12})$ such that $\Phi_0:(\cC^{\kappa})^{op}\to\cD_0$ is $\kappa$-left exact. In this case we have
\begin{equation*}
j^L(\Phi_0,\Phi_1,\Phi_2,\psi_{01},\psi_{12})=(\Phi_0,\Phi_1^{\kappa\hy\lex},\Phi_2^{\kappa\hy\lex}\sqcup_{(F_{12}\circ \Phi_1)^{\kappa\hy\lex}}(F_{12}\circ\Phi_1^{\kappa\hy\lex})^{\kappa\hy\lex},\psi_{01}',\psi_{12}').
\end{equation*} 
Here $\psi_{01}':F_{01}\circ \Phi_0\xto{\psi_{01}}\Phi_1\to\Phi_1^{\kappa\hy\lex},$ and $\psi_{12}'$ is obtained by composing $F_{12}\circ \Phi_1^{\kappa\hy\lex}\to (F_{12}\circ \Phi_1^{\kappa\hy\lex})^{\kappa\hy\lex}$ with the structural morphism to the pushout. Now we see that the functor $\cC\otimes i_{02}:\cC\otimes\cT_{02}\to\cC\otimes\cT_{012}$ is identified with the following one:
\begin{equation*}
\Fun^{\kappa\hy\lex}((\cC^{\kappa})^{op},\cT_{02})\to \Fun^{\kappa\hy\lex}((\cC^{\kappa})^{op},\cT_{012}),
\end{equation*}
\begin{equation*}
(\Phi_0,\Phi_2,\psi)\mapsto (\Phi_0,(F_{01}\circ\Phi_0)^{\kappa\hy\lex},\Phi_2\sqcup_{(F_{02}\circ \Phi_0)^{\kappa\hy\lex}}(F_{12}\circ (F_{01}\circ \Phi_0)^{\kappa\hy\lex})^{\kappa\hy\lex},\psi_{01}',\psi_{12}').
\end{equation*}
Denote by $i_0:\cD_0\to\cT_{02}$ and $i_2:\cD_2\to\cT_2$ the inclusions. Then the functor $\cC\otimes i_0$ is identified with
\begin{equation*}
\Fun^{\kappa\hy\lex}((\cC^{\kappa})^{op},\cD_0)\to \Fun^{\kappa\hy\lex}((\cC^{\kappa})^{op},\cT_{02}),\quad \Phi\mapsto (\Phi,(F_{02}\circ \Phi)^{\kappa\hy\lex},\psi_{02}), 
\end{equation*}
where $\psi_{02}:F_{02}\circ \Phi\to (F_{02}\circ \Phi)^{\kappa\hy\lex}$ is the adjunction counit. Also, the functor $\cC\otimes i_2$ is identified with
\begin{equation*}
\Fun^{\kappa\hy\lex}((\cC^{\kappa})^{op},\cD_2)\to \Fun^{\kappa\hy\lex}((\cC^{\kappa})^{op},\cT_{02}),\quad \Psi\mapsto (0,\Psi,0).
\end{equation*}
Putting $i_0' = i_{02}\circ i_0,$ $i_2' = i_{02}\circ i_2,$ we obtain a commutative square in $\h \Fun^{\kappa\hy\cont}(\cC\otimes\cD_0,\cC\otimes\cD_2):$
\begin{equation*}
\begin{tikzcd}
(\cC\otimes i_2)^R\circ (\cC\otimes i_0) \ar{r}\ar[d, "\sim" {anchor=south, rotate=90}] & (\cC\otimes i_2')^R\circ (\cC\otimes i_0') \ar[d, "\sim" {anchor=south, rotate=90}]\\
(F_{02}\circ -)^{\kappa\hy\lex} \ar{r} & (F_{12}\circ (F_{01}\circ -)^{\kappa\hy\lex})^{\kappa\hy\lex},
\end{tikzcd}
\end{equation*}
where the lower horizontal arrow is induced by the adjunction counit $F_{12}\circ -\to (F_{12}\circ -)^{\kappa\hy\lex}.$ In other words, the above $2$-morphism $\cC\otimes(\cT_{12}\circ\cT_{01})\to (\cC\otimes\cT_{12})\circ (\cC\otimes\cT_{01})$ in $\Corr_{\kappa}(\cC\otimes\cD_0,\cD_2)$ corresponds to the structural $2$-morphism $\cC\otimes (F_{12}\circ F_{01})\to (\cC\otimes F_{12})\circ (\cC\otimes F_{01})$ in $\mPr^{\acc}_{\st,\kappa}(\cC\otimes\cD_0,\cC\otimes\cD_2).$ This proves the proposition.
\end{proof}

\begin{cor}\label{cor:if_flat_then_actual_2_functor}
Let $\cC$ be a $\kappa$-presentable stable category. If $\cC$ is flat in $\Pr^L_{\st,\kappa},$ then the oplax $2$-functor $\cC\otimes-:\mPr^{\acc}_{\st,\kappa}\to\mPr^{\acc}_{\st,\kappa}$ from Proposition \ref{prop:oplax_2_functor_direct_construction} is an actual $2$-functor.
\end{cor}

\begin{proof}
By Proposition \ref{prop:tensoring_with_C_via_correspondences} it suffices to check the following: if $\cT_{01}\in\Corr_{\kappa}(\cD_0,\cD_1)$ and $\cT_{12}\in\Corr_{\kappa}(\cD_1,\cD_2)$ are composable $1$-morphisms in $\Corr_{\kappa},$ then the $2$-morphism $\cC\otimes(\cT_{12}\circ\cT_{01})\to (\cC\otimes\cT_{12})\circ (\cC\otimes\cT_{01})$ is an isomorphism in $\Corr_{\kappa}(\cC\otimes\cD_0,\cC\otimes\cD_2).$ In terms of the categories $\cT_{02}$ and $\cT_{012}$ as above this means that the (continuous) functor $\cC\otimes\cT_{02}\to\cC\otimes\cT_{012}$ is fully faithful. But this holds by the definition of flatness: the categories $\cT_{02}$ and $\cT_{012}$ are $\kappa$-presentable and the inclusion $\cT_{02}\hto\cT_{012}$ is $\kappa$-strongly continuous.
\end{proof}

\begin{proof}[Proof of Theorem \ref{th:flat=dualizable}] By Proposition \ref{prop:AB6_criterion}, the final assertion follows from the equivalence \Iff{flatness}{AB6_kappa}. Below for brevity we say that $\cC$ satisfies (AB6)$_{\kappa}$ if (AB6) for $\kappa$-small products holds in $\cC.$
	
\Implies{AB6_kappa}{flatness}. Suppose that $\cC$ satisfies (AB6)$_{\kappa}.$ Take some fully faithful $\kappa$-strongly continuous functor $F:\cD\to\cE$ between $\kappa$-presentable stable categories. The right adjoint to $\cC\otimes F$ is identified with the precomposition functor
	\begin{equation}\label{eq:precomposition_for_C}
		-\circ F^{\kappa,op}:\Fun^{\kappa\hy\lex}(\cE^{\kappa,op},\cC)\to \Fun^{\kappa\hy\lex}(\cD^{\kappa,op},\cC).
	\end{equation}
	We need to show that this is a colocalization. By the proof of Proposition \ref{prop:AB6_criterion}, our assumption on $\cC$ means that the colimit functor $\colim:\Ind(\cC^{\kappa})\to\cC$ commutes with $\kappa$-small limits. Its right adjoint commutes with all limits, hence the functor \eqref{eq:precomposition_for_C} is a retract of the functor
	\begin{equation}\label{eq:precomposition_for_Ind_C_kappa}
		-\circ F^{\kappa,op}:\Fun^{\kappa\hy\lex}(\cE^{\kappa,op},\Ind(\cC^{\kappa}))\to \Fun^{\kappa\hy\lex}(\cD^{\kappa,op},\Ind(\cC^{\kappa}))
	\end{equation}
	in the $(\infty,1)$-category $\Pr^{\acc}_{\st,\kappa}.$ The functor \eqref{eq:precomposition_for_Ind_C_kappa} is a colocalization since the category $\Ind(\cC^{\kappa})$ is dualizable, hence flat. The class of colocalizations is closed under retracts in $\Pr^{\acc}_{\st}$ by Proposition \ref{prop:localizations_stable_under_pullbacks}. Hence, \eqref{eq:precomposition_for_C} is a colocalization, as required.
	
The implication \Implies{flatness}{2_functoriality} follows from Corollary \ref{cor:if_flat_then_actual_2_functor}.

It remains to prove \Implies{2_functoriality}{AB6_kappa}. Suppose that we have a $2$-functor $\cC\otimes-:\mPr^{\acc}_{\st,\kappa}\to \mPr^{\acc}_{\st,\kappa}$ which induces the usual Lurie tensor product $\cC\otimes-:\mPr^L_{\st,\kappa}\to \mPr^L_{\st,\kappa}.$ Consider a $\kappa$-small family of filtered categories $J_i,$ $i\in I.$ We have the following natural exact $\kappa$-accessible functors between $\kappa$-presentable stable categories:
	
	- $T_{\cC}:\prodd[i\in I]\Fun(J_i,\cC)\to \prodd[i\in I] \cC,$ the product of the colimit functors $\Fun(J_i,\cC)\to \cC;$
	
	- $U_{\cC}:\prodd[i\in I] \cC\to \cC,$ the product functor;
	
	- $V_{\cC}:\prodd[i\in I]\Fun(J_i,\cC)\to \Fun(\prodd[i\in I]J_i,\cC),$ taking $(F_i)_i$ to the functor $(j_i)_i\mapsto \prod_{i\in I}F_i(j_i);$
	
	- $W_{\cC}:\Fun(\prodd[i\in I]J_i,\cC)\to\cC,$ the colimit functor.
	
	The functors $T_{\cC}$ and $W_{\cC}$ are continuous, and the functors $U_{\cC}$ and $V_{\cC}$ have left adjoints $U_{\cC}^L$ and $V_{\cC}^L$ respectively. Here $U_{\cC}^L$ is the diagonal functor $\cC\to\prodd[i\in I]\cC,$ and the components of $V_{\cC}^L$ are the left Kan extension functors $\Fun(\prod_{i\in I}J_i,\cC)\to \Fun(J_j,\cC).$
	Tautologically, we have a commutative square
	\begin{equation}\label{eq:square_for_AB6}
		\begin{tikzcd}
			\prodd[i\in I]\Fun(J_i,\cC) \ar{r}{T_{\cC}} & \prodd[i\in I] \cC\\
			\Fun(\prodd[i\in I]J_i,\cC) \ar{r}{W_{\cC}} \ar{u}{V_{\cC}^L} & \cC \ar{u}{U_{\cC}^L}.
		\end{tikzcd}
	\end{equation}
	
	Now, (AB6)$_{\kappa}$ for $\cC$ means that the square \eqref{eq:square_for_AB6} satisfies the dual Beck-Chevalley condition: the induced map $W_{\cC}\circ V_{\cC}\to U_{\cC}\circ T_{\cC}$ is an isomorphism.
	
	Note that the square \eqref{eq:square_for_AB6} is equivalent to the following tensor product:
	$$
	\begin{tikzcd}
		\cC\otimes \prodd[i\in I]\Fun(J_i,\Sp) \ar{r}{\cC\otimes T_{\Sp}} & \cC\otimes\prodd[i\in I] \Sp\\
		\cC\otimes \Fun(\prodd[i\in I]J_i,\Sp) \ar{r}{\cC\otimes W_{\Sp}} \ar{u}{\cC\otimes V_{\Sp}^L} & \cC\otimes\Sp \ar{u}{\cC\otimes U_{\Sp}^L}.
	\end{tikzcd}
	$$
	
	Since the axiom (AB6) holds in the category $\Sp,$ the map $W_{\Sp}\circ V_{\Sp}\to U_{\Sp}\circ T_{\Sp}$ is an isomorphism. Applying the $2$-functor $\cC\otimes-:\mPr^{\acc}_{\st,\kappa}\to\mPr^{\acc}_{\st,\kappa},$ we conclude that the map  $W_{\cC}\circ V_{\cC}\to U_{\cC}\circ T_{\cC}$ is an isomorphism. Therefore, (AB6)$_{\kappa}$ holds in $\cC.$
\end{proof}

\section{Extensions of compactly generated categories}
\label{sec:extensions_of_comp_gen}

Consider a short exact sequence of presentable stable categories of the form
\begin{equation}\label{eq:extension_of_comp_gen} 0\to\Ind(\cA)\xto{F}\cC\xto{G}\Ind(\cB)\to 0,\end{equation}
where $\cA$ and $\cB$ are small idempotent-complete stable categories, and the functors $F$ and $G$ are strongly continuous.

\begin{ques}Is it true that the category $\cC$ is automatically dualizable?\end{ques}

It turns out that in general the answer is ``no''. To give a precise criterion of dualizability of $\cC,$ we introduce some notation.

Denote by $G^R$ the right adjoint to $G,$ and by $G^{RR}$ the right adjoint to $G^R.$ Note that we have a semi-orthogonal decomposition $\cC=\langle G^R(\Ind(\cB)),F(\Ind(\cA))\rangle.$ Putting $\Phi:=G^{RR}\circ F:\Ind(\cA)\to\Ind(\cB),$ we obtain an equivalence
$\cC\simeq \Ind(\cB)\oright_{\Phi}\Ind(\cA)$ (the notation for the semi-orthogonal gluing is from Definition \ref{defi:gluing_general}).

Note that we have equivalences of categories of functors:
\begin{multline*}\Fun^{\acc}(\Ind(\cA),\Ind(\cB))\simeq\Fun^{\acc}(\Ind(\cA),\Fun(\cB^{op},\Sp))\\
\simeq\Fun(\cB^{op},\Fun^{\acc}(\Ind(\cA),\Sp))\simeq \Fun(\cB^{op},\Ind(\Ind(\cA)^{op}))\simeq\Fun(\cB,\Pro(\Ind(\cA)))^{op}.\end{multline*}

We denote by $\Psi:\cB\to\Pro(\Ind(\cA))$ the functor corresponding to $\Phi.$

Recall the following full subcategories of $\Pro(\Ind(\cA))$ formed by (elementary) Tate objects.

\begin{defi}\cite[Definition 1]{Hen} We denote by $\Tate_{\el}(\cA)\subset\Pro(\Ind(\cA))$ the full subcategory which is generated by $\Ind(\cA)$ and $\Pro(\cA)$ as a stable subcategory. We denote by $\Tate(\cA)$ the Karoubi completion of $\Tate_{\el}(\cA).$\end{defi}

The following result gives a criterion when $\cC$ is dualizable.

\begin{prop}\label{prop:when_extension_is_dualizable} Within the above notation, the following are equivalent.

\begin{enumerate}[label=(\roman*),ref=(\roman*)]
\item $\cC$ is dualizable. \label{ext_dual1}

\item $\cC$ is compactly generated. \label{ext_dual2}

\item the image of $\Psi$ is contained in $\Tate(\cA)\subset\Pro(\Ind(\cA)).$ \label{ext_dual3}
\end{enumerate}
\end{prop}

Assuming Proposition \ref{prop:when_extension_is_dualizable}, we deduce the following result, which first appeared in \cite[Appendix A]{CDH+20}.

\begin{theo}\label{th:Ext^1_between_categories} Let $\cA$ and $\cB$ be small stable Karoubi complete categories. Denote by $\Ext^1(\cB,\cA)$ the (large, locally small) space of short exact sequences in $\Cat^{\perf}$ of the form
\begin{equation}\label{eq:extension_of_small_cats} 0\to\cA\to\cD\to\cB\to 0.\end{equation}
Then we have an equivalence $\Ext^1(\cB,\cA)\simeq \Fun(\cB,\Tate(\cA))^{\simeq}.$\end{theo}

\begin{proof}The functor $\Ind(-)$ gives an equivalence between the space of short exact sequences of the form \eqref{eq:extension_of_small_cats} and the space of short exact sequences of the form \eqref{eq:extension_of_comp_gen} such that $\cC$ is compactly generated. Applying Proposition \ref{prop:when_extension_is_dualizable}, \Iff{ext_dual2}{ext_dual3}, we see that the latter space is equivalent to $\Fun(\cB,\Tate(\cA))^{\simeq}.$\end{proof}

\begin{remark}The special case of Theorem \ref{th:Ext^1_between_categories} for split extensions is well known. Namely, short exact sequences of the form \eqref{eq:extension_of_small_cats} where both functors have right (resp. left) adjoints correspond to semi-orthogonal decompositions of the form $\cD=\langle\cB,\cA\rangle$ (resp. $\cD=\langle\cA,\cB\rangle$), which in turn correspond to functors $\cB\to\Pro(\cA)\subset\Tate(\cA)$ (resp. $\cB\to\Ind(\cA)\subset\Tate(\cA)$), i.e. $\cA\mhyphen\cB$-bimodules (resp. $\cB\mhyphen\cA$-bimodules).\end{remark}

\begin{proof}[Proof of Proposition \ref{prop:when_extension_is_dualizable}] Clearly, \Implies{ext_dual2}{ext_dual1}. 

\Implies{ext_dual1}{ext_dual2}. If $\cC$ is dualizable, then by Proposition \ref{prop:compact_objects_in_quotients} for any object $x\in\cB$ there exists an object $y\in\cC^{\omega}$ such that $G(y)\cong x\oplus x[1]$. Such objects together with the objects of $F(\cA)$ form a generating collection of compact objects of $\cC,$ hence $\cC$ is compactly generated.

\Iff{ext_dual2}{ext_dual3}. We use the above observation: $\cC$ is compactly generated if and only if for any object $x\in \cB$ there exists a compact object $y\in\cC^{\omega}$ such that $G(y)\cong x\oplus x[1]=:z.$ Via the identification $\cC\simeq\Ind(\cB)\oright_{\Phi}\Ind(\cA),$ $\Phi=G^{RR}\circ F,$ the object $y$ must be of the form $(z,w,\varphi),$ $\varphi:z\to\Phi(w).$

We need to describe the necessary and sufficient conditions for $(z,w,\varphi)$ to be a compact object of $\cC.$ Note that $\Hom_{\cC}((z,w,\varphi),(-,0,0))\cong\Hom(z,-)$ commutes with colimits since $z$ is a compact object of $\Ind(\cB).$ We have
\begin{equation}\label{eq:Hom_as_fiber} \Hom_{\cC}((z,w,\varphi),(0,-,0))=\Fiber(\Hom_{\Ind(\cA)}(w,-)\to \Hom_{\Ind(\cB)}(z,\Phi(-))).\end{equation}
We deduce that the object $(z,w,\varphi)$ is compact if and only if the right-hand side in \eqref{eq:Hom_as_fiber} commutes with coproducts.

Note that the functor $\Hom_{\Ind(\cB)}(z,\Phi(-)))$ is exactly the object $\Psi(z)\in\Pro(\Ind(\cA)),$ where $\Psi:\cB\to\Pro(\Ind(\cA))$ corresponds to $\Phi,$ as above. We conclude that the existence of $y\in\cC^{\omega}$ such that $G(y)\cong z$ is equivalent to the existence of objects $u\in\Pro(\cA)\subset\Pro(\Ind(\cA))$ and $v\in\Ind(\cA)\subset\Pro(\Ind(\cA)),$ and a cofiber sequence $$u\to\Psi(z)\to v\quad \text{in }\Pro(\Ind(\cA)).$$
By \cite[Theorem 3]{Hen}, this is equivalent to the inclusion $\Psi(z)\in\Tate_{\el}(\cA).$ By Thomason's theorem the latter condition is equivalent to the inclusion $\Psi(x)\in\Tate(\cA).$  This proves the equivalence of \ref{ext_dual2} and \ref{ext_dual3}.
\end{proof}

\begin{remark}It follows from the proof of Proposition \ref{prop:when_extension_is_dualizable} that the universal extension of $\Tate(\cA)$ by $\cA$ is the following short exact sequence:
$$0\to\cA\xto{F}\Ind(\cA)\oright_{T}\Pro(\cA)\xto{G}\Tate(\cA)\to 0.$$
Here the functor $T:\Pro(\cA)\to\Ind(\cA)$ is given by $T(\proolim[i]x_i)=\prolim[i] x_i.$ The functor $F$ is given by $F(x)=(x,x,\id_x).$ The functor $G$ is given by $G(x,y,\varphi)=\Cone(x\xto{\varphi} y).$ The category $\Ind(\cA)\oright_{T}\Pro(\cA)$ can be thought of as the category of elementary Tate objects of $\cA$ with a choice of a lattice. The functor $G$ then simply forgets the lattice.\end{remark}

\begin{example} Let $p$ be a prime number, and consider the short exact sequence
$$0\to\Perf_{p\hy\tors}(\Z)\to\Perf(\Z)\to\Perf(\Z[p^{-1}])\to 0.$$
The associated $\Z$-linear functor $\Perf(\Z[p^{-1}])\to \Tate(\Perf_{p\hy\tors}(\Z))$ sends $\Z[p^{-1}]$ to $\Q_p,$ where $\Q_p$ is considered as a Tate object in the usual way.\end{example}

We give a natural ``linear-algebraic'' example of a non-dualizable category which is an extension of compactly generated categories. We use the notation $x^{(S)}=\bigoplus\limits_S x$ for the direct sum of copies of an object $x$ indexed by the elements of a set $S.$

\begin{prop}Let $\mk$ be a field. Consider the category $\cC=D(\mk)\oright_{\Phi} D(\mk),$ where $\Phi(V)=(V^{(\N)})^{\N}.$ Equivalently, $\cC$ is the category of triples $(V,W,\varphi),$ where $V,W\in D(\mk)$ and $\varphi:V^{(\N)}\to W^{(\N)}.$ Then $\cC$ is not dualizable.\end{prop}

\begin{proof}Consider the functor $\Psi:\Perf(k)\to\Pro(D(\mk))$ corresponding to $\Phi.$ By Proposition \ref{prop:when_extension_is_dualizable}, we need to show that $\Psi(\mk)\not\in\Tate(\Perf(\mk)).$ 

Identifying $\Pro(D(\mk))$ with the opposite of the category of $\mk$-linear accessible functors $D(\mk)\to D(\mk),$ we see that $\Psi(\mk)$ corresponds to the functor $\Phi.$ Using the fact that the axiom (AB6) (from Subsection \ref{ssec:AB_axioms}) holds in $D(\mk),$ we obtain the isomorphisms
$$\Phi(V)=(V^{(\N)})^{\N}\cong \indlim[f:\N\to\N] \prodd[n\in\N]\bigoplus\limits_{i=0}^{f(n)-1} V\cong\indlim[f:\N\to\N]\Hom(\bigoplus\limits_{n\in \N}\mk^{f(n)},V).$$
Putting $X_f:=\bigoplus\limits_{n\in \N}\mk^{f(n)},$ we obtain an isomorphism $\Psi(\mk)\cong\proolim[f:\N\to\N] X_f.$

Note that the quotient $\Pro(D(\mk))/\Pro(\Perf(\mk))$ is equivalent to $\Pro(\Calk(\mk)).$ Moreover, this equivalence identifies the quotient $\Tate(\Perf(\mk))/\Pro(\Perf(\mk))$ with the full subcategory of pro-constant objects $\Calk(\mk)\subset\Pro(\Calk(\mk)).$ Therefore, to show that $\Psi(\mk)\not\in\Tate(\Perf(\mk)),$ it suffices to show that the image of $\Psi(\mk)$ in $\Pro(\Calk(\mk))$ is not a pro-constant object. Assume the contrary. For a function $f:\N\to\N$ denote by $\bar{X}_f\in \Calk(\mk)$ the image of $X_f.$ Since the maps $X_g\to X_f$ are split epimorphisms for $f\leq g,$ we deduce that the objects $\bar{X}_f$ must form an eventually constant pro-system $(\bar{X}_f)_{f:\N\to\N}$ in $\Calk(\mk).$ But this is clearly false: for any $f:\N\to\N,$ the map $\bar{X}_{f+1}\to \bar{X}_f$ is not an isomorphism in $\Calk(\mk)$ since $\Fiber(X_{f+1}\to X_f)\cong\mk^{(\N)}.$\end{proof}

Another application of Proposition \ref{prop:when_extension_is_dualizable} is the classification of dualizable localizing subcategories of $D(R),$ where $R$ is a noetherian commutative ring (we do not require the inclusion functor to be strongly continuous). For a subset $S\subset\Spec R$ denote by
$D_S(R)\subset D(R)$ the full localizing subcategory generated by the residue fields $\mk(\mfp)=R_{\mfp}/\mfp R_{\mfp},$ $\mfp\in S.$  
By \cite[Theorem 2.8]{Nee92} the assignment $S\mapsto D_S(R)$ gives a bijection
$$\{\text{subsets of }\Spec R\}\cong \{\text{localizing subcategories of }D(R)\}.$$

Moreover, by \cite[Theorem 3.3]{Nee92} the inclusion $D_S(R)\to D(R)$ is strongly continuous if and only if $S$ is closed under specialization. In this case the category $D_S(R)$ is compactly generated by the Koszul complexes $\Kos(R;f_1,\dots,f_n),$ where $f_1,\dots,f_n\in R$ are elements such that $V(f_1,\dots,f_n)\subset S.$

\begin{defi}Given a poset $(P,\leq)$ and two elements $x,y\in P,$ we put $[x,y]:=\{z\in P\mid x\leq z\leq y\}.$ We call a subset $S\subset P$ convex if for any $x,y\in S$ we have $[x,y]\subset S.$\end{defi}

\begin{theo}\label{th:dualizability_equiv_convexity} Let $R$ be a noetherian commutative ring and $S\subset\Spec R$ a subset. Then the category $D_S(R)$ is dualizable if and only if $S\subset\Spec R$ is convex with respect to the specialization order. In this case the category $D_S(R)$ is compactly generated.\end{theo}

\begin{example}Let $\mk$ is a field. Consider the following full subcategory of the derived category of quasi-coherent sheaves on the affine plane:
$$\cC:=\{M\in D(\mk[x,y])\mid M[x^{-1}]/y M[x^{-1}]=0\}\subset D(\mk[x,y]).$$ Then Theorem \ref{th:dualizability_equiv_convexity} implies that $\cC$ is not dualizable. Indeed, the corresponding subset of $\A_{\mk}^2=\Spec\mk[x,y]$ is given by $S=\A_{\mk}^2\setminus\{x\ne 0,\, y=0\}.$ We see that $S$ contains the ideals $\{0\}$ and $\mk[x,y],$ but does not contain the ideal $(y),$ hence $S$ is not convex.\end{example}

We need the following auxiliary statement. 

\begin{lemma}\label{lem:perfectnessOver_local_rings} Let $(R,\m)$ be a noetherian local commutative ring, and $M\in D(R)$ a complex of $R$-modules. Denote by $\iota:D_{\m\hy\tors}(R)\to D(R)$ the inclusion functor. Suppose that the functor $\Hom_R(M,\iota(-))$ commutes with coproducts. Then $M$ is a perfect complex of $R$-modules.\end{lemma}

\begin{proof}We recall the following terminology. If $\cA$ is a Grothendieck abelian category, then an object $I\in\cA$ is an {\it injective cogenerator} if $I$ is injective in $\cA$ and for any non-zero object $X\in\cA$ there exists a non-zero map $X\to I.$ 
	
Now let $\cI$ be the injective hull of the $R$-module $R/\m.$ Then $\cI$ is a locally $\m$-torsion $R$-module, hence $\cI\in D_{\m\hy\tors}(R).$ Moreover, $\cI$ is an injective cogenerator of the abelian category of $R$-modules. 

We first observe that $M$ has bounded cohomology. Indeed, suppose that $H^{n_i}(M)\ne 0$ for infinitely many integers $n_i.$ For each $i,$ choose a non-zero map $H^{n_i}(M)\to\cI$ (which exists since $\cI$ is a cogenerator). Since $\cI$ is injective, we obtain the map $M\to\prodd[i]\cI[-n_i]\cong \bigoplus\limits_i\cI[-n_i].$ This map does not factor through a finite direct sum, which contradicts the commutation of $\Hom_R(M,\iota(-))$ with coproducts.

Next, we show that each $R$-module $H^n(M)$ is finitely generated. We may and will assume that $n=0.$ Suppose that $H^0(M)$ is not finitely generated. Then we can construct a strictly increasing sequence of submodules $$0\ne N_0\subsetneq N_1\subset\dots\subset H^0(M)$$ such that each $N_n$ is finitely generated. Put $N=\bigcup_{n} N_n.$ For $n\geq 0,$ choose a non-zero map $f_n:N_n\to \cI,$ such that $(f_n)_{\mid N_{n-1}}=0$ (where $N_{-1}=0$). Choose some $g_n:N\to \cI$ which extends $f_n,$ for each $n\geq 0.$ The sequence $(g_n)_{n\geq 0}$ gives a well-defined map $g:N\to\bigoplus\limits_{\N}\cI=\cI^{(\N)},$ which does not factor through a finite direct sum. Since $R$ is noetherian, the module $\cI^{(\N)}$ is injective. Choosing an extension $h:H^0(M)\to\cI^{(\N)}$ of $g,$ we obtain a morphism $M\to\cI^{(\N)},$ which does not factor through a finite direct sum. This contradicts the commutation of $\Hom_R(M,\iota(-))$ with coproducts.

So we showed that $M$ has bounded finitely generated cohomology. To prove that $M$ is perfect, it suffices to observe that $M\tens{R}R/\m$ is perfect over $R/\m:$ by adjunction, the functor $\Hom_{R/\m}(M\tens{R}R/\m,-)$ commutes with coproducts.
\end{proof}

\begin{lemma}\label{lem:nonconvex_special} Let $(R,\m)$ be a local noetherian commutative ring without zero-divisors. Let $S\subsetneq\Spec R$ be a proper subset such that $\m,\{0\}\in S,$ and $S\setminus\{\{0\}\}$ is closed under specialization. Then the category $D_S(R)$ is not dualizable.\end{lemma}

\begin{proof}Put $S':=S\setminus\{\{0\}\},$ and denote by $K$ the field of fractions of $R.$ We have a short exact sequence
$$0\to D_{S'}(R)\xto{F} D_S(R)\xto{G} D(K)\to 0,$$
where $F$ and $G$ are strongly continuous, and the categories on the left and on the right are compactly generated.

Suppose that $D_S(R)$ is a dualizable category. The proof of Proposition \ref{prop:when_extension_is_dualizable} shows that there exists an object $M\in D_{S'}(R)$ and a map $K\oplus K[1]\xto{\varphi} M$ such that the functor
$$\Fiber(\Hom_R(M,-)\xto{\Hom_R(\varphi,-)} \Hom_R(K\oplus K[1],F(-))):D_{S'}(R)\to D(R)$$ commutes with coproducts. Since $D_{\m\hy\tors}(R)\subset D_{S'}(R),$ denoting by $\iota:D_{\m\hy\tors}(R)\to D(R)$ the inclusion functor, we see that the functor
$\Hom_R(\Cone(\varphi),\iota(-))$ commutes with coproducts. By Lemma \ref{lem:perfectnessOver_local_rings}, we deduce that $\Cone(\varphi)\in\Perf(R).$ But choosing some $\mfp\in(\Spec R)\setminus S,$ we have
$$\Cone(\varphi)\tens{R}R_{\mfp}\cong K[1]\oplus K[2]\not\in\Perf(R_{\mfp}),$$
a contradiction.\end{proof}

\begin{proof}[Proof of Theorem \ref{th:dualizability_equiv_convexity}] Suppose that $S$ is convex. Denote by $S'\subset\Spec R$ the set of all points which are specializations of elements of $S.$ Put $S''=S'\setminus S.$ Then $S'$ and $S''$ are closed under the specialization, hence the categories $D_{S'}(R)$ and $D_{S''}(R)$ are compactly generated and the inclusion $D_{S''}(R)\to D_{S'}(R)$ is strongly continuous. Hence, the quotient $D_S(R)\simeq D_{S'}(R)/D_{S''}(R)$ is also compactly generated.

Now let $S\subset \Spec R$ be a non-convex subset. Suppose that the category $D_S(R)$ is dualizable. Choose some $x,y\in S$ such that $y\leadsto x$ and $[x,y]\not\subset S.$ Since $R$ is noetherian, we may and will assume that $y$ is a minimal element of the set $\{z\in S\mid z\leadsto x\text{ and }[x,z]\not\subset S\}.$ Denote by $\mfp\subset R$ and $\mfq\subset R$ respectively the prime ideals corresponding to the points $x$ and $y.$ Put $R':=R_{\mfp}/\mfq R_{\mfp}.$ We identify the set $\Spec R'$ with its image in $\Spec R,$ and put $S':=S\cap \Spec R'.$

Since the category $D_S(R)$ is dualizable, so is the category $D_S(R)\tens{R}D(R').$ Moreover, since $D(R')$ is dualizable, the functor
$$D_S(R)\tens{R}D(R')\to D(R)\tens{R}D(R')\simeq D(R')$$
is fully fiathful. Its essential image is clearly the category $D_{S'}(R'),$ which is therefore dualizable. But our choice of the ideals $\mfp$ and $\mfq$ implies that the subset $S'\subset\Spec R'$ satisfies the conditions of Lemma \ref{lem:nonconvex_special}. Hence, the category $D_{S'}(R')$ is not dualizable, a contradiction.\end{proof}

\begin{remark}
 The ``only if'' part of Theorem \ref{th:dualizability_equiv_convexity} fails for non-noetherian rings. Namely, let $R$ be a valuation ring, i.e. a commutative ring with no zero divisors such that for any element $x$ of its fraction field $K$ we have $x\in R$ or $x^{-1}\in R.$ The (unique) maximal ideal $\m\subset R$ consists of elements $x\in R$ such that $x^{-1}\not\in R.$ Suppose that moreover the valuation has rank $\geq 2,$ which equivalently means that there is a prime ideal $\mfp\subset R$ such that $\{0\}\subsetneq \mfp\subsetneq \m.$ Take the non-convex subset $S = \{\m,\{0\}\}\subset\Spec R.$ Then we have
 $$D_S(R)\simeq D_{\m\hy\tors}(R)\times D(K)$$ -- a compactly generated category.
 
 To see this, take some $x\in \m\setminus\mfp$ and $y\in \mfp\setminus\{0\},$ which implies that $y$ is divisible by any power of $x.$ Then $y$ annihilates any $\m^{\infty}$-torsion $R$-module $M.$ Indeed, if $\alpha\in M$ is an element such that $x^n\alpha=0,$ then $y\alpha=(x^{-n}y)x^n\alpha=0.$ It follows that for $N\in D_{\m\hy\tors}(R)$ we have
 $$\Hom_R(R[y^{-1}],N)\cong\prolim[](N\xlto{y}N\xlto{y}\dots)\cong 0.$$
 Hence we also have $\Hom_R(K,N)=0.$ Therefore, the full subcategories $D_{\m\hy\tors}(R)$ and $D(K)$ in $D_S(R)$ are mutually orthogonal, as stated. 
\end{remark}




\section{Essential image of a homological epimorphism}
\label{sec:image_of_hom_epi}

Consider a homological epimorphism $F:\cA\to\cB,$ where $\cA$ and $\cB$ are small stable idempotent-complete categories. We know that if the category $\ker(\Ind(F))\subset\Ind(\cA)$ is compactly generated, then $\cB\simeq (\cA/\ker(F))^{\Kar}.$ In particular, the essential image of $F$ is a stable subcategory of $\cB.$ In this case by Thomason's theorem every object of $\cB$ of the form $x\oplus x[1]$ can be lifted to $\cA.$

It turns out that if $\ker(\Ind(F))$ is not compactly generated then the image of $F$ might not be a stable subcategory of $\cB.$ More precisely, we show that in general for a dualizable category $\cC$ there may exist an object $x\in\Calk_{\omega_1}^{\cont}(\cC)$ such that the object $x\oplus x[1]$ cannot be lifted to $\cC^{\omega_1}.$ 

For the construction we use the dualizability of the category of sheaves $\Shv(X;\cC),$ where $X$ is a locally compact Hausdorff space and $\cC$ is a dualizable category. This is already established in \cite{Lur18}, see also Proposition \ref{prop:Shv_is_dualizable} below for a more general statement. For an object $P\in\cC$ and for an open subset $U\subset X$ we denote by $P_U$ the extension by zero of the constant sheaf on $U$ with value $P.$ By the proof of loc. cit. if $P$ is compact then we have $\hat{\cY}(P_U)\cong\inddlim[V\Subset U]P_V.$ 

\begin{prop}\label{prop:image_not_a_stable_subcat} Let $\mk$ be a field and consider the dualizable category $\cC=\Sh(\R;D(k)).$ Then there exists a non-zero idempotent-complete full stable subcategory $\cA\subset\Calk_{\omega_1}^{\cont}(\cC)$ such that $\cA\simeq\Perf(\mk)$ and the intersection of $\cA$ with the essential image of $\cC^{\omega_1}$ is zero. In particular, the essential image of the homological epimorphism $\cC^{\omega_1}\to \Calk_{\omega_1}^{\cont}(\cC)$ is not a stable subcategory.\end{prop}

\begin{proof}Consider the constant sheaf $\mk_{\R}.$ By the above description of $\hat{\cY}(\mk_{\R})$ and by Proposition \ref{prop:Homs_in_Calk} we have
$$\End_{\Calk_{\omega_1}^{\cont}(\cC)}(\mk_{\R})\cong \indlim[C>0]\Gamma(\R\setminus [-C,C];\mk_{\R})\cong \mk\times \mk$$ -- the cohomology of the neighborhood of infinity of $\R.$ The first resp. second copy of $\mk$ corresponds to the neighborhood of $-\infty$ resp. $+\infty.$ We denote by $e_-,e_+:\mk_{\R}\to\mk_{\R}$ the corresponding idempotent endomorphisms in the Calkin category, and denote their images by $\mk_-,\mk_+\in\Calk_{\omega_1}^{\cont}(\cC).$ We obtain fully faithful functors $\Phi_-,\Phi_+:\Perf(\mk)\to \Calk_{\omega_1}^{\cont}(\cC)$ which send $\mk$ to $\mk_-$ resp. $\mk_+.$ We claim that any non-zero object of the image of $\Phi_+$ is not contained in the essential image of $\cC^{\omega_1}.$

Assume the contrary: suppose that for some sheaf $\cF\in\cC^{\omega_1}$ the image of $\cF$ in the Calkin category is isomorphic to a non-zero object of $\Phi_+(\Perf(\mk)).$ Take some real number $C>0.$ By Proposition \ref{prop:compact_sheaves} the constant sheaf $\mk_{[-C,C]}$ is a compact object of $\Sh([-C,C];D(\mk)),$ hence the image of $\mk_+$ in $\Calk_{\omega_1}^{\cont}(\Sh([-C,C];D(\mk)))$ is zero. Thus, the pullback of $\cF$ to $[-C,C]$ is a compact object. Applying Proposition \ref{prop:compact_sheaves} again, we see that $\cF_{\mid [-C,C]}\cong M_{[-C,C]}$ for some perfect complex $M\in\Perf(\mk).$ This holds for any $C>0,$ hence $\cF\cong M_{\R}$ for some $M\in\Perf(\mk).$ Then $M\ne 0,$ and the image of $M_{\R}$ in $\Calk_{\omega_1}^{\cont}(\cC)$ is isomorphic to the direct sum $\Phi_-(M)\oplus \Phi_+(M),$ which is not contained in the essential image of $\Phi_+.$ This gives a contradiction.\end{proof}

\part{Localizing invariants of dualizable categories}
\label{part:loc_invar_dualizable}

\section{General theory of localizing invariants}
\label{sec:localizing_invariants}

\subsection{Universal localizing invariants}
\label{ssec:U_loc}

To avoid set-theoretic issues, we will consider only accessible localizing invariants with values in accessible stable categories.

Let $\cE$ be an accessible stable category (not necessarily cocomplete). Recall that a functor $F:\Cat^{\perf}\to\cE$ is a localizing invariant if the following conditions hold:

\begin{enumerate}[label=(\roman*),ref=(\roman*)]
\item $F(0)=0;$ \label{kappa_loc1}

\item for any fiber-cofiber sequence of the form
$$
\begin{tikzcd}
\cA \ar{r}\ar{d} & \cB\ar{d}\\
0 \ar{r} & \cC
\end{tikzcd}
$$ 
in $\Cat^{\perf},$ the square
$$
\begin{tikzcd}
	F(\cA) \ar{r}\ar{d} & F(\cB)\ar{d}\\
	0 \ar{r} & F(\cC)
\end{tikzcd}
$$
is a fiber sequence (equivalently, a cofiber sequence). \label{kappa_loc2}
\end{enumerate}

If $\kappa$ is a regular cardinal such that $\cE$ has $\kappa$-filtered colimits, we denote by $\Fun_{\loc,\kappa}(\Cat^{\perf},\cE)$ the category of localizing invariants which commute with $\kappa$-filtered colimits. It follows from \cite{BGT} that for any regular cardinal $\kappa$ there exists a universal localizing invariant
$\cU_{\loc,\kappa}:\Cat^{\perf}\to\Mot^{\loc}_{\kappa}$
such that the category $\Mot^{\loc}_{\kappa}$ is accessible and has $\kappa$-filtered colimits, the functor $\cU_{\loc,\kappa}$ is $\kappa$-continuous and for any other such category $\cE$ we have
$$\Fun_{\loc,\kappa}(\Cat^{\perf},\cE)\simeq \Fun^{\kappa\hy\cont}(\Mot^{\loc}_{\kappa},\cE).$$
Here the right-hand side is the category of exact functors which commute with $\kappa$-filtered colimits.

A clarification is required when $\kappa$ is uncountable. In this case the construction of \cite[Section 8.2]{BGT} gives a solution to a different universal problem: the authors construct a presentable stable category $\un{\cM}_{\loc}^{\kappa}$ and a localizing invariant $\cU^{\kappa}_{\loc}:\Cat^{\perf}\to\un{\cM}_{\loc}^{\kappa}$ which commutes with $\kappa$-filtered colimits, such that for any presentable stable category $\cE$ we have
$$\Fun_{\loc,\kappa}(\Cat^{\perf},\cE)\simeq \Fun^{L}(\un{\cM}_{\loc}^{\kappa},\cE).$$

The following proposition describes a relation between these two constructions.

\begin{prop}Let $\kappa$ be an uncountable regular cardinal. The category $\Mot^{\loc}_{\kappa}$ is $\kappa$-accessible, and the category $\un{\cM}_{\loc}^{\kappa}$ is compactly generated. We have an equivalence $$\Mot^{\loc}_{\kappa}\simeq \Ind_{\kappa}((\un{\cM}_{\loc}^{\kappa})^{\omega}).$$\end{prop} 

\begin{proof}It suffices to show that if $\cE$ is an accessible category with $\kappa$-filtered colimits, then the category $\Fun_{\loc,\kappa}(\Cat^{\perf},\cE)$ is equivalent to the category of partially defined localizing invariants $(\Cat^{\perf})^{\kappa}\to\cE$ (where we consider only short exact sequences of $\kappa$-compact categories). This follows from Lemma \ref{lem:approximating_ses} below.
\end{proof}

\begin{lemma}\label{lem:approximating_ses} Let $\kappa$ be an uncountable regular cardinal. Any short exact sequence in $\Cat^{\perf}$ is a $\kappa$-filtered colimit of short exact sequences of $\kappa$-compact categories.\end{lemma}

\begin{proof} Consider a short exact sequence 
\begin{equation}\label{eq:random_ses} 0\to\cA\to\cB\to\cC\to 0\end{equation}
 in $\Cat^{\perf}.$ Choose a $\kappa$-directed system $(\cB_i)_{i\in I}$ such that each $\cB_i$ is in $(\Cat^{\perf})^{\kappa}$ and $\cB\simeq\indlim[i]\cB_i.$ Consider a poset $J$ of pairs $(i,\cD),$ where $i\in I$ and $\cD\subset \cA\times_{\cB}\cB_i$ is a full stable idempotent-complete subcategory generated by a $\kappa$-small collection of objects. By Proposition \ref{prop:kappa_compact_small_cats}, such a category $\cD$ is automatically $\kappa$-compact. The poset $J$ is $\kappa$-directed, and we conclude that \eqref{eq:random_ses} is a $J$-indexed colimit of short exact sequences
$$0\to\cD\to\cB_i\to (\cB_i/\cD)^{\Kar}\to 0$$
of $\kappa$-compact categories.
\end{proof}

Localizing invariants which commute with filtered colimits are called {\it finitary}. In this case we write $$\Mot^{\loc}=\Mot^{\loc}_{\omega},\quad \cU_{\loc}=\cU_{\loc,\omega}.$$

Recall that for each regular cardinal $\kappa$ the category $\Mot^{\loc}_{\kappa}$ has a natural symmetric monoidal structure, and the functor $\cU_{\loc,\kappa}$ is symmetric monoidal. Below we will often formulate results in terms of the universal localizing invariants. They automatically imply the corresponding results for general localizing invariants.

\begin{remark} In \cite{E25b} we study the category $\Mot^{\loc}$ in great detail. In particular, in loc. cit. we prove that the category $\Mot^{loc}$ and its relative versions are rigid as symmetric monoidal categories in the sense of Gaitsgory and Rozenblyum \cite[Definition 9.1.2]{GaRo17}. This in particular implies that $\Mot^{\loc}$ is a dualizable category, which is quite non-obvious from the definition. In the present paper however we will use only the universal property of the categories $\Mot^{\loc}_{\kappa}.$\end{remark}
 
We recall the following standard notion of a $K$-equivalence.

\begin{defi}\label{defi:K-equivalence} Let $F:\cA\to\cB$ be an exact functor between small stable idempotent-complete categories. We say that $F$ is a $K$-equivalence if there exists an exact functor $G:\cB\to\cA$ such that we have $[F\circ G]=[\Id]$ in $K_0(\Fun(\cB,\cB))$ and $[G\circ F]=[\Id]$ in $K_0(\Fun(\cA,\cA)).$\end{defi}
 
It is well-known that a $K$-equivalence induces isomorphisms on all additive invariants, hence on all localizing invariants, see \cite{BGT}. 

\subsection{Localizing invariants of dualizable categories}
\label{ssec:loc_invar_dualizable}

The notion of a localizing invariant $\Cat_{\st}^{\dual}\to\cE$ is defined in the same way as for $\Cat^{\perf}.$ By Ramzi's theorem \cite{Ram24a} (see also Theorem \ref{th:presentability_of_Cat^dual}), the category $\Cat_{\st}^{\dual}$ is $\omega_1$-presentable, and again we will consider only accessible localizing invariants. It turns out that any localizing invariant $F$ defined on small categories has a unique extension to a localizing invariant $F^{\cont}$ defined on dualizable categories, so that $F^{\cont}(\Ind(\cA))=F(\cA).$ 

Roughly the idea is the following. Proposition \ref{prop:Calk_of_cg} shows that for a compactly generated presentable category $\cC$ the Calkin category $\Calk_{\omega_1}^{\cont}(\cC)$ is, up to direct summands, the quotient of $\cC^{\omega_1}$ by the full subcategory of compact objects $\cC^{\omega}.$ For a general dualizable category $\cC,$ one can think of $\Calk_{\omega_1}^{\cont}(\cC)$ as a ``virtual quotient of $\cC^{\omega_1}$ by compact objects''. We will use this point of view to extend localizing invariants from $\Cat^{\perf}$ to $\Cat_{\st}^{\dual}.$   





\begin{defi}\label{defi:F_cont} Let $\cC$ be a dualizable presentable $\infty$-category. For any accessible stable category $\cE$ and for any accessible localizing invariant $F:\Cat^{\perf}\to\cE,$ we define the functor $F^{\cont}:\Cat_{\st}^{\dual}\to \cE$ by the formula
$$F^{\cont}(\cC):=\Omega F(\Calk_{\omega_1}^{\cont}(\cC)).$$ 
If $F$ is the $K$-theory functor $K:\Cat^{\perf}\to\Sp,$ then we refer to the functor $K^{\cont}:\Cat_{\st}^{\dual}\to\Sp$ as continuous $K$-theory.
\end{defi}


\begin{prop}\label{prop:F_cont_is_loc}Let $\cE$ and $F$ be as in Definition \ref{defi:F_cont}. The functor $F^{\cont}:\Cat_{\st}^{\dual}\to\cE$ is a localizing invariant. In particular, $F$ is additive in semi-orthogonal decompositions.\end{prop}

\begin{proof}This follows immediately from Proposition \ref{prop:exactness_of_Calk}.\end{proof}


\begin{prop}\label{prop:F_cont_of_cg} If $\cE$ and $F$ are as in Definition \ref{defi:F_cont}, and $\cC\in\Cat_{\st}^{\cg},$ then we have a natural isomorphism $$F^{\cont}(\cC)\cong F(\cC^{\omega}).$$

More precisely, we have a natural isomorphism $F^{\cont}\circ\Ind\cong F$ of functors $\Cat^{\perf}\to\cE.$\end{prop}

\begin{proof}
This follows immediately from the short exact sequence (in $\Cat^{\perf}$)
$$0\to\cC^{\omega}\to \cC^{\omega_1}\to \Calk_{\omega_1}^{\cont}(\cC)\to 0,$$
and from the vanishing $F(\cC^{\omega_1})=0.$
\end{proof}

We have the following (expected) statement about filtered colimits.

\begin{prop}\label{prop:F_cont_filtered_colimits} Let $\cE$ be an accessible stable category with $\kappa$-filtered colimits, and $F:\Cat^{\perf}\to\cE$ a localizing invariant which commutes with $\kappa$-filtered colimits. Then the functor $F^{\cont}:\Cat_{\st}^{\dual}\to \cE$ commutes with $\kappa$-filtered colimits.\end{prop}

\begin{proof}Let $I$ be a $\kappa$-filtered category, and consider a functor $I\to\Cat_{\st}^{\dual},$ $i\mapsto \cC_i.$ Since a short exact sequence
$$0\to \cC\to\Ind(\cC^{\omega_1})\to\Ind(\Calk^{\cont}_{\omega_1}(\cC))\to 0$$ is functorial in $\cC\in\Cat_{\st}^{\dual},$ it follows from Proposition \ref{prop:weak_AB5_Pr^LL} that we have a short exact sequence
\begin{equation}\label{eq:colim_as_kernel}0\to \liminj_i^{\cont}\cC_i\to\Ind(\liminj_i\cC_i^{\omega_1})\to\Ind(\liminj_i\Calk_{\omega_1}^{\cont}(\cC_i))\to 0.\end{equation}
Now, we have $$F^{\cont}(\Ind(\indlim[i]\cC_i^{\omega_1}))\cong F(\indlim[i]\cC_i^{\omega_1})\cong\liminj_iF(\cC_i^{\omega_1})=0.$$ Applying $F^{\cont}$ to \eqref{eq:colim_as_kernel}, we get $$F^{\cont}(\indlim[i]^{\cont}\cC_i)\cong \Omega F(\indlim[i]\Calk_{\omega_1}^{\cont}(\cC_i))\cong\indlim[i]\Omega F(\Calk_{\omega_1}^{\cont}(\cC_i))\cong\indlim[i]F^{\cont}(\cC_i).$$ This proves the proposition.\end{proof}

\begin{cor}Continuous $K$-theory of dualizable categories commutes with filtered colimits.\end{cor}

Summarizing, we can formulate the following theorem.

\begin{theo}\label{th:equivalence_between_localizing} Let $\kappa$ be a regular cardinal and $\cE$ an accessible stable category with $\kappa$-filtered colimits. The precomposition functor
$$\Fun^{\acc}(\Cat_{\st}^{\dual},\cE)\xto{-\circ\Ind}\Fun^{\acc}(\Cat^{\perf},\cE)$$
induces an equivalence between the full subcategories $\Fun_{\loc,\kappa}(\Cat_{\st}^{\dual},\cE)$ and $\Fun_{\loc,\kappa}(\Cat^{\perf},\cE)$ of localizing invariants which commute with $\kappa$-filtered colimits. The inverse functor is given by $F\mapsto F^{\cont}.$\end{theo}

\begin{proof} By Propositions \ref{prop:F_cont_is_loc} and \ref{prop:F_cont_filtered_colimits}, the assignment $F\mapsto F^{\cont}$ is indeed a functor $\Fun_{\loc,\kappa}(\Cat^{\perf},\cE)\to \Fun_{\loc,\kappa}(\Cat_{\st}^{\dual},\cE).$ By Proposition \ref{prop:F_cont_of_cg}, it is a right inverse to the precomposition $-\circ\Ind.$

On the other hand, for any $G\in\Fun_{\loc}(\Cat_{\st}^{\dual},\cE),$ and any dualizable category $\cC,$ we have a short exact sequence 
$$0\to\cC\to\Ind(\cC^{\omega_1})\to\Ind(\Calk_{\omega_1}^{\cont}(\cC))\to 0,$$ which is functorial in $\cC.$ This gives an isomorphism $G\cong (G\circ\Ind)^{\cont}$ of functors $\Cat_{\st}^{\dual}\to \cE.$ This proves the theorem.\end{proof}

Further, we have two straightforward results about commutation of localizing invariants with sufficiently nice pullbacks and pushouts in $\Cat_{\st}^{\dual}.$

\begin{prop}\label{prop:F_cont_nice_pullbacks} Let $\Phi:\Cat^{\perf}\to\cE$ be an accessible localizing invariant. Let $\cA\xto{F}\cB\xlto{G}\cC$ be a pair of strongly continuous functors between dualizable categories, and assume that $F$ is a localization. Then the fiber product $\cA\times_{\cB}\cC,$ taken in $\Pr^L_{\st},$ is also the fiber product taken in $\Cat_{\st}^{\dual},$ and the natural map $\Phi^{\cont}(\cA\times_{\cB}\cC)\to \Phi^{\cont}(\cA)\times_{\Phi^{\cont}(\cB)}\Phi^{\cont}(\cC)$ is an isomorphism.\end{prop}

\begin{proof}The first assertion is the part of Proposition \ref{prop:nice_pullbacks}. Also by loc. cit. the functor $\cA\times_{\cB}\cC\to\cC$ is a localization. Clearly, the natural functor $\ker(\cA\times_{\cB}\cC\to\cC)\to \ker(\cA\to\cB)$ is an equivalence. Applying $\Phi^{\cont},$ we deduce the commutation of $\Phi^{\cont}$ with this fiber product.\end{proof}

\begin{prop}\label{prop:F_cont_nice_pushouts} Let $\Phi:\Cat^{\perf}\to\cE$ be an accessible localizing invariant. Let $\cA\xlto{F}\cB\xto{G}\cC$ be a pair of strongly continuous functors between dualizable categories, and assume that $F$ is fully faithful. Then the natural map $\Phi^{\cont}(\cA)\sqcup_{\Phi^{\cont}(\cB)}\Phi^{\cont}(\cC)\to \Phi^{\cont}(\cA\sqcup_{\cB}^{\cont}\cC)$ is an isomorphism.\end{prop}

\begin{proof}By Proposition \ref{prop:pushouts_of_mono}, the functor $\cC\to \cA\sqcup_{\cB}^{\cont}\cC$ is fully faithful. Clearly, the functor between the quotients $\cB/\cA\to (\cA\sqcup_{\cB}^{\cont}\cC)/\cC$ is an equivalence. Applying $\Phi^{\cont},$ we obtain the commutation of $\Phi^{\cont}$ with the pushout.\end{proof}

We observe that there is an alternative equivalent way to extend localizing invariants to dualizable categories.

\begin{prop}\label{prop:F_cont_hom_epi} Let $F:\Cat^{\perf}\to\cE$ be an accessible localizing invariant. Then for a dualizable category $\cC$ we have
$$F^{\cont}(\cC)\cong F(\cC^{\omega_1}\times_{\Calk_{\omega_1}^{\cont}(\cC)}\cC^{\omega_1}).$$\end{prop}

\begin{proof}It suffices to show that for any homological epimorphism $\cA\to\cB$ in $\Cat^{\perf}$ we have $F(\cA\times_{\cB}\cA)\simeq F(\cA)\times_{F(\cB)}F(\cA).$ This follows from Propositions \ref{prop:fiber_product_hom_epi} and \ref{prop:F_cont_nice_pullbacks}.\end{proof}

Finally, we mention the following straightforward result about infinite semi-orthogonal decompositions.

\begin{prop}\label{prop:additivity_in_infinite_SOD} Let $\cC$ be a dualizable category, $I$ a poset, and $\cC=\langle\cC_i;i\in I\rangle$ -- an $I$-indexed semi-orthogonal decomposition. Then the natural map $\biggplus[i\in I]\cU_{\loc}^{\cont}(\cC_i)\to \cU_{\loc}^{\cont}(\cC)$ is an isomorphism. Hence, the same holds for any finitary localizing invariant.\end{prop}

\begin{proof}As in the proof of Proposition \ref{prop:infinite_SOD_of_dualizable}, we have an ind-system of dualizable subcategories $\cC_J\subset\cC$ (with strongly continuous inclusions), where $J$ runs through finite subsets of $I,$ and $\cC_J$ is generated by $\cC_j,$ $j\in J.$ We have $\cC\simeq\indlim[J]^{\cont}\cC_J,$ and $\cC_J$ has a finite semi-orthogonal decomposition with components $\cC_j,$ $j\in J.$ We obtain the isomorphisms
\begin{equation*}\cU_{\loc}^{\cont}(\cC)\cong \indlim[J\subset I]\cU_{\loc}^{\cont}(\cC_J)\cong \indlim[J\subset I]\bigoplus_{j\in J}\cU_{\loc}^{\cont}(\cC_i)\cong \bigoplus_{i\in I} \cU_{\loc}^{\cont}(\cC_i).\qedhere\end{equation*}\end{proof}

\begin{remark}Note that the functor $\cU_{\loc,\kappa}^{\cont}:\Cat_{\st}^{\dual}\to\Mot^{\loc}_{\kappa}$ is the universal localizing invariant of dualizable categories which commutes with $\kappa$-filtered colimits. It follows formally that $\cU_{\loc,\kappa}^{\cont}$ is symmetric monoidal: we only need to know that for a dualizable category $\cC$ the functor $\cC\otimes-$ preserves short exact sequences and filtered colimits. An alternative way to argue is to use Theorem \ref{th:F_cont_as_right_Kan_extension} below which says that $\cU_{\loc,\kappa}^{\cont}$ is the right Kan extension of $\cU_{\loc,\kappa}.$\end{remark}

\subsection{Right Kan extension}
\label{ssec:F_cont_via_right_Kan_extension}

In this subsection we explain another equivalent way to extend localizing invariants from small stable categories to dualizable categories. We prove the following result.

\begin{theo}\label{th:F_cont_as_right_Kan_extension} Let $F:\Cat^{\perf}\to\cE$ be an accessible localizing invariant, where $\cE$ is an accessible stable category. Then the functor $F^{\cont}:\Cat_{\st}^{\dual}\to\cE$ is the right Kan extension of $F$ via $\Ind(-):\Cat^{\perf}\to\Cat_{\st}^{\dual}.$\end{theo}

We will again use the fact that the category $\Cat_{\st}^{\dual}$ is $\omega_1$-presentable (Ramzi's theorem), and also one of the characterizations of $\kappa$-compact objects in $\Cat_{\st}^{\dual}$ (for $\kappa>\omega$), given by Theorem \ref{th:presentability_of_Cat^dual}.

It is convenient to define the following topologies on the opposite categories $(\Cat^{\perf})^{op}$ and $(\Cat_{\st}^{\dual})^{op}.$  In both cases we define a pretopology for which a family of functors $\{\cA\to\cA_i\}_{i\in I}$ is a cover of $\cA$ if $I$ consists of one element $i$ and the functor $\cA\to\cA_i$ is fully faithful. To deal with set-theoretic issues we also consider the induced pretopologies on the small subcategories $(\Cat^{\perf,\kappa})^{op}$ and $(\Cat_{\st}^{\dual,\kappa})^{op}$ (note that these subcategories are closed under finite limits), where $\kappa$ is an uncountable regular cardinal. In each case we denote by $\tau$ the corresponding Grothendieck topology. Note that in each case $\tau$ is a subcanonical topology.

Although the category $\cE$ is not necessarily complete (it is only assumed to be accessible and stable), we still have a notion of an $\cE$-valued sheaf. Namely, the sheaf condition only requires the existence of the relevant limits, not all limits. We make the following observation.

\begin{prop}\label{prop:loc_invar_is_a_sheaf} Let $F:\Cat^{\perf}\to\cE$ be an accessible localizing invariant. Then $F$ and $F^{\cont}$ are sheaves for the $\tau$-topology.\end{prop}

\begin{proof}It suffices to prove this for $F^{\cont}.$ Let $\Phi:\cC\to\cD$ be a fully faithful strongly continuous functor between dualizable categories. Consider the \v{C}ech conerve $\cD^{\bullet}$ of $\Phi,$ i.e. $\cD^n=\cD\sqcup_{\cC}^{\cont}\dots\sqcup_{\cC}^{\cont}\cD$ -- the cofiber product of $n+1$ copies of $\cD.$ It follows from Proposition \ref{prop:F_cont_nice_pushouts} that $F^{\cont}(\cD^{\bullet})$ is the \v{C}ech conerve of $F^{\cont}(\cC)\to F^{\cont}(\cD).$ Since $\cE$ is stable, we conclude that $F^{\cont}(\cC)\cong \Tot(F^{\cont}(\cD^{\bullet})).$
More precisely, for any morphism $x\to y$ in $\cE,$ its \v{C}ech conerve $y^{\bullet}$ has the following property: for $n\geq 1$ the map $x\to\Tot_n(y^{\bullet})$ is an isomorphism.
\end{proof}




\begin{proof}[Proof of Theorem \ref{th:F_cont_as_right_Kan_extension}.] Choose an uncountable regular cardinal $\kappa$ such that $\cE$ is $\kappa$-accessible, and $F$ and $F^{\cont}$ commute with $\kappa$-filtered colimits. It suffices to show that the restriction of $F^{\cont}$ to $\Cat_{\st}^{\dual,\kappa}$ is the right Kan extension of the restriction of $F$ to $\Cat^{\perf,\kappa}.$ 

Let $\cC$ be a $\kappa$-compact dualizable category. By Theorem \ref{th:presentability_of_Cat^dual}, we can find a fully faithful  strongly continuous functor $\cC\to\cD,$ such that $\cD$ is compactly generated and $\cD^{\omega}$ is $\kappa$-compact in $\Cat^{\perf}.$ Denote by $\cD^{\bullet}$ the \v{C}ech conerve of $\cC\to\cD$ in $\Cat_{\st}^{\dual}.$ Then each $\cD^n$ is a compactly generated category, and $(\cD^n)^{\omega}$ is $\kappa$-compact.

Now, the functor $\Cat^{\perf,\kappa}\to\cS,$ $\cA\mapsto \Fun^{LL}(\cC,\Ind(\cA))^{\simeq},$ is a sheaf of spaces, covered by the representable sheaf $h_{\cD^{\omega}},$ with the \v{C}ech nerve $h_{(\cD^{\bullet})^{\omega}}.$ By Proposition  \ref{prop:loc_invar_is_a_sheaf}, the functors $F$ and $F^{\cont}$ are sheaves with values in $\cE.$ It follows that the right Kan extension of $F$ from $\Cat^{\perf,\kappa}$ to $\Cat_{\st}^{\dual,\kappa}$ exists, and its value on $\cC$ is given by
\begin{equation*} \Tot(F((\cD^{\bullet})^{\omega}))\cong \Tot(F^{\cont}(\cD^{\bullet}))\cong F^{\cont}(\cC).\qedhere\end{equation*}
\end{proof}

\subsection{Example: sheaves on the real line}
\label{ssec:sheaves_on_R}

As an example, we will compute the finitary localizing invariants of the category of sheaves on the real line. Let $\cC$ be a dualizable category. Recall that the category $\Sh(\R;\cC)$ is dualizable \cite{Lur18}. We will give a detailed proof of the following special case of Theorem \ref{th:main_intro}.

\begin{prop}\label{prop:U_loc_sheaves_on_R} For a dualizable category $\cC$ we have
$$\cU_{\loc}^{\cont}(\Sh(\R;\cC))\cong\Omega \cU_{\loc}^{\cont}(\cC), \quad \cU_{\loc}^{\cont}(\Sh(\R\cup\{-\infty\};\cC))=0.$$ Hence, the same holds for any finitary localizing invariant, for example, for $K$-theory.
\end{prop}

\begin{example}Let $\mk$ be a field, and consider the dualizable category $\cC=\Sh(\R;D(\mk)).$ Combining Propositions \ref{prop:U_loc_sheaves_on_R} and \ref{prop:F_cont_hom_epi}, we see that $$K_0^{\cont}(\cC)\cong K_0(\cC^{\omega_1}\times_{\Calk_{\omega_1}^{\cont}(\cC)}\cC^{\omega_1})\cong\mk^{\ast}.$$ Given $\lambda\in \mk^{\ast},$ we can describe its representative in the fiber product as follows. Consider the object $\mk_{\R}\in\cC^{\omega_1}.$ The endomorphism algebra of its image in the Calkin category is $\mk\times\mk$ (see the proof of Proposition \ref{prop:image_not_a_stable_subcat}). The object $(\mk_{\R},\mk_{\R},(\lambda,1))\in \cC^{\omega_1}\times_{\Calk_{\omega_1}^{\cont}(\cC)}\cC^{\omega_1}$ represents the class $\lambda\in K_0^{\cont}(\cC).$\end{example}


In Sections \ref{sec:presentability_Cat_dual} and \ref{sec:Urysohn} we already considered the category $\Shv_{\geq 0}(\R;\Sp),$ which was defined to be the continuous stabilization of the poset $(\R\cup\{+\infty\})_{\leq}$ (considered as a compactly assembled category). Here we give its interpretation via sheaves, and also consider the version with more general coefficients. Recall that if $\cC$ is compactly generated then we have a notion of a singular support for $\cC$-valued sheaves on $C^1$-manifolds \cite[Definition 4.5]{RS}. In fact there is a reasonable notion of singular support for sheaves with values in arbitrary presentable stable categories, see Remarks \ref{rem:SS_general} and \ref{rem:constant_functor_criterion}. We will only consider sheaves on the real line whose singular support is contained in $\R\times\R_{\geq 0}\subset T^*\R$ (and similarly for $\R\times\R_{\leq 0}$). We first give a direct description of the category of such sheaves, and the comparison with the standard definition is formulated in Proposition \ref{prop:description_of_SS_on_the_real_line} below (it is due to Kashiwara and Schapira). 

Denote by $\gamma=\R_{\leq 0}\subset \R$ the cone of non-positive real numbers. The $\gamma$-topology on $\R$ is defined as follows: a subset $U\subset\R$ is $\gamma$-open if $U$ is open and $U+\gamma=U.$ In other words, $U$ is $\gamma$-open if either $U=\emptyset$ or $U=\R$ or $U=(-\infty,a)$ for some $a\in\R.$ We denote by $\R_{\gamma}$ the set $\R$ equipped with $\gamma$-topology. Denote by $\varphi_{\gamma}:\R\to\R_{\gamma}$ the identity map of $\R$ considered as a map of topological spaces. Similarly, we denote by $\varphi_{-\gamma}:\R\to\R_{-\gamma}$ the map to $\R$ equipped with $(-\gamma)$-topology. 

{\noindent{\bf Notation.}} {\it For a presentable stable category $\cC,$ we denote by $\Sh_{\geq 0}(\R;\cC)$ resp. $\Sh_{\leq 0}(\R;\cC)$ the category $\Sh(\R_{\gamma};\cC)$ resp.  $\Sh(\R_{-\gamma};\cC)$ of $\cC$-valued sheaves on $\R_{\gamma}$ resp. $\R_{-\gamma}.$}

Note that for $\cC=\Sp$ we indeed have a tautological equivalence $\Sh_{\geq 0}(\R;\Sp)\simeq \Stab^{\cont}((\R\cup\{+\infty\})_{\leq}),$ where the latter category was defined in Subsection \ref{ssec:cont_stab}. 
In particular, by Proposition \ref{prop:continuous_stab} the category $\Sh_{\geq 0}(\R;\Sp)$ is dualizable.

Moreover, by Remark \ref{rem:Stab_cont_otimes_something} for any presentable stable category $\cC$ we have 
\begin{equation}\label{eq:sheaves_otimes_something}\Sh_{\geq 0}(\R;\cC)\simeq\Sh_{\geq 0}(\R;\Sp)\otimes\cC.\end{equation}
In particular, if $\cC$ is dualizable, then so is the category $\Sh_{\geq 0}(\R;\cC).$ 

Note that for a sheaf $\cF\in\Sh(\R_{\gamma};\cC)$ the stalk at $a\in\R$ has a very simple description: we have $\cF_a=\indlim[b>a]\cF((-\infty,b)).$ In particular, for $a'<a$ we have a natural map $\cF_a\to\cF_{a'};$ formally this map comes from the fact that $a$ is a specialization of $a'$ in $\R_{\gamma}.$ We obtain a functor $\Psi:\Shv_{\geq 0}(\R;\cC)\to \Fun(\R_{\leq}^{op},\cC),$ $\Psi(\cF)(a)=\cF_a.$ This functor is in fact fully faithful, which is part of Proposition \ref{prop:resolution_of_Sh_geq_0} below.

In the formulation of the next result we make a choice of a dense subset $A\subset \R.$ This choice is not relevant for part 2) and for the proof of Proposition \ref{prop:U_loc_sheaves_on_R}, but in the next subsection (proof of Proposition \ref{prop:all_loc_invar_zero_for_Sh_geq_0}) it will be useful to choose $A$ to be countable.    



\begin{prop}\label{prop:resolution_of_Sh_geq_0} Let $\cC$ be a dualizable category, and choose some dense subset $A\subset\R.$ Denote by $A_{\leq}$ the set $A$ considered as a linearly ordered set with the usual order. Consider the composition $\Psi_A:\Shv_{\geq 0}(\R;\cC)\xto{\Psi}\Fun(\R_{\leq}^{op},\cC)\to\Fun(A_{\leq}^{op},\cC),$ so we have $\Psi_A(\cF)(a)=\cF_a.$

1) We have a short exact sequence in $\Cat_{\st}^{\dual}:$
\begin{equation}\label{eq:ses_for_Sh_geq_0} 0\to\Sh_{\geq 0}(\R;\cC)\xto{\Psi_A}\Fun(A_{\leq}^{op},\cC)\xto{\Phi}\prodd[A]\cC\to 0.\end{equation}
Here the functor $\Phi$ is given by $\Phi(F)_a=\Cone(\indlim[b>a]F(b)\to F(a)).$

2) We have $\cU_{\loc}^{\cont}(\Sh_{\geq 0}(\R;\cC))=0.$\end{prop}


\begin{proof} 1) We first check that $\Phi$ is indeed a strongly continuous localization. First, for $a\in A,$ define the functors 
\begin{equation}\label{eq:functors_P_a} P_a:\cC\to \Fun(A_{\leq}^{op},\cC),\quad P_a(x)(b)=\begin{cases}x & \text{ for }b\leq a;\\
0 & \text{ for }b>a,\end{cases}\end{equation}
$$I_a:\cC\to \Fun(A_{\leq}^{op},\cC),\quad I_a(x)(b)=\begin{cases}x & \text{ for }b\geq a;\\
0 & \text{ for }b<a.\end{cases}$$ 
The right adjoint to $P_a$ is given by $F\mapsto F(a),$ which is also the left adjoint to $I_a.$ It follows that for $a\in A$ the right adjoint to the functor $F\mapsto \Cone(\indlim[b>a]F(b)\to F(a))$ sends $x\in\cC$ to the functor $G_a(x):A_{\leq}^{op}\to\cC,$ given by 
$$G_a(x)=\Fiber(I_a(x)\to\limproj\limits_{b>a}I_b(x)),\quad G_a(x)(b)=\begin{cases}x & \text{ for }b=a;\\
0 & \text{ for }b\ne a.\end{cases}$$
Here the latter description of $G_a(x)$ follows from the fact that $A_{\leq}$ is a dense linearly ordered set. Now the right adjoint to $\Phi$ is given by \begin{equation}\label{eq:formula_for_Phi^R}\Phi^R((x_a)_a)=\prodd[a\in\R]G_a(x_a)=\bigoplus\limits_{a\in\R}G_a(x_a).\end{equation}
Thus, $\Phi^R((x_a)_a)(b)=x_b,$ and the maps $\Phi^R(x)(b)\to\Phi^R(x)(a)$ are zero for $a<b.$ In particular, $\Phi$ is strongly continuous and the adjunction counit $\Phi\circ\Phi^R\to\Id$ is an isomorphism, so $\Phi$ is a localization.

Now, we observe that the category $\Sh_{\geq 0}(\R;\cC)$ is naturally equivalent to the right orthogonal $\im(\Phi^R)^{\perp}\subset\Fun(A_{\leq}^{op},\cC).$ This equivalence sends a sheaf $\cF$ to the functor $a\mapsto \cF((-\infty,a)).$ Note that for $a\in A,$ $x\in\cC,$ we have $G_a(x)\cong\Cone(\indlim[b<a]P_b(x)\to P_a(x)),$ hence $$\Hom(G_a(x),F)=\Hom_{\cC}(x,\Fiber(F(a)\to\prolim[b<a]F(b))).$$ Using \eqref{eq:formula_for_Phi^R}, we see that the category $\im(\Phi^R)^{\perp}$ consists of functors $F:A_{\leq}^{op}\to \cC$ such that for any $a\in\R$ the map $F(a)\to\prolim[b<a]F(b)$ is an isomorphism. This corresponds exactly to the fact that $\cF$ is a sheaf on $\R_{\gamma}.$

Since the functor $\Phi^R$ has a right adjoint, we have equivalences
\begin{equation}\label{eq:right_and_left_orthogonals_to_image_of_Phi^R}
\Sh_{\geq 0}(\R;\cC)\xto{\sim} \im(\Phi^R)^{\perp}\xto{\sim} ^{\perp}\im(\Phi^R)=\ker(\Phi).
\end{equation}

Now the right adjoint to the inclusion $^{\perp}\im(\Phi^R)\to \Fun(A_{\leq}^{op},\cC)$ sends $F$ to the functor $a\mapsto \indlim[b>a]F(b).$ Hence, the composition of equivalences \eqref{eq:right_and_left_orthogonals_to_image_of_Phi^R} is identified with the functor $\Psi_A:\Sh_{\geq 0}(\R;\cC)\to \ker(\Phi)\subset \Fun(A_{\leq}^{op},\cC).$ This proves that we have a short exact sequence \eqref{eq:ses_for_Sh_geq_0} in $\Cat_{\st}^{\dual}.$ 
This proves part 1). 

2) By \eqref{eq:sheaves_otimes_something}, it is sufficient to prove the vanishing for $\cC=\Sp.$ Take $A=\R$ and consider the short exact sequence \eqref{eq:ses_for_Sh_geq_0}. We need to show that the map $\cU_{\loc}^{\cont}(\Phi)$ is an equivalence in $\Mot^{\loc}.$ 

Note that the categories $\cA=\Fun(\R_{\leq}^{op},\Sp)$ and $\cB=\prodd[\R]\Sp$ are compactly generated. We have an $\R$-indexed semi-orthogonal decomposition $\cA^{\omega}=\la P_a(\Sp^{\omega}), a\in\R\ra,$ where the (fully faithful, strongly continuous) functors $P_a$ are defined in \eqref{eq:functors_P_a}. Also, we have a tautological $\R$-indexed orthogonal decomposition of the category $\cB^{\omega}\simeq\biggplus[\R]\Sp^{\omega}.$ The functor $\Phi^{\omega}:\cA^{\omega}\to\cB^{\omega}$ is compatible with these decompositions and it induces the identity functors on the components. Hence, the map $\cU_{\loc}^{\cont}(\Phi)=\cU_{\loc}(\Phi^{\omega})$ is an isomorphism.
\end{proof}

We mention the following result which is due to Kashiwara and Schapira. We use the following notation: for a $C^1$-manifold $M,$ a conical closed subset $L\subset T^* M$ and a presentable stable compactly generated category $\cC$ we denote by $\Sh_L(M;\cC)\subset\Sh(M;\cC)$ the full subcategory of sheaves whose singular support is contained in $L.$

\begin{prop}\label{prop:description_of_SS_on_the_real_line}\cite{KS18} Let $\cC$ be a presentable stable compactly generated category. The pullback functor $\varphi_{\gamma}^*:\Sh(\R_{\gamma};\cC)\to\Sh(\R;\cC)$ is fully faithful and its essential image is exactly the subcategory $\Sh_{\R\times\R_{\geq 0}}(\R;\cC).$ In particular, we have $\Sh_{\R\times\R_{\geq 0}}(\R;\cC)\simeq \Sh(\R_{\gamma};\cC).$
\end{prop}

\begin{proof}This is a straightforward generalization of the special case of \cite[Theorem 2.5]{KS18}.\end{proof}

In the following proof we will need only the fully faithfulness of $\varphi_{\gamma}^*$ and $\varphi_{-\gamma}^*.$ 


\begin{proof}[Proof of Proposition \ref{prop:U_loc_sheaves_on_R}]
Again, it suffices to prove the result for $\cC=\Sp.$ We have a short exact sequence in $\Cat_{\st}^{\dual}$

$$0\to \Sh(\R;\Sp)\xto{j_!}\Sh(\R\cup\{-\infty\};\Sp)\to\Sp\to 0,$$
where the $j:\R\to\R\cup\{-\infty\}$ is the inclusion. Hence, it suffices to prove the result for sheaves on $\R\cup\{-\infty\}.$

We claim that there is a natural short exact sequence in $\Cat_{\st}^{\dual}:$
\begin{equation}\label{eq:extension_of_SS_categories}0\to\Sh_{\geq 0}(\R;\Sp)\xto{\Phi}\Sh(\R\cup\{-\infty\};\Sp)\xto{\Theta} \Sh_{\leq 0}(\R;\Sp)\to 0.\end{equation}
Assuming this, the result follows from Proposition \ref{prop:resolution_of_Sh_geq_0}.

The functor $\Theta$ in \eqref{eq:extension_of_SS_categories} is the left adjoint to $j_!\circ\varphi_{-\gamma}^*.$ The functor $\Phi$ is given by $\bar{\varphi_{\gamma}^*},$ where $\bar{\varphi}_{\gamma}:\R\cup\{-\infty\}\to\bar{\R}_{\gamma}$ is the map (extending $\varphi_{\gamma}$) to the sobrification of $\R_{\gamma}.$ Note that as a set $\bar{\R}_{\gamma}$ is naturally identified with $\R\cup\{-\infty\}$ (where $-\infty$ is the generic point), and clearly $\Sh(\R_{\gamma};\Sp)\simeq \Sh(\bar{\R}_{\gamma};\Sp)$

An explicit description of the functors $\Phi$ and $\Theta$ is the following. For a convex open subset $U\subset\R\cup\{-\infty\}$ and for $a\in\R$ we have
$$\Phi(\cF)(U)=\cF((U\cap\R)+\gamma), \quad \Phi^R(\cG)((-\infty,a))=\cG([-\infty,a)),$$
$$\Theta(\cG)((a,+\infty))=\Gamma_{[-\infty,a]}(\R\cup\{-\infty\},\cG)[1],\quad \Theta^R(\cF)(U)=\begin{cases}\cF(U-\gamma) & \text{if }-\infty\not\in U;\\
0 & \text{if }-\infty\in U.\end{cases}$$
With these descriptions it is straightforward to check that \eqref{eq:extension_of_SS_categories} is indeed a short exact sequence. 
\end{proof}

A reader familiar with singular supports can easily describe the functors in \eqref{eq:extension_of_SS_categories} via objects on the product $\R\times\R,$ we leave it as an exercise. 



\begin{remark}It follows from Proposition \ref{prop:all_loc_invar_zero_for_Sh_geq_0} below and from the above proof of Proposition \ref{prop:U_loc_sheaves_on_R} that the result holds for all accessible localizing invariants, not necessarily finitary. We use this to prove Theorem \ref{th:sheaves_on_finite_CW_complexes} on localizing invariants of categories of sheaves on finite CW complexes.\end{remark}

\begin{remark}\label{rem:SS_general} If $X$ is a $C^1$-manifold and $\cC$ is a compactly generated presentable stable category, then the singular support of a $\cC$-valued sheaf on $X$ is defined in \cite[Definition 4.5]{RS}, which is a straightforward generalization of \cite[Definition 4.4.1]{KS}. However, if $\cC$ is a dualizable category which is not compactly generated, then already in the $1$-dimensional case the non-characteristic deformation lemma \cite[Theorem 4.1]{RS} fails and the conditions of \cite[Proposition 5.1.1]{KS} are no longer equivalent. 

To illustrate this, let $\cC=\Sh(\R;\Sp).$ 
Consider a $\cC$-valued sheaf $\cF$ on $\R$ given by $\cF((a,b))=\bS_{(b,+\infty)},$ for $a<b.$ Here $\bS_{(b,+\infty)}$ is the extension by zero of the constant sheaf on $(b,+\infty)$ with value $\bS.$ Then according to \cite[Definition 4.5]{RS} the singular support $SS(\cF)$ would be the zero section $T^*_{\R}\R\subset T^*\R.$ But clearly the sheaf $\cF$ is not locally constant, hence the definition of the singular support has to be modified. 

We mention several equivalent ways to define the singular support for a sheaf $\cF$ on a $C^1$-manifold $X$ with values in a presentable stable category $\cC.$

1) A point $(x_0,\xi_0)\in T^*X$ is not contained in $SS(\cF)$ if and only if for some (any) choice of local coordinates it satisfies the equivalent conditions $(2)$ and $(3)$ of \cite[Proposition 5.1.1]{KS}.

2) The singular support $SS(\cF)\subset T^*X$ is the smallest conical closed subset $L\subset T^*X$ such that $\cF$ is contained in the essential image of the functor $$\Sh_L(X;\Sp)\otimes\cC\to \Sh(X;\Sp)\otimes\cC\simeq\Sh(X;\cC).$$ This functor is fully faithful since the inclusion $\Sh_L(X;\Sp)\to \Sh(X;\Sp)$ has a left adjoint.

3) The singular support of $\cF$ is the support of the conical sheaf (of spectra) $\mu hom(\cF,\cF)$ on $T^*X$ (microlocal Hom), where the bifunctor $\mu hom$ is defined as in \cite[Definition 4.4.1]{KS}.

4) Choose a regular cardinal $\kappa$ such that $\cF$ takes values in $\cC^{\kappa}.$ Then the composition $\Open(X)^{op}\xto{\cF}\cC^{\kappa}\xto{\cY}\Ind(\cC^{\kappa})$ defines a sheaf $\cF'$ with values in the compactly generated category $\Ind(\cC^{\kappa})$ (since the Yoneda embedding $\cY:\cC^{\kappa}\to\Ind(\cC^{\kappa})$ commutes with existing limits). We put $\SS(\cF):=\SS(\cF'),$ where the latter singular support is defined as in \cite[Definition 4.5]{RS}.

5) If $\cC$ is dualizable, we can define the singular support of $\cF$ to be the singular support of the sheafification of the presheaf $\Open(X)^{op}\xto{\cF}\cC\xto{\hat{\cY}}\Ind(\cC^{\omega_1})$ where the latter is defined as in \cite[Definition 4.5]{RS}.
\end{remark}

\begin{remark}\label{rem:constant_functor_criterion} Essentially the above issue with the definition of the singular support reduces to the fact that Kashiwara's criterion for a functor $\R_{\leq}\to\cC$ to be constant fails when $\cC$ is not compactly generated (even if we require that $\cC$ is dualizable). Namely, if $\cC$ is compactly generated and $F:\R_{\leq}\to\cC$ is a functor, then $F$ is constant if and only if we have
\begin{equation}\label{eq:constant_functor}\indlim[b<a]F(b)\xto{\sim}F(a)\xto{\sim}\prolim[b>a]F(b),\quad\text{ for any }a\in\R,\end{equation}
see \cite[Corollary 3.2]{RS}.
However, if $\cC=\Sh(\R;\Sp)$ and $F:\R_{\leq}\to\cC$ is given by $F(a)=\bS_{(-\infty,a)},$ then the condition \eqref{eq:constant_functor} holds but $F$ is not constant.\end{remark}

\subsection{Commutation of K-theory with infinite products}
\label{ssec:K_theory_of_products}

Recall that by \cite{KW19} $K$-theory commutes with infinite products of small stable Karoubi-complete categories. Their proof is based on the Grayson's construction \cite{Gra12} (more precisely, its version for stable $\infty$-categories) which roughly speaking allows to identify $K_n$ of some category with a direct summand of $K_{n-1}$ of another category. We will give an alternative (and in our opinion more conceptual) proof of the commutation of $K$-theory with products, using the category of sheaves on the real line instead of the Grayson's construction.

\begin{remark}While the paper was being written, Vladimir Sosnilo explained to me a considerable simplification of the proof from \cite{KW19}. In particular, his argument allows to prove Theorem \ref{th:map_of_localizing_invariants} below (except for the assertion about $K^{\cont}$) without using dualizable categories at all.\end{remark}

We first recall the argument showing the commutation of $K_0$ of small triangulated categories with infinite products.

\begin{prop}(Heller's criterion) Let $T$ be a small triangulated category, and take two objects $x,y\in T.$ The following are equivalent:

\begin{enumerate}[label=(\roman*),ref=(\roman*)]
\item we have $[x]=[y]$ in $K_0(T);$

\item there exist objects $z,u,v\in T$ and exact triangles
$$u\to x\oplus z\to v\to u[1],$$
$$u\to y\oplus z\to v\to u[1].$$
\end{enumerate}\end{prop}

\begin{proof}
The proof is completely analogous to the proof of \cite[Lemma 3.4]{KW20}.
\end{proof}

\begin{cor}\label{cor:K_0_comm_with_products}\cite[Lemma 5.1]{KW19} For a collection of small triangulated categories $T_i,$ $i\in I,$ we have $K_0(\prodd[i] T_i)\cong \prodd[i] K_0(T_i).$\end{cor}

We will prove the following general result. Recall that a $t$-structure $(\cE_{\geq 0},\cE_{\leq 0})$ on a stable category $\cE$ is called {\it non-degenerate} if it is both left and right separated, i.e. $\bigcap\limits_n \cE_{\geq n}=0,$ $\bigcap\limits_n \cE_{\leq n}=0.$ Equivalently this means that for any non-zero object $x\in\cE$ there exists $n\in\Z$ such that $\pi_n(x)\ne 0.$

\begin{theo}\label{th:map_of_localizing_invariants} 1) Let $\cE$ be an accessible stable category with a non-degenerate $t$-structure. Let $F,G:\Cat^{\perf}\to\cE$ be accessible localizing invariants and $\varphi:F\to G$ a morphism such that $\pi_0(\varphi):\pi_0(F)\to\pi_0(G)$ is an isomorphism. Then $\varphi$ is an isomorphism.

2) In particular, both $K$-theory $K:\Cat^{\perf}\to\Sp$ and continuous $K$-theory $K^{\cont}:\Cat_{\st}^{\dual}\to\Sp$ commute with infinite products.\end{theo}

\begin{proof}[Proof that 1)$\Rightarrow$2).] By Proposition \ref{prop:products_in_Cat^dual} it suffices to prove the statement for continuous $K$-theory. Consider a collection $\{\cC_i\}_{i\in I}$ of dualizable categories. We need to show that $K^{\cont}(\prodd[i]^{\dual}\cC_i)\xto{\sim}\prodd[i] K^{\cont}(\cC_i).$ We may assume that the categories $\cC_i$ are all equivalent: indeed, each $\cC_i$ is a retract in $\Cat_{\st}^{\dual}$ of the category $\prod_j\cC_j.$ 

In other words, we need to show that for a set $I$ and a dualizable category $\cC$ the map $\varphi_{\cC}:K^{\cont}(\prod_I^{\dual}\cC)\to\prod_I K^{\cont}(\cC)$ is an isomorphism. But this is the special case of 1) for the standard $t$-structure on $\Sp.$ Indeed, by Proposition \ref{prop:AB4*_for_Cat^dual} the functor $\cC\mapsto K^{\cont}(\prod_I^{\dual}\cC)$ is a localizing invariant, and so is the functor $\cC\mapsto \prod_I K^{\cont}(\cC).$ When $\cC$ is compactly generated, the map $\pi_0(\varphi_{\cC})$ is an isomorphism by Corollary \ref{cor:K_0_comm_with_products}.\end{proof}

To prove part 1), we need the following result on $K$-equivalences.

\begin{prop}\label{prop:complete_SOD} Let $\cA$ be a small stable idempotent-complete category with an $\N$-indexed semi-orthogonal decomposition
$\cA=\langle\cA_0,\cA_1,\dots\rangle.$ Denote by $\cB_n\subset\cA$ the full stable subcategory generated by $\cA_0,\cA_1,\dots,\cA_n,$ and denote by $p_n:\cB_n\to\cA_n$ the right adjoint to the inclusion functor. Consider the sequence $(\cB_n)_{n\geq 0}$ as an inverse sequence where the transition functors $\cB_{n+1}\to\cB_n$ are left adjoint to the inclusions. Then the compositions $\pi_n:\prolim[k]\cB_k\to\cB_n\xto{p_n}\cA_n$ define a $K$-equivalence $F:\prolim[n]\cB_n\to\prodd[n\in\N]\cA_n.$ 
\end{prop}

\begin{proof}Put $\cB:=\prolim[n]\cB_n.$ 
For $n\geq 0$ denote by $U_n:\cA_n\to\cB_n$ the inclusion functor. We denote by the same symbol the compositions $\cA_n\xto{U_n}\cB_n\hto \cB_k$ for $k\geq n.$ For $k\geq 0$ consider the functor
\begin{equation*}
G_k:\prodd[n\in\N]\cA_n\to\cB_k,\quad G_k((x_n)_{n\geq 0}) = \biggplus[0\leq n\leq k]U_n(x_n).
\end{equation*} Then the functors $G_k$ are compatible with the functors $\cB_{k+1}\to\cB_k,$ hence we obtain a functor $G:\prodd[n\in\N]\cA_n\to\cB.$ By construction, we have $F\circ G\simeq\Id.$ We need to show that $[G\circ F]=[\Id]$ in $K_0(\Fun(\cB,\cB)).$

For each $n\geq 0$ the inclusions $\cB_n\to\cB_k,$ $k\geq n,$ define a fully faithful functor $\cB_n\hto\cB.$ Its left adjoint is simply the structural functor $\cB\to\cB_n.$ Denote by $\iota_n:\cA_n\xto{U_n}\cB_n\to\cB$ the composition functors, $n\geq 0.$ Since the compositions $\cA_n\xto{\iota_n}\cB\to\cB_k$ vanish for $k<n,$ we see that the direct sum of the functors $\iota_n\circ \pi_n$ is well-defined. Moreover, we have $\biggplus[n\geq 0]\iota_n\circ \pi_n\cong G\circ F.$ 

Now, denote by $\Psi_n:\cB\to\cB_n\to\cB$ the composition, and put $\Phi_n:=\Fiber(\Id\to \Psi_{n-1})$ (where $\Psi_{-1}=0$). Equivalently, $\Phi_n:\cB\to\cB$ is the semi-orthogonal projection onto the left orthogonal to $\cB_{n-1}$ (where $\cB_{-1}=0$). Denote by $u_n:\Phi_{n+1}\to\Phi_n$ the natural transformations, $n\geq 0.$ Then $\Cone(u_n)\cong \iota_n\circ \pi_n.$

Since the composition $\cB\xto{\Phi_n}\cB\to\cB_k$ vanishes for $k<n,$ we see that the direct sum $\biggplus[n\geq 0]\Phi_n$ is also a well-defined endofunctor of $\cB.$ By the above, we obtain a n cofiber sequence
$$\biggplus[n\geq 1]\Phi_n\xto{(u_n)_{n\geq 0}}\biggplus[n\geq 0]\Phi_n\to G\circ F\quad\text{in }\Fun(\cB,\cB).$$
We conclude that
$$[\Id]=[\Phi_0]=[\biggplus[n\geq 0]\Phi_n]-[\biggplus[n\geq 1]\Phi_n]=[G\circ F]$$
in $K_0(\Fun(\cB,\cB)),$ as required.
\end{proof}

We deduce the following statement about commutation of $K_0$ with certain ``extremely nice'' sequential limits.

\begin{cor}\label{cor:K_0_of_nice_limits} Let $\cB_0\xlto{F_0}\cB_1\xlto{F_1}\dots$ be an inverse sequence in $\Cat^{\perf}$ such that each functor $F_n$ has a fully faithful right adjoint $F_n^R.$ Then we have an isomorphism of abelian groups $K_0(\prolim[n]\cB_n)\xto{\sim}\prolim[n]K_0(\cB_n).$\end{cor}

\begin{proof}Denote by $\cA_n\subset\cB_n$ the kernel of $F_{n-1}$ (i.e. the left orthogonal to the image of $F_{n-1}^R$), $n\geq 0,$ where we put $\cB_{-1}=0,$ $F_{-1}=0.$ It follows from Proposition \ref{prop:complete_SOD} and Corollary \ref{cor:K_0_comm_with_products} that
\begin{equation*}K_0(\prolim[n]\cB_n)\cong K_0(\prodd[n\geq 0]\cA_n)\cong \prodd[n\geq 0]K_0(\cA_n)\cong\prolim[n]K_0(\cB_n).\qedhere\end{equation*}\end{proof}

We deduce that {\it all} accessible localizing invariants vanish for the categories of the form $\Sh_{\geq 0}(\R;\cC).$

\begin{prop}\label{prop:all_loc_invar_zero_for_Sh_geq_0} Let $\cC$ be a dualizable category. Then for any regular cardinal $\kappa$ we have
$\cU_{\loc,\kappa}^{\cont}(\Sh_{\geq 0}(\R;\cC))=0.$\end{prop}

\begin{proof}Since $\Sh_{\geq 0}(\R;\cC)\simeq \Sh_{\geq 0}(\R;\Sp)\otimes\cC,$ we may assume that $\cC=\Sp.$ Applying  Proposition \ref{prop:resolution_of_Sh_geq_0} to $A=\Q,$ we obtain a short exact sequence
$$0\to\Sh_{\geq 0}(\R;\Sp)\to\Fun(\Q_{\leq}^{op},\Sp)\xto{\Phi}\prodd[\Q]\Sp\to 0$$
in $\Cat_{\st}^{\dual}.$ Moreover, the source and the target of $\Phi$ are compactly generated categories. Denote by $\Phi^{\omega}$ the induced functor on the subcategories of compact objects. It suffices to show that $\Phi^{\omega}:\Fun(\Q_{\leq},\Sp)^{\omega}\to (\prodd[\Q]\Sp)^{\omega}\simeq\biggplus[\Q]\Sp^{\omega}$ is a $K$-equivalence. As in the proof of proposition \ref{prop:resolution_of_Sh_geq_0}, we put $\cA=\Fun(\Q_{\leq},\Sp)$ and $\cB=\prodd[\Q]\Sp.$

Note that the functor $\Phi^{\omega}:\cA^{\omega}\to\cB^{\omega}$ has an obvious section $\Psi:\cB^{\omega}\simeq \biggplus[\Q]\Sp^{\omega}\to \cA,$ which corresponds to the collection of ``representable'' functors $P_a(\bS),$ $a\in\Q,$ where $P_a$ is defined in \eqref{eq:functors_P_a}. We need to show that $[\Psi\circ\Phi^{\omega}]=[\Id]$ in $K_0(\Fun(\cA^{\omega},\cA^{\omega})).$

Choose any bijection $\N\xto{\sim}\Q,$ $n\mapsto a_n.$ Denote by $\cC_n\subset \cA^{\omega}$ the full stable subcategory generated by the objects $P_{a_i}(\bS),$ $0\leq i\leq n.$ Then $\cC_n$ is equivalent to the category $\Fun([n],\Sp^{\omega}),$ where as usual $[n]$ denotes the poset $\{0,1,\dots,n\}.$ Clearly, each inclusion $\cC_n\to\cC_{n+1}$ has both left and right adjoints. Moreover, we see that the elements $[P_a(\bS)],$ $a\in\Q,$ form a basis of $K_0(\cA^{\omega})\cong\biggplus[\Q]\Z.$

It follows from Corollary \ref{cor:K_0_of_nice_limits} that $K_0(\Fun(\cA^{\omega},\cA^{\omega}))\cong\prolim[n]K_0(\Fun(\cC_n,\cA^{\omega})).$ Since $\Fun(\cC_n,\cA^{\omega})\simeq\Fun([n],\cA^{\omega}),$ we deduce that 
$$K_0(\Fun(\cA^{\omega},\cA^{\omega}))\cong\Hom_{\Z}(K_0(\cA^{\omega}),K_0(\cA^{\omega}))=\End_{\Z}(\biggplus[\Q]\Z).$$ It remains to observe that the functor $\Psi\circ\Phi^{\omega}$ induces the identity map on $K_0(\cA^{\omega}):$ it sends each object $P_a(\bS)$ to itself, $a\in\Q.$\end{proof}


\begin{proof}[Proof of Theorem \ref{th:map_of_localizing_invariants}]. We first prove that $\pi_n(\varphi)$ is an isomorphism for $n\leq 0.$ The argument is basically the same as in \cite{KW19}. We know that $\pi_0(\varphi)$ is an isomorphism by assumption. Assuming that $\pi_n(\varphi)$ is an isomorphism for some $n,$ we see that for any $\cA\in\Cat^{\perf}$ the composition
$$\pi_{n-1}(\varphi):\pi_{n-1}F(\cA)\cong \pi_n F(\Calk_{\omega_1}(\cA))\xto{\pi_n(\varphi)}\pi_n G(\Calk_{\omega_1}(\cA))\cong\pi_{n-1}G(\cA)$$ 
is also an isomorphism. By induction, we see that $\pi_n(\varphi)$ is an isomorphism for $n\leq 0.$

Denote by $\varphi^{\cont}:F^{\cont}\to G^{\cont}$ the induced map between the localizing invariants of dualizable categories. Arguing as above, we see that $\pi_n(\varphi^{\cont})$ is an isomorphism for $n\leq -1.$ It remains to show that if $\pi_n(\varphi^{\cont})$ is an isomorphism for some $n,$ then $\pi_{n+1}(\varphi^{\cont})$ is also an isomorphism.

Let $\cC$ be a dualizable category. Then the functor $\Gamma_c(\R,\varphi_{\gamma}^*(-))[1]:\Sh_{\geq 0}(\R;\cC)\to\cC,$ $\cF\mapsto \indlim[a]\cF((-\infty,a)),$ is a strongly continuous localization (the right adjoint sends an object of $\cC$ to the corresponding constant sheaf). We denote by $\Sh_{>0}(\R;\cC)$ the kernel of $\Gamma_c(\R,\varphi_{\gamma}^*(-)).$ Applying Proposition \ref{prop:all_loc_invar_zero_for_Sh_geq_0}, we see that the composition
\begin{multline*}\pi_{n+1}(\varphi^{\cont}):\pi_{n+1}F^{\cont}(\cC)\cong \pi_n F^{\cont}(\Sh_{>0}(\R;\cC))\xto{\pi_n(\varphi^{\cont})} \pi_n G^{\cont}(\Sh_{>0}(\R;\cC))\cong \\
\pi_{n+1}G^{\cont}(\cC)\end{multline*} is an isomorphism. Therefore, the map $\varphi^{\cont}:F^{\cont}\to G^{\cont}$ is an isomorphism. 
\end{proof}

\section{Sheaves and cosheaves on continuous posets}
\label{sec:sh_cosh_continuous_posets}

\subsection{Continuous posets}
\label{ssec:cont_posets}

Recall that a partially ordered set $(P,\leq)$ is called continuous if $P$ has directed supremums (joins), and the functor $\colim:\Ind(P)\to P$ has a left adjoint $\hat{\cY}:P\to\Ind(P).$ Equivalently, this means the following.

Assuming that $P$ has directed supremums, one defines another relation $\ll$ on $P:$ namely, $j\ll k$ if for any directed system $(l_i)_{i\in I}$ of elements of $P$ such that $k\leq \sup\limits_{i\in I}l_i,$ we have $j\leq l_i$ for some $i\in I.$ Then $P$ is continuous if and only if for each $j\in P$ the poset $\{k\in P\mid k\ll j\}$ is directed and we have equality $$j=\sup\limits_{k\ll j}k.$$ In this case we have $\hat{\cY}(j)=\inddlim[k\ll j]k.$ We refer to \cite{GHKLMS, BH81} for details.

\begin{example}\label{ex:extended_real_line_cont_poset} Let $P=\R\cup\{+\infty\}$ with the usual linear order. Then $P$ is continuous and we have $x\ll y$ iff $x<y.$\end{example}

\begin{example}\label{ex:Open_X_cont_poset} Let $X$ be a locally compact Hausdorff space, and let $P=\Open(X)$ -- the set of open subsets of $X,$ ordered by inclusion. Then $P$ is continuous and we have $U\ll V$ iff $U\Subset V,$ i.e. $\bar{U}\subset V$ and $\bar{U}$ is compact.\end{example}

\begin{example}\label{ex:Comp_X_op_cont_poset} Let $X$ be as in Example \ref{ex:Open_X_cont_poset}, and let $P = \msK(X)^{op}$ -- the set of compact subsets of $X$ with the reverse inclusion order. Then $P$ is continuous and we have $Z\ll Y$ iff $Z\Supset Y,$ i.e. $Z$ contains an open neighborhood of Y.
\end{example}

From now on we assume that $P$ is a continuous poset. We define a Grothendieck topology on $P$ (considered as a category) as follows. A sieve on $j$ is a covering sieve if and only if it contains all $k\ll j.$ This is indeed a topology, since for any $k\ll j$ there exists $l\in P$ such that $k\ll l \ll j.$ Below we will consider sheaves and cosheaves on $P$ with respect to this topology. 

\subsection{Sheaves of spectra on continuous posets}
\label{ssec:sh_of_sp_on_cont_posets}

Let $P$ be a continuous poset. We denote by $\Sh(P;\Sp)$ the category of sheaves on $P$ (where $P$ is considered as a category) for the above topology. This is in fact the category $\Stab^{\cont}(P)$ from Subsection \ref{ssec:cont_stab}.

Namely, since every object of $P$ has the smallest covering sieve, a presheaf $\cF$ is a sheaf if and only if
$$\cF(j)=\prolim[k\ll j]\cF(k),\quad j\in P.$$
Denote by $(-)^{\sharp}:\PSh(P;\Sp)\to\Sh(P;\Sp)$ the sheafification functor. Explicitly, we have
\begin{equation}\label{eq:sheafification_on_cont_posets} \cF^{\sharp}(j)=\prolim[k\ll j]\cF(k),\quad j\in P.\end{equation} For $j\in P,$ denote by $h_j$ the ``representable'' presheaf of spectra, 
$$h_j(k)=\begin{cases}\bS & \text{if }k\leq j\\
0 & \text{else.}\end{cases}$$

\begin{prop}\label{prop:Stab_cont_of_cont_posets} The category $\Sh(P;\Sp)$ is dualizable. For any dualizable category $\cC,$ we have an equivalence
$$\Fun^{LL}(\Sh(P;\Sp),\cC)\xto{\sim} \Fun^{\strcont}(P,\cC),\quad F\mapsto(j\mapsto F(h_j^{\sharp})).$$
\end{prop}

\begin{proof}This is a special case of Proposition \ref{prop:continuous_stab}.
\end{proof}

\subsection{Localizing invariants of categories of sheaves on continuous posets}
\label{ssec:U_loc_sh_on_cont}

In this subsection we compute the (finitary) localizing invariants of categories of sheaves on a continuous poset $P.$ This computation will be relevant in Section \ref{sec:sh_loc_comp}, although formally we will only use a different but related computation for cosheaves from Subsection \ref{ssec:cosheaves_cont_posets}. More precisely, we will consider the sheaves on the poset from Example \ref{ex:Open_X_cont_poset}, i.e. $P=\Open(X)$ for a locally compact Hausdorff space $X.$ It is important to note that the notion of a sheaf on the poset $\Open(X)$ considered here is different from the standard notion of a sheaf on $X:$ we only require that for an open $U\subset X$ the descent holds for the cover of $U$ by all $V\Subset U.$ 

Let $P$ be a continuous poset, and consider a presheaf $\un{\cC}$ on $P$ with values in $\Cat_{\st}^{\dual},$ $\cC_j=\un{\cC}(j)$ for $j\in P.$ We denote by $\res_{jk}:\cC_j\to\cC_k$ the restriction functors for $j\geq k,$ and by $\res_{jk}^R:\cC_k\to\cC_j$ their right adjoints. Consider the category $\PSh(P;\un{\cC})$ of $\un{\cC}$-valued presheaves on $P$ (i.e. the category of sections of the associated cocartesian fibration over $P^{op},$ which is also cartesian). For $j\in P,$ we denote by $\tilde{\iota}_j:\cC_j\to\PSh(P;\un{\cC})$ the functor given by
$$\tilde{\iota}_j(x)(k)=\begin{cases}\res_{jk}(x) & \text{ for }k\leq j;\\
0 & \text{else.}\end{cases}$$
The right adjoint to $\tilde{\iota}_j$ is given by $\tilde{\iota}_j^R(\cF)=\cF(j).$

The notion of a $\un{\cC}$-valued sheaf on $P$ is defined essentially in the same way as in the above special case when $\un{\cC}=\Sp$ (i.e. $\cC_j=\Sp,$ $\res_{jk}=\Id$). More precisely, a presheaf $\cF\in \PSh(P;\un{\cC})$ is a sheaf if and only if
\begin{equation*}
\cF(j)\xto{\sim}\prolim[k\ll j]\res_{jk}^R(\cF(k)),\quad j\in P.
\end{equation*} 
We denote by $\Sh(P;\un{\cC})$ the category of $\un{\cC}$-valued sheaves on $P.$ As above, we denote by $(-)^{\sharp}:\PSh(P;\un{\cC})\to \Shv(P;\un{\cC})$ the sheafification functor. For $\cF\in \PSh(P;\un{\cC})$ we have
\begin{equation}\label{eq:sheafification_simple}
\cF^{\sharp}(j) = \prolim[k\ll j]\res_{jk}^R(\cF(k)),\quad j\in P.
\end{equation}

 For $j\in P$ we have a functor $\iota_j:\cC_j\xto{\tilde{\iota}_j}\PSh(P;\un{\cC})\xto{(-)^{\sharp}} \Sh(P;\un{\cC}).$ The right adjoint to $\iota_j$ is given by $\iota_j^R(\cF)=\cF(j).$ 

Our goal in this subsection is to prove the following result.

\begin{theo}\label{th:U_loc_sheaves_cont_posets} 1) The category $\Sh(P;\un{\cC})$ is dualizable.

2) For each compact element $i\in P^{\omega},$ the functor $\iota_i:\cC_i\to\Sh(P;\un{\cC})$ is strongly continuous and fully faithful.

3) The functors $\iota_i,$ $i\in P^{\omega},$ induce an isomorphism $\bigoplus\limits_{i\in P^{\omega}}\cU_{\loc}^{\cont}(\cC_i)\xto{\sim} \cU_{\loc}^{\cont}(\Sh(P;\cC)).$\end{theo}

\begin{proof}
For a presheaf $\cF\in\PSh(P;\un{\cC})$ by \eqref{eq:sheafification_simple} we have $\cF^{\sharp}(j)=\cF(j)$ for $j\in P^{\omega}.$ This shows that for $j\in P^{\omega}$ the functor $\iota_j^R:\Sh(P;\un{\cC})\to\cC$ commutes with coproducts. This proves 2).

To prove 1), we first observe that the category $\PSh(P;\un{\cC})$ is dualizable. Indeed, we have a $P$-indexed semi-othogonal decomposition of the category $\PSh(P;\un{\cC})$ in $\Pr^{LL}_{\st}:$ \begin{equation}\label{eq:SOD_for_PSh}
\PSh(P;\un{\cC})=\langle \tilde{\iota}_j(\cC_j); j\in P \rangle.\end{equation}
Since $\cC_j$ are dualizable, so is the category $\PSh(P;\un{\cC})$ by Proposition \ref{prop:infinite_SOD_of_dualizable}.

Now, it follows from \eqref{eq:sheafification_simple} that the sheafification functor has a {\it left} adjoint $L:\Sh(P;\un{\cC})\to\PSh(P;\un{\cC}).$ We deduce that the category $\Sh(P;\un{\cC})$ is dualizable.

To prove 3), consider the full subcategory $\cD=\ker((-)^{\sharp})\subset\PSh(P;\un{\cC}).$ By the above we have a short exact sequence in $\Cat_{\st}^{\dual}:$
\begin{equation}\label{eq:resolution_of_Sh_on_P}0\to \Sh(P;\un{\cC})\xto{L}\PSh(P;\un{\cC})\xto{\Phi}\cD\to 0,\end{equation}
where the functor $\Phi$ is the left adjoint to the inclusion. We first observe the following.

{\noindent{\bf Claim 1.}} {\it Let $\cF\in\PSh(P;\un{\cC})$ be a presheaf such that for any $k\ll j$ the map $\res_{jk}(\cF(j))\to\cF(k)$ is zero. Then $\cF$ is contained in $\cD.$}

\begin{proof}[Proof of Claim 1.] This follows almost directly from the description of the sheafification functor \eqref{eq:sheafification_simple}. Namely, take some $j\in P.$ Then for any $k\ll j$ there exists $l\in P$ such that $k\ll l\ll j.$ By assumption the map $\res_{lk}(\cF(l))\to \cF(k)$ is zero, or equivalently the map $\cF(l)\to \res_{lk}^R(\cF(k))$ is zero. Applying $\res_{jl}^R,$ we see that the map $\res_{jl}^R(\cF(l))\to \res_{jk}^R(\cF(k))$ is zero. We conclude that the pro-object $\proolim[k\ll j]\res_{jk}^R(\cF(k))$ is zero, hence so is its limit, i.e. $\cF^{\sharp}(j)=0.$ In other words $\cF\in\cD,$ as stated.\end{proof}

Next, we observe that the category $\cD$ has the following semi-orthogonal decomposition.

{\noindent{\bf Claim 2.}} {\it For $j\in P,$ the functor $\Phi\circ\tilde{\iota}_j:\cC_j\to\cD$ is zero (resp. fully faithful) for $j\in P^{\omega}$ (resp. for $j\in P\setminus P^{\omega}$). Moreover, we have a semi-orthogonal decomposition
$$\cD=\langle \Phi(\tilde{\iota}_j(\cC_j)); j\in P\setminus P^{\omega}\rangle.$$}  

\begin{proof}[Proof of Claim 2.] Let us describe the functor $L:\Sh(P;\un{\cC})\to \PSh(P;\un{\cC})$ more explicitly. For $j\in P,$ we have
\begin{equation}\label{eq:description_of_L}L\circ\iota_j\cong \indlim[k\ll j](\tilde{\iota}_k\circ \res_{jk}).\end{equation} Indeed, it follows from \eqref{eq:sheafification_simple} (by adjunction) that for $x\in\cC_j,$ the object $\liminj\limits_{k\ll j}\tilde{\iota}_k(\res_{jk}(x))$ is left orthogonal to $\cD.$ On the other hand, it follows from Claim 1 that the object $\Cone(\liminj\limits_{k\ll j}\tilde{\iota}_k(\res_{jk}(x))\to\tilde{\iota}_j(x))$ is contained in $\cD$ (since the latter presheaf can take non-zero values only for $k\in P$ such that $k\leq j$ but $k\not\ll j$). This proves the isomorphism \eqref{eq:description_of_L}.

We deduce that $$\Phi\circ\tilde{\iota}_j\cong\Cone(\liminj\limits_{k\ll j}(\tilde{\iota}_k\circ\res_{jk})\to \tilde{\iota}_j).$$
In particular, we see that $\Phi\circ\tilde{\iota}_j$ is zero for $j\in P^{\omega}.$ If $j\in P\setminus P^{\omega},$ then $$\Phi(\tilde{\iota}_j(x))(j)=\Cone(\liminj\limits_{k\ll j}\tilde{\iota}_k(x)(j)\to\tilde{\iota}_j(x)(j))=\Cone(0\to x)=x.$$ Hence, the adjunction unit $x\to (\Phi\circ\tilde{\iota}_j)^R\circ (\Phi\circ\tilde{\iota}_j)$ is an isomorphism. We conclude that the functor $\Phi\circ\tilde{\iota}_j$ is fully faithful for $j\in P\setminus P^{\omega}.$

Similarly, if $j,k\in P\setminus P^{\omega}$ and $k\not\leq j,$ then $\Phi(\tilde{\iota}_j(x))(k)=0$ for $x\in\cC_j.$ This implies the semi-orthogonality. 

Finally, taking into account the vanishing of $\Phi\circ\tilde{\iota}_j$ for $j\in P^{\omega},$ we deduce that the categories $\Phi(\tilde{\iota}_j(\cC_j)),$ $j\in P\setminus P^{\omega},$ generate $\cD.$ This proves the claim.\end{proof}

It follows from Claim 2 and the semi-orthogonal decomposition \eqref{eq:SOD_for_PSh} that after applying $\cU_{\loc}^{\cont}$ to the short exact sequence \eqref{eq:resolution_of_Sh_on_P} we get a split exact triangle:
$$
\begin{tikzcd}
\cU_{\loc}^{\cont}(\Sh(P;\un{\cC})) \ar{r}{\cU_{\loc}^{\cont}(L)} &[1em] \cU_{\loc}^{\cont}(\PSh(P;\un{\cC})) \ar{r}{\cU_{\loc}^{\cont}(\Phi)} & \cU_{\loc}^{\cont}(\cD)\\
\bigoplus\limits_{j\in P^{\omega}}\cU_{\loc}^{\cont}(\cC_j) \ar{r}\ar{u}{\sim} & \bigoplus\limits_{j\in P}\cU_{\loc}^{\cont}(\cC_j) \ar{r}\ar{u}{\sim} & \bigoplus\limits_{j\in P\setminus P^{\omega}}\cU_{\loc}^{\cont}(\cC_j).\ar{u}{\sim}
\end{tikzcd}
$$ 
This proves part 3).
\end{proof}

\subsection{Localizing invariants of categories of cosheaves on continuous posets}
\label{ssec:cosheaves_cont_posets}

In this subsection we compute the (finitary) localizing invariants of categories cosheaves on continuous poset. This computation is very similar to Theorem \ref{th:U_loc_sheaves_cont_posets} and we will apply it in Subsection \ref{ssec:loc_invar_cats_of_sheaves} for the poset from Example \ref{ex:Comp_X_op_cont_poset}, i.e. $P=\msK(X)^{op}$ for a locally compact Hausdorff space $X.$ Again, we note that the notion of a cosheaf on $\msK(X)^{op}$ considered here is weaker than the standard notion of a $\msK$-sheaf on $X.$ Namely, we only require the continuity property without the finite descent condition. We refer to Subsection \ref{ssec:sheaves_K_sheaves_loc_comp} for the notions of a $\msK$-sheaf and a continuous $\msK$-presheaf. 

Let now $\un{\cC}:P\to\Cat_{\st}^{\dual}$ be a copresheaf of dualizable categories on $P,$ $\cC_j=\un{\cC}(j)$ for $j\in P.$ By a $\un{\cC}$-valued copresheaf on $P$ we mean a section of the associated cocartesian fibration over $P$ (which is also cartesian). We denote by $\copsh(P;\un{\cC})$ this category of copresheaves. Further, we denote by $\Cosh(P;\un{\cC})\subset \copsh(P;\un{\cC})$ the full subcategory of cosheaves: it consists of copresheaves $\cF$ such that $\cF(j)\cong \indlim[k\ll j]\cores_{kj}(\cF(k)),$ $j\in P.$ We denote by $(-)^{\sharp}:\copsh(P;\un{\cC})\to \Cosh(P;\un{\cC})$ the cosheafification functor, i.e. the right adjoint to the inclusion.

Again, we denote by $\tilde{\iota_j}:\cC_j\to\copsh(P;\un{\cC})$ the left Kan extension functor, given by
$$\tilde{\iota}_j(x)(k)=\begin{cases}\cores_{jk}(x) & \text{ for }k\geq j;\\
0 & \text{else.}\end{cases}$$
The right adjoint to $\tilde{\iota}_j$ is given by $\tilde{\iota}_j^R(\cF)=\cF(j).$ We have a composition $\iota_j:\cC_j\xto{\tilde{\iota_j}}\copsh(P;\un{\cC})\xto{(-)^{\sharp}}\Cosh(P;\un{\cC}).$ The following result is basically dual to Theorem \ref{th:U_loc_sheaves_cont_posets}.

\begin{theo}\label{th:U_loc_cosheaves_on_cont_posets} 1) The category $\Cosh(P;\un{\cC})$ is dualizable.

2) For each compact element $j\in P^{\omega},$ the functor $\tilde{\iota_j}$ takes values in cosheaves, hence we have a fully faithful strongly continuous functor $\iota_j:\cC_j\to\Cosh(P;\un{\cC}).$

3) The functors $\iota_j,$ $j\in P^{\omega},$ induce an isomorphism $\biggplus[j\in P^{\omega}]\cU_{\loc}^{\cont}(\cC_j)\xto{\sim} \cU_{\loc}^{\cont}(\Cosh(P;\un{\cC})).$ 
\end{theo}

\begin{proof}As in the proof of Theorem \ref{th:U_loc_sheaves_cont_posets}, the category $\copsh(P;\un{\cC})$ is dualizable and we have a $P^{op}$-indexed semi-orthogonal decomposition $\copsh(P;\un{\cC})=\la \tilde{\iota_j}(\cC_j); j\in P^{op} \ra.$ The cosheafification functor is continuous, since we have 
\begin{equation}\label{eq:cosheafification_simple}\cF^{\sharp}(j)=\indlim[k\ll j]\cores_{kj}(\cF(k)),\quad j\in P.\end{equation}
It follows that the category $\Cosh(P;\un{\cC})$ is dualizable. This proves 1).

Part 2) is clear from the definitions.

We prove 3). Consider the short exact sequence
\begin{equation}\label{eq:resolution_of_Cosh_on_P}0\to \Cosh(P;\un{\cC})\to\copsh(P;\un{\cC})\xto{\Psi}\cE\to 0,\end{equation} 
where $\cE\subset \copsh(P;\un{\cC})$ is the kernel of the cosheafification, and $\Psi$ is the left adjoint to the inclusion. We claim that the functors $\Psi\circ\tilde{\iota_j}:\cC_j\to \cE$ are fully faithful for $j\not\in P^{\omega}$ and we have a semi-orthogonal decomposition
\begin{equation}\label{eq:SOD_kernel_of_cosheafification}
\cE=\la \Psi(\tilde{\iota_j}(\cC_j)), j\in (P\setminus P^{\omega})^{op} \ra.\end{equation}
Indeed, this follows directly from \eqref{eq:cosheafification_simple}: we have $$\Psi(\tilde{\iota_j}(x))=\Cone(\iota_j(x)\to\tilde{\iota_j}(x)),\quad \iota_j(x)(k) = \begin{cases}\cores_{jk}(x) & \text{for }k\gg j;\\
	0 & \text{else,}\end{cases}$$
hence
\begin{equation*}
\Psi(\tilde{\iota_j}(x))(k) = \begin{cases}\cores_{jk}(x) & \text{for }k\geq j,\, k\not\gg j;\\
0 & \text{else.}\end{cases}
\end{equation*} It follows by adjunctions that the functors $\Psi\circ\tilde{\iota_j}:\cC_j\to \cE$ are fully faithful ($j\not\in P^{\omega}$) and their images are semi-orthogonal as required. Since the functors $\tilde{\iota_j},$ $j\in P^{\omega},$ take values in cosheaves, we obtain the semi-orthogonal decomposition \eqref{eq:SOD_kernel_of_cosheafification}.

As in the proof of Theorem \ref{th:U_loc_sheaves_cont_posets}, we see that after applying $\cU_{\loc}^{\cont}$ to the sequence \eqref{eq:resolution_of_Cosh_on_P} we get a split exact triangle:
$$
\begin{tikzcd}
\cU_{\loc}^{\cont}(\Cosh(P;\un{\cC})) \ar{r}{\cU_{\loc}^{\cont}(L)} &[1em] \cU_{\loc}^{\cont}(\copsh(P;\un{\cC})) \ar{r}{\cU_{\loc}^{\cont}(\Psi)} & \cU_{\loc}^{\cont}(\cE)\\
\bigoplus\limits_{j\in P^{\omega}}\cU_{\loc}^{\cont}(\cC_j) \ar{r}\ar{u}{\sim} & \bigoplus\limits_{j\in P}\cU_{\loc}^{\cont}(\cC_j) \ar{r}\ar{u}{\sim} & \bigoplus\limits_{j\in P\setminus P^{\omega}}\cU_{\loc}^{\cont}(\cC_j). \ar{u}{\sim}
\end{tikzcd}
$$ 
This proves 3).
\end{proof}
 
\section{Sheaves on locally compact Hausdorff spaces}
\label{sec:sh_loc_comp}

\subsection{Sheaves on finite CW complexes}
\label{ssec:sheaves_finite_CW}

The following result describes arbitrary (accessible) localizing invariants of categories of sheaves on finite CW complexes, in the case of constant coefficients. Here we consider finite CW complexes both as topological spaces and as $\infty$-groupoids (i.e. as objects of the $\infty$-category $\cS^{\fin}$ of finite spaces).

\begin{theo}\label{th:sheaves_on_finite_CW_complexes} Let $X$ be a finite CW complex, considered as a topological space. Let $\cC$ be a dualizable category. Then for any accessible localizing invariant $F:\Cat^{\perf}\to\cE$ we have an isomorphism
$$F^{\cont}(\Shv(X;\cC))\simeq F^{\cont}(\cC)^X.$$
Here in the right hand side we consider $X$ as an object of $\cS^{\fin}.$\end{theo}

\begin{proof}Suppose that $\cE$ has $\kappa$-filtered colimits and $F$ commutes with $\kappa$-filtered colimits. By Proposition \ref{prop:all_loc_invar_zero_for_Sh_geq_0}, we have $\cU_{\loc,\kappa}^{\cont}(\Sh_{\geq 0}(\R;\Sp))=0.$ Arguing as in the proof of Proposition \ref{prop:U_loc_sheaves_on_R} we obtain $$\cU_{\loc,\kappa}^{\cont}(\Sh([0,1),\Sp))=0.$$
%
From the short exact sequence $$0\to \Sh([0,1);\Sp)\to \Sh([0,1];\Sp)\to\Sp\to 0$$
in $\Cat_{\st}^{\dual}$ we obtain
$$\cU_{\loc,\kappa}^{\cont}(\Sh([0,1],\Sp))\cong \cU_{\loc,\kappa}^{\cont}(\Sp).$$
It follows that for any finite CW complex $X$ and for any dualizable category $\cC$ we have
$$F^{\cont}(\Sh(X;\cC))\cong F^{\cont}(\Sh(X\times [0,1];\cC)).$$
Therefore we have a well defined functor $G:(\cS^{\fin})^{op}\to\cE,$ such that for a finite CW complex $X$ we have $G(X)=F^{\cont}(\Sh(X;\cC)).$
Since $G(*)=F^{\cont}(\cC)$ and $G(\emptyset)=0,$ it suffices to prove that $G$ commutes with pullbacks.

This is straightforward. Namely, consider cellular closed embeddings of finite CW complexes $X\hto Y,$ $X\hto Z.$ Then both pullback functors $\Sh(Y;\cC)\to \Sh(X;\cC)$ and $\Sh(Z;\cC)\to \Sh(X;\cC)$ are strongly continuous localizations and we have a pullback square 
$$
\begin{tikzcd}
\Sh(Y\sqcup_X Z;\cC) \ar{r}\ar{d} & \Sh(Y;\cC) \ar{d}\\
\Sh(Z;\cC) \ar{r} & \Sh(X;\cC).
\end{tikzcd}
$$
By Proposition \ref{prop:F_cont_nice_pullbacks} we obtain an isomorphism
$$F^{\cont}(\Sh(Y\sqcup_X Z;\cC))\cong F^{\cont}(\Sh(Y;\cC))\times_{F^{\cont}(\Sh(X;\cC))}F^{\cont}(\Sh(Z;\cC)).$$
Hence, $G$ commutes with pullbacks. This proves the theorem.\end{proof}

\subsection{Sheaves and K-sheaves on locally compact Hausdorff spaces}
\label{ssec:sheaves_K_sheaves_loc_comp}

Let now $X$ be a locally compact Hausdorff space, and consider a presheaf $\un{\cC}$ on $X$ with values in $\Pr^{LL}_{\st}.$ In particular, we require that the restriction functors $\res_{UV}:\un{\cC}(U)\to\un{\cC}(V)$ are strongly continuous.

We will consider $\un{\cC}$-valued (pre)sheaves on $X.$ Recall that $\cF\in\PSh(X;\un{\cC})$ is a sheaf if and only if it satisfies the following conditions:

\begin{enumerate}[label=(\roman*),ref=(\roman*)]
\item $\cF(\emptyset)=0;$ \label{sheafcond1}

\item for any open $U,V\subset X,$ we have a pullback square in $\un{\cC}(U\cup V):$
$$
\begin{tikzcd}
\cF(U\cup V) \ar{r}\ar{d} & \res_{U\cup V,V}^R\cF(V)\ar{d}\\
\res_{U\cup V,U}^R\cF(U) \ar{r} & \res_{U\cup V,U\cap V}^R\cF(U\cap V).
\end{tikzcd}
$$ 
\label{sheafcond2}

\item for any open $U\subset X,$ the map $\cF(U)\to\limproj\limits_{V\Subset U}\res_{UV}^R\cF(V)$ is an isomorphism in $\un{\cC}(U).$ \label{sheafcond3}
\end{enumerate}

Recall from Example \ref{ex:Open_X_cont_poset} that the poset $\Open(X)$ of open subsets of $X$ is continuous. In particular, we have a topology on $\Open(X)$ as in Subsection \ref{ssec:cont_posets}. This topology is weaker than the usual one, and we denote by $\PSh^{\cont}(X;\cC)$ the category of sheaves on $\Open(X)$ with respect to this topology. We call the objects of $\PSh^{\cont}(X;\cC)$ ``continuous presheaves''. These are exactly the presheaves satisfying the condition \ref{sheafcond3}. We denote by $(-)^{\cont}:\PSh(X;\un{\cC})\to \PSh^{\cont}(X;\cC)$ the left adjoint to the inclusion. As in Subsection \ref{ssec:U_loc_sh_on_cont} we have
$$\cF^{\cont}(U)=\prolim[V\Subset U]\res_{UV}^R(\cF(V)).$$

We denote by $(-)^{\sharp}:\PSh(X;\un{\cC})\to\Sh(X;\un{\cC})$ the sheafification functor. It has a factorization:
$$\PSh(X;\un{\cC})\xto{(-)^{\cont}}\PSh^{\cont}(X;\un{\cC})\xto{(-)^{\sharp}}\Sh(X;\un{\cC}).$$

We call a presheaf $\cF\in\PSh(X;\un{\cC})$ {\it reduced} if $\cF(\emptyset)=0,$ i.e. if the condition \ref{sheafcond1} holds for $\cF.$ Denote by $\PSh^{\red}(X;\un{\cC})\subset \PSh(X;\un{\cC})$ the full subcategory of reduced presheaves. The inclusion functor has a left adjoint $\cF\mapsto\cF^{\red},$ where
\begin{equation}\label{eq:F_red} \cF^{\red}(U)=\begin{cases}\cF(U) & \text{if }U\neq\emptyset;\\
0 & \text{if }U=\emptyset.\end{cases}
\end{equation}
Note that if $\cF$ is continuous, then so is $\cF^{\red}.$

\begin{remark}
In what follows, we crucially use the following property of the continuous lattice $\Open(X):$ if $U'\Subset U$ and $V'\Subset V,$ then $U'\cap V'\Subset U\cap V.$ In other words, the way below relation in $\Open(X)$ is multiplicative in the terminology of \cite{GHKLMS}. 
\end{remark}

We make the following observation.

\begin{prop}Suppose that $\cF\in\PSh(X;\un{\cC})$ is a presheaf which satisfies descent for finite covers, i.e. the conditions \ref{sheafcond1} and \ref{sheafcond2}. Then $\cF^{\cont}$ is a sheaf. Equivalently, the natural map $\cF^{\cont}\to\cF^{\sharp}$ is an isomorphism.

In particular, the inclusion functor $\Sh(X;\un{\cC})\to \PSh^{\cont}(X;\un{\cC})$ is continuous.\end{prop}

\begin{proof} Assuming that $\cF$ satisfies finite descent, it suffices to show that $\cF^{\cont}$ satisfies finite descent.

Clearly, $\cF^{\cont}(\emptyset)=\cF(\emptyset)=0,$ so \ref{sheafcond1} holds. To see that \ref{sheafcond2} holds for $\cF^{\cont},$ we first introduce the following notation: for an open $U\subset X$ we denote by $P_X(U)$ the poset $\{W:W\Subset U\}.$ Now take some open subsets $U,V\subset X,$ and consider the maps
\begin{equation}\label{eq:cup} P_X(U)\times P_X(V)\to P_X(U\cup V),\quad (W_1,W_2)\mapsto W_1\cup W_2;\end{equation}
\begin{equation}\label{eq:cap} P_X(U)\times P_X(V)\to P_X(U\cap V),\quad (W_1,W_2)\mapsto W_1\cap W_2.\end{equation}
The standard arguments show that both maps \eqref{eq:cup} and \eqref{eq:cap} are cofinal (hence, the maps between the same posets with the opposite order are final). This directly implies that \ref{sheafcond2} holds for $\cF^{\cont}.$ Thus, $\cF^{\cont}$ is indeed a sheaf.

For the final assertion of the proposition, it suffices to notice that the class of presheaves satisfying finite descent is closed under colimits.   
\end{proof}

We consider the category $\PSh^{\cont}(X;\un{\cC})$ as the first approximation of the category $\Sh(X;\un{\cC}).$ Note that if $\un{\cC}$ is a presheaf with values in $\Cat_{\st}^{\dual},$ then we know the finitary localizing invariants of the category $\PSh^{\cont}(X;\un{\cC})$ by Theorem \ref{th:U_loc_sheaves_cont_posets}. To construct the further approximations, it is convenient to use the language of $\msK$-sheaves, which we now recall.

We first extend the presheaf $\un{\cC}$ to locally closed subsets by left Kan extension: for a locally closed $Y\subset X$ we put
$$\un{\cC}(Y):=\indlim[U\supset Y]^{\cont}\un{\cC}(U),$$ where $U$ runs over open subsets $U\subset X$ such that $Y$ is a closed subset of $U.$ 
In particular, we obtain a presheaf $\un{\cC}_{\mid Y}$ of presentable stable categories on $Y.$ 

As in Example \ref{ex:Comp_X_op_cont_poset} we denote by $\msK(X)$ the poset of compact subsets of $X.$ Recall from loc. cit. that the poset $\msK(X)^{op}$ is continuous, and $Z\ll Y$ iff $Z\Supset Y,$ i.e. $Z$ contains an open neighborhood of $Y.$

\begin{defi}1) We denote by $\PSh_{\msK}(X;\un{\cC})$ the category of $\msK$-presheaves, i.e. the category of $\un{\cC}$-valued presheaves on $\msK(X).$

2) We denote by $\PSh_{\msK}^{\cont}(X;\un{\cC})\subset \PSh_{\msK}(X;\un{\cC})$ the full subcategory of continuous $\msK$-presheaves, i.e. $\msK$-presheaves such that for any compact $Y\subset X$ we have an isomorphism
$$\indlim[Z\Supset Y]\res_{Z,Y}(\cF(Z))\xto{\sim}\cF(Y).$$

3) We denote by $\Sh_{\msK}(X;\un{\cC})\subset\PSh^{\cont}(X;\un{\cC})$ the full subcategory formed by continuous $\msK$-presheaves $\cF$ satisfying additional conditions:

i) $\cF(\emptyset)=0;$

ii) For compact subsets $Y_1,Y_2\subset X,$ the following is a pullback square:
$$
\begin{tikzcd}
\cF(Y_1\cup Y_2) \ar{r}\ar{d} & \res_{Y_1\cup Y_2,Y_1}^R\cF(Y_1)\ar{d}\\
\res_{Y_1\cup Y_2,Y_2}^R\cF(Y_2) \ar{r} & \res_{Y_1\cup Y_2,Y_1\cap Y_2}^R\cF(Y_1\cap Y_2).
\end{tikzcd}
$$\end{defi}

Note that the category $\PSh_{\msK}^{\cont}(X;\un{\cC})$ is equivalent to the category of $\un{\cC}$-valued cosheaves on $\msK(X)^{op}$ for the topology from Subsection \ref{ssec:cont_posets}.

Again, it is convenient to call a $\msK$-presheaf $\cF$ {\it reduced} if $\cF(\emptyset)=0.$ We denote by $\PSh_{\msK}^{\red}(X;\un{\cC})\subset \PSh_{\msK}(X;\un{\cC})$ the full subcategory of reduced $\msK$-presheaves. As above, the inclusion functor has a left adjoint $\cF\mapsto\cF^{\red},$ where
\begin{equation}\label{eq:F_red_msK} \cF^{\red}(Y)=\begin{cases}\cF(Y) & \text{if }Y\neq\emptyset;\\
		0 & \text{if }Y=\emptyset.\end{cases}
\end{equation} If $\cF$ is a continuous $\msK$-presheaf, then so is $\cF^{\red}.$

The following is an (almost) straightforward generalization of Lurie's theorem for the case of a constant presheaf $\un{\cC}.$ 

\begin{prop}\label{prop:PSh_and_PSh_K}\cite[Theorem 7.3.4.9, Corollary 7.3.4.10]{Lur09} 1) Consider the following functors
\begin{equation}\label{eq:PSh_to_PSh_K} \PSh^{\cont}(X;\un{\cC})\to \PSh^{\cont}_{\msK}(X;\un{\cC}),\quad \cF\mapsto (Y\mapsto \liminj\limits_{U\supset Y}\res_{U,Y}(\cF(U)));\end{equation}
\begin{equation}\label{eq:PSh_K_to_PSh} \PSh^{\cont}_{\msK}(X;\un{\cC})\to \PSh^{\cont}(X;\un{\cC}),\quad \cG\mapsto (U\mapsto \limproj\limits_{Y\subset U}\res_{U,Y}^R(\cG(Y))).\end{equation}
They are quasi-inverse equivalences.

2) The functors \eqref{eq:PSh_to_PSh_K} and \eqref{eq:PSh_K_to_PSh} induce equivalences between the full subcategories $\Sh(X;\un{\cC})\subset \PSh^{\cont}(X;\un{\cC})$ and $\Sh_{\msK}(X;\un{\cC})\subset \PSh^{\cont}_{\msK}(X;\un{\cC}).$
\end{prop}

\begin{proof}Part 1) is proved by the same argument as in \cite[Theorem 7.3.4.9]{Lur09}, and in fact we do not need the strong continuity of the restriction functors $\res_{UV}.$

We prove 2). Let $\cF$ be a $\un{\cC}$-valued sheaf, and denote by $\cG$ the associated continuous $\msK$-presheaf. We need to prove that $\cG$ is a $\msK$-sheaf. Clearly, $\cG(\emptyset)=0.$
Consider compact subsets $Y_1,Y_2\subset X,$ and put $Y=Y_1\cup Y_2,$ $Y_{12}=Y_1\cap Y_2.$
It suffices to show that the square
\begin{equation}\label{eq:to_be_pullback_square}
\begin{tikzcd}
\cG(Y) \ar{r}\ar{d} & \res_{Y,Y_1}^R\cG(Y_1)\ar{d}\\
\res_{Y,Y_2}^R\cG(Y_2) \ar{r} & \res_{Y,Y_{12}}^R\cG(Y_{12})
\end{tikzcd}
\end{equation}
is the directed colimit of squares
\begin{equation}\label{eq:already_pullback_square}
\begin{tikzcd}
\res_{U_1\cup U_2,Y}\cF(U_1\cup U_2) \ar{r}\ar{d} & \res_{U_1\cup U_2,Y}\res_{U_1\cup U_2,U_1}^R\cF(U_1)\ar{d}\\
\res_{U_1\cup U_2,Y}\res_{U_1\cup U_2,U_2}^R\cF(U_2)) \ar{r} & \res_{U_1\cup U_2,Y}\res_{U_1\cup U_2,U_1\cap U_2}^R\cF(U_1\cap U_2)
\end{tikzcd}
\end{equation}
over open neighborhoods $U_1\supset Y_1,$ $U_2\supset Y_2.$ Indeed, this would imply that \eqref{eq:to_be_pullback_square} is a pullback square since so are \eqref{eq:already_pullback_square} (by the sheafiness of $\cF$).

Clearly, we have $$\indlim[\substack{U_1\supset Y_1,\\ U_2\supset Y_2}]\res_{U_1\cup U_2,Y}\cF(U_1\cup U_2)\cong \cG(Y).$$
Next, we consider the colimit of the top right objects in \eqref{eq:already_pullback_square}. Denote by $I$ the above directed poset of pairs of open neighborhoods $(U_1\supset Y_1, U_2\supset Y_2),$ ordered by the reverse inclusion. Let $J$ be a similar poset of pairs open neighborhoods $(U_1\supset Y_1,V\supset Y)$ such that $V\supset U_1.$ Again, $J$ is ordered by the reverse inclusion. Then the map $I\to J,$ $(U_1,U_2)\mapsto (U_1,U_1\cup U_2),$ is cofinal: it has a left adjoint $(U_1,V)\mapsto (U_1,V).$
Hence, we have an isomorphism
\begin{equation}\label{eq:isomorphism_by_cofinality}
\indlim[\substack{U_1\supset Y_1,\\ U_2\supset Y_2}]\res_{U_1\cup U_2,Y}\res_{U_1\cup U_2,U_1}^R\cF(U_1)\xto{\sim} \indlim[\substack{U_1\supset Y_1,\\ V\supset U_1\cup Y}] \res_{VY}\res_{V,U_1}^R\cF(U_1).
\end{equation} 

Applying Lemma \ref{lem:auxiliary_on_colimits}, for a fixed open neighborhood $V\supset Y$ we obtain an isomorphism
\begin{equation*}
\indlim[V\supset U_1\supset Y_1] \res_{V,U_1}^R \cF(U_1)\xto{\sim} \res_{V,Y_1}^R \cF(Y_1).
\end{equation*}
Here we used the strong continuity of the restriction functors. We obtain a simplification of the target of \eqref{eq:isomorphism_by_cofinality}:
\begin{equation}\label{eq:isomorphism_by_auxiliary_on_colimits}
\indlim[\substack{U_1\supset Y_1,\\ V\supset U_1\cup Y}]\res_{VY}\res_{V,U_1}^R\cF(U_1)\xto{\sim} \indlim[V\supset Y] \res_{VY}\res_{V,Y_1}^R\cF(Y_1).
\end{equation}
Next, by Proposition \ref{prop:colimit_of_presentable_compositions_of_adjoints} we have
\begin{equation}\label{eq:isomorphism_by_playing_with_adjoints}
\indlim[V\supset Y] \res_{VY}\res_{V,Y_1}^R\cF(Y_1)\cong \indlim[V\supset Y] \res_{VY}\res_{V,Y}^R\res_{Y,Y_1}^R\cF(Y_1) \xto{\sim} \res_{Y,Y_1}^R\cF(Y_1)
\end{equation}
(we again used the strong continuity of the restriction functors). The composition of the isomorphisms \eqref{eq:isomorphism_by_cofinality}, \eqref{eq:isomorphism_by_auxiliary_on_colimits} and \eqref{eq:isomorphism_by_playing_with_adjoints} gives the isomorphism between the top right object of \eqref{eq:to_be_pullback_square} and the colimit of top right objects of \eqref{eq:already_pullback_square}. 
The proof for the remaining two objects of the square \eqref{eq:to_be_pullback_square} is similar. Therefore, $\cG$ is a $\msK$-sheaf.

The proof in the other direction is as in \cite[Theorem 7.3.4.9]{Lur09} (and here we do not need the strong continuity of the restriction functors).
\end{proof}

We introduce some notation. Suppose that $X$ is compact. Given a finite collection $\{Y_1,\dots,Y_n\}$ of closed subsets of $X$ such that $X=Y_1\cup\dots\cup Y_n,$ we put \begin{equation}\label{eq:sieve_of_compacts}S_{(X;Y_1,\dots,Y_n)}:=\{Y\subset X\mid Y\text{--compact, and }Y\subset Y_i\text{ for some }i\in\{1,\dots,n\}\}\end{equation} We consider $S$ as a poset with the reverse inclusion order. 

Consider the set $J_X$ of all collections $S$ of compact (=closed) subsets of $X$ of the form $S_{(X;Y_1,\dots,Y_n)}$ as in \eqref{eq:sieve_of_compacts}. Define the partial order on $J_X$ by setting $S\leq S'$ if $S'\subset S.$ Then $J_X$ is a directed poset.

\begin{remark}\label{remark:finitary_topology} 1) If $X$ is locally compact, then the poset $\msK(X)$ has a natural Grothendieck pretopology, where the family $\{Y_i\to Y\}_{i\in I}$ is a cover if $I$ is finite and $\bigcup_{i\in I} Y_i=Y.$ Then for each compact $Y\subset X,$ the poset $J_Y$ is a cofinal family of covering sieves for $Y.$

2) If $X$ is compact and $S=S_{(X;Y_1,\dots,Y_n)}\in J_X$ is a sieve as above, then we have a final functor $P_{\geq 1}(\{1,\dots,n\})\to S,$ $I\mapsto Y_I:=\bigcap_{i\in I} Y_i.$ Here for a set $T$ we denote by $P_{\geq 1}(T)$ the poset of non-empty subsets of $T.$\end{remark}

We recall that the notation $\cF^{\red}$ for a $\msK$-presheaf $\cF$ is introduced in \eqref{eq:F_red_msK}. 

\begin{prop}Let $X$ and $\un{\cC}$ be as above. Consider the functor $(-)^{\sharp}:\PSh_{\msK}^{\cont}(X;\un{\cC})\to\Sh_{\msK}(X;\un{\cC}),$ which under the equivalence of Proposition \ref{prop:PSh_and_PSh_K} corresponds to the sheafification functor. Then for $\cF\in \PSh_{\msK}^{\cont}(X;\un{\cC})$ we have
\begin{equation}\label{eq:K-sheafification}\cF^{\sharp}(Y)=\indlim[S\in J_Y]\prolim[Z\in S]\res_{Y,Z}^R(\cF^{\red}(Z)).\end{equation}\end{prop}

\begin{proof}
Consider the full subcategory $\Sh_{\fin,\msK}(X;\un{\cC})\subset \PSh_{\msK}(X;\un{\cC})$ of $\msK$-presheaves satisfying only the finite descent, i.e. sheaves for the finitary topology on $\msK(X)$ described in Remark \ref{remark:finitary_topology} 1). Then the corresponding finitary sheafification functor $(-)^{\sharp,\fin}:\PSh_{\msK}(X;\un{\cC})\to \Sh_{\fin,\msK}(X;\un{\cC})$ is given by \eqref{eq:K-sheafification}. Thus, we need to show the following: for $\cF\in \PSh^{\cont}_{\msK}(X;\un{\cC}),$ the $\msK$-presheaf $\cF^{\sharp,\fin}$ is continuous. This is straightforward.
\end{proof}




We observe the following.

\begin{prop}\label{prop:closed_embedding_pullback_PSh^cont} For a closed embedding $\iota:Y\subset X,$ the functor $\PSh^{\cont}_{\msK}(X;\un{\cC})\xto{\iota^*} \PSh^{\cont}_{\msK}(Y;\un{\cC}_{\mid Y})$ is a strongly continuous localization.\end{prop}

\begin{proof}
The functor $\iota^*$ and its right adjoint $\iota_*$ are given by the formulas
$$\iota^*(\cF)(Z)=\cF(Z),\quad \iota_*(\cF)(Z)=\res_{Z,Z\cap Y}^R(\cF(Z\cap Y)).$$
Hence, $\iota_*$ is continuous and the adjunction counit $\iota^*\iota_*\to \Id$ is an isomorphism.\end{proof}

We will denote by $\PSh^{\cont,\red}_{\msK}(X;\un{\cC})=\PSh^{\cont}_{\msK}(X;\un{\cC})\cap \PSh^{\red}_{\msK}(X;\un{\cC})$ the full subcategory of continuous reduced $\msK$-presheaves. If $X$ is compact and $S\in J_X,$ we define the presentable stable category
$$\cD_{(X,S)}:=\prolim[Y\in S]\PSh^{\cont,\red}_{\msK}(Y;\un{\cC}_{\mid Y}).$$

\begin{prop}\label{prop:approximating_sheaves} Suppose that $X$ is compact.
	
1) For each $S\in J_X,$ the composition functor $\PSh^{\cont}_{\msK}(X;\un{\cC})\xto{(-)^{\red}} \PSh^{\cont,\red}_{\msK}(X;\un{\cC}) \to \cD_{(X,S)}$ is a strongly continuous localization. The essential image of the right adjoint $\cD_{(X,S)}\to \PSh^{\cont}_{\msK}(X;\un{\cC})$ consists of continuous reduced $\msK$-presheaves $\cF$ such that for each compact $Y\subset X$ we have
\begin{equation}\label{eq:description_of_D_X_S} \cF(Y)\cong \limproj\limits_{Z\in S}\res_{Y,Y\cap Z}^R(\cF(Y\cap Z)).\end{equation}

2) Put $\cD:=\liminj\limits_{S\in J_X}^{\cont}\cD_{(X,S)}.$ Then the functor $\PSh^{\cont}_{\msK}(X;\un{\cC})\to\cD$ is a localization, and the essential image of its right adjoint is exactly the category of $\msK$-sheaves $\Sh_{\msK}(X;\un{\cC}).$ In particular, we have an equivalence $\cD\simeq \Sh_{\msK}(X;\un{\cC}).$\end{prop} 

\begin{proof}

Observe that the analogue of part 1) for categories $\PSh^{\red}_{\msK}(Y;\un{\cC}_{\mid Y})$ holds trivially. To prove the statement 1), we only need to show the following.

{\noindent{\bf Claim.}} {\it Suppose that $S=S_{(X;Y_1,\dots,Y_n)}\in J_X.$ Let $\cF\in\PSh^{\red}_{\msK}(X;\un{\cC})$ be a $\msK$-presheaf satisfying \eqref{eq:description_of_D_X_S}. If each of the pullbacks $\cF_{\mid Y_i}$ is in $\PSh^{\cont}_{\msK}(Y_i;\un{\cC}_{\mid Y_i}),$ then also $\cF\in\PSh^{\cont}_{\msK}(X;\un{\cC}).$}

\begin{proof}[Proof of Claim.] We may assume that $n=2,$ i.e. $S=S_{(X;Y_1,Y_2)}.$ Take a compact subset $T\subset X.$ It suffices to show that the map
$$\indlim[T'\Supset T]\res_{T',T}\res_{T',T'\cap Y_1}^R\cF(T'\cap Y_1)\to \res_{T,T\cap Y_1}^R\cF(T\cap Y_1)$$
is an isomorphism, and similarly if $Y_1$ is replaced by $Y_2$ or $Y_1\cap Y_2.$

This follows directly from Proposition \ref{prop:colimit_of_presentable_compositions_of_adjoints} and Lemma \ref{lem:auxiliary_on_colimits}: we have
\begin{multline*}\indlim[T'\Supset T]\res_{T',T}\res_{T',T'\cap Y_1}^R\cF(T'\cap Y_1)\cong \indlim[T''\supset T'\Supset T]\res_{T'',T}\res_{T'',T'\cap Y_1}^R\cF(T'\cap Y_1)\\ \cong \indlim[T''\Supset T]\res_{T'',T}\res_{T'',T\cap Y_1}^R\cF(T\cap Y_1)\cong \res_{T,T\cap Y_1}^R\cF(T\cap Y_1).\end{multline*}
The remaining two isomorphisms are proved in the same way. 
\end{proof} 

2) It follows from 1) that the functor $\PSh^{\cont}_{\msK}(X;\un{\cC})\to\cD$ is a strongly continuous localization, and the essential image of its right adjoint consists of continuous reduced $\msK$-presheaves $\cF$ such that for any compact $Y\subset X$ and for any $S\in J_X$ we have 
\begin{equation}\label{eq:objects_of_D}\cF(Y)\cong \limproj\limits_{Z\in S}\res_{Y,Y\cap Z}^R \cF(Y\cap Z),\quad\text{ for }Y\in\msK(X),\,S\in J_X.\end{equation}
We need to show that this condition holds if and only if $\cF$ is a $\msK$-sheaf.

The ``if'' direction is obvious. We show the ``only if'' direction. Take some pair of compact subsets $Y_1,Y_2\subset X.$ For any open neighborhoods $U_1\supset Y_1,$ $U_2\supset Y_2,$ denote by $S(U_1,U_2)\in J_X$ the sieve $S_{(X;\bar{U_1},\bar{U_2},X\setminus(U_1\cup U_2))}.$ Applying the condition \eqref{eq:objects_of_D} to $Y=Y_1\cup Y_2$ and $S=S(U_1,U_2),$ we see that the following is a pullback square:
$$
\begin{tikzcd}
\cF(Y) \ar{r}\ar{d} & \res_{Y,Y\cap \bar{U_1}}^R\cF(Y\cap \bar{U_1}) \ar{d}\\
\res_{Y,Y\cap \bar{U_2}}^R\cF(Y\cap \bar{U_2}) \ar{r} & \res_{Y,Y\cap \bar{U_1}\cap\bar{U_2}}^R\cF(Y\cap \bar{U_1}\cap\bar{U_2}).
\end{tikzcd}
$$  
Taking the colimit over $U_1$ and $U_2$ (and using Lemma \ref{lem:auxiliary_on_colimits} again), we obtain the following square:
$$
\begin{tikzcd}
\cF(Y) \ar{r}\ar{d} & \res_{Y,Y_1}^R\cF(Y_1)\ar{d}\\
\res_{Y,Y_2}^R\cF(Y_2) \ar{r} & \res_{Y,Y_1\cap Y_2}^R\cF(Y_1\cap Y_2).
\end{tikzcd}
$$
Hence, this is a pullback square and $\cF$ is a $\msK$-sheaf.
\end{proof}

\subsection{Localizing invariants of categories of sheaves}
\label{ssec:loc_invar_cats_of_sheaves}

From now on we assume that our presheaf of categories $\un{\cC}$ takes values in dualizable categories. Our goal in this section is to compute (describe) the finitary localizing invariants of the category $\Shv(X;\un{\cC}),$ which is dualizable by Proposition \ref{prop:Shv_is_dualizable} below. It suffices to consider the universal case, i.e. the invariant $\cU_{\loc}^{\cont}:\Cat_{\st}^{\dual}\to \Mot^{\loc}.$ 

We first explain a certain natural construction. By Proposition \ref{prop:U_loc_of_sheaves_gives_K_sheaf} 1) the functor 
\begin{equation*}
\cF_0:\msK(X)^{op}\to \Mot^{\loc},\quad \cF_0(Y) = \cU_{\loc}^{\cont}(\Shv(Y;\un{\cC}_{\mid Y})),
\end{equation*} 
is a $\msK$-sheaf on $X.$ Note that $\cF_0$ is well-defined since for compact $Y\supset Z$ the pullback functor $\Shv(Y;\un{\cC}_{\mid Y})\to \Shv(Z;\un{\cC}_{\mid Z})$ is strongly continuous. We denote by $\cG_0\in \Shv(X;\Mot^{\loc})$ the corresponding sheaf via the equivalence from Proposition \ref{prop:PSh_and_PSh_K}. By Proposition \ref{prop:U_loc_of_sheaves_gives_K_sheaf} 2) the compactly supported global sections of $\cG_0$ are given by 
\begin{equation*}
\Gamma_c(X,\cG_0)\cong \cU_{\loc}^{\cont}(\Shv(X;\un{\cC}).
\end{equation*} 
We also have a $\msK$-presheaf 
\begin{equation*}
\cF_0':\msK(X)^{op}\to \Mot^{\loc},\quad \cF_0'(Y) = \cU_{\loc}^{\cont}(\un{\cC}(Y)).
\end{equation*}
Clearly, $\cF_0'$ is a continuous $\msK$-presheaf. Moreover, we have a map $\cF_0'\to\cF_0$ induced by the strongly continuous functors $\un{\cC}(Y)\to \Shv(Y;\un{\cC}_{\mid Y})$ (left adjoint to the global section functors). If we denote by $\cG_0'\in \PSh^{\cont}(X;\Mot^{\loc})$ the continuous presheaf corresponding to $\cF_0',$ then we get a map $\cG_0'\to\cG,$ hence also a map of sheaves $\cG_0^{'\sharp}\to \cG_0.$ Note that by construction $\cG_0^{'\sharp}$ is simply the sheafification of the presheaf $\cU_{\loc}^{\cont}(\un{\cC}),$ given by $U\mapsto \cU_{\loc}^{\cont}(\un{\cC}(U)).$ This gives a natural map
\begin{equation}\label{eq:construction_of_isom_for_U_loc_of_sheaves}
\Gamma_c(X,\cU_{\loc}^{\cont}(\un{\cC})^{\sharp})\to \Gamma_c(X,\cG_0)\cong \cU_{\loc}^{\cont}(\Shv(X;\un{\cC})).
\end{equation}


\begin{theo}\label{th:U_loc_locally_compact} Let $X$ be a locally compact Hausdorff space and $\un{\cC}$ is a presheaf on $X$ with values in $\Cat_{\st}^{\dual}.$ Then the category $\Sh(X;\un{\cC})$ is dualizable and the above map \eqref{eq:construction_of_isom_for_U_loc_of_sheaves} is an isomorphism in $\Mot^{\loc}:$
$$\Gamma_c(X,(\cU_{\loc}^{\cont}(\un{\cC}))^{\sharp})\xto{\sim} \cU_{\loc}^{\cont}(\Sh(X;\un{\cC})).$$
Hence, the same also holds for any finitary localizing invariant, for example, for $K$-theory.\end{theo}

We first prove the dualizability statement (this is certainly known, but we include a proof for completeness).

\begin{prop}\label{prop:Shv_is_dualizable}
Let $(X,\un{\cC})$ be as in Theorem \ref{th:U_loc_locally_compact}. Then the category $\Shv(X;\un{\cC})$ is dualizable.
\end{prop}

\begin{proof}
Given an open subset $U\subset X$ we denote by $(-)_U:\un{\cC}(U)\to\Sh(X;\un{\cC})$ the left adjoint to the functor $\cF\mapsto\cF(U).$ Namely, for $P\in\un{\cC}(U),$ the sheaf $P_U$ is the sheafification of the presheaf
$$V\mapsto\begin{cases}\res_{U,V}(P) & \text{if }V\subset U;\\
	0 & \text{else.}\end{cases}$$
We observe the following: for any compact inclusion $V\Subset U$ of open subsets, for any objects $P,Q\in\un{\cC}(U),$ and for any {\it compact} map $P\to Q$ in $\un{\cC}(U),$ the induced map of sheaves $\res_{U,V}(P)_V\to Q_U$ is compact in $\Sh(X;\un{\cC}).$ Indeed, this follows from the factorization 
\begin{multline*}\Hom(Q_U,-)\to \Hom(\res_{U,\bar{V}}(Q),\Gamma(\bar{V},-))\to \Hom(\res_{U,\bar{V}}(P),\Gamma(\bar{V},-))\\
	\to \Hom(\res_{U,V}(P)_V,-)
\end{multline*}
and the fact that the functor $\Gamma(\bar{V},-):\Sh(X;\un{\cC})\to\un{\cC}(V)$ is continuous.

Now observe that we have a generating collection of sheaves $$\{P_U\mid U\text{-- countable at }\infty,\,P\in\un{\cC}(U)^{\omega_1}\}.$$
For any such object $P_U,$ since $\un{\cC}(U)$ is a dualizable category, we can choose a sequence $Q_1\to Q_2\to\dots$ in $\un{\cC}(U)^{\omega_1}$ such that $P=\indlim[n] Q_n$ and all the maps $Q_n\to Q_{n+1}$ are compact. Choose a sequence $V_1\Subset V_2\Subset\dots$ of open subsets and compact inclusions such that $U=\bigcup_n V_n$ and each $V_n$ is countable at $\infty.$ Then $$P_U\cong \indlim[n]\res_{U,V_n}(Q_n)_{V_n},$$ and by the above discussion each map $\res_{U,V_n}(Q_n)_{V_n}\to \res_{U,V_{n+1}}(Q_{n+1})_{V_{n+1}}$ is compact. This shows that the category $\Sh(X;\un{\cC})$ is dualizable.
\end{proof}

Next, we prove the above assertions about a $\Mot^{\loc}$-valued $\msK$-sheaf on $X$ and the compactly supported global sections of the associated sheaf.

\begin{prop}\label{prop:U_loc_of_sheaves_gives_K_sheaf} Let $(X,\un{\cC})$ be as in Theorem \ref{th:U_loc_locally_compact}.

1)  The assignment 
\begin{equation*}
Y\mapsto \cF_0(Y)=\cU_{\loc}^{\cont}(\Shv(Y;\un{\cC}_{\mid Y})),\quad Y\in\msK(X),
\end{equation*} 
defines a $\msK$-sheaf $\cF_0$ on $X$ with values in $\Mot^{\loc}.$

2) As above, denote by $\cG_0\in\Shv(X;\Mot^{\loc})$ the sheaf corresponding to the $\msK$-sheaf $\cF_0$ via the equivalence from Proposition \ref{prop:PSh_and_PSh_K}. Then we have a natural isomorphism $\Gamma_c(X;\cG_0)\cong \cU_{\loc}^{\cont}(\Shv(X;\un{\cC})).$ 
\end{prop}

\begin{proof}
1) Clearly, we have $\cF_0(\emptyset) = 0.$ For a compact $Y\subset X$ we have an equivalence $\indlim[Z\Supset Y]^{\cont}\Shv(Z;\un{\cC}_{\mid Z})\xto{\sim}\Shv(Y;\un{\cC}_{\mid Y}),$ which is easily seen by identifying the colimit with the corresponding limit in $\Pr^R_{\st}.$ Finally, for compact subsets $Y_1,Y_2\subset X$ the square
\begin{equation*}
\begin{tikzcd}
\Shv(Y_1\cup Y_2;\un{\cC}_{\mid Y_1\cup Y_2}) \ar[r]\ar[d] & \Shv(Y_1;\un{\cC}_{\mid Y_1})\ar[d]\\
\Shv(Y_2;\un{\cC}_{\mid Y_2}) \ar[r] & \Shv(Y_1\cap Y_2;\un{\cC}_{\mid Y_1\cap Y_2})
\end{tikzcd}
\end{equation*}
is cartesian in $\Pr^L_{\st}$ and both right vertical and lower horizontal arrows are localizations. Hence, by Proposition \ref{prop:F_cont_nice_pullbacks} the above square is also a pullback square in $\Cat_{\st}^{\dual}$ and we have an equivalence $\cF_0(Y_1\cup Y_2)\xto{\sim}\cF_0(Y_1)\times_{\cF_0(Y_1\cap Y_2)}\cF_0(Y_2).$ This proves that $\cF_0$ is a $\msK$-sheaf.

2) If $X$ is compact, then the assertion is tautological. The general case follows by taking the one-point compactification of $X.$ However, it is probably more instructive to give a direct computation as follows.

For an open $U\Subset X$ we have natural isomorphisms 
\begin{equation}\label{eq:Gamma_c_with_coeff_in_G}
\Gamma_c(U,\cG_0)\xto{\sim} \Fiber(\cF_0(\bbar{U})\to \cF_0(\bbar{U}\setminus U))\xto{\sim} \cU_{\loc}^{\cont}(\Shv(U;\un{\cC}_{\mid U})).
\end{equation}
The latter isomorphism is obtained from the short exact sequence in $\Cat_{\st}^{\dual}:$
\begin{equation*}
0\to \Shv(U;\un{\cC}_{\mid U})\to \Shv(\bbar{U};\un{\cC}_{\mid \bbar{U}})\to \Shv(\bbar{U}\setminus U;\un{\cC}_{\mid (\bbar{U}\setminus U)})\to 0.
\end{equation*}
Here the inclusion functor is given by extension by zero. The composite isomorphism \eqref{eq:Gamma_c_with_coeff_in_G} is naturally functorial in $U,$ where for $U\subset V\Subset X$ the morphism $\cU_{\loc}^{\cont}(\Shv(U;\un{\cC}_{\mid U}))\to \cU_{\loc}^{\cont}(\Shv(V;\un{\cC}_{\mid V}))$ is induced by the extension by zero functor between the categories of sheaves. We obtain
\begin{equation*}
\Gamma_c(X,\cG_0)\cong \indlim[U\Subset X]\Gamma_c(U,\cG_0)\cong \indlim[U\Subset X]\cU_{\loc}^{\cont}(\Shv(U;\cC_{\mid U}))\cong \cU_{\loc}^{\cont}(\Shv(X;\un{\cC})),
\end{equation*}
as stated.
\end{proof}

The idea of the proof of Theorem \ref{th:U_loc_locally_compact} is the following. We reduce to the case when $X$ is compact. Then we will use the presentation of the category of sheaves as a directed colimit of simpler categories as in Proposition \ref{prop:approximating_sheaves}. This allows to reduce the computation to these approximations, which essentially follows from the already established results. The details are as follows.

First, note that we already know the finitary localizing invariants of the categories of continuous $\msK$-presheaves.

\begin{prop}\label{prop:U_loc_cont_presheaves} The category $\PSh_{\msK}^{\cont}(X;\un{\cC})$ is dualizable and we have
$$\cU_{\loc}^{\cont}(\PSh_{\msK}^{\cont}(X;\un{\cC}))\cong\biggplus[\substack{Y\subset X\\ \text{open-compact}}]\cU_{\loc}^{\cont}(\un{\cC}(Y)).$$\end{prop}

\begin{proof}This is a special case of Theorem \ref{th:U_loc_cosheaves_on_cont_posets} since the compact elements of the continuous poset $\msK(X)^{op}$ are exactly the open-compact subsets of $X.$\end{proof}





We need the following observation about the categories from Proposition \ref{prop:approximating_sheaves}.

\begin{prop}\label{prop:U_loc_of_approximations} Suppose that $X$ is compact.  For a sieve $S\in J_X$ the natural functor
$$\prolim[Y\in S]^{\dual}\PSh_{\msK}^{\cont,\red}(Y;\un{\cC}_{\mid Y})\to \prolim[Y\in S]\PSh_{\msK}^{\cont,\red}(Y;\un{\cC}_{\mid Y})$$
is an equivalence. Moreover, we have $$\cU_{\loc}^{\cont}(\prolim[Y\in S]^{\dual}\PSh_{\msK}^{\cont,\red}(Y;\un{\cC}_{\mid Y}))\cong \prolim[Y\in S]\cU_{\loc}^{\cont}(\PSh_{\msK}^{\cont,\red}(Y;\un{\cC}_{\mid Y})).$$\end{prop}

\begin{proof} This is basically an iterated application of Propositions \ref{prop:F_cont_nice_pullbacks} and \ref{prop:approximating_sheaves}. For convenience we may and will assume that $\un{\cC}(\emptyset)=0.$ We prove the statement for a sieve $S=S_{(X;Y_1,\dots,Y_n)}$ by induction on $n,$ where the case $n=1$ is clear. 

Let $n>1.$ We put $$Z=Y_2\cup\dots\cup Y_n,\quad S'=S_{(Z;Y_2,\dots,Y_n)},\quad S''=S_{(Y_1\cap Z;Y_1\cap Y_2,\dots,Y_1\cap Y_n)}.$$ Using the induction hypothesis, we see that the functors
$$\prolim[Y\in S']^{\dual}\PSh_{\msK}^{\cont}(Y;\un{\cC}_{\mid Y})\to\prolim[Y\in S']\PSh_{\msK}^{\cont}(Y;\un{\cC}_{\mid Y}),\quad \prolim[Y\in S'']^{\dual}\PSh_{\msK}^{\cont}(Y;\un{\cC}_{\mid Y})\to\prolim[Y\in S'']\PSh_{\msK}^{\cont}(Y;\un{\cC}_{\mid Y})$$ are equivalences. 

Applying Propositions \ref{prop:closed_embedding_pullback_PSh^cont} and \ref{prop:approximating_sheaves}, we see that the composition
$$\PSh_{\msK}^{\cont}(Y_1;\un{\cC}_{\mid Y_1})\to\PSh_{\msK}^{\cont}(Y_1\cap Z;\un{\cC}_{\mid Y_1\cap Z})\to \prolim[Y\in S'']\PSh_{\msK}^{\cont}(Y;\un{\cC}_{\mid Y})$$ is a localization (both functors are localizations). It follows from Proposition \ref{prop:F_cont_nice_pullbacks} that the pullback square
$$\begin{tikzcd}
\prolim[Y\in S]^{\dual}\PSh_{\msK}^{\cont}(Y;\un{\cC}_{\mid Y}) \ar{r}\ar{d} & \prolim[Y\in S']^{\dual}\PSh_{\msK}^{\cont}(Y;\un{\cC}_{\mid Z})\ar{d}\\
\PSh_{\msK}^{\cont}(Y_1;\un{\cC}_{\mid Y_1}) \ar{r} & \prolim[Y\in S'']^{\dual}\PSh_{\msK}^{\cont}(Y;\un{\cC}_{\mid Y})
\end{tikzcd}$$
in $\Cat_{\st}^{\dual}$ is also a pullback square in $\Pr_{\st}^L,$ and $\cU_{\loc}^{\cont}$ applied to this square is a pullback square in $\Mot^{\loc}.$ This proves the proposition.
\end{proof}

As in Subsection \ref{ssec:analogy_dualizable_comphaus}, for a topological space $Y$ we denote by $\clopen(Y)$ the Boolean algebra of open-closed subsets of $Y.$

\begin{prop}\label{prop:funny_operation} Let $X$ be locally compact Hausdorff space, and $\cE$ a presentable stable category. Given a reduced $\msK$-presheaf $\cF\in\PSh_{\msK}^{\red}(X;\cE),$ we define another $\msK$-presheaf $\wt{\cF}$ by the formula
$$\wt{\cF}(Y):=\bigoplus\limits_{Z\in \clopen(Y)}\cF(Y).$$ We have a natural map $\cF\to\wt{\cF}$ which is given on $Y\in\msK(X)$ by the inclusion of $\cF(Y)$ as the direct summand of $\wt{\cF}(Y)$ corresponding to $Y\in\clopen(Y).$

1) If $\cF$ is continuous, then so is $\wt{\cF}.$

2) Assuming that $\cF$ is continuous, the map of $\msK$-sheaves $\cF^{\sharp}\to\wt{\cF}^{\sharp}$ is an isomorphism.\end{prop}

\begin{proof}1) This follows immediately from the following observation: for any compact $Y\subset X,$ the map
$$\liminj\limits_{Y'\Supset Y}\clopen(Y')\to\clopen(Y)$$ is a bijection of sets (this is a special case of Proposition \ref{prop:Q_commutes_with_cofiltered_limits}),

2) Recall that $\cF^{\sharp}$ is the sheafification of $\cF$ for the finitary topology on $\msK(X).$ Consider the weaker pretopology on $\msK(X),$ for which a family $\{Y_i\to Y\}_{i\in I}$ is a covering if $I$ is finite and $Y=\bigsqcup_{i\in I}Y_i.$ Denote by $(-)^{\sharp,weak}$ the corresponding sheafification functor. It suffices to show that the map $\cF^{\sharp,weak}\to\wt{\cF}^{\sharp,weak}$ is an isomorphism. 

Let us observe that the weak sheafification functor on reduced $\msK$-presheaves is given by
$$\cF\mapsto (Y\mapsto \cF^{\sharp,weak}(Y)=\indlim[Y=\bigsqcup_i Y_i]\biggplus[i]\cF(Y_i),$$
where the colimit is over the decompositions of $Y$ into a finite disjoint union of compact subsets. This implies that the map $\cF^{\sharp,weak}\to\wt{\cF}^{\sharp,weak}$ is an isomorphism. To see this, consider the (continuous, reduced) $\msK$-presheaf $\cG=\Cone(\cF\to\wt{\cF}),$ so that
\begin{equation*}
\cG(Y) \cong \biggplus[\substack{Z\in\clopen(Y),\\ Z\ne Y}] \cF(Z),\quad Y\in\msK(X).
\end{equation*}
We need to show that $\cG^{\sharp,weak} = 0.$ Take $Y\in\msK(X)$ such that $Y\ne \emptyset,$ and consider the following poset $T_{Y}.$ Its elements are collections $(\{Y_i\}_{i\in I}, \{M_i\}_{i\in I}),$ where  $I$ is a finite set, $Y_i\subset Y$ are non-empty compact subsets such that $\bigsqcup\limits_{i\in I} Y_i = Y,$ and $M_i\subset \clopen(Y_i)\setminus\{Y_i\}$ are finite subsets. The partial order is the following:
\begin{equation*}
(\{Y_i\}_{i\in I}, \{M_i\}_{i\in I})\leq (\{Y'_j\}_{j\in J}, \{M'_j\}_{j\in J})
\end{equation*}
if each $Y'_j$ is contained in some $Y_i,$ and for such $i$ and $j$ for each $Z\in M_i$ we have either $Y'_j\subset Z$ or $Z\cap Y'_j\in M_j'.$ Then $T_{Y}$ is directed and $\cG^{\sharp,weak}(Y)$ is identified with the colimit of the functor
\begin{equation*}
\Phi:T_{Y}\to \cE,\quad \Phi((\{Y_i\}_{i\in I}, \{M_i\}_{i\in I}) = \biggplus[i\in I,\, Z\in M_i] \cF(Z).
\end{equation*}
But this ind-system is ind-zero in $\cE.$ Indeed, let $T'_{Y}\subset T_{Y}$ be the full subposet consisting of collections $(\{Y_i\},\{M_i\})$ such that $M_i=\{\emptyset\}$ for all $i.$ Then $\Phi_{\mid P'_{Y}}=0.$ It suffices to observe that $T'_{Y}$ is cofinal in $T_{Y}.$ Namely, for an element $(\{Y_i\}_{i\in I}, \{M_i\}_{i\in I})\in T_{Y}$ we can take $\{Y'_j\}_{j\in J}$ to be the collection of atomic elements of the (finite) Boolean subalgebra of $\clopen(Y)$ generated by all $Y_i$ and $Z\in M_i.$ Putting $M'_j=\{\emptyset\},$ $j\in J,$ we obtain $(\{Y_i\},\{M_i\})\leq (\{Y'_{j}\},\{M'_j\}).$ Hence, $T'_{Y}$ is cofinal in $T_{Y},$ as stated. 
\end{proof}

The reduction to the case when $X$ is compact is as follows.

\begin{prop}\label{prop:compatibility_of_maps_after_compactification}
Denote by $\bar{X}=X\cup\{\infty\}$ the one-point compactification of $X,$ and denote by $j:X\to\bar{X}$ the embedding. Let $j_!\un{\cC}$ be the presheaf of dualizable categories on $\bar{X}$ obtained by extension by zero of $\un{\cC},$ i.e.
\begin{equation*}
	j_!\un{\cC}(U) = \begin{cases}
		\un{\cC}(U) & \text{if }\infty\not\in U;\\
		0 & \text{if }\infty\in U.
	\end{cases}
\end{equation*} Then we have an equivalence $j^*:\Sh(\bar{X},j_!\un{\cC})\xto{\sim}\Sh(X;\un{\cC}).$ Moreover, the morphisms \eqref{eq:construction_of_isom_for_U_loc_of_sheaves} for pairs $(X,\un{\cC})$ and $(\bar{X},j_!\un{\cC})$ are naturally identified, i.e. we have a commutative square
\begin{equation}\label{eq:compatibility_of_maps_after_compactification}
\begin{tikzcd}
\Gamma(\bar{X},\cU_{\loc}^{\cont}(j_! \un{\cC})^{\sharp})\ar[r]\ar{d}{\sim} & \cU_{\loc}^{\cont}(\Sh(\bar{X};j_! \un{\cC}))\ar{d}{\sim}\\
\Gamma_c(X,\cU_{\loc}^{\cont}(\un{\cC})^{\sharp})\ar[r] & \cU_{\loc}^{\cont}(\Sh(X;\un{\cC})).
\end{tikzcd}
\end{equation}
\end{prop}

\begin{proof}
The equivalence between the categories of sheaves is evident: its inverse is given by extension by zero. In particular we already have the right vertical equivalence in \eqref{eq:compatibility_of_maps_after_compactification}. Next, the presheaf $\cU_{\loc}^{\cont}(j_!\un{\cC})$ on $\bar{X}$ is clearly identified with the extension by zero of the presheaf $\cU_{\loc}^{\cont}(\un{\cC})$ on $X.$ This gives the left vertical isomorphism in \eqref{eq:compatibility_of_maps_after_compactification}.

Consider the sheaf $\cG_0\in\Sh(X;\Mot^{\loc})$ from Proposition \ref{prop:U_loc_of_sheaves_gives_K_sheaf}, and denote by $\bbar{\cG}_0\in \Sh(\bar{X};\Mot^{\loc})$ a similar sheaf for $(\bar{X},j_!\un{\cC}).$ By construction we have a natural isomorphism $\bbar{\cG}_0\cong j_!\cG_0,$ and this identification is compatible with the isomorphisms 
\begin{equation*}
\Gamma_c(X,\cG_0)\cong \cU_{\loc}^{\cont}(\Sh(X;\un{\cC})),\quad \Gamma(\bar{X},\bbar{\cG}_0)\cong \cU_{\loc}^{\cont}(\bar{X},j_!\un{\cC}).
\end{equation*}
Finally, the map $\cU_{\loc}^{\cont}(j_!\un{\cC})^{\sharp}\to \bbar{\cG}_0$ (of sheaves on $\bar{X}$) is identified with the extension by zero of the map $\cU_{\loc}^{\cont}(\un{\cC})^{\sharp}\to \cG_0.$ This shows that the square \eqref{eq:compatibility_of_maps_after_compactification} naturally commutes.
\end{proof}

\begin{proof}[Proof of Theorem \ref{th:U_loc_locally_compact}.]
By Proposition \ref{prop:compatibility_of_maps_after_compactification} we may and will assume that $X$ is compact. Also, we may and will assume that $\un{\cC}(\emptyset)=0.$

Consider the above $\msK$-presheaf $\cF_0'\in\PSh_{\msK}^{\cont}(X;\Mot^{\loc}),$ given by $Y\mapsto \cU_{\loc}^{\cont}(\un{\cC}(Y)).$ As we already mentioned, $\cF_0'$ is a continuous $\msK$-presheaf. By Proposition \ref{prop:U_loc_cont_presheaves}, for a compact $Y\subset X$ we have a natural isomorphism
$$\cU_{\loc}^{\cont}(\PSh^{\cont}_{\msK}(Y;\un{\cC}_{\mid Y}))\cong \wt{\cF_0'}(Y)=\bigoplus\limits_{Z\in\clopen(Y)}\cF_0'(Z).$$
Applying Propositions \ref{prop:approximating_sheaves}, \ref{prop:U_loc_of_approximations} and \ref{prop:funny_operation}, we obtain the following chain of isomorphisms
\begin{multline*}\cF_0^{'\sharp}(X)\cong\wt{\cF_0'}^{\sharp}(X)=\liminj\limits_{S\in J_X}\limproj\limits_{Y\in S}\wt{\cF_0'}(Y)\cong \liminj\limits_{S\in J_X}\limproj\limits_{Y\in S}\cU_{\loc}^{\cont}(\PSh^{\cont}_{\msK}(Y;\un{\cC}_{\mid Y}))\\ \cong
\liminj\limits_{S\in J_X}\cU_{\loc}^{\cont}(\limproj^{\dual}\limits_{Y\in S}\PSh^{\cont}_{\msK}(Y;\un{\cC}_{\mid Y}))\cong \cU_{\loc}^{\cont}(\liminj\limits_{S\in J_X} \limproj^{\dual}\limits_{Y\in S}\PSh^{\cont}_{\msK}(Y;\un{\cC}_{\mid Y}))\\
\cong \cU_{\loc}^{\cont}(\Sh_{\msK}(X;\un{\cC})).\end{multline*}
By construction, the composite isomorphism is identified with the morphism \eqref{eq:construction_of_isom_for_U_loc_of_sheaves}. This proves the theorem.
\end{proof}

\subsection{Compact objects}
\label{ssec:compact_sheaves}

Let 
$(X,\un{\cC})$ 
be as above. The following result describes the category of compact objects of the category 
$\Sh(X;\un{\cC}).$ It is a generalization of \cite[Theorem 0.1]{Nee01b}, and the proof is based on the ideas of loc. cit.

\begin{prop}\label{prop:compact_sheaves} Let $\cF\in\Sh(X;\un{\cC})$ be a sheaf. The following are equivalent.

\begin{enumerate}[label=(\roman*),ref=(\roman*)]
\item $\cF$ is a compact object in the category $\Sh(X;\un{\cC}).$ \label{comp_sheaf1}

\item $\cF$ has compact support and $X$ can be covered by open subsets $U$ such that $\cF_{\mid U}\cong P_U$ for some compact object $P\in\un{\cC}(U)^{\omega}.$ \label{comp_sheaf2}
\end{enumerate}\end{prop}

\begin{proof} \Implies{comp_sheaf2}{comp_sheaf1}. First suppose that $\cF$ only satisfies the second condition of \ref{comp_sheaf2} (but not necessarily has compact support). We claim that the functor
\begin{equation*}
\Phi = \un{\Hom}(\cF,-):\Shv(X;\un{\cC})\to \Shv(X;\Sp),\quad \Phi(\cG)(U) = \Hom_{\Shv(U;\un{\cC}_{\mid U})}(\cF_{\mid U},\cG_{\mid U}),
\end{equation*} 
commutes with colimits. It suffices to consider the case when $\cF = P_X$ for some compact object $P\in\un{\cC}(X)^{\omega}.$ Then for any compact subset $Y\subset X$ we have $\Phi(\cG)(Y)\cong \Hom(\res_{X,Y} (P), \cG(Y)).$ The latter functor commutes with colimits (in $\cG$), hence so does $\Phi.$

Now suppose that $\cF$ also has compact support and put $Z = \supp(\cF).$ Then $\Phi$ takes values in the full subcategory of sheaves supported on $Z.$ We obtain isomorphisms 
\begin{equation*}
\Hom(\cF,-)\cong \Gamma(X,\Phi(-))\cong \Gamma_c(X,\Phi(-)).
\end{equation*}
The latter functor commutes with colimits, hence $\cF$ is compact. This proves the implication.

\Implies{comp_sheaf1}{comp_sheaf2}. We first make the following observation: for a sheaf $\cG\in\Shv(X;\un{\cC})$ we have an isomorphism
\begin{equation}\label{eq:approx_of_sheaves_by_extension_by_zero}
\indlim[U\Subset X]j_{U!}j_U^*\cG\xto{\sim}\cG.
\end{equation}
This can be considered as a special case of Proposition \ref{prop:colimit_of_presentable_compositions_of_adjoints} since we have $\indlim[U\Subset X]^{\cont}\Shv(U;\un{\cC}_{\mid U})\simeq \Shv(X;\un{\cC}).$ But this is also easy to see directly: for any $V\Subset X$ the ind-system $(j_V^*j_{U!}j_U^*\cG)_{U\Subset X}$ is eventually constant with value $j_V^*\cG.$

Now suppose that $\cF$ is compact. We first show the following: if $\cF_x=0\in\un{\cC}_x$ for some $x\in X,$ then there is an open neighborhood $U\ni x$ such that $\cF_{\mid U}=0.$ Indeed, the assumption implies that $\cF$ is extended by zero from $X\setminus\{x\}.$ Applying \eqref{eq:approx_of_sheaves_by_extension_by_zero} to $\cF_{\mid X\setminus\{x\}}$ and extending by zero to $X$ we obtain $\cF\cong\indlim[V\Subset X\setminus \{x\}]j_{V!}j_V^*(\cF).$  Since $\cF$ is compact, we deduce that $\cF$ is a direct summand of some sheaf of the form $j_{V!}j_V^*(\cF),$ where $V\Subset X\setminus \{x\}.$ Hence the restriction of $\cF$ to $(X\setminus\bar{V})\ni x$ vanishes.

Next, let $x\in X$ be a point such that $\cF_x\ne 0.$ Since the functor $i_x^*:\Sh(X;\un{\cC})\to\un{\cC}_x$ is strongly continuous, the object $\cF_x$ is compact in $\un{\cC}_x.$ By Proposition \ref{prop:compact_objects_in_filtered_colimits} there is an open neighborhood $U\ni x$ and a compact object $P\in\un{\cC}(U)^{\omega}$ such that $\res_{U,\{x\}}(P)\cong \cF_x.$ Shrinking $U$ if necessary, we can lift the identity map of $\cF_x$ to a map $P\to\cF(U)$ in the category $\un{\cC}(U).$ Denote by $\varphi:P_U\to\cF$ the associated map of $\un{\cC}$-valued sheaves on $X.$ Choosing an open $V\Subset U$ such that $x\in V,$ we see that $\Cone(\varphi)_{\mid \bar{V}}$ is a compact object of $\Sh(\bar{V};\un{\cC}_{\mid\bar{V}})$ whose stalk at $x$ vanishes. By the above, we deduce that $\Cone(\varphi)$ vanishes in a neighborhood of $x.$ Hence, shrinking $U$ further, we may assume that $\cF_{\mid U}\cong P_U$ as required.

Finally, to see that $\cF$ has compact support, we simply apply \eqref{eq:approx_of_sheaves_by_extension_by_zero} to $\cF.$ Since $\cF$ is compact, it is a direct summand of some sheaf $j_{U!}j_U^*\cF,$ $U\Subset X,$ which has compact support.
\end{proof}

\begin{cor}\label{cor:when_sheaves_comp_gen} Let $X$ be a non-empty locally compact Hausdorff space, and let $\cC$ be a non-zero presentable stable category. The following are equivalent.

\begin{enumerate}[label=(\roman*),ref=(\roman*)]
\item The category $\Sh(X;\cC)$ is compactly generated. \label{comp_gen1}

\item The category $\cC$ is compactly generated and $X$ is locally profinite. \label{comp_gen2}
\end{enumerate}
\end{cor}

\begin{proof} \Implies{comp_gen2}{comp_gen1} Consider the collection of sheaves $$S=\{P_U\in\Sh(X;\cC): P\in\cC^{\omega},\,U\subset X\text{ is open-compact}\}.$$ 
Then each sheaf in $S$ is compact and the right orthogonal to $S$ consists of sheaves $\cF$ such that $\cF(U)=0$ for every open-compact $U\subset X.$ Since $X$ is locally profinite, this means that $\cF=0.$ Hence, $S$ generates the category $\Sh(X;\cC).$

\Implies{comp_gen1}{comp_gen2} Choose a point $x\in X.$ The pullback to $\{x\}$ is a strongly continuous localization $\Sh(X;\cC)\to\cC.$ Since the category $\Sh(X;\cC)$ is compactly generated, so is $\cC.$

Next, we reduce to the case when $X$ is compact. If not, consider the short exact sequence
$$0\to \Sh(X;\cC)\to \Sh(X\cup\{\infty\};\cC)\to\cC\to 0$$
in $\Cat_{\st}^{\dual}.$ Since the categories on the left and on the right are compactly generated, so is the category in the middle (by Proposition \ref{prop:when_extension_is_dualizable}). Hence, we may and will assume that $X$ is compact.

Suppose that $X$ is not profinite. Choose a connected component $Y\subset X$ such that $Y$ is not a single point. Then $Y$ is closed, hence compact Hausdorff. The functor $\Sh(X;\cC)\to\Sh(Y;\cC)$ is a strongly continuous localization. Since the category $\Sh(X;\cC)$ is compactly generated, so is the category $\Sh(Y;\cC).$ Replacing $X$ with $Y,$ we may and will assume that $X$ is connected.

Choose a point $x_0\in X.$ Choose a non-zero compact object $P\in\cC^{\omega},$ and consider the skyscraper sheaf $P_{x_0}$ on $X.$ Then we can find a directed system $(\cF_i)_{i\in I}$ of compact $\cC$-valued sheaves such that $\indlim[i]\cF_i\cong P_{x_0}.$ By Proposition \ref{prop:compact_sheaves} each $\cF_i$ is a locally constant sheaf with compact stalks. 

Since $P\cong\indlim[i]\cF_{i,x_0}$ is compact in $\cC,$ we can find $i_0\in I$ and a splitting $\alpha:P\to \cF_{i_0,x_0}.$ 
Using compactness of $P$ again, we can find an open neighborhood $U\ni x_0$ and a map $\beta:P\to\cF_{i_0}(U)$ which lifts $\alpha.$ Choose a compact $Y\subset U$ such that $Y$ is a neighborhood of $x_0.$ Let $Z$ be the connected component of $x_0$ in $Y.$ Then $Z\neq\{x_0\},$ hence we may and will assume that $X=U=Y=Z.$

So we have the map $\beta:P\to\cF_{i_0}(X).$ Take some $i\geq i_0.$ Since the sheaf $\cF_i$ is locally constant, the set
$$T_i=\{y\in X: \text{the composition }P\xto{\beta}\cF_{i_0}(X)\to\cF_i(X)\to\cF_{i,y}\text{ is non-zero}\}$$
is open-closed in $X.$ We have $x_0\in T_i.$ Since $X$ is connected, we have $T_i=X.$ Choosing a point $x_1\in X\setminus\{x_0\},$ we see that the map $P\to\indlim[i]\cF_{i,x_1}$ is non-zero. But $\indlim[i]\cF_{i,x_1}\cong (P_{x_0})_{x_1}=0,$ a contradiction.  
\end{proof}

\appendix

\section{Phantom maps, pure injectivity and Adams representability}
\label{app:Adams_representability}

In this section we prove two closely related analogues of Adams representability theorem for $\omega_1$-compact dualizable categories. 
We first clarify what we mean by Adams representability and Brown representability, and recall the known results due to Neeman.

Recall from \cite[Definition 8.2.1]{Nee01} that a triangulated category $T$ with infinite coproducts is said to satisfy Brown representability if any product-preserving cohomological functor $T^{op}\to\Ab$ is representable. As we already mentioned in Subsection \ref{ssec:well_generated}, by \cite[Theorem 1.17]{Nee01} any well generated triangulated category satisfies Brown representability. We will crucially use this result below.

One can also consider a much more subtle question. Suppose that $T$ is compactly generated. Let us say that $T$ satisfies Adams representability if any cohomological functor $(T^{\omega})^{op}\to\Ab$ is isomorphic to a restriction of some representable functor $\Hom_T(-,x)_{\mid T^{\omega}},$ where $x\in T.$ By \cite[Theorem 5.1, Proposition 4.11]{Nee97}, if the category $T^{\omega}$ is countable, then $T$ satisfies Adams representability. The countability condition menas that $T^{\omega}$ has at most countably many isomorphism classes of objects, and the sets of morphisms are at most countable. If $T=\h\cC$ for a compactly generated presentable stable category $\cC,$ then the countability of $T^{\omega}$ equivalently means that $\cC^{\omega}\in (\Cat^{\perf})^{\omega_1},$ which is also equivalent to the condition $\cC\in(\Cat_{\st}^{\dual})^{\omega_1}.$ The Adams representability for $\h\cC$ almost directly implies that any homological functor $(\h\cC^{\vee})^{\omega}\to\Ab$ is of the form $\pi_0\ev_{\cC}(x,-)_{\mid (\h\cC^{\vee})^{\omega}}$ for some $x\in\cC.$ This statement can be considered as Adams representability for covariant functors (or ``homology theories''). 

We will prove the generalizations of the above Adams representability theorems for $\omega_1$-compact dualizable categories: Theorem \ref{th:Adams_rep_contravar} for contravariant functors and Theorem \ref{th:Adams_rep_covar} for covariant functors. We now explain some preliminary notions for the formulation of Theorem \ref{th:Adams_rep_contravar}.

Let $\cC$ be a dualizable category, and denote by $T=\h\cC,$ $T'=\h\cC^{\vee}$ the triangulated homotopy categories of $\cC$ and $\cC^{\vee}.$ We say that a morphism in $T$ is compact if it is an image (i.e. homotopy class) of a compact morphism in $\cC.$ The following are straightforward generalizations of the usual notions for compactly generated categories \cite{Kr00, GP04}.

\begin{defi}\label{defi:phantom_pure_epi_pure_mono} Let $f:x\to y$ be a morphism in $T.$
	
	1) $f$ is phantom if for any compact morphism $g:z\to x$ with $z\in T^{\omega_1}$ we have $f\circ g=0.$
	
	2) $f$ is a pure monomorphism if the map $\Fiber(f)\to x$ is phantom.
	
	3) $f$ is a pure epimorphism if the map $y\to\Cone(f)$ is phantom. 
\end{defi}

We will consider the following categories of almost modules. 
We denote by $J$ resp. $J'$ the ideal of compact maps in $T^{\omega_1}$ resp. $T^{'\omega_1}.$ Recall from Subsection \ref{ssec:dual_via_compact} that $J^2=J,$ $J^{'2}=J'.$

\begin{defi}1) We say that a right $T^{\omega_1}$-module $M:T^{\omega_1,op}\to\Ab$ is almost zero if it is annihilated by $J.$ In other words, for any compact morphism $f:x\to y$ in $T^{\omega_1},$ the map $M(f):M(y)\to M(x)$ is zero. The definition of an almost zero left $T^{'\omega_1}$-module is similar.
	
	2) We denote by $\Mod_a\hy T^{\omega_1}$ the quotient of $\Mod\hy T^{\omega_1}=\Fun^{\add}(\cT^{\omega_1,op},\Ab)$ by the full (Serre) subcategory of almost zero right modules. Similarly, we denote by  $T^{'\omega_1}\hy\Mod_a$ the quotient of $T^{'\omega_1}\hy\Mod=\Fun^{\add}(T^{'\omega_1},\Ab)$ by the full subcategory of almost zero left modules.\end{defi}

The subcategories of almost zero modules are indeed Serre subcategories since $J^2=J$ and $J^{'2}=J'.$  If $T$ is compactly generated, then both categories $\Mod_a\hy T^{\omega_1}$ and $T^{'\omega_1}\hy\Mod_a$ are identified with $\Mod\hy T^{\omega}$ (Proposition \ref{prop:almost_modules_T^omega_1_basics} below). If $T$ is not compactly generated, then the categories $\Mod_a\hy T^{\omega_1}$ and $T^{'\omega_1}\hy\Mod_a$ are still equivalent (Proposition \ref{prop:equivalence_on_almost_modules}). In fact the category $T^{'\omega_1}\hy\Mod_a$ admits a more explicit description: it is identified with the category of additive functors $\cC^{\vee,\omega_1}\to\Ab$ which commute with sequential colimits (Corollary \ref{cor:left_almost_modules_explicit}). 

Recall that for a small triangulated category $S$ an additive functor $F:S\to\Ab$ is called homological if for any exact triangle $x\to y\to z$ in $S$ the sequence $F(x)\to F(y)\to F(z)$ is exact. Clearly, homological functors are closed under extensions in the abelian category of left $S$-modules (i.e. the category $\Fun^{\add}(S,\Ab)$ of additive functors $S\to\Ab$). We will also refer to homological functors $S^{op}\to \Ab$ as cohomological. We say that an almost module $M\in\Mod_a\hy T^{\omega_1}$ is almost cohomological if it is an image of a right $T^{\omega_1}$-module corresponding to a cohomological functor $T^{\omega_1,op}\to\Ab.$ 

\begin{theo}\label{th:Adams_rep_contravar} Let $\cC$ be a dualizable category which is $\omega_1$-compact in $\Cat_{\st}^{\dual}.$ Denote by $T$ the homotopy category of $\cC.$ Then the functor $G:T\to \Mod_a\hy T^{\omega_1},$ $G(x)=\Hom_T(-,x),$ is full and conservative, and its essential image is exactly the full subcategory of almost cohomological almost modules. More precisely, for $x,y\in T$ we have a functorial short exact sequence of abelian groups:
	\begin{equation}\label{eq:ses_for_Homs1}
		0\to \Ext^1(G(x),G(y[-1]))\to\Hom_T(x,y)\to \Hom(G(x),G(y))\to 0.
	\end{equation}
	The abelian group of phantom maps $x\to y$ is identified with $\Ext^1(G(x),G(y[-1])).$ Moreover, we have $\Ext^n(G(x),G(y)) = 0$ for $n\geq 2.$ The composition of two phantom maps is zero. \end{theo}

We first explain the relation with Neeman's theorem on Adams representability. Suppose that $\cC$ is compactly generated. Under the equivalence $\Mod_a\hy T^{\omega_1}\simeq \Mod\hy T^{\omega},$ almost cohomological almost $T^{\omega_1}$-modules correspond to cohomological $T^{\omega}$-modules. Moreover, $\cC$ is an $\omega_1$-compact object of $\Cat_{\st}^{\dual}$ if and only if the triangulated category $T^{\omega}$ is countable. Thus, Theorem \ref{th:Adams_rep_contravar} is a generalization of \cite[Theorem 5.1, Proposition 4.11]{Nee97} if we restrict to triangulated categories with an $\infty$-enhancement.

Before proving Theorem \ref{th:Adams_rep_contravar} we explain some basic properties of the abelian category $\Mod_a\hy T^{\omega_1}.$ First we recall the following characterization of (co)homological functors. Recall that if $\cA$ is a small additive category, then a right $\cA$-module $M$ is called fp-injective if for any finitely presented right $\cA$-module $N$ we have $\Ext^1_{\cA}(N,M)=0.$ We also recall that for a morphism $f:x\to y$ in $\cA$ a {\it weak kernel} of $f$ is a morphism $g:z\to x$ such that $f\circ g = 0,$ and for any morphism $g':z'\to x$ such that $f\circ g' = 0$ there exists a morphism $h:z'\to z$ such that $g' = g\circ h.$ The notion of a weak cokernel is defined dually.

\begin{prop}\label{prop:cohom_functors} Let $S$ be a small triangulated category and $F:S^{op}\to\Ab$ an additive functor. The following are equivalent. 
	
	\begin{enumerate}[label=(\roman*),ref=(\roman*)]
		\item $F$ is cohomological. \label{cohom1}
		
		\item $F$ is flat as a right $S$-module. \label{cohom2}
		
		\item $F$ is fp-injective as a right $S$-module. \label{cohom3}
	\end{enumerate}
\end{prop}

\begin{proof}
	For an object $x\in S,$ we denote by $h_x$ the corresponding representable right $S$-module.
	
	\Iff{cohom1}{cohom2}. Flatness of $F$ means that for any finitely presented $S$-module $N,$ any morphism $N\to F$ factors through a representable module. Let $N=\coker(h_x\to h_y)$ for some morphism $f:x\to y$ in $S.$ Then the map $N\to F$ is given by an element $\alpha\in\ker(F(y)\to F(x)).$ Note that $z=\Cone(f)$ is a weak cokernel of $f$ in $S.$ Hence, the map $N\to F$ factors through a representable module iff $\alpha$ is in the image of the map $F(z)\to F(y).$ We conclude that $F$ is flat if and only if $F$ is cohomological.
	
	\Iff{cohom1}{cohom3}. Let $N=\coker(h_x\xto{f} h_y)$ be a finitely presented $S$-module, for some morphism $f:x\to y$ in $S.$ Let $z=\Fiber(f).$ Then we have an exact sequence of $S$-modules
	$$h_z\to h_x\to h_y\to N\to 0.$$ 
	It follows that $\Ext^1(N,F)\cong \ker(F(x)\to F(z))/\im(F(y)\to F(x)).$ We conclude that $F$ is fp-injective if and only if $F$ is cohomological.
\end{proof}

Next, recall from Proposition \ref{prop:compact_maps_and_wavy_arrows} that the tensor product $I=J\tens{\cT^{\omega_1}}J$ is the quasi-ideal in $T^{\omega_1}$ given by $I(x,y)=\pi_0\Hom(\cY_{\cC}(x),\hat{\cY}_{\cC}(y)).$ We denote by $I'\cong J'\tens{T^{'\omega_1}}J'$ the similar quasi-ideal in $T^{'\omega_1},$ given by $I'(x,y)=\pi_0\Hom(\cY_{\cC^{\vee}}(x),\hat{\cY}_{\cC^{\vee}}(y))$ We make the following observation.

\begin{prop}\label{prop:left_adjoint_almost_modules} The quotient functor $\Mod\hy\cT^{\omega_1}\to \Mod_a\hy\cT^{\omega_1}$ resp. $T^{'\omega_1}\hy\Mod\to T^{'\omega_1}\hy\Mod_a$ has an exact left adjoint $\Mod_a\hy\cT^{\omega_1}\to \Mod\hy\cT^{\omega_1}$ resp. $T^{'\omega_1}\hy\Mod_a\to T^{'\omega_1}\hy\Mod,$ given by $$M\mapsto M\tens{\cT^{\omega_1}}I\cong M\tens{\cT^{\omega_1}}J\tens{\cT^{\omega_1}}J$$ resp. $$N\mapsto I'\tens{T^{'\omega_1}}N\cong J'\tens{T^{'\omega_1}}J'\tens{T^{'\omega_1}}N.$$\end{prop}

\begin{proof}The description of the left adjoints is analogous to the corresponding statement in the usual (commutative) almost mathematics \cite{GR03}. It follows from Proposition \ref{prop:cohom_functors} that both $I$ and $I'$ are flat on the left and on the right, so the functors $-\tens{\cT^{\omega_1}}I$ and $I'\tens{T^{'\omega_1}}-$ are exact.
\end{proof}

We deduce an explicit description of the category $T^{'\omega_1}\hy\Mod_a.$

\begin{cor}\label{cor:left_almost_modules_explicit}
	The functor $I'\tens{T^{'\omega_1}}-:T^{'\omega_1}\hy\Mod_a\to T^{'\omega_1}\hy\Mod$ is exact, fully faithful and its essential image is identified with the category of additive functors $F:T^{'\omega_1}\to\Ab$ such that the composition $\cC^{\vee,\omega_1}\to T^{'\omega_1}\to\Ab$ commutes with sequential colimits.
\end{cor}  

\begin{proof}
	The fully faithfulness and exactness follow from Proposition \ref{prop:left_adjoint_almost_modules}. To describe the essential image, it is convenient to identify the category $T^{'\omega_1}\hy\Mod$ with the category $\Fun^{\add}(\cC^{\vee,\omega_1},\Ab)$ of additive (i.e. finite coproduct-preserving) functors $\cC^{\vee,\omega_1}\to\Ab.$
	
	Now, if we consider $I'\tens{T^{'\omega_1}}-$ as a functor $\Theta$ from $\Fun^{\add}(\cC^{\vee,\omega_1},\Ab)$ to itself, then it sends $F$ to the composition
	\begin{equation}\label{eq:Theta_of_F_as_composition}
		\Theta(F):\cC^{\vee,\omega_1}\xto{\hat{\cY}_{\cC^{\vee}}}\Ind(\cC^{\vee,\omega_1})\xto{\Ind(F)}\Ind(\Ab)\xto{\colim}\Ab. 
	\end{equation}
	This follows from the definition of $I':$ if $x\in \cC^{\vee,\omega_1}$ with $\hat{\cY}(x)\cong \inddlim[n]x_n,$ then the right $T^{'\omega_1}$-module $I'(-,x)$ is identified with $\indlim[n]\h_{x_n},$ which gives the above description of $\Theta(F).$ Since all the functors in \eqref{eq:Theta_of_F_as_composition} commute with sequential colimits, so does $\Theta(F).$ On the other hand, if an additive functor $F:\cC^{\vee,\omega_1}\to\Ab$ commutes with sequential colimits, then $\Theta(F)$ is isomorphic to $F.$ This gives the stated description of the essential image of $\Theta.$ 
\end{proof}

There is no clear analogue of this description for $\Mod_a\hy T^{\omega_1}$ in terms of functors $\cC^{\omega_1,op}\to \Ab,$ but as we already mentioned the categories $\Mod_a\hy T^{\omega_1}$ and $T^{'\omega_1}\hy\Mod_a$ will turn out to be equivalent.

 We observe the following interpretation of local coherence of abelian categories.

\begin{prop}\label{prop:local_coherence} Let $\cA$ be a Grothendieck abelian category which is compactly generated. The following are equivalent.
	\begin{enumerate}[label=(\roman*),ref=(\roman*)]
		\item $\cA$ is locally coherent. \label{loccoh1}
		
		\item The functor $\hat{\cY}:\cA\to\Ind(\cA)$ is exact. \label{loccoh2}
	\end{enumerate}
\end{prop}

\begin{proof}\Implies{loccoh1}{loccoh2}. If $\cA$ is locally coherent, then the category $\cA^{\omega}$ is abelian and the inclusion functor $\cA^{\omega}\to\cA$ is exact. Passing to ind-completions, we get the functor $\hat{\cY},$ which is therefore exact.
	
	\Implies{loccoh2}{loccoh1}. Let $f:x\to y$ be a morphism in $\cA^{\omega}.$ We need to show that the kernel of $f$ in $\cA$ is also a compact object. Note that the maps $\hat{\cY}(x)\to \cY(x)$ and $\hat{\cY}(y)\to \cY(y)$ are isomorphisms. Since both functors $\hat{\cY}$ and $\cY$ are left exact, we conclude that the map $\hat{\cY}(\ker(f))\to \cY(\ker(f))$ is also an isomorphism. This exactly means that $\ker(f)\in\cA^{\omega},$ as required.\end{proof}

\begin{remark}\label{rem:local_coherence_of_Mod_S}
	Note that for a small triangulated category $S$ the category $\Mod\hy S$ is locally coherent. This follows from the fact that $S$ has weak kernels, see also \cite[Lemma 5.1.7, Proposition 5.1.10]{Nee01} for a detail proof of an equivalent statement. More precisely, in loc. cit. the category of finitely presented objects in $\Mod\hy S$ is denoted by $A(S)$ and it is proved (among other things) that $A(S)$ is abelian. 
	
	A generalization of this statement holds for homotopy categories of dualizable categories, this is the final assertion of Proposition \ref{prop:almost_modules_T^omega_1_basics} below. 
\end{remark}

We deduce the following basic properties of the abelian category $\Mod_a\hy T^{\omega_1}.$ Here the axioms (AB4*) and (AB6) are stated for an abelian category in the original sense \cite{Gro}, see also Subsection \ref{ssec:AB_axioms}.

\begin{prop}\label{prop:almost_modules_T^omega_1_basics} 1) If $\cC$ is compactly generated, then we have $\Mod_a\hy T^{\omega_1}\simeq \Mod\hy T^{\omega}.$
	
	2) In general, the category $\Mod_a\hy T^{\omega_1}$ is a Grothendieck abelian category which satisfies (AB4*) and (AB6). In particular, it is compactly assembled. The functor $\hat{\cY}:\Mod_a\hy T^{\omega_1}\to \Ind(\Mod_a\hy T^{\omega_1})$ is exact.\end{prop}

\begin{proof}1) If $\cC$ is compactly generated, then $J$ is an ideal of maps in $T^{\omega_1}$ which factor through an object of $T^{\omega}.$ Hence, the kernel of the (exact) restriction functor $\Mod\hy T^{\omega_1}\to \Mod\hy T^{\omega}$ is the category of almost zero modules. Since its left adjoint is fully faithful, the assertion follows.
	
	2) Recall that the category of almost zero modules is closed under limits and colimits. Since the axioms (AB4*) and (AB6) hold in $\Mod\hy T^{\omega_1},$ they also hold in $\Mod_a\hy T^{\omega_1};$ this is a special case of Roos' theorem \cite{Roo65}. It is clear that $\Mod_a\hy T^{\omega_1}$ is a Grothendieck category. Finally, the functor $\hat{\cY}$ is the composition of exact functors:
	$$\Mod_a\hy T^{\omega_1}\xto{-\tens{T^{\omega_1}}J\tens{T^{\omega_1}}J}\Mod\hy T^{\omega_1}\xto{\hat{\cY}}\Ind(\Mod\hy T^{\omega_1})\to \Ind(\Mod_a\hy T^{\omega_1}).$$ Here the middle functor is exact by Proposition \ref{prop:local_coherence} and Remark \ref{rem:local_coherence_of_Mod_S}.
\end{proof}

Consider the functor $G:T\to\Mod_a\hy T^{\omega_1}$ from Theorem \ref{th:Adams_rep_contravar}, i.e. $G(x)=\Hom_T(-,x),$ considered as an almost $T^{\omega_1}$-module. 

\begin{lemma}\label{lem:G_cohom_functor} $G$ is a cohomological functor which commutes with infinite coproducts. Moreover, $G$ is conservative.\end{lemma}

\begin{proof}It is clear that $G$ is a cohomological functor. Let $\{z_i\}_{i\in I}$ be a collection of objects of $T.$ By the definition of compact morphisms, we see that the $T^{\omega_1}$-module
	$$\coker(\biggplus[i]\Hom_T(-,z_i)\to \Hom_T(-,\biggplus[i] z_i))$$ is almost zero. This shows that $G$ commutes with coproducts.
	
	To show that $G$ is conservative, it suffices to check that if $G(x)=0$ then $x=0.$ But the vanishing of $G(x)$ means that any compact morphism $y\to x$ is zero. This implies that $x$ is zero.\end{proof}

Note that $f$ is a phantom morphism resp. a pure monomorphism resp. a pure epimorphism in $T$ if and only if the morphism $G(f)$ is zero resp. a monomorphism resp. an epimorphism in $\Mod_a\hy\cT^{\omega_1}.$ 

\begin{defi}We say that an object $x\in T$ is pure injective if for any phantom morphism $f:y\to z$ and for any morphism $g:z\to x$ we have $g\circ f=0.$ Equivalently, for any pure monomorphism $f:y\to z$ and for any map $g:y\to x$ there exists a map $h:z\to x$ such that $g=h\circ f.$\end{defi}

We refer to \cite{Kr00, GP04} for the usual notion of pure injective objects for compactly generated triangulated categories.
We obtain the following description of the category of pure injective objects, which is a straightforward generalization of the corresponding result for compactly generated categories \cite[Theorem 4.1]{GP04}.

\begin{prop}\label{prop:pure_injective} The functor $G:T\to \Mod_a\hy T^{\omega_1}$ induces an equivalence from the category of pure injective objects of $T$ to the category of injective objects of $\Mod_a\hy T^{\omega_1}.$\end{prop}

\begin{proof}Let $M$ be an injective object of $\Mod_a\hy T^{\omega_1}.$ Consider the functor $T^{op}\to\Ab$ given by $\Hom(G(-),M).$ This is a cohomological functor (since $M$ is injective), which also commutes with products by Lemma \ref{lem:G_cohom_functor}. By \cite[Theorem 1.17]{Nee01} this functor is represented by some object $\Lambda(M)\in T.$ Moreover, $\Lambda(M)$ is a pure injective object of $T,$ since $G$ annihilates phantom morphisms. We obtain a functor $\Lambda$ from the category of injective almost $T^{\omega_1}$-modules to the category of pure injective objects of $T.$ Note that we have a functorial isomorphism $G(\Lambda(M))\cong M.$
	
	Let now $x\in T$ be a pure injective object, and choose a monomorphism $G(x)\to M,$ where $M$ is injective. Then the map $x\to \Lambda(M)$ is a pure monomorphism. Since $x$ is pure injective, it follows that $x$ is a direct summand of $\Lambda(M).$ Hence, $G(x)$ is a direct summand of $G(\Lambda(M))\cong M.$ We conclude that $G(x)$ is injective.
	
	It remains to recall that $G$ is conservative, hence the map $x\to \Lambda(G(x))$ is an isomorphism for any pure injective $x\in T.$\end{proof}

\begin{prop}\label{prop:G_maps_to_cohom_modules} The functor $G:T\to \Mod_a\hy T^{\omega_1}$ takes values in almost cohomological almost $T^{\omega_1}$-modules. In particular, any injective almost $T^{\omega_1}$-module is almost cohomological.\end{prop}

\begin{proof}Indeed, for any object $x\in T$ the $T^{\omega_1}$-module $\Hom_T(-,x)$ is cohomological. If $M\in \Mod_a\hy T^{\omega_1}$ is injective, then it is in the essential image of $G$ by Proposition \ref{prop:pure_injective}, hence $M$ is almost cohomological.
\end{proof}

\begin{proof}[Proof of Theorem \ref{th:Adams_rep_contravar}] Essentially the theorem follows from the following statement.
	
	{\noindent{\bf Claim.}} {\it Any almost cohomological almost $T^{\omega_1}$-module $M\in\Mod_a\hy T^{\omega_1}$ has injective dimension at most $1.$}
	
	Assume that Claim holds. Let $M\in\Mod_a\hy T^{\omega_1}$ be an almost cohomological almost $T^{\omega_1}$-module. Choose an injective resolution of $M$ of length $1:$
	$$0\to M\to \cJ_0\xto{f} \cJ_1\to 0.$$
	Consider the pure injective objects $\Lambda(\cJ_0), \Lambda(\cJ_1)\in T,$ and the map $\Lambda(f):\Lambda(\cJ_0)\to \Lambda(\cJ_1),$ where the functor $\Lambda$ is defined in the proof of Proposition \ref{prop:pure_injective}. Now the map $\Lambda(f)$ is a pure epimorphism since $G(\Lambda(f))$ is identified with the epimorphism $f.$ Hence, the map $\Lambda(f)[-1]$ is also a pure epimorphism, which implies that $G(\Lambda(f)[-1])$ is an epimorphism. Putting $z=\Fiber(\Lambda(f))$ and using the fact that $G$ is homological, we obtain
	$$G(z)\cong\ker(G(\Lambda(f))\cong\ker(f)=M.$$
	Hence, the essential image of $G$ is exactly the full subcategory of almost cohomological almost modules.
	
	Now take some objects $x,y\in T,$ and choose an injective resolution of $G(y)$ as above: 
	$$0\to G(y)\to \cI_0\xto{g} \cI_1\to 0.$$
	We obtain a (non-unique) map $y\to \Fiber(\Lambda(g)),$ which becomes an isomorphism after applying $G.$ Since $G$ is conservative, we obtain an exact triangle
	$$y\to \Lambda(\cI_0)\to \Lambda(\cI_1)\to y[1].$$
	It follows that we have a short exact sequence
	\begin{multline}\label{eq:ses_for_Homs2} 0\to \coker(\Hom_T(x,\Lambda(\cI_0)[-1])\to \Hom_T(x,\Lambda(\cI_1)[-1]))\to \Hom_\Lambda(x,y)\\
		\to\ker(\Hom_T(x,\Lambda(\cI_0))\to \Hom_T(x,\Lambda(\cI_1)))\to 0.\end{multline}
	By construction of the functor $\Lambda,$ we see that the sequence \eqref{eq:ses_for_Homs2} is exactly of the form \eqref{eq:ses_for_Homs1}.
	
	In particular, we see that any phantom map $h:x\to y$ factors through $\Lambda(\cI_1)[-1]$ (and vice versa). If $w:z\to x$ is another phantom map, then $h\circ w=0$ since $\Lambda(\cI_1)[-1]$ is pure injective.
	\begin{proof}[Proof of Claim.]
		By Theorem \ref{th:presentability_of_Cat^dual} there exists a fully faithful strongly continuous functor $F:\cC\to\Ind(\cA),$ where $\cA$ is $\omega_1$-compact in $\Cat^{\perf}.$ Denote by $S$ the homotopy category of $\Ind(\cA).$ By Proposition \ref{prop:kappa_compact_small_cats}, the category $S^{\omega}$ is countable.
		
		By Proposition \ref{prop:almost_modules_T^omega_1_basics} 1) the category $\Mod_a\hy S^{\omega_1}$ is equivalent to $\Mod\hy S^{\omega}.$ Since the functor $T^{\omega_1}\to S^{\omega_1}$ preserves compact morphisms, we have a well-defined quotient functor $$U:\Mod\hy S^{\omega}\simeq\Mod_a\hy S^{\omega_1}\to\Mod_a\hy T^{\omega_1}.$$ Its left adjoint $V:\Mod_a\hy T^{\omega_1}\to \Mod\hy S^{\omega}$ is fully faithful. Explicitly, $V$ is induced by the functor 
		\begin{equation*}
			\wt{V}:\Mod\hy T^{\omega_1}\to \Mod\hy S^{\omega},\quad \wt{V}(M)(x) = M\tens{T^{\omega_1}} \Hom_S(x,F(-)),
		\end{equation*}
		where $M\in \Mod\hy T^{\omega_1},$ $x\in S^{\omega}.$ It follows from Proposition \ref{prop:cohom_functors} that $V$ takes almost cohomological almost modules to flat modules. Namely, by loc. cit. this equivalently means that $\wt{V}$ preserves flatness, which holds since each right $S^{\omega}$-module $\Hom_S(-,F(y))$ is flat, where $y\in T^{\omega_1}.$
		We also observe that $V$ is exact: this is equivalent to exactness of $\wt{V}$ which holds since by loc. cit. each left $T^{\omega_1}$-module $\Hom_S(x,F(-))$ is flat.
		
		Now, applying Proposition \ref{prop:cohom_functors}, we see that for any almost cohomological almost module $M\in\Mod_a\hy T^{\omega_1}$ the $S^{\omega}$-module $V(M)$ is fp-injective. Recall that the category $\Mod\hy S^{\omega}$ is locally coherent (Remark \ref{rem:local_coherence_of_Mod_S}). Since $S^{\omega}$ is also countable, by \cite[Proposition 2.3]{Pos17} any fp-injective object in $\Mod\hy S^{\omega}$ has injective dimension at most $1.$ More precisely, in loc. cit. the result is stated for modules over a ring, but the same argument applies in our situation. In particular, the $S^{\omega}$-module $V(M)$ has injective dimension at most $1.$ Since $V$ is exact, the functor $U$ preserves injective objects. Since $U$ is also exact, we conclude that $M\cong U(V(M))$ has injective dimension at most $1.$
	\end{proof}
	Theorem is proved.
\end{proof}

\begin{remark}One can also prove Theorem \ref{th:Adams_rep_contravar} without choosing an embedding $\cC\to\Ind(\cA).$ One can define the notion of fp-injective objects in any compactly assembled Grothendieck abelian category $\cB$ (equivalently, a Grothendieck abelian category satisfying (AB6)), such that the functor $\hat{\cY}:\cB\to\Ind(\cB)$ is exact. Namely, an object $x\in\cB$ is fp-injective if for any $y,z\in\cB^{\omega_1}$ and for any compact morphism $f:y\to z,$ the map $$\Ext^1(f,x):\Ext^1(z,x)\to\Ext^1(y,x)$$
	is zero. If moreover $\cB^{\omega_1}$ is a Serre subcategory of $\cB,$ then one can show that any fp-injective object of $\cB$ has injective dimension at most $1.$
	
	Now, for any dualizable category $\cC,$ an almost $\h\cC^{\omega_1}$-module is almost cohomological if and only if it is fp-injective in $\Mod_a\hy \h\cC^{\omega_1}.$ If $\cC$ is $\omega_1$-compact, it can be shown that $(\Mod_a\hy \h\cC^{\omega_1})^{\omega_1}$ is a Serre subcategory in $\Mod_a\hy \h\cC^{\omega_1}.$ This gives an alternative (intrinsic) way to show that $G(x)$ has injective dimension at most $1$ for any $x\in \h\cC.$\end{remark}

\begin{remark}If $\cB$ is a Grothendieck abelian category with (AB6) and (AB4*), then one can define flat objects of $\cB$ intrinsically. Namely, an object $x\in\cB$ is flat if for any compact morphism $y\to x$ and for any object $z\in\cB$ the map $\Ext^1(x,z)\to \Ext^1(y,z)$ is zero.
	
	If $\cC$ is a dualizable category, then a right almost $\h \cC^{\omega_1}$-module is almost cohomological if and only if it is flat in $\Mod_a\hy\h\cC^{\omega_1}.$\end{remark}

Now we formulate a closely related result on Adams representability for covariant functors.

\begin{theo}\label{th:Adams_rep_covar}
Let $\cC$ be a dualizable category which is $\omega_1$-compact in $\Cat_{\st}^{\dual}.$ We denote by $T$ resp. $T'$ the triangulated homotopy category of $\cC$ resp. $\cC^{\vee}.$ Consider the functor $H:T\to \Fun^{\add}(T^{'\omega_1},\Ab),$ given by $H(x)=\pi_0\ev_{\cC}(x,-).$ Then the functor $H$ is full and conservative, and its essential image consists of homological functors which commute with countable coproducts. Moreover, for $x,y\in T$ we have a functorial short exact sequence of abelian groups
\begin{equation*}
0\to \Ext^1(H(x),H(y[-1]))\to \Hom_T(x,y)\to \Hom(H(x),H(y))\to 0.
\end{equation*}
The subgroup of phantom maps $x\to y$ is identified with $\Ext^1(H(x),H(y[-1])).$ Moreover, we have $\Ext^n(H(x),H(y)) = 0$ for $n\geq 2.$
\end{theo} 

We will deduce Theorem \ref{th:Adams_rep_covar} from Theorem \ref{th:Adams_rep_contravar} using the following equivalence of categories.


\begin{prop}\label{prop:equivalence_on_almost_modules}
Consider the colimit-preserving functor $\Phi:\Mod\hy T^{\omega_1}\to T^{'\omega_1}\hy\Mod,$ which is given on representable modules by $\Phi(\h_x)=H(x)=\pi_0\ev_{\cC}(x,-).$ Then $\Phi$ is exact and it vanishes on almost zero right $T^{\omega_1}$-modules. The induced functor 
\begin{equation}\label{eq:equivalence_on_almost_modules}
\bbar{\Phi}:\Mod_a\hy\cT^{\omega_1}\to T^{'\omega_1}\hy\Mod_a
\end{equation}
is an equivalence.
\end{prop}

\begin{proof}
First we note that (up to a unique isomorphism) $\Phi$ is given by the formula
\begin{equation*}
\Phi(M)(y) = M\tens{T^{\omega_1}} \pi_0 \ev_{\cC}(-,y),\quad M\in \Mod\hy T^{\omega_1},\, y\in T^{'\omega_1}.
\end{equation*}
Now the exactness of $\Phi$ means that for any object $y\in T^{'\omega_1}$ the left $T^{\omega_1}$-module $\pi_0\ev_{\cC}(-,y)$ is flat. This follows from Proposition \ref{prop:cohom_functors}. 

By Proposition \ref{prop:left_adjoint_almost_modules}, to show the vanishing of $\Phi$ on almost zero modules, it suffices to check that for any $M\in \Mod\hy T^{\omega_1}$ the map $\Phi(M\tens{T^{\omega_1}}I)\to\Phi(M)$ is an isomorphism. Since both $\Phi$ and $-\tens{T^{\omega_1}}I$ commute with colimits, we may assume that $M=\h_x$ for some $x\in T^{\omega_1}.$ Let $\hat{\cY}(x)=\inddlim[n]x_n$ (in $\Ind(\cC^{\omega_1})$). Then $I(-,x)\cong\indlim[n]\h_{x_n},$ and
\begin{equation*}
\Phi(\h_x\tens{T^{\omega_1}}I)\cong \Phi(I(-,x))\cong\indlim[n]H(x_n)\cong H(x)=\Phi(\h_x).
\end{equation*}
This proves the vanishing of $\Phi$ on almost zero modules. Hence, we have a well-defined induced functor $\bbar{\Phi}$ from \eqref{eq:equivalence_on_almost_modules}. It remains to prove that $\bbar{\Phi}$ is an equivalence. 

We will construct a functor 
\begin{equation}\label{eq:inverse_equivalence_on_almost_modules}
\bbar{\Psi}:T^{'\omega_1}\hy\Mod_a\to \Mod_a\hy\cT^{\omega_1}
\end{equation} 
and prove that it is quasi-inverse to $\bbar{\Phi}.$ The construction goes as follows. First, recall the functor $(-)^{\vee}:\cC^{\vee,op}\to\cC$ from Remark \ref{rem:functor_-^vee} (with the roles of $\cC$ and $\cC^{\vee}$ interchanged). To avoid confusion, we denote the corepresentable left $T^{'\omega_1}$-modules by $\h_y^{\vee},$ $y\in T^{'\omega_1}.$ Consider the colimit-preserving functor $\Psi:\cT^{'\omega_1}\hy\Mod\to\Mod\hy T^{\omega_1},$ given on corepresentable objects by $\Psi(\h_y^{\vee})=\Hom_T(-,y^{\vee}).$ In other words, (up to a unique isomorphism) $\Psi$ is given by
\begin{equation*}
\Psi(N)(x) = \Hom_T(x,(-)^{\vee})\tens{T^{'\omega_1}} N,\quad N\in T^{'\omega_1}\hy\Mod,\, x\in T^{\omega_1}.
\end{equation*}
Arguing as above, we see that $\Psi$ is exact. We claim that $\Psi$ preserves almost zero modules, hence it induces a functor $\bbar{\Psi}$ as in  \eqref{eq:inverse_equivalence_on_almost_modules}. To see this, it suffices to prove that we have an isomorphism of functors
\begin{equation}\label{eq:functor_Psi_commutation_with_tensoring_by_I}
\Psi(I'\tens{T^{'\omega_1}}-)\cong \Psi(-)\tens{T^{\omega_1}}I.
\end{equation}
We only need to construct this isomorphism functorially on corepresentable left $T^{'\omega_1}$-modules. So let $y\in \cC^{\vee,\omega_1},$ and let $\hat{\cY}(y^{\vee})=\inddlim[i\in I]x_i$ in $\Ind(\cC^{\omega_1}),$ where $I$ is directed. We claim that the pro-object $\proolim[i]x_i^{\vee}\in\Pro(\cC^{\vee})$ is contained in (the essential image of) $\Pro(\cC^{\vee,\omega_1}).$ Indeed, for any $i\in I$ there exists $i'\geq i$ such that the map $x_i\to x_{i'}$ is compact, hence so is the map $x_{i'}^{\vee}\to x_i^{\vee}.$ The latter map therefore factors through an object of $\cC^{\omega_1},$ which gives the assertion. 

Now choose a pro-system $(z_j)_{j\in J}$ in $\cC^{\vee,\omega_1}$ with an isomorphism $\proolim[j]z_j\cong\proolim[i]x_i^{\vee},$ where $J$ is codirected. We claim that there is a natural isomorphism 
\begin{equation}\label{eq:pro_representability_of_I'_y_-}
\indlim[i]h_{z_j}^{\vee}\xto{\sim} I'(y,-).
\end{equation}
To see this, first note that for $i\in I$ the structural map $\cY_{\cC}(x_i)\to \hat{\cY}_{\cC}(y^{\vee})$ is given by an element of $\pi_0\ev_{\cC}(y^{\vee},x_i^{\vee}),$ which also gives a map $\cY_{\cC^{\vee}}(y)\to \hat{\cY}_{\cC^{\vee}}(x_i^{\vee}).$ The latter maps are compatible with transition maps $x_i\to x_{i'}.$ By the above pro-isomorphism we obtain a compatible collection of elements of $I'(y,z_j),$ $j\in J.$ This gives the map \eqref{eq:pro_representability_of_I'_y_-}. To see that it is an isomorphism, note that for $x\in \cC$ we have a natural composition map of left $T^{'\omega_1}$-modules:
\begin{equation*}
\pi_0\ev_{\cC}(x,-)\to \pi_0\ev_{\cC}(x^{\vee\vee},-)\to\Hom_T(x^{\vee},-).
\end{equation*} 
Now take some $i\leq i'$ in $I$ such that the map $x_i\to x_{i'}$ is compact in $\cC.$ Choosing a morphism $\cY(x_i)\to \hat{\cY}(x_{i'})$ witnessing the compactness, we obtain a commutative diagram of left $T^{'\omega_1}$-modules:
\begin{equation*}
\begin{tikzcd}
\pi_0\ev_{\cC}(x_i,-) \ar[r]\ar[d] & \pi_0\ev_{\cC}(x_{i'},-)\ar[d]\\
\Hom_{T'}(x_i^{\vee},-)\ar[ru]\ar[r] & \Hom_{T'}(x_{i'}^{\vee},-).
\end{tikzcd}
\end{equation*} 
Here the diagonal map is induced by the element of $\pi_0\ev_{\cC}(x_{i'},x_i^{\vee}),$ corresponding to the chosen map $\cY(x_i)\to \hat{\cY}(x_{i'}).$ Since for any $i\in I$ there exists $i'\geq i$ such that the map $x_i\to x_{i'}$ is compact, we obtain a chain of isomorphisms of left $T^{'\omega_1}$-modules:
\begin{equation*}
I'(y,-)\cong \pi_0\ev_{\cC}(y^{\vee},-)\cong \indlim[i\in I]\pi_0\ev_{\cC}(x_i,-)\cong \indlim[i\in I]\Hom_{T'}(x_i^{\vee},-)\cong \indlim[j\in J] \h_{z_j}^{\vee}.
\end{equation*}
To deduce the stated isomorphism \eqref{eq:functor_Psi_commutation_with_tensoring_by_I} computed on $\h_y^{\vee},$ we need to observe the following. Let again $i\leq i'$ be elements of $I$ such that the map $x_i\to x_{i'}$ is compact, and choose a morphism $\cY(x_i)\to\hat{\cY}(x_{i'})$ witnessing the compactness. Then we obtain a commutative diagram in $T:$
\begin{equation*}
\begin{tikzcd}
x_i\ar[r]\ar[d] & x_{i'}\ar[d]\\
x_i^{\vee\vee}\ar[ru]\ar[r] & x_{i'}^{\vee\vee}
\end{tikzcd}
\end{equation*}
Here the diagonal map is obtained by applying the composition
\begin{equation*}
\pi_0\ev_{\cC}(x_{i'},x_i^{\vee})\to \pi_0\ev_{\cC}(x_{i'},x_i^{\vee\vee\vee})\to \Hom_T(x_i^{\vee\vee},x_{i'}).
\end{equation*}
to the chosen element (corresponding to the map $\cY(x_i)\to\hat{\cY}(x_{i'})$). Again, since for any $i\in I$ there exists $i'\geq i$ such that the map $x_i\to x_{i'}$ is compact, we obtain isomorphisms of right $T^{\omega_1}$-modules:
\begin{multline*}
\Psi(I'\tens{T^{'\omega_1}}h_y^{\vee})\cong \Psi(I'(y,-))\cong\indlim[j]\Psi(h_{z_j}^{\vee})\cong\indlim[j]\Hom_T(-,z_j^{\vee})\cong\indlim[i]\Hom_T(-,x_i^{\vee\vee})\\
	\cong\indlim[i]\Hom_T(-,x_i)\cong \Hom_T(-,y^{\vee})\tens{T^{\omega_1}}I\cong \Psi(\h_y^{\vee})\tens{T^{\omega_1}}I.
\end{multline*}
The composition is functorial in $y,$ hence we obtain an isomorphism of functors \eqref{eq:functor_Psi_commutation_with_tensoring_by_I}. In particular, $\Psi$ indeed induces the functor $\bbar{\Psi}$ from \eqref{eq:inverse_equivalence_on_almost_modules}.

To prove that $\bbar{\Psi}$ is quasi-inverse to $\bbar{\Phi},$ it suffices to prove the following isomorphisms of functors:
\begin{equation}\label{eq:Psi_circ_Phi}
\Psi\circ\Phi\cong -\tens{T^{\omega_1}}I,
\end{equation}
\begin{equation}\label{eq:Phi_circ_Psi}
	\Phi\circ\Psi\cong I'\tens{T^{'\omega_1}}-.
\end{equation}

We first prove \eqref{eq:Phi_circ_Psi}. Let $y\in T^{'\omega_1}$ and choose an $\omega_1$-directed system $(w_k)_{k\in K}$ in $\cC^{\omega_1}$ such that $y^{\vee}\cong\indlim[k]w_k.$ We obtain
\begin{multline*}
	\Phi(\Psi(\h_y^{\vee}))\cong \Phi(\Hom_T(-,y^{\vee}))\cong \Phi(\indlim[k]\h_{w_k})\cong  \indlim[k]\pi_0\ev_{\cC}(w_k,-)\\
	\cong\pi_0\ev_{\cC}(y^{\vee},-)\cong I'(y,-)\cong I'\tens{T^{'\omega_1}} \h_y^{\vee}.
\end{multline*}
The composition is functorial in $y,$ which gives an isomorphism \eqref{eq:Phi_circ_Psi}.

Finally, we prove \eqref{eq:Psi_circ_Phi}. Let $x\in T^{\omega_1}$ with $\hat{\cY}(x)=\indlim[n]x_n$ (in $\Ind(\cC^{\omega_1})$). We may and will assume that each map $x_n\to x_{n+1}$ is compact in $\cC.$ Arguing as above, we obtain factorizations $x_{n+1}^{\vee}\to y_n\to x_n^{\vee}$ in $\cC^{\vee}$ with $y_n\in\cC^{\vee,\omega_1}.$ Again, the above arguments give the isomorphisms $\Phi(\h_x)\cong \indlim[n]\h_{y_n}^{\vee}$ and
\begin{multline*}
\Psi(\Phi(\h_x))\cong\indlim[n]\Psi(\h_{y_n}^{\vee})\cong \indlim[n]\Hom_T(-,y_n^{\vee})\cong \indlim[n]\Hom_T(-,x_n^{\vee\vee})\\
\cong \indlim[n]\h_{x_n}\cong I(-,x)\cong \h_x\tens{T^{\omega_1}} I.
\end{multline*}
The composition is functorial in $x,$ hence we get an isomorphism of functors \eqref{eq:Psi_circ_Phi}. This proves the proposition.
\end{proof}

\begin{proof}[Proof of Theorem \ref{th:Adams_rep_covar}] 
The proof is obtained from Theorem \ref{th:Adams_rep_contravar} using the relation between left $T^{'\omega_1}$-modules and right $T^{\omega_1}$-modules. 

We say that a left almost module $M\in T^{'\omega_1}\hy\Mod_a$ is almost homological if it is an image of a left $T^{'\omega_1}$-module corresponding to a homological functor $T^{'\omega_1}\to\Ab.$ Recall from Proposition \ref{prop:left_adjoint_almost_modules} that the (exact, fully faithful) left adjoint to the quotient functor $T^{'\omega_1}\hy\Mod\to T^{'\omega_1}\hy\Mod_a$ is given by $I'\tens{T^{'\omega_1}}-,$ where $I'\cong J'\tens{T^{'\omega_1}}J'$ is the quasi-ideal in $T^{'\omega_1}$ given by $I'(x,y)=\pi_0\Hom(\cY_{\cC^{\vee}}(x),\hat{\cY}_{\cC^{\vee}}(y)).$ By Corollary \ref{cor:left_almost_modules_explicit} the essential image of the functor $I'\tens{T^{'\omega_1}}-$ consists of functors $F:T^{'\omega_1}\to\Ab$ such that the composition $\cC^{\vee,\omega_1}\to T^{'\omega_1}\xto{F}\Ab$ commutes with sequential colimits. We deduce that the functor $I'\tens{T^{'\omega_1}}-$ gives an equivalence
\[
\xymatrix@1{
	\left\{
	\text{
		\begin{minipage}[c]{2.5in}
			almost homological left almost modules $M\in T^{'\omega_1}\hy\Mod_a$
		\end{minipage}
	}
	\right\}
	\ar[r]^-{\sim}
	&
	\left\{
	\text{
		\begin{minipage}[c]{2.2in}
			homological functors $F:T^{'\omega_1}\to\Ab,$ commuting with countable coproducts
		\end{minipage}
	}
	\right\}
}
\]
Indeed, since each left $T^{'\omega_1}$-module $I'(x,-)$ is flat, we see by Proposition \ref{prop:cohom_functors} that $I'\tens{T^{'\omega_1}}-$ sends almost homological almost modules to homological functors. It remains to observe that for a homological functor $F:T^{'\omega_1}\to\Ab$ the composition $\cC^{\vee,\omega_1}\to T^{'\omega_1}\xto{F}\Ab$ commutes with sequential colimits if and only if $F$ commutes with countable coproducts.

Therefore, in the formulation of the theorem we can replace the functor $H$ with the composition
\begin{equation*}
\bbar{H}:T\xto{H}T^{'\omega_1}\hy\Mod\to T^{'\omega_1}\hy\Mod_a,
\end{equation*}
and prove that this functor has the same properties as the functor $G$ from Theorem \ref{th:Adams_rep_contravar}. It remains to observe that the following square commutes
\begin{equation*}
\begin{tikzcd}
T \ar[r, equal] \ar{d}{G} & T\ar{d}{\bbar{H}}\\
\Mod_a\hy T^{\omega_1} \ar[r, "\bbar{\Phi}", "\sim"'] & T^{'\omega_1}\hy\Mod_a,
\end{tikzcd}
\end{equation*}   
where $\bbar{\Phi}$ is the equivalence from Proposition \ref{prop:equivalence_on_almost_modules}. It is clear that $\bbar{\Phi}$ sends almost cohomological right almost $T^{\omega_1}$-modules to almost homological left almost $T^{'\omega_1}$-modules. Namely, this follows from Proposition \ref{prop:cohom_functors} since each left $T^{'\omega_1}$-module $\pi_0\ev_{\cC}(x,-)$ is homological, where $x\in\cC.$ By a similar argument we see that the quasi-inverse functor to $\bbar{\Phi}$ (from the proof of Proposition \ref{prop:equivalence_on_almost_modules}) sends almost homological left almost $T^{'\omega_1}$-modules to almost cohomological right almost $T^{\omega_1}$-modules. Therefore, the theorem follows from Theorem \ref{th:Adams_rep_contravar}.  
\end{proof}

\section{K-theory of products of exact infinity-categories}
\label{app:K_theory_products_exact}

In this section we prove that $K$-theory of exact $\infty$-categories commutes with infinite products (Corollary \ref{cor:K_theory_of_exact_cats_comm_with_products}).

\subsection{Products of additive infinity-categories}
\label{ssec:products_of_additive_cats}

Given a small additive $\infty$-category $\cA,$ we denote by $\Stab(\cA)$ its stabilization in the following sense. The category $\Stab(\cA)$ is a (small) stable category with a functor $j:\cA\to\Stab(\cA),$ and for any stable category $\cC$ and an additive functor $F:\cA\to\cC$ there is a unique (up to a contractible space of choices) factorization $F:\cA\xto{j}\Stab(\cA)\xto{G}\cC,$ where $G$ is exact. This notion of stabilization is studied for example by Sosnilo in \cite{Sos19} in the context of Bondarko's weight structures. By \cite[Theorem 5.3.1, Proposition 6.2.1]{Bon}, if $\cA$ is Karoubi complete, then so is $\Stab(\cA).$ Given $a,b\in\Z,$ $a\leq b,$ we denote by $\Stab(\cA)_{[a,b]}\subset\Stab(\cA)$ the full subcategory generated by $\cA[c],$ $a\leq c\leq b.$ We put $\Stab(\cA)_{(-\infty,b]}=\bigcup\limits_{c\leq b}\Stab(\cA)_{[c,b]},$ and similarly $\Stab(\cA)_{[a,+\infty)}=\bigcup\limits_{c\geq a}\Stab(\cA)_{[a,c]}.$   

The following result is due to A. C\'ordova Fedeli.

\begin{theo}\label{th:Cordova_U_loc}\cite[Proposition 2.10]{Cor} Let $\{\cA_i\}_{i\in I}$ be a collection of Karoubi complete additive small $\infty$-categories. We have an isomorphism in $\Mot^{\loc}:$
$$\cU_{\loc}(\Stab(\prodd[i\in I]\cA_i))\cong \cU_{\loc}(\prodd[i\in I]\Stab(\cA_i)).$$\end{theo}

Since we will need an analogous statement for $t$-structures, we recall the proof from \cite{Cor} with some simplifications.

We introduce some notation. Put $\cC=\prodd[i\in I]\Stab(\cA_i).$ Given functions $f,g:I\to\Z,$ $f\leq g,$ we put \begin{equation}\label{eq:C_f_g}\cC_{[f,g]}=\bigcup\limits_{c>0}\prod\limits_i \Stab(\cA_i)_{[f(i)-c,g(i)+c]},\end{equation}
\begin{equation}\label{eq:C_-infty_g}\cC_{(-\infty,g]}=\bigcup\limits_{c>0}\prod\limits_i \Stab(\cA_i)_{(-\infty,g(i)+c]},\end{equation}
\begin{equation}\label{eq:C_f_+infty}\cC_{[f,+\infty)}=\bigcup\limits_{c>0}\prod\limits_i \Stab(\cA_i)_{[f(i)-c,+\infty)}\end{equation}

For any function $f:I\to\Z,$ we denote by $\Sigma^f:\prod\limits_i \Stab(\cA_i)\to \prod\limits_i \Stab(\cA_i)$ the product of the functors $\Sigma^{f(i)}:\Stab(\cA_i)\to \Stab(\cA_i).$

\begin{lemma}\label{lem:functors_with_same_class}\cite[Lemma 2.9]{Cor} 1) For each $f:I\to\N,$ the functors $$\Id,\Sigma^{2f}:\cC_{[0,+\infty)}\to \cC_{[0,+\infty)}$$
have the same class in $K_0(\Fun(\cC_{[0,+\infty)},\cC_{[0,+\infty)})).$

2) Analogous statement holds for the same functors
$$\Id,\Sigma^{2f}:\cC_{[0,0]}\to \cC_{[0,2f]}$$
\end{lemma}

\begin{proof}In both cases, we have the following equalities in $K_0:$
\begin{multline*}[\Sigma^{2f}]=[\prod\limits_i (\bigoplus\limits_{k=0}^{f(i)}\Sigma^{2k})]-[\prod\limits_i (\bigoplus\limits_{k=0}^{f(i)-1}\Sigma^{2k})]=[\prod\limits_i (\bigoplus\limits_{k=0}^{f(i)}\Sigma^{2k})]-[\Sigma^2(\prod\limits_i (\bigoplus\limits_{k=0}^{f(i)-1}\Sigma^{2k}))]=\\ [\prod\limits_i (\bigoplus\limits_{k=0}^{f(i)}\Sigma^{2k})]-[\prod\limits_i (\bigoplus\limits_{k=1}^{f(i)}\Sigma^{2k})]=[\Id].\end{multline*}This proves the lemma.\end{proof}

\begin{lemma}\label{lem:product_of_prestab}\cite[Proof of Proposition 2.7]{Cor} The map $\cU_{\loc}(\cC_{[0,+\infty)})\to \cU_{\loc}(\cC)$ is an isomorphism in $\Mot^{\loc}.$\end{lemma}

\begin{proof}Recall the axiom (AB6) from Subsection \ref{ssec:AB_axioms}. Since it holds in $\Cat^{\perf},$ we have
$$\cU_{\loc}(\cC)\cong \cU_{\loc}(\indlim[f:I\to\N]\cC_{[-2f,+\infty)})\cong \indlim[f:I\to\N] \cU_{\loc}(\cC_{[-2f,+\infty)}).$$
Hence, it suffices to prove that for any $f:I\to \N$ the map $$\cU_{\loc}(\cC_{[0,+\infty)})\to \cU_{\loc}(\cC_{[-2f,+\infty)})$$
is an isomorphism. This follows from Lemma \ref{lem:functors_with_same_class}.\end{proof}

\begin{lemma}\label{lem:equivalence_of_quotients}\cite[Proof of Proposition 2.10]{Cor} Let $f:I\to\N$ be a function. Then the functor $$\Phi:\cC_{[0,f]}/\cC_{[f,f]}\to\cC_{[0,+\infty)}/\cC_{[f,+\infty)}$$
is an equivalence.\end{lemma}

\begin{proof} {\noindent{\bf Essential surjectivity of $\Phi$}}. Take some object $x=(x_i)_i\in \cC_{[0,+\infty)}.$ Then for some $c\geq 0$ we have $x_i\in \Stab(\cA_i)_{[-c,+\infty)}$ for all $i\in I.$ Choose exact triangles 
$$y_i\to x_i\to z_i,\quad y_i\in\Stab(\cA_i)_{[-c,f(i)]},\,\, z_i\in \Stab(\cA_i)_{[f(i)+1,+\infty)}.$$
Then $y=(y_i)_i\in \cC_{[0,f]},$ $z=(z_i)_i\in\cC_{[f,+\infty)},$ and the map $y\to x$ becomes an isomorphism in $\cC_{[0,+\infty)}/\cC_{[f,+\infty)}.$ This proves that $\Phi$ is essentially surjective.

Let now $x=(x_i)_i,$ $y=(y_i)_i$ be objects of $\cC_{[0,f]},$ and denote by $\bar{x}$ and $\bar{y}$ their images in  $\cC_{[0,f]}/\cC_{[f,f]}.$

{\noindent{\bf Fullness of $\h\Phi.$}} Consider a map $f:\Phi(\bar{x})\to \Phi(\bar{y}),$ represented by a roof $x\to z\leftto y,$ where $z=(z_i)_i\in\cC_{[0,+\infty)},$ and $\Cone(y\to z)\in \cC_{[f,+\infty)}.$ Take $c\geq 0$ such that $x_i,y_i\in\Stab(\cA_i)_{[-c,f(i)+c]},$ $z_i\in\Stab(\cA_i)_{[-c,+\infty)},$ and $\Cone(y_i\to z_i)\in\Stab(\cA_i)_{[f(i)-c,+\infty)}.$ Choose $w_i\in\Stab(\cA_i)_{[-c,f(i)+c]}$ and the maps $w_i\to z_i,$ $i\in I,$ such that $\Cone(w_i\to z_i)\in\Stab(\cA_i)_{[f(i)+c+1,+\infty)}.$ Then we can choose factorizations $x_i\to w_i\to z_i$ and $y_i\to w_i\to z_i,$ and automatically $\Cone(y_i\to w_i)\in \Stab(\cA_i)_{[f(i)-c,f(i)+c+1]}.$

Putting $w=(w_i)_i,$ we have $w\in\cC_{[0,f]},$ $\Cone(y\to w)\in\cC_{[f,f]},$ and $\Cone(w\to z)\in \cC_{[f,+\infty)}.$ Therefore, the roof $x\to w\leftto y$ represents a morphism $g:\bar{x}\to\bar{y}$ such that $\Phi(g)=f.$

{\noindent{\bf Faithfulness of $\h\Phi.$}} Consider a map $f:\bar{x}\to\bar{y}$ such that $\Phi(f)=0.$ We may assume that $f$ can be lifted to a morphism $\wt{f}:x\to y.$ Then there exists a map $g:y\to z,$ $z=(z_i)_i\in\cC_{[0,+\infty)},$ such that $\Cone(g)\in \cC_{[f,+\infty)}$ and $g\wt{f}=0.$ Arguing as above, we can construct a factorization $y\xto{h} w\xto{s} z$ of $g,$ such that $w\in\cC_{[0,f]},$ $\Cone(h)\in\cC_{[f,f]},$ and $\h\wt{f}=0.$ Hence $f=0.$ 
\end{proof}

\begin{proof}[Proof of Theorem \ref{th:Cordova_U_loc}] By Lemma \ref{lem:product_of_prestab}, it is sufficient to prove that the map $\cU_{\loc}(\cC_{[0,0]})\to \cU_{\loc}(\cC_{[0,+\infty)})$ is an isomorphism. Applying again (AB6) in $\Cat^{\perf},$ we obtain
$$\cU_{\loc}(\cC_{[0,+\infty)})\cong \cU_{\loc}(\indlim[f:I\to\N]\cC_{[0,2f]})\cong \indlim[f:I\to\N] \cU_{\loc}(\cC_{[0,2f]}).$$
Hence, it is sufficient to prove that for any $f:I\to\N$ the map $\cU_{\loc}(\cC_{[0,0]})\to \cU_{\loc}(\cC_{[0,2f]})$ is an isomorphism. Combining Lemmas \ref{lem:functors_with_same_class} and \ref{lem:equivalence_of_quotients}, we obtain equivalences
\begin{multline*}\Cone(\cU_{\loc}(\cC_{[0,0]})\to \cU_{\loc}(\cC_{[0,2f]}))\cong
 \Cone(\cU_{\loc}(\cC_{[2f,2f]})\to \cU_{\loc}(\cC_{[0,2f]}))\\ \cong \cU_{\loc}(\cC_{[0,2f]}/\cC_{[2f,2f]})\cong \cU_{\loc}(\cC_{[0,+\infty)}/\cC_{[2f,+\infty)})\\ \cong
  \Cone(\cU_{\loc}(\cC_{[2f,+\infty)})\to \cU_{\loc}(\cC_{[0,+\infty)}))\cong 0.\end{multline*}
  This proves the theorem.\end{proof}
  
\subsection{Products of stable categories with bounded $t$-structures}
\label{ssec:products_bounded_t_structures}

We will need the following analogue of Theorem \ref{th:Cordova_U_loc}. Let $\{\cC_i\}_{i\in I}$ be a family of small stable categories with bounded $t$-structures. Denote by $(\prodd[i]\cC_i)^{\bnd}\subset \prodd[i]\cC_i$ the full subcategory formed by objects $x=(x_i)_i$ such that the objects $x_i$ have uniformly bounded homological amplitude. 

\begin{theo}\label{th:U_loc_products_bounded_t-structures} We have an isomorphism $$\cU_{\loc}((\prodd[i]\cC_i)^{\bnd})\to \cU_{\loc}(\prodd[i]\cC_i).$$\end{theo}

\begin{proof}The proof is completely analogous to the proof of Theorem \ref{th:Cordova_U_loc} with minor changes, which we indicate.

For a stable category $\cB$ with a bounded $t$-structure and for integers $a\leq b$ we denote by $\cB_{[a,b]}\subset\cB$ the full subcategory formed by objects whose homology is concentrated in degrees between $a$ and $b,$ and similarly for $\cB_{(-\infty,b]},$ $\cB_{[a,+\infty)}.$

Put $\cD:=\prod\limits_i\cC_i.$ Given functions $f,g:I\to\Z,$ $f\leq g,$ we define the subcategories $\cD_{[f,g]},$ $\cD_{(-\infty,g]},$ $\cD_{[f,+\infty)}$ analogously to \eqref{eq:C_f_g}-\eqref{eq:C_f_+infty}.

Analogously to Lemma \ref{lem:product_of_prestab}, one shows that we have an isomorphism
$\cU_{\loc}(\cD_{(-\infty,0]})\cong \cU_{\loc}(\cD).$

Next, one proves that for any function $f:I\to\N$ the functor $\Psi:\cD_{[-f,0]}/\cD_{[-f,-f]}\to \cD_{(-\infty,0]}/\cD_{(-\infty,-f]}$ is an equivalence. The proof is analogous to Lemma \ref{lem:equivalence_of_quotients}, but simpler since in the case of $t$-structures we have (canonical) truncations. More precisely, for an object $x = (x_i)_i\in \cD_{(-\infty,0]}$ we can take $y = (\tau_{\geq -f(i)}x_i)_i\in \cD_{[-f,0]},$ and the natural map $y\to x$ becomes an isomorphism in $\cD_{(-\infty,0]}/\cD_{(-\infty,f]}.$ This gives the essential surjectivity of $\h \Psi.$ Next, let $x\to z\leftto y$ be a roof, where $x=(x_i)_i,y=(y_i)_i\in\cD_{[-f,0]},$ $z\in\cD_{(-\infty,0]}$ and $\Cone(y\to z)\in \cD_{(-\infty,-f]}.$ Choosing $c\geq 0$ such that $x_i,y_i\in (\cC_i)_{\geq -f(i)-c}$ for all $i,$ we put $w=(\tau_{\geq -f(i)-c}z_i)_i.$ We have (unique) factorizations $x\to w\to z,$ $y\to w\to z,$ and $w\in\cD_{[-f,0]},$ $\Cone(w\to z)\in \cD_{(-\infty,f]},$ $\Cone(y\to w)\in\cD_{[-f,-f]}.$ This gives the fullness of $\h\Psi.$ The faithfulness of $\h\Psi$ is proved similarly: again, replace the components $z_i$ of a relevant object of $\cC_{(-\infty,0]}$ with truncations $\tau_{\geq -f(i)-c}z_i$ for suitable $c\geq 0.$.   

Then one concludes that $\cU_{\loc}(\cD_{[0,0]})\cong \cU_{\loc}(\cD_{(-\infty,0]}),$ proving the theorem.\end{proof}

\subsection{Products of exact $\infty$-categories}
\label{ssec:products_of_exact_cats}

Recall the notion of an exact $\infty$-category. 

\begin{defi}\cite[Definition 3.1]{Bar15} An exact $\infty$-category is a triple $(\cE,\cE_\dagger,\cE^\dagger),$ where $\cE$ is an additive $\infty$-category and $\cE_\dagger,\cE^\dagger\subset \cE$ are subcategories such that the following conditions hold.

1) Every morphism $0\to x$ (resp. $x\to 0$) is in $\cE_\dagger$ (resp. $\cE^\dagger$).

2) Pushouts of morphisms in $\cE_\dagger$ exist, and the class of morphisms in $\cE_\dagger$ is closed under pushouts. Pullbacks of morphisms in $\cE^\dagger$ exist and the class of morphisms in $\cE^\dagger$ is closed under pullbacks.

3) For a commutative square 
$$
\begin{tikzcd}
x \ar{r}{f} \ar{d}{g} & y \ar{d}{g'}\\
x' \ar{r}{f'} & y'
\end{tikzcd}
$$
the following are equivalent:

a) the square is cocartesian, $f\in\cE_{\dagger}$ and $g\in\cE^\dagger;$

b) the square is cartesian, $f'\in\cE_\dagger$ and $g'\in\cE^\dagger.$\end{defi}

By definition, a short exact sequence is a bicartesian square of the form
$$
\begin{tikzcd}
x \ar{r}{f} \ar{d} & y \ar{d}{g}\\
0 \ar{r} & z\\
\end{tikzcd}
$$
where $f\in\cE_\dagger$ and  $g\in\cE^\dagger.$ A total cofiber of such a square in $\Stab(\cA)$ is by definition a {\it primitive} acyclic object. The full subcategory of acyclic objects $\Stab(\cE)^{\acycl}\subset \Stab(\cE)$ is the stable subcategory generated by the primitive acyclic objects. The stable hull of $\cE$ is the category $\cH^{\st}(\cE):=\Stab(\cE)/\Stab(\cE)^{\acycl};$ we follow the notation of \cite{Kle}.

\begin{prop}\label{prop:t-structure_on_acyclic} For an exact $\infty$-category $\cE,$ the category $\Stab(\cE)^{\acycl}$ has a bounded $t$-structure, and the functor $\Stab(\cE)^{\acycl}\to\Ind(\Stab(\cE))$ is $t$-exact.\end{prop}

\begin{proof}Recall that for an abelian category $\cA$ a full subcategory $\cB\subset\cA$ is called a weak Serre subcategory if $\cB$ is non-empty and is closed under taking kernels, cokernels and extensions in $\cA.$ Note that if $T$ is a triangulated category with a $t$-structure, and $\cB\subset T^{\heartsuit}$ is a weak Serre subcategory of the heart, then the full triangulated subcategory $S\subset T$ generated by $\cB$ has a bounded $t$-structure with the heart $\cB.$ In particular, the functor $S\to T$ is $t$-exact.

Consider the category $\Ind(\Stab(\cE))$ with its standard $t$-structure. By \cite[Proposition 3.10]{Kle}, the primitive acyclic objects are contained in the heart $\Ind(\Stab(\cE))^{\heartsuit}.$ Moreover, by \cite[Lemma 3.11 and Lemma 3.12]{Kle} they form a weak Serre subcategory.  The proposition follows.\end{proof}

\begin{theo}\label{th:U_loc_products_exact_cats} Let $\{\cE_i\}_{i\in I}$ be a family of small idempotent-complete exact $\infty$-categories. We have an isomorphism
$$\cU_{\loc}(\cH^{\st}(\prodd[i]\cE_i))\cong \cU_{\loc}(\prod\limits_{i}\cH^{\st}(\cE_i)).$$\end{theo} 

\begin{cor}\label{cor:K_theory_of_exact_cats_comm_with_products} $K$-theory of exact $\infty$-categories commutes with infinite products.\end{cor}

\begin{proof}Indeed, by Theorem \ref{th:U_loc_products_exact_cats} and Theorem \ref{th:map_of_localizing_invariants} we have 
\begin{equation*}K(\prod\limits_i\cE_i)\cong K(\cH^{\st}(\prod\limits_i\cE_i))\cong K(\prod\limits_i\cH^{\st}(\cE_i))\cong \prod\limits_i K(\cH^{\st}(\cE_i))\cong \prod\limits_i K(\cE_i).\qedhere\end{equation*}\end{proof}

\begin{proof}[Proof of Theorem \ref{th:U_loc_products_exact_cats}.] We have a map of short exact sequences of stable categories:
$$
\begin{tikzcd}
0 \ar{r} & \Stab(\prod\limits_i\cE_i)^{\acycl} \ar{r}\ar{d} & \Stab(\prod\limits_i\cE_i) \ar{r}\ar{d} & \cH^{\st}(\prod\limits_i\cE_i) \ar{r}\ar{d} & 0\\
0 \ar{r} & \prod\limits_i\Stab(\cE_i)^{\acycl} \ar{r} & \prod\limits_i\Stab(\cE_i) \ar{r} & \prod\limits_i\cH^{\st}(\cE_i) \ar{r} & 0.
\end{tikzcd}$$

By Theorem \ref{th:Cordova_U_loc}, $\cU_{\loc}$ applied to the middle vertical arrow is an isomorphism. By Theorem \ref{th:U_loc_products_bounded_t-structures} applied to the bounded $t$-structures on $\Stab(\cE_i)^{\acycl}$ from Proposition \ref{prop:t-structure_on_acyclic}, $\cU_{\loc}$ applied to the left vertical arrow is an isomorphism. Hence, $\cU_{\loc}$ applied to the right vertical arrow is an isomorphism. This proves the theorem.\end{proof}

\end{document}